\documentclass[twoside, reqno]{amsart}
\usepackage[english]{babel}
\usepackage[utf8]{inputenc}
\usepackage[
backend=biber,
style=alphabetic, 
giveninits=true 
]{biblatex}

    \usepackage{amsthm}
    \usepackage{thmtools}
     \usepackage{amsfonts}
    \usepackage{mathtools}
    \usepackage{amssymb}
    \usepackage{amsmath}
    \usepackage{tikz-cd}
    \usepackage[utf8]{inputenc}
    \usepackage{csquotes}
    \usepackage[headheight=14pt]{geometry} 
    \usepackage{appendix}
    \usepackage{url}
    \usepackage{hyperref}
    \usepackage{graphicx}
    \usepackage{xcolor}
    \usepackage{tikz}
    \usepackage{nameref}
    \usepackage{fancyhdr} 
    \usepackage{fancybox}
    \usepackage{enumitem} 
    \usepackage{trimclip}
    \usepackage{epigraph}
    \usepackage{faktor}
    \usepackage{ stmaryrd }
    \usepackage{comment}
    \usepackage{scalerel,stackengine}
    \usepackage{lscape}
    \usepackage{multirow}
    \usepackage{mathrsfs}
    \usepackage{calc}
    \usepackage{relsize}
    \usepackage[bbgreekl]{mathbbol}
\theoremstyle{plain}
\newtheorem{thm}{Theorem}[section]
\newtheorem{lemma}[thm]{Lemma}
\newtheorem{proposition}[thm]{Proposition}
\newtheorem{cor}[thm]{Corollary}

\newtheorem{notn}[thm]{Notation}

\newtheorem{thm solo}{Theorem}
\newtheorem{prop solo}[thm solo]{Proposition}
\newtheorem{cor solo}[thm solo]{Corollary}
\newtheorem{conj solo}[thm solo]{Conjecture}

\theoremstyle{definition}
\newtheorem{definition}[thm]{Definition}
\newtheorem{example}[thm]{Example}

\theoremstyle{remark}
\newtheorem{remark}[thm]{Remark}

\numberwithin{equation}{section}

\DeclareSymbolFontAlphabet{\mathbb}{AMSb}
\DeclareSymbolFontAlphabet{\mathbbl}{bbold}
\newcommand{\Prism}{{\mathlarger{\mathbbl{\Delta}}}}

\DeclareMathOperator{\ga}{\mathfrak{g}}

\DeclareMathOperator{\ma}{\mathfrak{m}}

\DeclareMathOperator{\ip}{\mathfrak{p}}

\DeclareMathOperator{\Xs}{\mathfrak{X}}

\DeclareMathOperator{\Ys}{\mathfrak{Y}}

\DeclareMathOperator{\Stab}{Stab}

\DeclareMathOperator{\Map}{Map}

\DeclareMathOperator{\tworlim}{2-\varinjlim}

\DeclareMathOperator{\id}{id}
\DeclareMathOperator{\Hom}{Hom}
\DeclareMathOperator{\Ext}{Ext}
\DeclareMathOperator{\RHom}{RHom}

\DeclareMathOperator{\Bun}{Bun}

\DeclareMathOperator{\HT}{HT}
\DeclareMathOperator{\GM}{GM}

\DeclareMathOperator{\BL}{BL}

\DeclareMathOperator{\Ker}{Ker}

\DeclareMathOperator{\Coker}{Coker}

\DeclareMathOperator{\Ab}{Ab}

\DeclareMathOperator{\Grp}{Grp}

\DeclareMathOperator{\Sht}{Sht}
\DeclareMathOperator{\Sch}{Sch}

\DeclareMathOperator{\Spec}{Spec}
\DeclareMathOperator{\Spf}{Spf}
\DeclareMathOperator{\Spa}{Spa}
\DeclareMathOperator{\Spd}{Spd}

\DeclareMathOperator{\Proj}{Proj}

\DeclareMathOperator{\Gr}{Gr}

\DeclareMathOperator{\GL}{GL}

\DeclareMathOperator{\Aut}{Aut}

\DeclareMathOperator{\Pic}{Pic}
\DeclareMathOperator{\BT}{BT}

\DeclareMathOperator{\Nil}{Nil}
\DeclareMathOperator{\Nilp}{Nilp}

\DeclareMathOperator{\adm}{adm}
\DeclareMathOperator{\Lie}{Lie}
\DeclareMathOperator{\Image}{Im}

\DeclareMathOperator{\End}{End}

\DeclareMathOperator{\QHom}{QHom}

\DeclareMathOperator{\Adic}{Adic}

\DeclareMathOperator{\Gal}{Gal}

\DeclareMathOperator{\ev}{ev}

\DeclareMathOperator{\rank}{rank}
\DeclareMathOperator{\height}{ht}

\DeclareMathOperator{\cts}{cts}

\DeclareMathOperator{\fppf}{fppf}

\DeclareMathOperator{\crys}{crys}

\DeclareMathOperator{\cyc}{cyc}
\DeclareMathOperator{\an}{an}

\DeclareMathOperator{\Loc}{Loc}

\DeclareMathOperator{\et}{\acute{e}t}
\DeclareMathOperator{\Et}{\acute{E}t}

\DeclareMathOperator{\proet}{pro\acute{e}t}

\DeclareMathOperator{\Zar}{Zar}

\DeclareMathOperator{\gr}{gr}

\DeclareMathOperator{\rig}{rig}
\DeclareMathOperator{\alg}{alg}

\DeclareMathOperator{\FF}{FF}
\DeclareMathOperator{\Perf}{Perf}
\DeclareMathOperator{\perf}{perf}
\DeclareMathOperator{\LSD}{LSD}

\DeclareMathOperator{\Sm}{Sm}
\DeclareMathOperator{\qcqs}{qcqs}

\DeclareMathOperator{\univ}{univ}
\DeclareMathOperator{\Isom}{Isom}
\DeclareMathOperator{\ttt}{tt}
\DeclareMathOperator{\ptt}{tpt}
\DeclareMathOperator{\tor}{tor}
\DeclareMathOperator{\tors}{tors}
\DeclareMathOperator{\torsion}{torsion}

\DeclareMathOperator{\qlog}{qlog}

\DeclareMathOperator{\Ainf}{\mathrm{A_{inf}}}

\DeclareMathOperator{\Acrys}{\mathrm{A_{crys}}}

\DeclareMathOperator{\Bcr+}{\mathrm{B_{crys}^+}}

\DeclareMathOperator{\BdeR}{\mathrm{B_{dR}}}

\DeclareMathOperator{\BdR+}{\mathrm{B_{dR}^+}}
\DeclareMathOperator{\OBdeR}{\mathcal{O}\mathbb{B}_{dR}}
\DeclareMathOperator{\OBdR+}{\mathcal{O}\mathbb{B}_{dR}^+}
\DeclareMathOperator{\OBHT}{\mathcal{O}\mathbb{B}_{HT}}
\DeclareMathOperator{\OC}{\mathcal{O}\mathbb{C}}
\newcommand{\cj}[1]{\overline{#1}}
\DeclareMathOperator{\RZ}{RZ}
\DeclareMathOperator{\Hdg}{Hdg}

\DeclareMathOperator{\kfl}{kfl}
\DeclareMathOperator{\Vect}{Vect}
\DeclareMathOperator{\DM}{DM}
\DeclareMathOperator{\qsyn}{qsyn}

\stackMath
\newcommand\widecheck[1]{%
\savestack{\tmpbox}{\stretchto{%
  \scaleto{%
    \scalerel*[\widthof{\ensuremath{#1}}]{\kern-.6pt\bigwedge\kern-.6pt}%
    {\rule[-\textheight/2]{1ex}{\textheight}}%WIDTH-LIMITED BIG WEDGE
  }{\textheight}% 
}{0.5ex}}%
\stackon[1pt]{#1}{\scalebox{-1}{\tmpbox}}%
}
\newcommand{\G}{\mathbb{G}}

\newcommand{\A}{\mathbb{A}}
\newcommand{\B}{\mathbb{B}}
\newcommand{\C}{\mathbb{C}}

\newcommand{\Pro}{\mathbb{P}}

\newcommand{\R}{\mathbb{R}}
\newcommand{\Z}{\mathbb{Z}}
\newcommand{\Q}{\mathbb{Q}}
\newcommand{\F}{\mathbb{F}}

\newcommand{\Lb}{\mathbb{L}}
\newcommand{\Mb}{\mathbb{M}}

\newcommand{\Sph}{\mathbb{S}}

\newcommand{\dR}{\mathrm{dR}}

\newcommand{\V}{\mathbb{V}}

\newcommand{\Es}{\mathcal{E}}

\newcommand{\As}{\mathcal{A}}
\newcommand{\Bs}{\mathcal{B}}
\newcommand{\Os}{\mathcal{O}}
\newcommand{\Fs}{\mathcal{F}}
\newcommand{\Gs}{\mathcal{G}}
\newcommand{\Hs}{\mathcal{H}}
\newcommand{\Ls}{\mathcal{L}}
\newcommand{\Ms}{\mathcal{M}}
\newcommand{\Ns}{\mathcal{N}}

\newcommand{\Us}{\mathcal{U}}

\newcommand{\Is}{\mathcal{I}}

\newcommand{\Ss}{\mathcal{S}}
\newcommand{\Yss}{\mathcal{Y}}

\newcommand{\BC}{\mathcal{BC}}
\newcommand{\Fl}{\mathcal{F}l}

\DeclareMathOperator{\PPic}{\mathbf{Pic}}

\DeclareMathOperator{\Fil}{Fil}

\newcommand{\sub}{\subseteq}
\newcommand{\fa }{\forall}

\newcommand{\restr}[1]{|_{#1}}

\DeclareFieldFormat[article,inbook,incollection,inproceedings,patent,thesis,unpublished, misc]{title}{#1\isdot}
\DeclareFieldFormat{pages}{#1}
\renewbibmacro*{volume+number+eid}{%
  \printfield{volume}
  \setunit*{\addnbspace}
  \printfield{number}
  \setunit{\addcomma\space}
  \printfield{eid}
}
\DeclareFieldFormat[article]{number}{\mkbibparens{#1}}
\renewbibmacro*{issue+date}{}% Remove date
\renewbibmacro*{note+pages}{% Add date after pages
  \printfield{note}
  \setunit{\bibpagespunct}
  \printfield{pages}
  \setunit{\addcomma\addspace}
  \usebibmacro{date}
  \newunit}

\NewBibliographyString{toappearin}
\NewBibliographyString{submittedto}
\DefineBibliographyStrings{english}{%
  toappearin  = {to appear in},
  submittedto = {submitted to},
}
\renewbibmacro{in:}{
\ifentrytype{article}
    {\addperiod
    \addspace
    \printfield{pubstate}
     \clearfield{pubstate}
  \printunit{\addspace}}
{\bibstring{in}\printunit{\addspace}}
}

 \title{A Dieudonné theory for analytic $p$-divisible groups and applications to Shimura varieties}
 \author{Lucas Gerth}
 \date{} 

\newcommand{\RNum}[1]{\uppercase\expandafter{\romannumeral #1\relax}}

\begin{document}

\maketitle
\setlength{\parskip}{0pt}
\setlength{\parindent}{15pt}
\setlist{topsep=5pt, leftmargin=*}

\begin{abstract}
We study families of analytic $p$-divisible groups over adic spaces $S$ defined over $\Q_p$. We prove an equivalence between such families and Hodge--Tate triples, generalizing a theorem of Fargues. For a perfectoid space $S$, we construct a functor associating to an analytic $p$-divisible group $\Gs \rightarrow S$ a coherent sheaf $\Es(\Gs)$ on the relative Fargues--Fontaine curve $X_S$. Restricting to analytic $p$-divisible groups admitting a Cartier dual, we obtain an equivalence of categories with local shtukas satisfying a minuscule condition, compatible with the prismatic Dieudonné theory of Anschütz--Le Bras. We conclude with applications to moduli spaces: we show that the local Shimura varieties of EL and PEL type of Scholze--Weinstein are moduli spaces of analytic $p$-divisible groups with extra structure, and we give a reinterpretation of the Hodge--Tate period map of Scholze in terms of topologically $p$-torsion subgroups of abelian varieties.
\end{abstract}

%\tableofcontents

\section{Introduction}
Let $p$ be a prime number. In their book, Rapoport and Zink \cite{rapoportzink96} constructed the so-called Rapoport--Zink spaces. These are $p$-adic formal schemes, classifying deformations of a fixed $p$-divisible group with extra structure. Using Huber’s theory of analytic adic spaces \cite{huber2013étale}, one can consider their generic fibers. Inspired by earlier work of Drinfeld, Rapoport--Zink envisioned certain towers of coverings of these generic fibers as local analogues of Shimura varieties, with their étale cohomology expected to reflect parts of the local Langlands correspondence. Scholze–Weinstein later formulated in \cite{SW20} a very general definition of local Shimura varieties as moduli spaces of $p$-adic shtukas, encompassing the earlier Rapoport–Zink examples. This formed the geometric framework in which the expected connections with the local Langlands correspondence were fully realized, notably in the work of Fargues–Scholze \cite{fargues2024geometrization}. 

In the global theory, a large class of Shimura varieties, the so-called varieties of PEL type, admit a description as moduli spaces of abelian varieties with extra structure. One might be tempted to ask for an analogous description of certain local Shimura varieties as parametrizing geometric objects. As Rapoport–Zink spaces are moduli spaces of $p$-divisible groups, it is natural to expect that their generic fibers should themselves be described as parametrizing suitable “generic fiber” objects attached to $p$-divisible groups. A target category for such a generic fiber functor was provided by Fargues in \cite{Farg19} with his notion of analytic $p$-divisible groups.

In this paper, we lay down the foundations for families of analytic $p$-divisible groups over adic spaces $S$ over $\Q_p$. We show that such groups are equivalent to Hodge--Tate triples over $S$, see Theorem \ref{Intro thm: extending Fargues' equivalence of categories}. Moreover, we define a Dieudonné functor with values in minuscule local shtukas, see Theorem \ref{intro thm: analytic Dieudonné theory in general} below, mirroring Dieudonné theory for classical $p$-divisible groups over (formal) schemes. We then use the above results to give a moduli-theoretic description of the local Shimura varieties of EL and PEL type in terms of analytic $p$-divisible groups with extra structure, see Theorem \ref{intro thm: Rapoport-Zink space in terms of analytic p-div groups}. At last, we show that the Hodge--Tate period map of Scholze \cite{Scholze2015torsion} on global Shimura varieties of PEL type can be reinterpreted as sending an abelian variety with extra structure to its topologically $p$-torsion subgroup, see Theorem \ref{intro thm: HT map reinterpreted}.

\subsection{Analytic $p$-divisible groups}
The main objects of study of this paper are relative analytic $p$-divisible groups $\Gs \rightarrow S$, where $S$ is an adic space over $\Q_p$. Following \cite[§2.1]{Farg19}, these are smooth families of commutative adic groups for which $[p]\colon \Gs \rightarrow \Gs$ is finite étale surjective and such that $\Gs$ is of topologically $p$-torsion. This means that we have an equality of $v$-sheaves $\Gs = \Gs\langle p^{\infty} \rangle$, where we define, following Heuer \cite[§2]{heuer2022geometric}
\begin{align} \Gs\langle p^{\infty} \rangle =  \Image(\underline{\Hom}(\underline{\Z_p},\Gs) \xrightarrow{\ev_1} \Gs).\end{align}
Important examples of analytic $p$-divisible groups are adic generic fibers $\mathfrak{G}_{\eta}$ of $p$-divisible groups over formal models $\Ss$ of $S$. More generally, we define a generic fiber functor on log $p$-divisible groups in the sense of Kato \cite{kato2023logarithmicdieudonnetheory}, see Proposition \ref{prop: generic fibers of log p-divisible groups}. Another interesting class of examples are the topologically $p$-torsion subgroups $A\langle p^{\infty} \rangle$ of abelian varieties $A$, viewed as adic spaces. More generally, we allow families of abeloid varieties, the rigid analytic analogue of complex tori \cite{lutkebohmert2016}.

As we extensively rely on the technique of reduction to the perfectoid case, we restrict our attention to a subcategory of adic spaces $S$ over $\Q_p$, the so-called good adic spaces (see Definition \ref{def: good adic spaces}), on which the structure sheaf is a $v$-sheaf. For example, $S$ could be a seminormal rigid space or a perfectoid space. There is also a meaningful definition of analytic $p$-divisible groups over arbitrary locally spatial diamonds, see Remark \ref{rmk: p-divisible groups over LSD}.

Such an analytic $p$-divisible group $\Gs$ is endowed with a functorial logarithm map $\log_{\Gs} \colon \Gs \rightarrow \Lie(\Gs)$, an étale surjective homomorphism of adic groups with kernel $\Gs[p^{\infty}]$. Our first theorem is the following equivalence, crucial for all our other applications.
\begin{thm}{(Theorem \ref{thm: extending Fargues' equivalence of categories})}\label{Intro thm: extending Fargues' equivalence of categories}
    Let $S$ be a good adic space over $\Q_p$. Then the following categories are canonically equivalent
    \begin{enumerate}
        \item The category of analytic $p$-divisible groups $\Gs \rightarrow S$, and
        \item The category of tuples $(\Lb,E,f)$ where $\Lb$ is a $\Z_p$-local system on $S_v$, $E$ is a vector bundle on $S_{\et}$ and $f\colon E\otimes_{\Os_{S}} \Os_{S_v} \rightarrow \Lb(-1)\otimes_{\Z_p} \Os_{S_v}$ is a morphism of $v$-vector bundles.
\end{enumerate}
If $\Gs$ and $(\Lb,E,f)$ correspond to each-other, we have $\Lb = T_p\Gs$, $E=\Lie(\Gs)$ and the map $f=f_{\Gs}$ fits in the following commutative diagram of group diamonds
% https://q.uiver.app/#q=WzAsMTAsWzIsMCwiXFxHcyJdLFszLDAsIlxcTGllKFxcR3MpXFxvdGltZXNfe1xcT3NfU31cXE9zX3tTX3Z9Il0sWzMsMSwiVF9wXFxHcygtMSlcXG90aW1lc197XFxaX3B9XFxPc197U192fSJdLFsyLDEsIlRfcFxcR3MoLTEpXFxvdGltZXNfe1xcWl9wfVxcR19tIFxcbGFuZ2xlIHBee1xcaW5mdHl9IFxccmFuZ2xlIl0sWzEsMCwiXFxHc1twXntcXGluZnR5fV0iXSxbMCwwLCIwIl0sWzQsMCwiMCJdLFsxLDEsIlRfcFxcR3MoLTEpXFxvdGltZXNfe1xcWl9wfVxcbXVfe3Bee1xcaW5mdHl9fSJdLFswLDEsIjAiXSxbNCwxLCIwLiJdLFswLDEsIlxcbG9nX3tcXEdzfSJdLFsxLDIsImZfe1xcR3N9Il0sWzMsMiwiXFxpZCBcXG90aW1lcyBcXGxvZyIsMl0sWzAsM10sWzQsMF0sWzUsNF0sWzEsNl0sWzQsNywiPSJdLFs4LDddLFs3LDNdLFsyLDldXQ==
\begin{equation}\label{intro eq: diagram for f}
\begin{tikzcd}
	0 & {\Gs[p^{\infty}]} & \Gs & {\Lie(\Gs)\otimes_{\Os_S}\!\Os_{S_v}} & 0 \\
	0 & {T_p\Gs(-1)\otimes_{\Z_p}\!\mu_{p^{\infty}}} & {T_p\Gs(-1)\otimes_{\Z_p}\!\G_m \langle p^{\infty} \rangle} & {T_p\Gs(-1)\otimes_{\Z_p}\!\Os_{S_v}} & {0.}
	\arrow[from=1-1, to=1-2]
	\arrow[from=1-2, to=1-3]
	\arrow["{=}", from=1-2, to=2-2]
	\arrow["{\log_{\Gs}}", from=1-3, to=1-4]
	\arrow[from=1-3, to=2-3]
	\arrow[from=1-4, to=1-5]
	\arrow["{f_{\Gs}}", from=1-4, to=2-4]
	\arrow[from=2-1, to=2-2]
	\arrow[from=2-2, to=2-3]
	\arrow["{\id \otimes \log}"', from=2-3, to=2-4]
	\arrow[from=2-4, to=2-5]
\end{tikzcd}
\end{equation}
\end{thm}
When $S=\Spa(K)$ for a non-archimedean field $K$ with completed algebraic closure $C$, the second category is equivalent to that of tuplets $(\Lb,E,f)$ where $\Lb$ is a continuous finite free $\Z_p$-linear representation of $\Gal_K$, $E$ is a finite-dimensional $K$-vector space and $f\colon \Lb\otimes_{\Z_p} C \rightarrow E \otimes_K C$ is a $C$-linear Galois-equivariant map. The above result then recovers a theorem of Fargues \cite[Thm. 0.1]{Farg19}.

We briefly describe the proof of Theorem \ref{Intro thm: extending Fargues' equivalence of categories}. The main difficulty, starting with an analytic $p$-divisible group $\Gs$, resides in constructing the map $f_{\Gs}$. Our strategy is first to construct the middle vertical map $\Gs \rightarrow T_p\Gs(-1) \otimes_{\Z_p} \G_m\langle p^{\infty} \rangle$ in the diagram (\ref{intro eq: diagram for f}), and to obtain the map $f_{\Gs}$ as its derivative at the identity. This is possible thanks to the following result, which is of independant interest.
\begin{thm}{(Theorem \ref{thm: existence of analytic Weil pairings})}
    Let $S$ be a good adic space over $\Q_p$. Let $\Gs \rightarrow S$ be an analytic $p$-divisible group and consider its Tate module $T_p\Gs$, a $\Z_p$-local system on $S_v$. Then the natural Weil pairing $e$ between $\Gs[p^{\infty}]$ and $T_p\Gs^{\vee}(1)$ uniquely extends to a pairing $\widehat{e}$ of group diamonds, as in the following diagram
    % https://q.uiver.app/#q=WzAsNCxbMCwwLCJcXEdzW3Bee1xcaW5mdHl9XSBcXHRpbWVzIFRfcFxcR3Nee1xcdmVlfSgxKSJdLFsxLDAsIlxcbXVfe3Bee1xcaW5mdHl9fSJdLFsxLDEsIlxcR19tXFxsYW5nbGUgcF57XFxpbmZ0eX1cXHJhbmdsZS4iXSxbMCwxLCJcXEdzIFxcdGltZXMgVF9wXFxHc157XFx2ZWV9KDEpIl0sWzAsMSwiZSJdLFsxLDIsIiIsMCx7InN0eWxlIjp7InRhaWwiOnsibmFtZSI6Imhvb2siLCJzaWRlIjoidG9wIn19fV0sWzAsMywiIiwyLHsic3R5bGUiOnsidGFpbCI6eyJuYW1lIjoiaG9vayIsInNpZGUiOiJ0b3AifX19XSxbMywyLCJcXHdpZGVoYXR7ZX0iLDIseyJzdHlsZSI6eyJib2R5Ijp7Im5hbWUiOiJkYXNoZWQifX19XV0=
\[\begin{tikzcd}
	{\Gs[p^{\infty}] \times T_p\Gs^{\vee}(1)} & {\mu_{p^{\infty}}} \\
	{\Gs \times T_p\Gs^{\vee}(1)} & {\G_m\langle p^{\infty}\rangle.}
	\arrow["e", from=1-1, to=1-2]
	\arrow[hook, from=1-1, to=2-1]
	\arrow[hook, from=1-2, to=2-2]
	\arrow["{\widehat{e}}"', dashed, from=2-1, to=2-2]
\end{tikzcd}\]
\end{thm}
For groups of good reduction, such pairings were already considered by Tate in his original article \cite{Tat67}. A related construction also appears in the work of Deninger--Werner \cite{DeningerWerner2005}\cite{Deninger2005} in the context of the $p$-adic Simpson correspondence. For an algebraic curve $X$ over $\cj{\Q}_p$, they associate a continuous $\C_p$-linear representation of the étale fundamental group $\pi_1(X,x)$ to certain vector bundles on $X_{\C_p}$. When specified to the rank $1$ objects, this yields a group homomorphism
\begin{align}\label{intro eq: morphism of Deninger--Werner}
\alpha\colon \Pic_X^0(\C_p) \rightarrow \Hom_{\cts}(\pi_1(X,x),\C_p^{\times}).\end{align}
It was shown by Heuer \cite{heuer2022geometric} that this morphism can be geometrized.

Theorem \ref{Intro thm: extending Fargues' equivalence of categories} can be used to single out the analytic $p$-divisible groups $\Gs \rightarrow S$ that admit a Cartier dual. These are the groups whose associated map $f_{\Gs}$ extends to a short exact sequence of $v$-vector bundles
% https://q.uiver.app/#q=WzAsNSxbMSwwLCJcXExpZShcXEdzKVxcb3RpbWVzX3tcXE9zX1N9XFxPc197U192fSJdLFsyLDAsIlRfcFxcR3MoLTEpXFxvdGltZXNfe1xcWl9wfVxcT3Nfe1Nfdn0iXSxbMCwwLCIwIl0sWzMsMCwiXFxvbWVnYSBcXG90aW1lc197XFxPc19TfVxcT3Nfe1Nfdn0oLTEpIl0sWzQsMCwiMCwiXSxbMCwxLCJmIl0sWzIsMF0sWzEsM10sWzMsNF1d
\begin{equation}\begin{tikzcd}
	0 & {\Lie(\Gs)\otimes_{\Os_S}\Os_{S_v}} & {T_p\Gs(-1)\otimes_{\Z_p}\Os_{S_v}} & {\omega \otimes_{\Os_S}\Os_{S_v}(-1)} & {0,}
	\arrow[from=1-1, to=1-2]
	\arrow["{f_{\Gs}}", from=1-2, to=1-3]
	\arrow[from=1-3, to=1-4]
	\arrow[from=1-4, to=1-5]
\end{tikzcd}\end{equation}
for a vector bundle $\omega$ on $S_{\et}$. In that case, there is a natural Cartier dual $\Gs^D$, compatible with the usual Cartier duality on $p$-divisible groups upon taking adic generic fibers (Proposition \ref{prop: groups of good reduction are dualizable}).

With our terminology, the celebrated result of Scholze--Weinstein \cite[Thm. B]{scholze2013moduli} states that, for a complete algebraically closed field $C/\Q_p$, the adic generic fiber functor $(\cdot)_{\eta}$ realises an equivalence
\begin{align*}
    \left\{\text{\begin{tabular}{l} {\parbox{4.0cm}{$p$-divisible groups over $\Os_C$}}\end{tabular}}\right\}
         \xrightarrow{\cong} \left\{\text{\begin{tabular}{l} {\parbox{4.5cm}{dualizable analytic $p$-divisible groups over $\Spa(C,\Os_C)$}}\end{tabular}}\right\}.
    \end{align*}
For an arbitrary base $S$, the situation is different. Already for $S=\Spa(C,C^+)$, where $C^+ \sub C$ is an open and bounded valuation subring of rank $\geq 2$, the functor $(\cdot)_{\eta}$ is generally not fully faithful, see Example \ref{ex: generic fiber not fully faithful.}. Moreover, we show the following. Let $K$ be a $p$-adic field. We remind the reader that this consists in a complete discretely valued field extension $K$ of $\Q_p$ with perfect residue field.

\begin{proposition}{(Proposition \ref{prop: dualizable groups over smooth rig over padic field})}\label{intro prop: dualizable groups over smooth rig over padic field}
    Let $S$ be a smooth rigid space over a $p$-adic field $K$. Then the functor taking an analytic $p$-divisible group $\Gs$ to its Tate module $T_p\Gs$ yields an equivalence of categories
\begin{align*}\label{intro eq: classif over smrig over padic field}
    \left\{\text{\begin{tabular}{l} {\parbox{3.9cm}{$\Gs \rightarrow S$ dualizable analytic $p$-divisible groups}}\end{tabular}}\right\}
         \xrightarrow{\cong} \left\{\text{\begin{tabular}{l} {\parbox{5.7cm}{Hodge--Tate $\Z_p$-local systems $\Lb$ on $S$ with Hodge--Tate weights $\in \{0,1\}$}}\end{tabular}}\right\}.
    \end{align*}
\end{proposition}
In contrast, by a well-known theorem of Breuil and Kisin, the category of $p$-divisible groups on $\Spf(\Os_K)$ is equivalent to the category of $\Z_p$-lattices in crystalline $\Gal_K$-representations with Hodge--Tate weights $\in\{0,1\}$. This shows that the functor $(\cdot)_{\eta}$ need not be essentially surjective in general. 

Still, the result of Scholze--Weinstein is enough to show that a dualizable analytic $p$-divisible group over $S=\Spa(A,A^+)$ aquires good reduction over the ring $A^{\circ}$ of powerbounded elements, $v$-locally on $S$, see Proposition \ref{prop: dualizable groups locally have good reduction on circ}. Moreover, we show, using prismatic Dieudonné theory (resp. log prismatic Dieudonné theory), that the adic generic fiber functor is fully faithful on $p$-divisible group (resp. log $p$-divisible groups) on a smooth (resp. semistable) formal model $\Ss$ of $S$ over $\Os_K$, for $K$ a $p$-adic field, see Proposition \ref{prop: good reduction versus crystalline local system} (resp. Proposition \ref{prop: semistable reduction versus semistable local system}). 

Other examples of dualizable analytic $p$-divisible groups are given by the topologically $p$-torsion Picard variety of smooth proper rigid spaces \cite{heuer2023diamantine}. Here we study it in a relative situation, obtaining the following result.
\begin{proposition}{(Corollary \ref{cor: ptop picard is dualizable})}\label{intro prop: ptop picard variety}
    Let $\pi\colon X \rightarrow S$ be a proper smooth morphism of seminormal rigid spaces over $K$, for some non-archimedean field $K/\Q_p$. 
    \begin{enumerate}
        \item The sheaf
    \[ \PPic_{X/S,\et}\langle p^{\infty} \rangle = R^1\pi_{\Et,*}\G_m\langle p^{\infty} \rangle\colon \Sm_{/S,\et} \rightarrow \Ab\]
    is representable by a smooth $S$-group. It is identified with the topologically $p$-torsion subsheaf of $\PPic_{X/S,\et} = R^1\pi_{\Et,*}\G_m$. 
    \item The group $\PPic_{X/S,\et}\langle p^{\infty} \rangle$ contains a maximal open analytic $p$-divisible subgroup $\Hs$, with 
    \[ \Lie(\Hs) = R^1\pi_{\et,*}\Os_X, \quad T_p\Hs = R^1\pi_{v,*}\Z_p(1)/\torsion,\]
    and the map $f_{\Hs}$ associated to $\Hs$ under Theorem \ref{Intro thm: extending Fargues' equivalence of categories} is the inclusion in the Hodge--Tate sequence for $X\rightarrow S$
    % https://q.uiver.app/#q=WzAsNSxbMCwwLCIwIl0sWzEsMCwiUl4xXFxwaV97XFxldCwqfVxcT3Nfe1h9IFxcb3RpbWVzX3tcXE9zX1N9XFxPc197U192fSJdLFsyLDAsIlJeMVxccGlfe3YsKn1cXFpfcCBcXG90aW1lc197XFxaX3B9XFxPc197U192fSJdLFszLDAsIlxccGlfKlxcT21lZ2Ffe1gvU31eMSBcXG90aW1lc197XFxPc19TfVxcT3Nfe1Nfdn0oLTEpIl0sWzQsMCwiMC4iXSxbMCwxXSxbMSwyXSxbMiwzXSxbMyw0XV0=
\begin{equation}\label{intro eq: contravariant Hodge--Tate sequence}\begin{tikzcd}
	0 & {R^1\pi_{\et,*}\Os_{X} \otimes_{\Os_S}\Os_{S_v}} & {R^1\pi_{v,*}\Z_p \otimes_{\Z_p}\Os_{S_v}} & {\pi_*\Omega_{X/S}^1 \otimes_{\Os_S}\Os_{S_v}(-1)} & {0.}
	\arrow[from=1-1, to=1-2]
	\arrow[from=1-2, to=1-3]
	\arrow[from=1-3, to=1-4]
	\arrow[from=1-4, to=1-5]
\end{tikzcd}\end{equation}
In particular, $\Hs$ is dualizable.
\item There exists a short exact sequence of $v$-sheaves, the multiplicative Hodge--Tate sequence
%  https://q.uiver.app/#q=WzAsNSxbMCwwLCIwIl0sWzEsMCwiXFxIcyJdLFsyLDAsIlJeMVxccGlfe3YsKn1cXFpfcCBcXG90aW1lc197XFxaX3B9XFxHX21cXGxhbmdsZSBwXntcXGluZnR5fSBcXHJhbmdsZSJdLFszLDAsIlxccGlfKlxcT21lZ2Ffe1gvU31eMVxcb3RpbWVzX3tcXE9zX1N9IFxcT3Nfe1Nfdn0oLTEpIl0sWzQsMCwiMC4iXSxbMCwxXSxbMSwyXSxbMiwzXSxbMyw0XV0=
\begin{equation}\label{intro eq: mult Hodge--Tate sequence}\begin{tikzcd}
	0 & \Hs & {R^1\pi_{v,*}\Z_p \otimes_{\Z_p}\G_m\langle p^{\infty} \rangle} & {\pi_*\Omega_{X/S}^1\otimes_{\Os_S} \Os_{S_v}(-1)} & {0.}
	\arrow[from=1-1, to=1-2]
	\arrow[from=1-2, to=1-3]
	\arrow[from=1-3, to=1-4]
	\arrow[from=1-4, to=1-5]
\end{tikzcd}\end{equation}
It lies over the sequence (\ref{intro eq: contravariant Hodge--Tate sequence}) above via the logarithm.
    \end{enumerate}
\end{proposition}
The inclusion in the sequence (\ref{intro eq: mult Hodge--Tate sequence}) generalizes (the topologically $p$-torsion part of) the analytic Weil pairings (\ref{intro eq: morphism of Deninger--Werner}) of Deninger--Werner and Heuer. If $X=\As$ is a relative abeloid variety over $S$, we show in Corollary \ref{cor: HT sequences for abeloids} that the Cartier dual $\Hs^D$ of the group $\Hs$ in the statement is isomorphic to the topologically $p$-torsion subgroup $\As\langle p^{\infty}\rangle \sub \As$. In general, this justifies thinking of the group $\Hs^D$ as the “topologically $p$-torsion Albanese variety” of $X/S$. 

The above result is part of a more general comparison theorem for analytic $p$-divisible groups, Theorem \ref{Thm: comparison theorem for analytic p-divisible groups}. The latter shows that for arbitrary $n \geq 0$, under some assumptions, the higher direct images $R^n\pi_{\Et,*}\Gs$ are representable and are analytic $p$-divisible, up to some torsion. When $n\geq 2$, these groups are typically not dualizable and thus don't have good reduction, nor do they arise from abeloid varieties. For example, the map $f$ associated with the group $R^n\pi_{\Et,*}\G_m\langle p^{\infty} \rangle$ under Theorem \ref{Intro thm: extending Fargues' equivalence of categories} is the inclusion
\[ f\colon R^n\pi_{\et,*}\Os_X \otimes \Os_{S_v} = \Fil_{\HT}^n \hookrightarrow  R^n\pi_{v,*}\Z_p \otimes_{\Z_p} \Os_{S_v},\]
given by the Hodge--Tate filtration
\[ \Fil_{\HT}^n \sub \Fil_{\HT}^{n-1} \sub \ldots \sub \Fil_{\HT}^0 = R^n\pi_{v,*}\Z_p \otimes_{\Z_p} \Os_{S_v}\]
whose graded pieces are
\[ \Fil_{\HT}^i/\Fil_{\HT}^{i+1} = R^i\pi_{\et,*}\Omega_{X/S}^{n-i}\otimes_{\Os_S} \Os_{S_v}(i-n).\]

\subsection{Dieudonné theory}
Let $S$ be a perfectoid space over $\Q_p$ and let $X_S$ denote the relative Fargues--Fontaine curve \cite[Def. 13.5.1]{SW20}. We can then associate to an analytic $p$-divisible group $\Gs \rightarrow S$ a sheaf $\Es(\Gs)$ on $X_S$. This generalizes a construction of Scholze--Weinstein \cite[§5]{scholze2013moduli}, which is the case $S=\Spa(C)$.
\begin{thm}{(Theorem \ref{thm: vb on FF associated to an pdiv groups})}\label{intro thm: vb on FF associated to an pdiv groups}
    Let $S$ be a perfectoid space over $\Q_p$ and let $\Gs \rightarrow S$ be an analytic $p$-divisible group. There exists a functorially associated sheaf of $\Os_{X_S}$-modules $\Es(\Gs)$ on the relative Fargues--Fontaine curve $X_S$. It comes with a natural short exact sequence
        % https://q.uiver.app/#q=WzAsNSxbMCwwLCIwIl0sWzEsMCwiVF9wXFxHcyBcXG90aW1lc197XFxaX3B9IFxcT3Nfe1hfU30iXSxbMiwwLCJcXEVzKFxcR3MpIl0sWzMsMCwiaV8qXFxMaWUoXFxHcykiXSxbNCwwLCIwLiJdLFswLDFdLFsxLDJdLFsyLDNdLFszLDRdXQ==
\begin{equation}\label{intro eq: modification defining E(G)}\begin{tikzcd}
	0 & {T_p\Gs \otimes_{\Z_p} \Os_{X_S}} & {\Es(\Gs)} & {i_*\Lie(\Gs)} & {0.}
	\arrow[from=1-1, to=1-2]
	\arrow[from=1-2, to=1-3]
	\arrow[from=1-3, to=1-4]
	\arrow[from=1-4, to=1-5]
\end{tikzcd}\end{equation}
Moreover, $\Gs$ is dualizable if and only if $\Es(\Gs)$ is a vector bundle.
\end{thm}
We show moreover that the sheaf $\Es(\Gs)$ is functorial in $\widetilde{\Gs}$, whereas the modification (\ref{intro eq: modification defining E(G)}) is not. Here we let $\widetilde{\Gs}$ denote the $p$-adic universal cover of $\Gs$
\[ \widetilde{\Gs} = \varprojlim_{[p]}\Gs,\] computed in the category of diamonds. When $S=\Spa(C)$, the group of global sections of $\widetilde{\Gs}$ is an example of an \textit{effective Banach--Colmez space} in the sense of \cite[§4]{Fontaine03}. We show in Proposition \ref{prop: minuscule sheaves vs minuscule BC spaces} that the association $\Gs \mapsto \Es(\Gs)$ factors through a fully faithful functor on a category of effective Banach--Colmez spaces over $S$. In particular, if $\Gs' \rightarrow S$ is another analytic $p$-divisible group, we have $\Es(\Gs) \cong \Es(\Gs')$ as $\Os_{X_S}$-modules if and only if $\widetilde{\Gs} \cong \widetilde{\Gs}'$ as $v$-sheaves of abelian groups. Coupled with Proposition \ref{intro prop: dRFF cohomology from ptop picard variety} and equation (\ref{intro eq: drFF versus dR}) below, this implies that, given two abeloid varieties $A,A'$, any isomorphism between their $p$-adic universal covers $\widetilde{A} \cong \widetilde{A}'$ yields an identification of their first de Rham homology. Such isomorphisms exist frequently between non-isomorphic (and even non-isogenous) abeloid varieties, by \cite[Thm. 1.5]{heuer2021proetaleuniformisationabelianvarieties}. We believe that this can be used to construct new interesting domains for de Rham period maps on global PEL Shimura varieties, outside of the good reduction locus. 
 
When restricting to dualizable groups, we obtain the following result, generalizing Fargues' theorem \cite[Thm. 14.1.1]{SW20} which is the case $S = \Spa(C)$.
\begin{thm}{(Theorem \ref{thm: equivalence for dualizable groups on perfd})}\label{intro thm: equivalence for dualizable groups on perfd}
Let $S$ be a perfectoid space over $\Q_p$, then there are canonical equivalences of exact categories between the following:
\begin{enumerate}
    \item Dualizable analytic $p$-divisible groups $\Gs \rightarrow S$,
    \item Pairs $(\Lb,E)$, where $\Lb$ is a $\Z_p$-local system on $S$ and $E \hookrightarrow\Lb(-1)\otimes_{\Z_p} \Os_{S_v}$ is a locally direct summand,
    \item Pairs $(\Lb,\Xi)$ consisting of a $\Z_p$-local system $\Lb$ and a minuscule $\B_{\dR}^+$-lattice
    \[ \Lb \otimes_{\Z_p} \B_{\dR}^+ \sub \Xi \sub \xi^{-1}(\Lb \otimes_{\Z_p}\B_{\dR}^+),\]
    \item Tuples $(\Lb,\Es,\alpha)$ consisting of a $\Z_p$-local system $\Lb$, a vector bundle $\Es$ on the Fargues--Fontaine curve $X_S$ and a minuscule modification $\alpha\colon \Lb\otimes_{\Z_p} \Os_{X_S} \dashrightarrow \Es$ (see Definition \ref{def: minuscule modification}), and
    \item Minuscule shtukas $(\Ms,\varphi_{\Ms})$ over $\mathcal{Y}_S= S \mathbin{\dot\times}\Spa(\Z_p)$ with one leg at $\varphi^{-1}(S)$ (see Definition \ref{def: minuscule shtuka}).
\end{enumerate}
\end{thm}
We note that the equivalences $(2)-(5)$ constitute the “easy” part of Fargues' theorem and are already implicitly contained in \cite{SW20}. The main difference with the case $S=\Spa(C)$ is that these objects are no longer equivalent to Breuil--Kisin--Fargues modules over $\Ainf$.

Using $v$-descent, we obtain a Dieudonné theory that applies to arbitrary good adic spaces $S$ as a base. For this, let $\Sht_{\min}$ be the prestack on perfectoid spaces sending $S$ to the groupoid of minuscule shtukas over $\Yss_S$ with a leg at $\varphi^{-1}(S)$. Using the descent result \cite[Prop 19.5.3]{SW20}, $\Sht_{\min}$ is easily seen to be a small $v$-stack, so that we can make sense of minuscule shtukas on good adic spaces.
\begin{definition}{(Definition \ref{def: analytic Dieudonné crystal})}
    Let $S$ be a good adic space over $\Q_p$. An analytic Dieudonné crystal over $S$ is a minuscule shtuka $(\Ms,\varphi_{\Ms})$ over $S$ such that the sheaves
    \[ \varphi_{\Ms}(\varphi^*\Ms)/\Ms \,\text{ and }\, \tfrac{1}{\varphi(\xi)}\Ms/\varphi_{\Ms}(\varphi^*\Ms)\]
    arise from vector bundles on $S_{\et}$.
\end{definition}
If $S$ is a perfectoid space, the condition is automatic, so that analytic Dieudonné crystals coincide with minuscule shtukas. We obtain the following version of Dieudonné theory.

\begin{thm}{(Theorem \ref{thm: analytic Dieudonné theory in general})}\label{intro thm: analytic Dieudonné theory in general}
Let $S$ be a good adic space over $\Q_p$, then there is a canonical equivalence of categories 
    \begin{align*}
 \left\{\text{\begin{tabular}{l} {\parbox{3.7cm}{dualizable analytic $p$-divisible groups $\Gs$ over $S$}}\end{tabular}}\right\}
         \xlongrightarrow{\cong} \left\{\text{\begin{tabular}{l} {\parbox{4.2cm}{analytic Dieudonné crystals $(\Ms,\varphi_{\Ms})$ over $S$}}\end{tabular}}\right\}.
    \end{align*}
The functor and its inverse are exact and are compatible with Cartier duality in the following sense:
\[ \Ms(\Gs^D) = \Ms(\Gs)^{\vee}\otimes \Ms(\G_m\langle p^{\infty}\rangle).\]
\end{thm}
We connect our functors with $p$-adic cohomology theories. Let $Z$ be a proper smooth rigid space over a complete algebraically closed field $C/\Q_p$.
% , we have the $\BdR+$-cohomology $H_{\crys}^n(Z/\BdR+)$ of Bhatt--Morrow--Scholze \cite[§13]{Bhatt2018}, see also \cite{Guo2021cryscoh}\cite{bosco2023rationalpadichodgetheory}. It consists of finite free $\BdR+(C)$-modules and deforms de Rham cohomology, i.e. it comes with a canonical isomorphism
% \begin{align} 
% H_{\crys}^n(Z/\BdR+) \otimes_{\BdR+(C),\theta} C = H_{\dR}^n(Z/C).
% \end{align}
% The $\BdR+$-cohomology further spreads to the Fargues--Fontaine curve. 
There is a cohomology theory $H_{\FF}^n(Z)$ \cite{Le_Bras_vezzani_2023}\cite{bosco2023rationalpadichodgetheory}, the Fargues--Fontaine cohomology, taking values in vector bundles on the Fargues--Fontaine curve $X_C$ and lifting de Rham cohomology, in the sense that
\begin{align}\label{intro eq: drFF versus dR} H_{\FF}^n(Z) \otimes_{\Os_{X_C}} C = H_{\dR}^n(Z/C).\end{align}
Our input is that we can recover the first de Fargues--Fontaine cohomology group from the group $\PPic_{Z/C,\et}\langle p^{\infty} \rangle$. This should be reminiscent of the isomorphism between the first crystalline cohomology of an abelian variety $A$ over a perfect field $k$ of characteristic $p$ and the Dieudonné isocrystal $D(\breve{A}[p^{\infty}])$ of the $p$-divisible group of the dual abelian variety $\breve{A}$.
\begin{proposition}{(Proposition \ref{prop: dRFF cohomology from ptop picard variety})}\label{intro prop: dRFF cohomology from ptop picard variety}
    Let $\Hs \sub \PPic_{Z/C,\et}\langle p^{\infty} \rangle$ denote the maximal analytic $p$-divisible open subgroup of Proposition \ref{intro prop: ptop picard variety}. Then we have a canonical isomorphism of vector bundles on the Fargues--Fontaine curve $X_C$
        \begin{align} \Es(\Hs) = H_{\FF}^1(Z)\otimes_{\Os_{X_C}} \Os_{X_C}(1).\end{align}
\end{proposition}

\subsection{Applications to local and global Shimura varieties}
At last, we obtain applications to moduli spaces. By Theorem \ref{Intro thm: extending Fargues' equivalence of categories}, analytic $p$-divisible groups satisfy descent along $v$-covers of perfectoid spaces (but not $v$-covers of arbitrary good adic spaces, see Remark \ref{remark: an pdiv groups don't descent outside of perfd}). We can obtain a more precise description. For fixed integers $0\leq d \leq n$, consider the small $v$-stack $\Ns_{n,d}$ on $\Perf_{\Q_p}$ which sends a perfectoid space $S$ to the groupoid of dualizable analytic $p$-divisible groups $\Gs \rightarrow S$ of dimension $d$ and height $n$. It comes with a Newton locally closed stratification, indexed by the Kottwitz set $B(\GL_n)$
\[ \Ns_{n,d} = \coprod_{b \in B(\GL_n)} \Ns_{n,d}^b, \]
where $\Ns_{n,d}^b$ is the subfunctor consisting of those groups $\Gs$ such that pointwise on $S$, there exists an isomorphism $\Es(\Gs) \cong \Es^b$. It then follows from Theorem \ref{intro thm: equivalence for dualizable groups on perfd} that 
\begin{align}\label{intro eq: stack of an pdiv versus quotients of flags}
\Ns_{n,d} =  \Big[ \Fl_{n,n-d}/\GL_n(\Z_p)\Big],
\end{align}
where $\Fl_{n,n-d}$ is the flag variety parametrizing $d$-dimensional subbundles $E \hookrightarrow \Os^{\oplus n}$. Moreover, the Newton stratification on $\Ns_{n,d}$ is obtained from the Newton stratification on $\Fl_{n,n-d}$ of \cite[Thm. 1.11]{ScholzeCariani2017}. Using our Cartier duality, there is a good theory of polarizations on dualizable analytic $p$-divisible groups. From this, there are immediate variants of the stacks $\Ns_{n,d}$ of EL and PEL type, admitting presentations as quotients of flag varieties analog to (\ref{intro eq: stack of an pdiv versus quotients of flags}). 

We use our results to reinterpret the Hodge--Tate period map of Scholze \cite{Scholze2015torsion} on global PEL Shimura varieties. For simplicity, we consider the case of the modular curve in this introduction; the general case is done in the body of the paper. For $N \geq 0$, this is a Deligne--Munford stack $A_{\Gamma(N)}$ over $\Spec(\Q)$, parametrising elliptic curves with level $N$-structure. For $N$ big enough, $A_{\Gamma(N)}$ is a smooth quasi-projective curve over $\Spec(\Q)$, so that we may consider its analytification
\[ \As_{\Gamma(N)} = (A_{\Gamma(N)} \otimes_{\Q} \Q_p)^{\an},\]
a smooth rigid space over $\Spa(\Q_p)$. We may then consider the inverse limit, viewed as a diamond
\[ \As_{\Gamma(p^{\infty})} = \varprojlim_{n\geq 1} \As_{\Gamma(p^n)}.\]
In this situation, the Hodge--Tate period map
\begin{align} 
\pi_{\HT}\colon \As_{\Gamma(p^{\infty})} \rightarrow \Pro^1
\end{align}
is a morphism of adic spaces, defined on $C$-points for a complete algebraically closed field $C$ as sending a pair $(E,\gamma)$ consisting of an elliptic curve $E$ over $C$ together with a trivialization $\gamma$ of the Tate module to the Hodge--Tate filtration
% https://q.uiver.app/#q=WzAsNSxbMCwwLCIwIl0sWzEsMCwiXFxMaWUoRSkoMSkiXSxbMiwwLCJUX3BFXFxvdGltZXNfe1xcWl9wfUNcXG92ZXJzZXR7XFxnYW1tYX17XFxjb25nfSBDXjIiXSxbMywwLCJcXG9tZWdhX3tcXGJyZXZle0V9fSJdLFs0LDAsIjAuIl0sWzAsMV0sWzEsMl0sWzIsM10sWzMsNF1d
\begin{equation}\begin{tikzcd}
	0 & {\Lie(E)(1)} & {T_pE\otimes_{\Z_p}C\overset{\gamma}{\cong} C^2} & {\omega_{\breve{E}}} & {0.}
	\arrow[from=1-1, to=1-2]
	\arrow[from=1-2, to=1-3]
	\arrow[from=1-3, to=1-4]
	\arrow[from=1-4, to=1-5]
\end{tikzcd}\end{equation}
The Hodge--Tate map is $\GL_2(\Z_p)$-equivariant, where the latter groups acts on $\Pro_{\Q_p}^1$ through its natural action on $\Z_p^{\oplus 2} \sub C^{\oplus 2}$. Therefore, it descends to a morphism of small $v$-stacks
\begin{align}\label{intro eq: descended Hodge--Tate period map}
\pi_{\HT}\colon \As \rightarrow [\Pro_{\Q_p}^1/\GL_2(\Z_p)].
\end{align}
We then obtain the following.
\begin{thm}{(Theorem \ref{thm: HT map reinterpreted})}\label{intro thm: HT map reinterpreted}
    Under the isomorphism (\ref{intro eq: stack of an pdiv versus quotients of flags}), the Hodge--Tate period map (\ref{intro eq: descended Hodge--Tate period map}) sends an elliptic curve $E$ to its topologically $p$-torsion subgroup $E \langle p^{\infty}\rangle \sub E$. 
\end{thm}
We extend this statement to arbitrary PEL Shimura varieties over $\Q$. Theorem \ref{thm: HT map reinterpreted} can be used to reprove Rapoport--Zink's $p$-adic uniformization of PEL Shimura varieties \cite[Thm. 6.36]{rapoportzink96}. In future work, we plan to use this perspective to study new cases of the fiber product conjecture of Scholze \cite[Conj. 1.1]{zhang2023peltypeigusastackpadic}.

Finally, we discuss applications to local Shimura varieties. We fix a $p$-divisible group $G_0$ over $\cj{\F}_p$ of dimension $d$ and height $n$ and we consider the associated Rapoport--Zink space $\RZ_{G_0}$. This is a smooth formal scheme over $\Spf(\breve{\Z}_p)$, where $\breve{\Z}_p = W(\cj{\F}_p)$, representing the moduli problem
\[ R \in \Nilp_{\breve{\Z}_p} \longmapsto \{\, (G,\rho) \mid G \, \,p\text{-divisible group over }R,\, \rho\colon G \otimes_R R/p\dashrightarrow G_0 \otimes_{\cj{\F}_p}R/p \,\}/\cong,\]
where $\rho$ is a quasi-isogeny. We let $\Ms_{G_0} =(\RZ_{G_0})_{\eta}$ denote the adic generic fiber, a smooth rigid space. We fix a deformation $\mathfrak{G}_0 \rightarrow \Spf(\breve{\Z}_p)$ of $G_0$ and we let $\Gs_0 \rightarrow \Spa(\breve{\Q}_p)$ be its adic generic fiber. We are now ready to state our theorem.
\begin{thm}{(Theorem \ref{main thm: moduli of an pdiv vs local SV})}\label{intro thm: Rapoport-Zink space in terms of analytic p-div groups}
    The rigid space $\Ms_{G_0}$ has the following functor of points, for $S$ a good adic space over $\Spa(\breve{\Q}_p)$
    \begin{align*}
         \Ms_{G_0}(S) = \{\, (\Gs,\beta) \mid \, &\Gs \text{ dualizable analytic }p\text{-divisible group over }S,\\ &\beta\colon \widetilde{\Gs}  \xrightarrow{\cong} \widetilde{\Gs}_0 \times_{\Spa(\breve{\Q}_p)} S \,\}/\cong.
    \end{align*}
%     Under this isomorphism, the Grothendieck--Messing period map sends a pair $(\Gs,\beta)$ to 
%     % https://q.uiver.app/#q=WzAsNCxbMiwwLCJEKFxcR3MpIl0sWzMsMCwiXFxMaWUoXFxHcykuIl0sWzEsMCwiRChcXEdzXzBcXHRpbWVzX3tcXFNwYShcXEJyZXZle1xcUX1fcCl9IFMpIl0sWzAsMCwiRChcXEdzXzApXFxvdGltZXNfe1xcQnJldmV7XFxRfV9wfVxcT3NfUyJdLFswLDEsIiIsMCx7InN0eWxlIjp7ImhlYWQiOnsibmFtZSI6ImVwaSJ9fX1dLFszLDIsIj0iXSxbMiwwLCJcXGJldGFeey0xfSJdLFsyLDAsIlxcY29uZyIsMl1d
% \[\begin{tikzcd}
% 	{D(\Gs_0)\otimes_{\Breve{\Q}_p}\Os_S} & {D(\Gs_0\times_{\Spa(\Breve{\Q}_p)} S)} & {D(\Gs)} & {\Lie(\Gs).}
% 	\arrow["{=}", from=1-1, to=1-2]
% 	\arrow["{\beta^{-1}}", from=1-2, to=1-3]
% 	\arrow["\cong"', from=1-2, to=1-3]
% 	\arrow[two heads, from=1-3, to=1-4]
% \end{tikzcd}\]
\end{thm}
We prove a more general result, Theorem \ref{main thm: EL and PEL moduli}, yielding a description of the local Shimura varieties of EL and PEL type of Scholze--Weinstein \cite[§24.3]{SW20} as moduli spaces of analytic $p$-divisible groups with extra structure. This should be seen as a local counterpart to the description of global Shimura varieties of PEL type in terms of abelian varieties. We note that Theorem \ref{main thm: EL and PEL moduli} is independent of any integral data and valid for any level subgroup $K$. In contrast, the comparison with Rapoport--Zink spaces \cite[Cor. 24.4.5]{SW20} only applies to parahoric level subgroups and relies on additional choices to be formulated.

\subsection{Limitations and generalizations}
Throughout this paper, we work under the assumption that the base $S$ is a good adic space over $\Q_p$. The reasons for this are twofold: the first is for sheafiness concerns. We often consider smooth maps $\Gs \rightarrow S$, which produce non-sheafy adic spaces for general $S$. The second, deeper reason is that our approach fundamentally relies on perfectoid methods. Therefore, we classify geometric and linear algebraic objects on the diamond $S^{\diamondsuit}$ associated to the adic space $S$, which only remembers topological information. Hence, the formalism of diamonds seems unsuitable for generalizing the present theory. Here are some concrete issues that arise:
\begin{enumerate}
    \item Let $S$ be a non-reduced rigid space in characteristic $0$. Then we are not aware of a satisfactory notion of pro-étale vector bundles on $S$, receiving a functor from $\Z_p$-local systems and containing étale vector bundles as a full subcategory.
    \item Let $S$ be a characteristic $p$ perfectoid space. Then the passage from analytic $p$-divisible groups $\Gs \rightarrow S$ to $v$-sheaves of groups on $S$ factors through the perfection $\Gs \mapsto \Gs^{\perf}$ and thus forgets information. For radicial groups, i.e. those where $\Gs[p] \rightarrow S$ is radicial, the perfection coincides with the $p$-adic universal cover $\widetilde{\Gs}$.
\end{enumerate}

Still, we expect some of our results to extend beyond the class of good adic spaces over $\Q_p$. We give some examples below.
\begin{enumerate}
    \item Let $S$ be any perfectoid space over $\Spa(\Z_p)$, e.g. of characteristic $p$. There is still a well-defined notion of shtuka on $\Yss_S$ with a leg at $\varphi^{-1}(S)$, and we expect Theorem \ref{intro thm: analytic Dieudonné theory in general} to extend to this case. This would then realize a minuscule variant $\Sht_{\min}$ of the stack of local shtuka over $\Spd(\Z_p)$ \cite[Def. 11.12]{zhang2023peltypeigusastackpadic}\cite[Def. 4.16]{gleason2025meromorphicvectorbundlesfarguesfontaine} as a moduli stack of dualizable analytic $p$-divisible groups, and extend our reinterpretation of the Hodge--Tate period map (Theorem \ref{intro thm: HT map reinterpreted}) to the mixed characteristic case.
    \item In another direction, it would be interesting to investigate if an analog of Grothendieck--Messing deformation theory \cite{Messing1972} exists for analytic $p$-divisible groups over non-reduced rigid spaces.
    \item Yet another interesting question would be to obtain a stacky formulation of Theorem \ref{intro thm: analytic Dieudonné theory in general}, in the spirit of stacky prismatic Dieudonné theory \cite[§3.4]{guo2023frobeniusheightprismaticcohomology}\cite{Mondal2024}, which uses the formalism of Drinfeld \cite{Drinfeld2024} and Bhatt--Lurie \cite{bhatt2022absoluteprismaticcohomology}\cite{bhatt2022prismatizationpadicformalschemes}. This last direction seems especially promising, given the recent development in non-archimedean geometry \cite{anschutz2025analyticrhamstacksfarguesfontaine}, based on the theory of analytic stacks of Clausen and Scholze \cite{ClausenScholze2024}.
\end{enumerate}  
We plan to return to these questions in future work.

\subsection{Plan of the paper}
Section \ref{section: preliminaries} reviews important concepts in relative $p$-adic Hodge theory. In Section \ref{section: Analytic $p$-divisible groups}, we introduce the main objects of the paper, the analytic $p$-divisible groups. We prove the theorem relating analytic $p$-divisible groups with Hodge--Tate triples, and we define a category of effective Banach--Colmez spaces. In Section \ref{section: Examples of analytic $p$-divisible groups}, we discuss various constructions yielding examples of analytic $p$-divisible groups. Section \ref{section Dieudonne theory} presents the Dieudonné theory, starting with some reminders about the Fargues--Fontaine curve. We then compare it with earlier constructions and with $p$-adic cohomology theories. Finally, in Sections \ref{section: EL and PEL local Shimura varieties} and \ref{section: moduli stack of an pdiv groups}, we discuss the moduli-theoretic aspects.
    
\subsection*{Acknowledgments}
I thank Ben Heuer for many interesting discussions and helpful comments on the project at its various stages of completion. This project was carried out while I was a PhD student at Goethe Universität Frankfurt, and I am especially grateful to my advisor, Annette Werner, for excellent guidance and valuable feedback. I thank Daniel Kim for pointing out an inaccuracy in an earlier version. Finally, I thank Laurent Fargues, Sean Howe, Pol van Hoften, Johannes Anschütz, Mingjia Zhang, Ian Gleason, Konrad Zou and Abhinandan for helpful discussions.

This project was funded by Deutsche Forschungsgemeinschaft (DFG, German Research Foundation) through the Collaborative Research Centre TRR 326 \textit{Geometry and Arithmetic of Uniformized Structures} - Project-ID 444845124. This work was also partially supported by the Simons Foundation through the Simons Collaboration on Perfection in Algebra, Geometry, and Topology. 
\subsection*{Notations}
\begin{itemize}
    \item Given a non-archimedean field $K$ (always assumed to be complete), we denote by $\Os_K$ (resp. $\ma_K$) its valuation ring (resp. maximal ideal), i.e. the set of elements of valuation $\leq 1$ (resp. $<1$). A $p$-adic field is by definition a non-archimedean field extension $K$ of $\Q_p$ that is discretely valued and whose residue field is perfect.
    \item Throughout, we will work with analytic adic spaces in the sense of Huber \cite{huber2013étale}. Adic spaces will always implicitly be assumed to be sheafy. For a Huber ring $A$ with subring of power-bounded elements $A^{\circ}$, we write $\Spa(A)$ for $\Spa(A,A^{\circ})$. We fix a prime number $p$ and we only consider adic spaces $X$ living over $\Spa(\Z_p)$, i.e. we always assume $p\in \Os(X)$ to be locally topologically nilpotent. Whenever we consider a pair $(K,K^+)$, it will always be implicitly assumed that $K$ is a non-archimedean field and $K^+$ is an open and bounded valuation subring. Such a pair will sometimes be referred to as a non-archimedean field.
    \item A rigid space $X$ over $(K,K^+)$ is by definition an adic space that is locally of topologically finite type over $\Spa(K,K^+)$. We will work with perfectoid spaces in the sense of \cite{scholze2011perfectoid} and, for an adic space $X$, we denote by $\Perf_X$ the category of perfectoid spaces together with a map to $X$. We will actually use the more general definition used in \cite{SW20} to allow for perfectoid spaces living over a non-perfectoid base field, such as $\Q_p$. More generally, we will consider sousperfectoid adic spaces, that is, adic spaces $X$ that admit an open cover by affinoid spaces $\Spa(R,R^+)$, where $R$ is a sousperfectoid ring in the sense of \cite[§6.3]{Kedlaya2020Sheafiness}. Most adic spaces that we will encounter will be of one of the above types.
    \item Given an adic space $X$, we may associate to it its étale site $X_{\et}$ \cite[Def. 8.2.19]{kedlaya2015relative}. In general, this may contain non-sheafy adic spaces. However, for all types of adic spaces listed above, this does not happen. A map of affinoid adic spaces is called standard étale if it is a composition of finite étale maps and rational immersions. 
    \item We will use the theory of diamonds and their very fine topologies introduced in \cite{scholze2022etale}. In particular, for an analytic adic space $X$ over $\Spa(\Z_p)$, we have an associated diamond $X^{\diamondsuit}$, a sheaf on the category $\Perf_{\F_p}$ of perfectoid spaces of characteristic $p$. It comes with an identification of their étale sites $X_{\et} \cong X_{\et}^{\diamondsuit}$ \cite[Lemma 15.6]{scholze2022etale}. We define the $v$-site of $X$ to be $X_v = \LSD_{X^{\diamondsuit},v}$ the category of locally spatial diamonds with a map to $X^{\diamondsuit}$, equipped with the $v$-topology \cite[§8]{scholze2022etale}. Under the equivalence $\Perf_{\F_p,/\Spd(\Z_p)} \cong \Perf_{\Z_p}$ sending a map $S \rightarrow \Spd(\Z_p)$ to the corresponding untilt $S^{\sharp}$, the diamond $X^{\diamondsuit}$ can be seen as a sheaf on $\Perf_{\Z_p}$, the functor of points of $X$. We will often omit the diamond symbol and write $X$ in place of $X^{\diamondsuit}$.
    \item Whenever we work with a formal scheme $\Xs$, we will always assume it to be locally of the form $\Spf(R)$, where $R$ is an $I$-adically complete ring for a finitely generated ideal $I$ containing a power of $p$. A $p$-divisible group on $\Xs$ is an fpqc sheaf $G$ of abelian groups on the category $\Sch_{\Xs}$ of schemes over $\Xs$ such that the multiplication by $p$ map $[p]\colon G \rightarrow G$ is surjective, with kernel $G[p]$ representable by a finite flat formal group scheme over $\Xs$ and such that $G = \varinjlim G[p^n]$. We denote by $\BT(\Xs)$ the category of $p$-divisible groups over $\Xs$.
    \item The notation $(\cdot)^{\vee}$ will be reserved for the dual of objects of linear algebraic nature, such as local systems or vector bundles. 
\end{itemize}
\section{Preliminaries}\label{section: preliminaries}
\subsection{$p$-adic local systems}
We review local systems on adic spaces and diamonds, following \cite[§2]{MannWerner2022loc}.

\begin{definition}
    Let $Y$ be a locally spatial diamond. Let $\tau$ be the étale topology and $\Lambda = \Z/p^n\Z$, or let $\tau \in \{\et, v\}$ and $\Lambda \in \{\Z/p^n\Z,\Z_p,\Q_p\}$. Let $\underline{\Lambda}$ denote the associated locally profinite sheaf on $X_{\tau}$ and let $\Lb$ be a sheaf of $\underline{\Lambda}$-modules on $Y_{\tau}$.
    \begin{enumerate}
        \item $\Lb$ is a $\Lambda$-local system if it is finite locally free, i.e. there exists a $\tau$-cover $\{Y_i \rightarrow Y\}$ such that $\Lb\restr{Y_i}\cong \underline{\Lambda}^{\oplus r_i}$ for some $r_i\geq 0$.
        \item $\Lb$ is a lisse $\Lambda$-sheaf if there exists a $\tau$-cover $\{Y_i \rightarrow Y\}$ such that $\Lb\restr{Y_i} \cong \underline{M_i}$, for some finite $\Lambda$-modules $M_i$. 
    \end{enumerate} 
\end{definition}
By \cite[Prop. 3.7, Thm. 3.11]{MannWerner2022loc}, we have an equivalence
\begin{align}\label{eq: from etale ls to v ls}
\Loc_{\Z/p^n\Z}(Y_{\et}) \cong \Loc_{\Z/p^n\Z}(Y_{v}).\end{align}
Furthermore, by \cite[Prop. 3.5]{MannWerner2022loc}, there is an equivalence between $\Z_p$-local systems $\Lb$ on $Y_v$ and inverse systems $(\Lb_n)_{n \geq 1}$ of $\Z/p^n\Z$-local systems on $Y$ with $\Lb_{n+1}/p^n \cong \Lb_n$. We let $\Z_p(1) = \varprojlim_n \mu_{p^n}$ denote the Tate twist.

Let $X$ be an analytic adic space. In that case, we may also consider $\Z/p^n\Z$-local systems on $X_{\et}$, which are identified with étale local systems on its associated diamond, using the equivalence $X_{\et} \cong X_{\et}^{\diamondsuit}$ of \cite[Lemma 15.6]{scholze2022etale}. We will often identify a $\Z/p^n\Z$-local system on $X_{\et}$ with its extension to $X_v$ given by (\ref{eq: from etale ls to v ls}). We will also talk about $\Z_p$ (resp. $\Q_p$)-local systems on $X$ when considering such local systems on $X^{\diamondsuit}$. 

In \cite{scholze2013padicHodge}, Scholze defines, for any locally noetherian adic space, a pro-étale site $X_{\proet}$, which is sometimes called the flattened pro-étale site. By \cite[Prop. 8.2]{scholze2013padicHodge}, it follows that $\Z_p$-local systems on $X_{\proet}$ are also equivalent to local systems on $X_{v}^{\diamondsuit}$.

We will use the following result of Scholze on the finiteness of cohomology of local systems. Let $K$ be a non-archimedean field extension of $\Q_p$ and let $K^+\sub K$ be an open and bounded valuation subring. Let $C$ be a completed algebraic closure of $K$ and let $C^+$ the completion of the integral closure of $K^+$ in $C$.
\begin{proposition}{(\cite[Thm. 1.1]{scholze2013padicHodge})}\label{prop: cohomology of local system}
    \begin{enumerate}
        \item Let $X$ be a proper smooth rigid space over $(K,K^+)$. Let $\Lb$ be a $\Z_p$-local system on $X$ and $n\geq 0$. Then $H_v^n(X_C, \Lb)$ is a finite $\Z_p$-module and we have
        \[ H_v^n(X_C, \Lb) = \varprojlim_m H_{\et}^n(X_C, \Lb/p^m).\]
        \item More generally, let $\pi\colon X \rightarrow S$ be a proper smooth morphism of analytic adic spaces over $\Z_p$. Let $\Lb$ be a $\Z_p$-local system on $X$ and $n\geq 0$. Then $R^n\pi_{v,*}\Lb$ is a lisse $\Z_p$-sheaf and we have
        \[ R^n\pi_{v,*}\Lb =\varprojlim_m \nu^*R^n\pi_{\et,*}\Lb/p^m,\]
        where $\nu\colon S_v \rightarrow S_{\et}$ is the natural map of sites.
    \end{enumerate}
\end{proposition}
    \begin{proof}
    It is enough to show the second point. Let us set $\Lb_m = \Lb/p^m$, for all $m\geq 1$. By \cite[Cor. 5.5]{heuer2024primitive}, we have
    \[ R^n\pi_{v,*}\Lb_m =\nu^*R^n\pi_{\et,*}\Lb_m.\]
    As a first step, we show that $R^n\pi_{\et,*}\Lb_1$ is a $\F_p$-local system on $S$.
    This is due to Scholze--Weinstein \cite[Thm. 10.5.1]{SW20} when $X$ and $S$ are rigid spaces over $(C,\Os_C)$. The general case follows from \cite{heuer2024primitive}. Indeed, it is enough to show that the stalks of this sheaf are finite-dimensional $\F_p$-vector spaces. If $s\in S(C,C^+)$ is any geometric point, we have by \cite[Lemma 4.10]{heuer2024primitive} 
    \[ (R^n\pi_{\et,*}\Lb_1)_s = H_{\et}^n(X_s, \Lb_1). \]
    By \cite[Prop. 8.2.3(ii)]{huber2013étale}, this latter group is unchanged if we replace $C^+$ by $\Os_C$. The desired finite-dimensionality now follows from \cite[Thm. 1.1]{scholze2013padicHodge} if $C$ has characteristic $0$ and \cite[Thm. 3.9]{heuer2024primitive} in general.
    % and we reduce to this case. Let $X_0$ (resp. $S_0$) denote the base change of $X$ (resp. $S$) along the immersion $\Spa(K,\Os_K) \rightarrow \Spa(K,K^+)$. By \cite[Prop. 8.2.3(ii)]{huber2013étale}, the sheaves $R^n\pi_{\et,*}\Lb_m$ are overconvergent. Furthermore, their formations are easily seen to commute with base change along the immersion $S_0 \sub S$. Hence, we may assume that $K^+=\Os_K$. Next, let $\pi_C\colon X_C \rightarrow S_C$ denote the map obtained by base change and let $f\colon X_C \rightarrow X$ the natural map. Then by \cite[Cor. 16.10(ii)]{scholze2022etale}, we have
    % \[ f^*R^n\pi_{\et,*}\Lb_1 \cong R^n\pi_{C,\et,*}\Lb_1.\]
    % Hence it follows from the result of Scholze--Weinstein that $R^n\pi_{\et,*}\Lb_1$ is a $\F_p$-local system $v$-locally on $X$ and thus an $\F_p$-local system. 
    The statement of the proposition now follows from this case by a standard devissage argument, following e.g. the proof of \cite[Thm. 4.5]{heuer2024relative}.
\end{proof}

\subsection{Proétale vector bundles}
We briefly discuss vector bundles on adic spaces and diamonds for the fine topologies introduced by Scholze in \cite{scholze2022etale}. The references are \cite[§2]{MannWerner2022loc} and \cite{heuer2022gtorsors}.

Let $X$ be an adic space over $\Q_p$ that is either sousperfectoid or a rigid space over a non-archimedean field $(K,K^+)$. By \cite[Thm. 8.2.22(4)]{kedlaya2015relative}, vector bundles on $X_{\an}$ and $X_{\et}$ are equivalent, and if $X = \Spa(A,A^+)$ is affinoid, these are further equivalent to finite projective $A$-modules. To a vector bundle $E$ on $X$, we can associate its associated geometric vector bundle $E \otimes_{\Os_X} \G_a$, which is again sousperfectoid or a rigid space.

Let now $Y$ be a locally spatial diamond over $\Spd(\Z_p)$. We consider the corresponding untilted structure sheaf $\Os_{Y_v}$, sending a characteristic $p$ perfectoid space $S \rightarrow Y$ to $\Os(S^{\sharp})$, where $S^{\sharp}$ is the untilt of $S$ corresponding to the composition $S \rightarrow Y \rightarrow \Spd(\Z_p)$. There also is an obvious integral variant $\Os_{Y_v}^{+}$. These extend to sheaves on $Y_v$ and we denote by $\Os_{Y_{\et}}$, $\Os_{Y_{\et}}^+$ their restriction to the étale site of $Y$.

\begin{definition}
    Let $\tau \in \{\et, v\}$. A $\tau$-vector bundle on $Y$ is defined to be a finite locally free sheaf of $\Os_{Y_{\tau}}$-modules on $Y_{\tau}$. We denote by $\Vect(Y_{\tau})$ the category of $\tau$-vector bundles on $Y$.
\end{definition}

Let $\nu\colon Y_{v} \rightarrow Y_{\et}$ denote the natural map of sites. Then there is a natural functor
\begin{align}\label{eq: from etale to v vb} \Vect(Y_{\et}) \rightarrow \Vect(Y_{v}),\end{align}
sending a vector bundle $E$ on $Y_{\et}$ to
\[ \nu^{*}E \otimes_{\nu^*\Os_{Y_{\et}}}\Os_{Y_{v}} \in \Vect(Y_{v}).\]
We will often abbreviate and simply write $E \otimes_{\Os_{Y_{\et}}}\Os_{Y_{v}}$. By \cite[Prop. 2.15-2.16]{MannWerner2022loc}, the above functor is fully faithful. We will often identify vector bundles on $Y_{\et}$ with their extensions to $Y_v$. Moreover, if $Y$ is a perfectoid space, then the above functor defines an equivalence of categories
\begin{align}\label{eq: etale and v vb equivalent for perfectoids}
\Vect(Y_{\et}) \cong \Vect(Y_{v}).
\end{align}
This is \cite[Thm. 3.5.8]{kedlaya2019relative2} if $Y$ lives over $\Q_p$, and \cite[Lemma 17.1.8]{SW20} in general.

For any locally noetherian adic spaces $X$ over $\Q_p$, we may also consider vector bundles on the flattened pro-étale site $X_{\proet}$ of Scholze \cite{scholze2013padicHodge}, equipped with the completed structure sheaf $\Os_{X_{\proet}}$. By \cite[Prop. 2.3.2]{heuer2022gtorsors}, we further have an equivalence
\begin{align}
    \Vect(X_{\proet}) \cong \Vect(X_{v}^{\diamondsuit}).
\end{align}

In general, étale vector bundles on an adic space $X$ and on its associated diamond do not coincide. This leads to the following definition.
\begin{definition}\label{def: good adic spaces}
    Let $X$ be an adic space over $\Q_p$ that is either a sousperfectoid space or a rigid space over some non-archimedean field $(K,K^+)$. We say that $X$ is good if we have an isomorphism of sheaves 
\[ \Os_{X_{\et}} = \nu_*\Os_{X_v},\]
where $\nu\colon X_v \rightarrow X_{\et}$ is the canonical map of sites. We denote by $\Adic_{\Q_p}$ the category of good adic spaces over $\Q_p$.
\end{definition}
Examples of good adic spaces include perfectoid spaces and seminormal rigid spaces. By \cite[Prop. 2.4-2.5]{gerth2024}, the category of good adic spaces is closed under smooth maps and the diamond functor restricts to a fully faithful embedding
\begin{align}\label{eq: diamond functor ff on good adic spaces}
     (\cdot)^{\diamondsuit} \colon \Adic_{\Q_p} \rightarrow \LSD_{\Q_p}.
\end{align}
Moreover, it is clear by definition that for $X$ a good adic space, we have an equivalence of ringed sites $(X_{\et},\Os_{X_{\et}}) \cong (X_{\et}^{\diamondsuit}, \Os_{X_{\et}^{\diamondsuit}})$ and thus an equivalence
\begin{align}
    \Vect(X_{\et}) \cong \Vect(X_{\et}^{\diamondsuit}).
\end{align}

We now discuss the Hodge--Tate spectral sequence of Faltings \cite[III Thm. 4.1]{Fal88} and Scholze \cite[Thm. 1.6]{scholze2013padicHodge}. We will need the following relative variant, considered in \cite[§2.2]{ScholzeCariani2017}, \cite{ABBES_2024}, \cite[§7.1]{gaisin2022relativearminfcohomology} and \cite{heuer2024relative}.
\begin{thm}{(\cite{heuer2024relative})}\label{thm: relative linear HT sequence}
    Let $(K,K^+)$ be a non-archimedean field over $\Q_p$. Let $X \rightarrow S$ be a proper smooth morphism of reduced rigid spaces over $(K,K^+)$. Then for $\tau\in \{ \et,v\}$ and $n\geq 0$, the sheaf
    \[ R^n\pi_{\tau,*}\Os_{X_\tau}\colon S_{\tau} \rightarrow \Ab\]
    is a $\tau$-vector bundle on $S$ whose formation commutes with any morphism $g\colon S' \rightarrow S$, where $S'$ is a rigid space or a sousperfectoid space. We have a spectral sequence of $v$-vector bundles
    \begin{align} \mathbf{E}_2^{ij} = (R^i\pi_{\et,*}\Omega_{X/S}^j) \otimes_{\Os_{S_{\et}}} \Os_{S_v}(-j) \Rightarrow (R^{i+j}\pi_{v,*}\Q_p) \otimes_{\Q_p} \Os_{S_v}.\end{align}
    Moreover, this sequence degenerates on the second page.
\end{thm}
\begin{proof}
    The first part is \cite[Thm. 3.18, Thm. 5.7.1]{heuer2024relative} for $\tau=\et$ and \cite[Corollary 4.6]{heuer2024relative} when $\tau=v$. The existence and degeneration of the spectral sequence is \cite[Cor. 5.12]{heuer2024relative} under the assumption that $K$ is perfectoid\footnote{In \cite{heuer2024relative}, Heuer works with Breuil--Kisin--Fargues twists $\Os_v\{-j\}$ (Definition 2.4 in \emph{loc. cit.}) instead of Tate twists $\Os_v(-j)$. However, these are canonically isomorphic, by adapting the proof of \cite[Prop. 6.7]{scholze2013padicHodge}.}. It is proven in \cite[Thm. 5.1]{heuer2024relative} that the sequence is natural in $S$, where we allow arbitrary morphisms of adic spaces between rigid spaces over different perfectoid fields. Let now $K$ be arbitrary and fix a completed algebraic closure $C$. The terms of the sequence for $X_C \rightarrow S_C$ are clearly defined over $S$. Using the above naturality property for the semilinear maps $(1\times\sigma) \colon S_C \rightarrow S_C$, where $\sigma \in \Gal_C$, we find that the maps in the spectral sequence over $C$ descend to $S$, as required
\end{proof}

\begin{example}\label{ex: HT sequence}
    Let $\pi\colon X \rightarrow S$ be a proper smooth morphism of reduced rigid spaces. Then for low degree $i+j=1$, the relative spectral sequence of Theorem \ref{thm: relative linear HT sequence} yields the following short exact sequence of $v$-vector bundles, the so-called Hodge--Tate sequence
% https://q.uiver.app/#q=WzAsNSxbMCwwLCIwIl0sWzEsMCwiKFJeMVxccGlfe1xcZXQsKn1cXE9zX1gpIFxcb3RpbWVzX3tcXE9zX3tTX3tcXGV0fX19IFxcT3Nfe1Nfdn0iXSxbMiwwLCIoUl4xXFxwaV97diwqfVxcdW5kZXJsaW5le1xcWl9wfSkgXFxvdGltZXNfe1xcdW5kZXJsaW5le1xcWl9wfX0gXFxPc197U192fSJdLFszLDAsIlxccGlfeyp9XFxPbWVnYV97WC9TfV4xXFxvdGltZXNfe1xcT3Nfe1Nfe1xcZXR9fX0gXFxPc197U192fSgtMSkiXSxbNCwwLCIwLiJdLFswLDFdLFsyLDMsIlxcSFRfe1gvU30iXSxbMyw0XSxbMSwyXV0=
\begin{equation}\label{eq: contravariant HT seq}
    \begin{tikzcd}
	0 & {(R^1\pi_{\et,*}\Os_X) \otimes_{\Os_{S}}\!\Os_{S_v}} & {(R^1\pi_{v,*}\underline{\Z_p}) \otimes\!\Os_{S_v}} & {\pi_{*}\Omega_{X/S}^1\otimes_{\Os_{S}} \!\Os_{S_v}(-1)} & {0.}
	\arrow[from=1-1, to=1-2]
	\arrow[from=1-2, to=1-3]
	\arrow["{\HT_{X/S}}", from=1-3, to=1-4]
	\arrow[from=1-4, to=1-5]
\end{tikzcd}
\end{equation}
\end{example}

We end this subsection by stating Scholze's Primitive Comparison Theorem.
\begin{thm}{(\cite[Thm. 1.3]{scholze2013padicHodge})}\label{thm: Primitive Comparison Theorem}
\begin{enumerate}
    \item Let $X$ be a proper smooth rigid space over $(K,K^+)$. Let $\Lb$ be a $\Z_p$-local system on $X$. Then there is a natural, Galois equivariant isomorphism
    \begin{align} H_v^n(X_C,\Lb)\otimes_{\Z_p} C \xrightarrow{\cong} H_v^n(X_C,\Lb \otimes_{\underline{\Z_p}}\Os_{X_v}). \end{align}
    \item Let $\pi\colon X \rightarrow S$ be a proper smooth map of analytic adic spaces over $\Z_p$, and let $\Lb$ be a $\Z_p$-local system on $X$. Then we have a natural isomorphism
    \begin{align}\label{eq: prim comp thm} (R^n\pi_{v,*}\Lb) \otimes_{\underline{\Z_p}}\Os_{S_v} \xrightarrow{\cong} R^n\pi_{v,*}(\Lb \otimes_{\underline{\Z_p}}\Os_{X_v}).\end{align}
\end{enumerate}
\end{thm}
\begin{proof}
    It is enough to show the second point. Let $\Lb_m = \Lb/p^m$. The crucial point is the isomorphism
    \[ (R^n\pi_{v,*}\Lb_1) \otimes_{\F_p}\Os_{S_v}^+/p \xrightarrow{\cong} R^n\pi_{v,*}(\Lb_1 \otimes_{\F_p}\Os_{X_v}^+/p),\]
    which is \cite[Thm. 1.3]{scholze2013padicHodge} if $X$ and $S$ are rigid spaces in characteristic $0$ and \cite[Thm. 4.3]{heuer2024primitive} in general. The statement of the theorem follows from this by devissage, arguing as in the proof of \cite[Thm. 4.4]{heuer2024relative}.
\end{proof}

\subsection{Hodge--Tate and de Rham local systems}
In this subsection, we recall the notion of de Rham and Hodge--Tate local systems, following \cite{scholze2013padicHodge}\cite{Liu2017}\cite{Shimizu_2018}.

We fix a smooth rigid space $X$ over a $p$-adic field $K$. We consider the pro-étale site $X_{\proet}$ of Scholze \cite[§3]{scholze2013padicHodge}, endowed with its completed structure sheaf $\Os_{X_{\proet}}$. We will shorten the notation and write $\Os_{X_{\et}}$ for the sheaf $\nu^{*}\Os_{X_{\et}}$ on $X_{\proet}$ and similarly for the integral variants, where $\nu\colon X_{\proet} \rightarrow X_{\et}$ is the natural map of sites. In \cite[§6]{scholze2013padicHodge}, Scholze further introduces the pro-étale de Rham period sheaves  $\B_{\dR}^+$ and $\B_{\dR} = \B_{\dR}^+[\xi^{-1}]$, where $\xi$ is a local generator of Fontaine's Theta map $\theta \colon \B_{\dR}^+ \rightarrow \Os_{X_{\proet}}$. We will endow $\B_{\dR}^+$ with its natural $\xi$-adic filtration and we also set
\[ \Fil^i\B_{\dR} = \xi^i\B_{\dR}^+, \quad \fa i \in \Z.\]
The map $\theta$ yields an isomorphism \cite[Prop. 6.7]{scholze2013padicHodge}
\[ \gr^i\B_{\dR} = \Os_{X_{\proet}}(i).\]
If $k$ is the residue field of $K$, there is an additional proétale sheaf $\OBdR+$ which is, roughly, a certain completion of $\Os_{X_{\et}} \otimes_{W(k)}\B_{\dR}^+$, see \cite[Def. 6.6]{bosco2023padicproetalecohomologydrinfeld} for a precise formula. The sheaf $\OBdR+$ comes equipped with a natural map $\theta\colon \OBdR+ \rightarrow \Os_{X_{\proet}}$ obtained by tensoring the theta map on $\B_{\dR}^+$ with the inclusion $\Os_{X_{\et}} \rightarrow  \Os_{X_{\proet}}$. We endow it with the filtration
\[ \Fil^i\OBdR+ = (\Ker(\theta))^i, \quad i\geq 0.\]
The sheaf $\OBdeR$ is defined as the completion of $\OBdR+[\xi^{-1}]$ with respect to the filtration
\[ \Fil^i = \sum_{j\in \Z}\Fil^{i+j}(\OBdR+) \xi^{-j}, \quad i\in \Z.\]

The sheaf $\OBdeR$ is naturally endowed with a $\B_{\dR}$-linear connection
\[ \nabla = d\otimes \id\colon \OBdeR \rightarrow \OBdeR \otimes_{\Os_{X_{\et}}} \Omega_{X/K}^1\]
which satisfies Griffiths transversality, i.e.
\[ \nabla(\Fil^i\OBdeR) \sub \Fil^{i-1}\OBdeR.\]

% The decisive property satisfied by the pair $(\OBdR+,\nabla)$ is the following $p$-adic version of the Poincaré Lemma.
% \begin{proposition}{(\cite
% [Cor. 6.13]{scholze2013padicHodge})}\label{prop: Poincaré lemma}
%     Let $X$ be a smooth rigid space over a $p$-adic field $K$, then there is an exact sequence on $X_{\proet}$
%     % https://q.uiver.app/#q=WzAsNixbMCwwLCIwIl0sWzEsMCwiXFxCX3tcXGRSfV4rIl0sWzIsMCwiXFxPQmRSKyJdLFszLDAsIlxcT0JkUitcXG90aW1lc197XFxPc197WF97XFxldH19fVxcT21lZ2Ffe1gvS31eMSJdLFs0LDAsIlxcT0JkUitcXG90aW1lc197XFxPc197WF97XFxldH19fVxcT21lZ2Ffe1gvS31eMiJdLFs1LDAsIlxcbGRvdHMiXSxbMCwxXSxbMyw0LCJcXG5hYmxhIl0sWzQsNSwiXFxuYWJsYSJdLFsyLDMsIlxcbmFibGEiXSxbMSwyXV0=
% \begin{equation}\begin{tikzcd}
% 	0 & {\B_{\dR}^+} & {\OBdR+} & {\OBdR+\otimes_{\Os_{X_{\et}}}\Omega_{X/K}^1} & {\OBdR+\otimes_{\Os_{X_{\et}}}\Omega_{X/K}^2} & \ldots
% 	\arrow[from=1-1, to=1-2]
% 	\arrow[from=1-2, to=1-3]
% 	\arrow["\nabla", from=1-3, to=1-4]
% 	\arrow["\nabla", from=1-4, to=1-5]
% 	\arrow["\nabla", from=1-5, to=1-6]
% \end{tikzcd}\end{equation} 
% Moreover, there exists an analog exact sequence for the pro-étale sheaves $\B_{\dR}$, $\OBdeR$. 
% \end{proposition}
% If $\nu\colon X_{\proet} \rightarrow X_{\et}$ is the map of sites, we have
% \begin{align}\label{eq: pushforward of period sheaves} 
% \nu_*\OBdR+ = \nu_*\OBdeR = \Os_{X_{\et}}
% \end{align}
% and pushing forward either of the exact sequence of Proposition \ref{prop: Poincaré lemma} along $\nu$ recovers the usual de Rham complex.

We now review de Rham local systems. Let $(\Es,\Fil^{\bullet},\nabla)$ be a filtered vector bundle with integrable connection, that is, $\Es$ is a vector bundle on $X_{\et}$ with a finite, separated and exhaustive filtration $\Fil^{\bullet}$ and an integrable connection 
\[ \nabla\colon \Es \rightarrow \Es\otimes_{\Os_{X_{\et}}}\Omega_{X/K}^1\]
satisfying Griffiths transversality with respect to $\Fil^{\bullet}$. Scholze considers the functor
\begin{align}\label{eq: from MIC to B+ modules}
(\Es,\Fil^{\bullet},\nabla) \mapsto \Fil^0(\Es \otimes_{\Os_{X_{\et}}} \OBdeR)^{\nabla=0},
\end{align}
where we take the tensor products of the connections and filtrations on the right-hand side. Scholze proves \cite[Thm. 7.6]{scholze2013padicHodge} that this formula yields a finite-locally free $\B_{\dR}^+$-module with same rank as $\Es$.
\begin{definition}{(\cite[Def. 8.3]{scholze2013padicHodge})}
    A $\Q_p$-local system $\V$ is called de Rham if there exists a filtered bundle with integrable connection $(\Es,\Fil^{\bullet},\nabla)$ and an isomorphism
\begin{align} \V\otimes_{\Q_p}\OBdeR = \Es \otimes_{\Os_{X_{\et}}} \OBdeR,\end{align}
compatible with filtrations and connections. Equivalently, by \cite[Thm. 7.6(i)]{scholze2013padicHodge}, there exists an isomorphism of $\B_{\dR}^+$-modules
\[ \V\otimes_{\Q_p}\B_{\dR}^+ = \Fil^0(\Es \otimes_{\Os_{X_{\et}}} \OBdeR)^{\nabla=0}.\]
\end{definition}

In the other direction, given a pro-étale $\Q_p$-local system $\V$ on $X$, Liu--Zhu \cite[§3.2]{Liu2017} attach to $\V$ a vector bundle on $X_{\et}$
\begin{align}\label{eq: DdR} D_{\dR}(\V) = \nu_*(\V \otimes_{\Q_p}\OBdeR),\end{align}
equipped with a filtration
\[ \Fil^iD_{\dR}(\V) = \nu_*(\V \otimes_{\Q_p}\Fil^i\OBdeR)\]
and a connection
\[ \nabla_{\V} = \nu_*(\nabla)\colon D_{\dR}(\V) \rightarrow D_{\dR}(\V)\otimes_{\Os_{X_{\et}}} \Omega_{X/K}^1\]
satisfying Griffiths transversality. By \cite[Thm. 7.6(i)]{scholze2013padicHodge}, if $\V$ is de Rham, then its associated tuplet is $(D_{\dR}(\Lb),\Fil^i, \nabla_{\V})$. From this, one can show that a $\Q_p$-local system $\V$ is de Rham if and only if $D_{\dR}(\V)$ and $\V$ have the same rank.

We record the following lemma which we will use later in Subsection \ref{subsection: Comparison with de Rham cohomology}.
\begin{lemma}{(\cite[Prop. 7.9]{scholze2013padicHodge})}\label{lemma: B+ lattice of Scholze}
    Let $\V$ be a de Rham $\Q_p$-local system on a smooth rigid space $X/K$ and define
    \[ \Xi = (D_{\dR}(\V) \otimes_{\Os_{X_{\et}}} \OBdR+)^{\nabla=0}.\]
    \begin{enumerate}
        \item The pro-étale sheaf $\Xi$ is a $\B_{\dR}^+$-lattice in $\V\otimes_{\Q_p} \B_{\dR}$ and comes with an isomorphism
        \[ \Xi\otimes_{\B_{\dR}^+,\theta}\Os_{X_{\proet}} \cong D_{\dR}(\V) \otimes_{\Os_{X_{\et}}} \Os_{X_{\proet}}.\]
        \item We have
    \begin{align*}
    \frac{(\V\otimes_{\Q_p} \B_{\dR}^+) \cap \xi^{i}\Xi}{(\V\otimes_{\Q_p} \B_{\dR}^+) \cap \xi^{i+1}\Xi} = \Fil^{-i}D_{\dR}(\V) \otimes_{\Os_{X_{\et}}} \Os_{X_{\proet}}(i)\end{align*}
    through the isomorphism 
    \[ \frac{\xi^i\Xi}{\xi^{i+1}\Xi} \cong D_{\dR}(\V) \otimes_{\Os_{X_{\et}}} \Os_{X_{\proet}}(i).\]
    \end{enumerate}
\end{lemma}

We now move on to the Hodge--Tate theory. Consider the following sheaf on $X_{\proet}$
\[ \OBHT = \gr^{\bullet}\OBdeR \cong \bigoplus_{i \in Z} \OC(i),\]
where we set
\[\OC = \gr^0\OBdeR.\]
The sheaf $\OBHT$ comes with a Higgs field
\[ \Theta = \gr^{\bullet}\nabla\colon \OBHT  \rightarrow \OBHT \otimes_{\Os_{X_{\et}}} \Omega_{X/K}^1\]
that is homogeneous of degree $-1$ with respect to the above grading. It thus restricts to a Higgs field
\[ \Theta\colon \OC \rightarrow \OC\otimes_{\Os_{X_{\et}}} \Omega_{X/K}^1(-1).\]
In \cite{Liu2017} and \cite{Hyodo1989}, the authors attach to a $\Q_p$-local system $\V$ on $X$ the coherent sheaf
\begin{align}\label{eq: DHT} D_{\HT}(\V) = \nu_*(\V\otimes_{\Q_p}\OBHT).\end{align}
It naturally comes with a grading 
\[ \gr^iD_{\HT}(\V) = \nu_*(\V\otimes_{\Q_p}\OC(i))\]
and a Higgs field
\[ \theta_{\V}\colon D_{\HT}(\V) \rightarrow D_{\HT}(\V)\otimes \Omega_{X/K}^1\]
that is homogeneous of degree $-1$ with respect to the above grading.
\begin{definition}{(\cite[Def. 5.6]{Shimizu_2018})}\label{def: HT local system}
A $\Q_p$-local system $\V$ is called Hodge--Tate if $D_{\HT}(\V)$ is a vector bundle of rank equal to $\rank_{\Q_p} \V$. Equivalently, by \cite[Thm. 5.5]{Shimizu_2018}, the natural map
\begin{align}\label{eq: HT comparison theorem} D_{\HT}(\V) \otimes_{\Os_{X_{\et}}} \OBHT \rightarrow \V \otimes_{\Q_p}\OBHT,\end{align}
is an isomorphism, which can be seen to be compatible with the grading and the Higgs fields, where we take the product Higgs field on the left-hand side
\[ \theta = \theta_{\V} \otimes \id + \id \otimes \Theta\colon D_{\HT}(\V) \otimes_{\Os_{X_{\et}}} \OBHT \rightarrow D_{\HT}(\V) \otimes_{\Os_{X_{\et}}} \OBHT \otimes_{\Os_{X_{\et}}} \Omega_{X/K}^1.\]
The weights of a Hodge--Tate local system $\V$ are defined to be the elements of the set
\[ \{\,i \in \Z \mid \gr^{-i}D_{\HT}(\V) \neq 0\,\}.\]
\end{definition}

Our input is the following characterization of Hodge--Tate local systems, which should be reminiscent of Hodge--Tate Galois representations.
\begin{proposition}\label{prop: characterization of HT ls}
    Let $\V$ be a $\Q_p$-local on a smooth rigid space $X/K$. Then $\V$ is Hodge--Tate, with weights $\in [a,b]$ for integers $a \leq b$, if and only if there exists a decreasing filtration
    \[ 0=\Fil^{-a+1} \sub \Fil^{-a} \sub \ldots \sub \Fil^{-b} =\V \otimes_{\Q_p}\Os_{X_{\proet}} \]
    by proétale sub-vector bundles such that
    \[ \Fil^{i}/\Fil^{i+1} \cong E_{i} \otimes_{\Os_{X_{\et}}}\Os_{X_{\proet}}(i+a+b),\quad \fa i,\]
    for étale vector bundles $E_i$ on $X$. Moreover, if these conditions hold, then $\Fil^{\bullet}$ and $E_{\bullet}$ are unique for this property and given by
    \[ \Fil_{\HT}^i = \Big(\bigoplus_{j\leq -i-a-b}\gr^{j}D_{\HT}(\V)\otimes \OC(-j)\Big)^{\theta=0}, \quad  E_i = \gr^{-i-a-b}D_{\HT}(\V).\]
\end{proposition} 
\begin{definition}\label{def: HT filtration}
    Let $\V$ be a Hodge--Tate local system on $X$. We call the filtration $\Fil_{\HT}^{\bullet} \sub \V \otimes_{\Q_p} \Os_{X_{\proet}}$ constructed in Proposition \ref{prop: characterization of HT ls} the Hodge--Tate filtration.
\end{definition}
\begin{proof}
For clarity, let us denote by $\sigma$ the unique order-reversing permutation of $\{-b,\ldots,-a\}$, given by
\[ \sigma(i) = -i-a-b. \]
Assume that $\V$ is a Hodge--Tate local system and consider the vector bundles
\[ D^i = \bigoplus_{j\leq \sigma(i)} \gr^{j}D_{\HT}(\V).\]
Then since the Higgs field $\theta_{\V}$ on $D_{\HT}(\V)$ decreases the rank by $1$, we see that $\theta_{\V}(D^i) \sub D^{i+1} \sub D^i$, so that $(D^i,\theta_{\V})$ are sub-Higgs bundles of $(D_{\HT}(\V),\theta_{\V})$. Using Scholze's version of the Poincaré Lemma \cite[Cor. 6.13]{scholze2013padicHodge}, it can be seen that the analog functor to (\ref{eq: from MIC to B+ modules}), sending a graded Higgs bundle $(E,\gr^{\bullet},\theta)$ to
\[ \gr^0(E \otimes_{\Os_{X_{\et}}} \OBHT)^{\theta=0}\]
takes values in pro-étale vector bundles and is exact. From this, the sheaves
\[ \Fil_{\HT}^i =\gr^0(D^i \otimes_{\Os_{X_{\et}}} \OBHT)^{\theta=0}\]
are pro-étale vector bundles and the short exact sequences
% https://q.uiver.app/#q=WzAsNSxbMCwwLCIwIl0sWzEsMCwiKERee2krMX0sXFx0aGV0YV97XFxWfSkiXSxbMiwwLCIoRF57aX0sXFx0aGV0YV97XFxWfSkiXSxbMywwLCIoXFxncl57XFxzaWdtYShpKX1EX3tcXEhUfShcXFYpLDApIl0sWzQsMCwiMCJdLFswLDFdLFsxLDJdLFsyLDNdLFszLDRdXQ==
\[\begin{tikzcd}
	0 & {(D^{i+1},\theta_{\V})} & {(D^{i},\theta_{\V})} & {(\gr^{\sigma(i)}D_{\HT}(\V),0)} & 0
	\arrow[from=1-1, to=1-2]
	\arrow[from=1-2, to=1-3]
	\arrow[from=1-3, to=1-4]
	\arrow[from=1-4, to=1-5]
\end{tikzcd}\]
give rise to short exact sequences
% https://q.uiver.app/#q=WzAsNSxbMCwwLCIwIl0sWzEsMCwiXFxGaWxfe1xcSFR9XntpKzF9Il0sWzIsMCwiXFxGaWxfe1xcSFR9XntpfSJdLFszLDAsIlxcZ3Jee1xcc2lnbWEoaSl9RF97XFxIVH0oXFxWKVxcb3RpbWVzIFxcT3Nfe1hfe1xccHJvZXR9fSgtXFxzaWdtYShpKSkiXSxbNCwwLCIwLiJdLFswLDFdLFsxLDJdLFsyLDNdLFszLDRdXQ==
\[\begin{tikzcd}
	0 & {\Fil_{\HT}^{i+1}} & {\Fil_{\HT}^{i}} & {\gr^{\sigma(i)}D_{\HT}(\V)\otimes \Os_{X_{\proet}}(-\sigma(i))} & {0.}
	\arrow[from=1-1, to=1-2]
	\arrow[from=1-2, to=1-3]
	\arrow[from=1-3, to=1-4]
	\arrow[from=1-4, to=1-5]
\end{tikzcd}\]
This shows that $\Fil_{\HT}^{\bullet}$ meets the requirements of the statement.

Conversely, assume that we are given a filtration on $\V\otimes_{\Q_p} \Os_{X_{\proet}}$ as in the statement. We show that $E_i \cong \gr^{\sigma(i)}D_{\HT}(\V)$, which by a dimension count will imply that $\V$ is Hodge--Tate. For this, consider the exact sequence of proétale vector bundles
% https://q.uiver.app/#q=WzAsNSxbMCwwLCIwIl0sWzEsMCwiXFxGaWxee2krMX0iXSxbMiwwLCJcXEZpbF57aX0iXSxbMywwLCJFX2lcXG90aW1lc197XFxPc197WF97XFxldH19fVxcT3Nfe1hfe1xccHJvZXR9fSgtXFxzaWdtYShpKSkiXSxbNCwwLCIwLiJdLFswLDFdLFsxLDJdLFsyLDNdLFszLDRdXQ==
\[\begin{tikzcd}
	0 & {\Fil^{i+1}} & {\Fil^{i}} & {E_i\otimes_{\Os_{X_{\et}}}\Os_{X_{\proet}}(-\sigma(i))} & {0.}
	\arrow[from=1-1, to=1-2]
	\arrow[from=1-2, to=1-3]
	\arrow[from=1-3, to=1-4]
	\arrow[from=1-4, to=1-5]
\end{tikzcd}\]
After twisting and tensoring along $\Os_{X_{\proet}} \rightarrow \OC$, we get the exact sequence
% https://q.uiver.app/#q=WzAsNSxbMCwwLCIwIl0sWzEsMCwiXFxGaWxee2krMX1cXG90aW1lc1xcT0MoXFxzaWdtYShpKSkiXSxbMiwwLCJcXEZpbF57aX1cXG90aW1lc1xcT0MoXFxzaWdtYShpKSkiXSxbMywwLCJFX2lcXG90aW1lc197XFxPc197WF97XFxldH19fVxcT0MiXSxbNCwwLCIwLiJdLFswLDFdLFsxLDJdLFsyLDNdLFszLDRdXQ==
\[\begin{tikzcd}
	0 & {\Fil^{i+1}\otimes\OC(\sigma(i))} & {\Fil^{i}\otimes\OC(\sigma(i))} & {E_i\otimes_{\Os_{X_{\et}}}\OC} & {0.}
	\arrow[from=1-1, to=1-2]
	\arrow[from=1-2, to=1-3]
	\arrow[from=1-3, to=1-4]
	\arrow[from=1-4, to=1-5]
\end{tikzcd}\]
We have
\[ \nu_*(\Fil^i \otimes_{\Os_{X_{\proet}}} \OC(\sigma(i))) \sub \nu_*(\V \otimes_{\Q_p} \OC(\sigma(i))) = \gr^{\sigma(i)}D_{\HT}(\V)\]
and we claim that this inclusion is an equality. By decomposing the inclusion $\Fil^i \hookrightarrow \V \otimes_{\Q_p} \Os_{X_{\proet}}$ into short exact sequences, it follows 
from the fact that, by \cite[Prop. 6.16(ii)]{scholze2013padicHodge}, for any $m\geq 0$, $n\neq 0$, and any vector bundle $E$ on $X_{\et}$,
\begin{align*}R^m\nu_*(E\otimes_{\Os_{X_{\et}}} \OC(n)) = E\otimes_{\Os_{X_{\et}}}R^m\nu_*\OC(n) = 0.
\end{align*}
Using these formulas, we also see that, in the following long exact sequence obtained by applying $R\nu*$ to the above short exact sequence
% https://q.uiver.app/#q=WzAsNSxbMCwwLCIwIl0sWzEsMCwiXFxudV8qKFxcRmlsXntpKzF9XFxvdGltZXNfe1xcT3Nfe1hfe1xccHJvZXR9fX1cXE9DKFxcc2lnbWEoaSkpKSJdLFsyLDAsIlxcZ3Jee1xcc2lnbWEoaSl9RF97XFxIVH0oXFxWKSJdLFszLDAsIlxcbnVfKihFX2lcXG90aW1lc197XFxPc197WF97XFxldH19fVxcT0MpPUVfaSJdLFsxLDEsIlJeMVxcbnVfKihcXEZpbF57aSsxfVxcb3RpbWVzX3tcXE9zX3tYX3tcXHByb2V0fX19XFxPQyhcXHNpZ21hKGkpKSksIl0sWzAsMV0sWzEsMl0sWzIsM10sWzMsNF1d
\[\begin{tikzcd}
	0 & {\nu_*(\Fil^{i+1}\otimes_{\Os_{X_{\proet}}}\OC(\sigma(i)))} & {\gr^{\sigma(i)}D_{\HT}(\V)} & {\nu_*(E_i\otimes_{\Os_{X_{\et}}}\OC)=E_i} \\
	& {R^1\nu_*(\Fil^{i+1}\otimes_{\Os_{X_{\proet}}}\OC(\sigma(i))),}
	\arrow[from=1-1, to=1-2]
	\arrow[from=1-2, to=1-3]
	\arrow[from=1-3, to=1-4]
	\arrow[from=1-4, to=2-2]
\end{tikzcd}\]
the first and last terms vanish. This proves that $E_i \cong \gr^{\sigma(i)}D_{\HT}(\V)$. Let us now show that $\Fil^i = \Fil_{\HT}^i$. We argue inductively. For $i=-b$, this is clear, since both are equal to $\V \otimes_{\Q_p} \Os_{X_{\proet}}$. Assume that the statement holds for $-b \leq i \leq i_0$ and let us show that it holds for $i=i_0+1$. Observe that the following composition vanishes
% https://q.uiver.app/#q=WzAsMyxbMCwwLCJcXEZpbF57aV8wKzF9Il0sWzEsMCwiXFxGaWxee2lfMH09XFxGaWxfe1xcSFR9XntpfSJdLFsyLDAsIlxcZ3Jee1xcc2lnbWEoaV8wKX1EX3tcXEhUfShcXFYpXFxvdGltZXNcXE9DKC1cXHNpZ21hKGlfMCkpLiJdLFswLDFdLFsxLDJdXQ==
\[\begin{tikzcd}
	{\Fil^{i_0+1}} & {\Fil^{i_0}=\Fil_{\HT}^{i_0}} & {\gr^{\sigma(i_0)}D_{\HT}(\V)\otimes\OC(-\sigma(i_0)).}
	\arrow[from=1-1, to=1-2]
	\arrow[from=1-2, to=1-3]
\end{tikzcd}\]
Indeed, by decomposing $\Fil^{i_0+1}$ further, it follows from the fact that there are no non-zero maps of pro-étale vector bundles
\[ E_i\otimes \Os_{X_{\proet}}(-\sigma(i)) \rightarrow \gr^{\sigma(i_0)}D_{\HT}(\V)\otimes \Os_{X_{\proet}}(-\sigma(i_0)),\]
for $i> i_0$. This shows that $\Fil^{i_0+1} \sub \Fil_{\HT}^{i_0+1}$. As the argument is symmetric, we obtain that $\Fil^{i_0+1} = \Fil_{\HT}^{i_0+1}$. This shows that the filtration is unique and concludes the proof.
\end{proof}
\begin{remark}
Note that the Hodge--Tate filtration on $\V\otimes_{\Q_p}\Os_{X_{\proet}}$ of Proposition \ref{prop: characterization of HT ls} generally does not split, unlike in the case $X=\Spa(K)$. Indeed, it is easy to see that if $\V$ is a local system such that
\[ \V \otimes_{\Q_p} \Os_{X_{\proet}} = \bigoplus_{i\in \Z} F_i\otimes_{\Os_{X_{\et}}} \Os_{X_{\proet}}(i),\]
for étale vector bundles $F_i$, then necessarily
\[ F_i = \gr^{-i}D_{\HT}(\V)^{\theta_{\V}=0}.\]
By a dimension count, this is only possible if $\theta_{\V}=0$, since $\V$ is Hodge--Tate by Proposition \ref{prop: characterization of HT ls}. In fact, given a Hodge--Tate local system $\V$, one can show that the Hodge--tate filtration splits if and only if $\theta_{\V} =0$.
\end{remark}

We can recover the Hodge--Tate filtration on the relative $p$-adic cohomology of a smooth proper morphism from the Hodge--Tate spectral sequence.
\begin{lemma}\label{lemma: comparison between the two HT filtrations}
Let $\pi \colon X \rightarrow S$ be a proper smooth morphism of smooth rigid space over a $p$-adic field $K$. Then $R^n\pi_{\proet,*}\Q_p$ is a Hodge--Tate $\Q_p$-local system, and the Hodge--Tate filtration from Definition \ref{def: HT filtration} coincides with the filtration induced by the Hodge--Tate spectral sequence of Theorem \ref{thm: relative linear HT sequence}.
\end{lemma}
\begin{proof}
    Observe that as the relative Hodge--Tate spectral sequence degenerates, the induced filtration $\Fil^i\sub R^n\pi_{\proet,*}\Q_p \otimes \Os_v$ takes the form
    \[ 0 = \Fil^{n+1} \sub \Fil^n \sub \ldots \Fil^0=R^n\pi_{\proet,*}\Q_p \otimes \Os_v \]
    with associated graded
    \[ \Fil^i/\Fil^{i+1} \cong R^i\pi_{\et,*}\Omega_{X/S}^{n-i}\otimes_{\Os_{X_{\et}} }\Os_{X_{\proet}}(i-n).\]
    It follows from the uniqueness part of Proposition \ref{prop: characterization of HT ls} that $\Fil^{\bullet} = \Fil_{\HT}^{\bullet}$.
\end{proof}

For de Rham local systems, we can relate the $\B_{\dR}^+$-lattice of Lemma \ref{lemma: B+ lattice of Scholze} and the Hodge--Tate filtration.

\begin{lemma}\label{lemma: BdR+-lattice induces HT filtration in arithmetic case}
    Let $X/K$ be a smooth rigid space and let $\V$ be a de Rham $\Q_p$-local system on $X$ and assume that the filtration on $D_{\dR}(\V)$ has range $[a,b]$, for integers $a\leq b$. Let $\Xi \sub \V \otimes_{\Q_p}\B_{\dR}$ be the $\B_{\dR}^+$-lattice of Lemma \ref{lemma: B+ lattice of Scholze} and let $\Fil_{\HT}^{\bullet} \sub \V \otimes_{\Q_p} \Os_{X_{\proet}}$ be the Hodge--Tate filtration of Definition \ref{def: HT filtration}. Then we have
    \begin{align*} 
    \frac{\xi^i\Xi \cap  (\V \otimes_{\Q_p}\B_{\dR}^+)}{\xi^i\Xi \cap  \xi(\V \otimes_{\Q_p}\B_{\dR}^+)} =\Fil_{\HT}^{i-b-a} , \quad \fa i \in \Z.
    \end{align*}
\end{lemma}
\begin{proof}
From Lemma \ref{lemma: B+ lattice of Scholze}(3), we have the following inclusions
\[  \xi^{-a}\Xi \sub \V \otimes_{\Q_p} \B_{\dR}^+ \sub \xi^{-b}\Xi.\]
Hence, the equation
\[ \Fil^{i} \coloneqq \frac{\xi^{i}\Xi \cap  (\V \otimes_{\Q_p}\B_{\dR}^+)}{\xi^{i}\Xi \cap  \xi(\V \otimes_{\Q_p}\B_{\dR}^+)}\]
defines a separated and exhaustive filtration on $\V \otimes_{\Q_p} \Os_{X_{\proet}}$. Moreover, an elementary manipulation yields an isomorphism
\begin{align*} 
\Fil^{i}/\Fil^{i+1} &= \frac{\xi^i\Xi \cap (\V \otimes \B_{\dR}^+)}{\xi^{i+1}\Xi \cap (\V \otimes \B_{\dR}^+)+ \xi^i\Xi  \cap \xi(\V\otimes \B_{\dR}^+)}\\
&= \frac{\Fil^{-i}D_{\dR}(\V)}{\Fil^{-i+1}D_{\dR}(\V)}\otimes \Os_{X_{\proet}}(i) \\
&= \gr^{-i}D_{\dR}(\V)\otimes \Os_{X_{\proet}}(i),
\end{align*}
where the second equality follows from Lemma \ref{lemma: B+ lattice of Scholze}(2). By the uniqueness part of Proposition \ref{prop: characterization of HT ls}, we conclude that $\Fil^{\bullet+a+b}=\Fil_{\HT}^{\bullet}$, as required.
\end{proof}

\subsection{Smooth commutative adic groups}
We recall some basic constructions on smooth families of commutative adic groups, following \cite{heuer2022gtorsors},\cite{heuer2024curve}. In the following, we fix a base adic space $S$ over $\Q_p$ which is either sousperfectoid or a rigid space over a non-archimedean field $(K,K^+)$.

\begin{definition}
    A smooth $S$-group, or smooth relative group over $S$, is a group object $\Gs \rightarrow S$ in the category of adic spaces smooth over $S$.
\end{definition}
Given such a group $\Gs \rightarrow S$, its diamond $\Gs^{\diamondsuit}$ defines a sheaf of groups on $S_v$, which is simply the functor of points
\[ \Gs^{\diamondsuit}\colon \Perf_S \rightarrow \Grp, \quad Y \mapsto \Hom_S(Y,\Gs). \]
Moreover, $\Gs$ is again sousperfectoid or a rigid space, and is good as soon as $S$ is. In that case, by \cite[Prop. 2.2(2)]{gerth2024}, the functor $\Gs \mapsto \Gs^{\diamondsuit}$ defines a fully faithful functor from the category of smooth $S$-groups to the category of sheaves of groups on $S_v$. We will usually omit $(\cdot)^{\diamondsuit}$ from the notation and simply write $\Gs$ when considering the associated diamond.

Given a smooth $S$-group $\Gs$ with identity section $e\in \Gs(S)$, we may associate to $\Gs$ its Hodge bundle and its Lie algebra
\begin{align}
    \omega_{\Gs} = e^*\Omega_{\Gs/S}^1\, , \quad \Lie(\Gs) = \omega_{\Gs}^{\vee},
\end{align}
which are vector bundle on $S_{\et}$. We denote by $\ga \rightarrow S$ the geometric vector bundle defined by $\Lie(\Gs)$.

In this article, we will be primarly concerned with commutative adic groups. In this case, we have access to a fundamental invariant, the logarithm. We recall some basic facts about it, following \cite{heuer2024curve}. Given an abelian sheaf $\Fs$ on $S_v$, we define the topologically $p$-torsion subsheaf $\Fs\langle p^{\infty} \rangle \sub \Gs$ as the following image
\begin{align}\label{eq: ptop torsion subsheaf} \Fs\langle p^{\infty} \rangle \coloneqq  \Image (\underline{\Hom}(\underline{\Z_p},\Fs) \xlongrightarrow{\ev_1} \Fs ). \end{align}

\begin{proposition}{(\cite[Prop. 3.2.4]{heuer2024curve})}\label{prop: log}
    Let $\Gs \rightarrow S$ be a smooth commutative group. Then the sheaf $\Gs\langle p^{\infty} \rangle$ is representable by a canonically defined open subgroup of $\Gs$ and there is a natural homomorphism $\log_{\Gs}$ of smooth $S$-groups, fitting in an exact sequence
    % https://q.uiver.app/#q=WzAsNCxbMSwwLCJcXEdzW3Bee1xcaW5mdHl9XSJdLFsyLDAsIlxcR3MgXFxsYW5nbGUgcF57XFxpbmZ0eX0gXFxyYW5nbGUiXSxbMywwLCJcXGdhLiJdLFswLDAsIjAiXSxbMywwXSxbMSwyLCJcXGxvZ197XFxHc30iXSxbMCwxXV0=
\begin{equation}\label{logarithm exact sequence}\begin{tikzcd}
	0 & {\Gs[p^{\infty}]} & {\Gs \langle p^{\infty} \rangle} & {\ga.}
	\arrow[from=1-1, to=1-2]
	\arrow[from=1-2, to=1-3]
	\arrow["{\log_{\Gs}}", from=1-3, to=1-4]
\end{tikzcd}\end{equation}
Moreover, the map $\log_{\Gs}$ is étale if $[p]\colon \Gs \rightarrow \Gs$ is.
\end{proposition}

\begin{example}\label{ex: Gm hat}
   Let $\Gs=\G_m$ over $S=\Spa(\Q_p)$, then
    \[ \G_m\langle p^{\infty} \rangle = 1 +\B_{< 1}^1 \sub \G_{m}\]
    is the open unit ball centered at $1$, with functor of points
    \[ \G_m\langle p^{\infty} \rangle(R,R^+) = 1+ R^{\circ\circ}.\]
    In that case, the logarithm is given by the usual power series
    \[ \log\colon \G_m\langle p^{\infty} \rangle \rightarrow \G_a, \quad 1+x \mapsto \sum_{n=1}^{\infty} \frac{(-1)^{n+1}}{n}x^n. \]
\end{example}

When the base $S$ is a rigid space, we have the following.
\begin{lemma}\label{lemma: p is etale over rigid spaces}
    Let $S$ be a rigid space over a non-archimedean field $(K,K^+)$ over $\Q_p$. Let $\Gs \rightarrow S$ be a smooth commutative adic group. Then $[p]\colon \Gs \rightarrow \Gs$ is étale.
\end{lemma}
\begin{proof}
By \cite[Prop. 1.7.5]{huber2013étale}, we can check that $[p]\colon \Gs \rightarrow \Gs$ is étale on stalks. We claim that for $x\in \Gs$, the ring $\Os_{\Gs,x}$ is noetherian. Indeed, since $\Os_{\Gs,x}$ is insensitive to replacing $\Gs$ by an open affinoid neighborhood of $x$ and pulling back along $\Spa(K,\Os_K) \rightarrow \Spa(K,K^+)$, we may reduce to the case of $\Gs$ being a classical rigid space in the sense of Tate, in which case this is \cite[§4.1, Prop. 6]{Bosch2014rigid}. By the fiberwise criteria for flatness \cite[\href{https://stacks.math.columbia.edu/tag/039D}{Tag 039D}]{stacks-project}, we may thus assume that $S=\Spa(K,\Os_K)$, in which case this follows from \cite[Lemme 1]{Farg19}.
\end{proof}

\section{Analytic $p$-divisible groups}\label{section: Analytic $p$-divisible groups}
\subsection{Definitions}
In this subsection, we present the key concept of analytic $p$-divisible groups and discuss examples. We continue to denote by $S$ a fixed adic space over $\Q_p$ which is either sousperfectoid or a rigid space over a non-archimedean field $(K,K^+)$.
\begin{definition}{(\cite[§2.1, Def. 2.1]{Farg19}\cite[Def. 2.19]{heuer2022geometric})}
    Let $\Gs$ be a commutative smooth $S$-group. We say that $\Gs$ is analytic $p$-divisible if $\Gs=\Gs\langle p^{\infty} \rangle$ and $[p]\colon\Gs \rightarrow \Gs$ is finite étale and surjective.
\end{definition}
The systematic study of analytic $p$-divisible groups over a non-archimedean field was first initiated by Fargues in \cite{Farg19} and later extended to rigid analytic families in \cite{Farg22}. Prior to this, the idea is already present in work of Fontaine \cite{Fontaine03}, in which an abelian variety $A$ over a finite extension $K/\Q_p$ is assigned a rigid analytic open subgroup $A^{(p)} \sub A$ which, in our terminology, recovers the topologically $p$-torsion subgroup $A\langle p^{\infty} \rangle$.

\begin{example}\label{ex: examples of analytic p-divisible groups}
    \begin{enumerate}
        \item The group $\G_m\langle p^{\infty} \rangle$ from Example \ref{ex: Gm hat} is analytic $p$-divisible. 
        \item Any vector bundle $\ga \rightarrow S$ is an analytic $
        p$-divisible group. In that case, $[p]$ is even an isomorphism. On the other extreme, let $\Lb$ be a $\Z_p$-local system on $S_v$, then 
        \[ \Gs = \Lb[\tfrac{1}{p}]/\Lb = \varinjlim_n \tfrac{1}{p^n}\Lb/\Lb\]
        is representable by an analytic $p$-divisible group with étale structure map $\Gs \rightarrow S$. Any analytic $p$-divisible group $\Gs$ is built from these two examples, using that the logarithm sequence (\ref{logarithm exact sequence}) for $\Gs$ is right-exact, by Lemma \ref{lemma: log is right-exact for an pdiv} below.
        \item Let $A$ be an abeloid variety over $K$, the rigid analytic analog of a complex torus \cite{Bosch1991DegeneratingAV}. Then the topologically $p$-torsion subgroup $A\langle p^{\infty} \rangle \sub A$ is an analytic $p$-divisible rigid group. More generally, let $S$ be a rigid space and let $\As \rightarrow S$ be a relative abeloid variety, i.e. a proper smooth adic group with geometrically connected fibers. Then $\As\langle p^{\infty} \rangle \rightarrow S$ is an analytic $p$-divisible group. Indeed, $\As\langle p^{\infty} \rangle$ is of topologically $p$-torsion by \cite[Lemma 2.9]{heuer2022geometric}, and $[p]\colon \As \rightarrow \As$ is étale by Lemma \ref{lemma: p is etale over rigid spaces}. Since $\As \rightarrow S$ is proper, also $[p]\colon \As \rightarrow \As$ is proper and hence finite, by \cite[Prop. 1.5.5, Cor. 1.7.4]{huber2013étale}. Finally, by the valuative criterion for properness, surjectivity can be checked on $(K,\Os_K)$-rational points, where it follows from \cite[Lemma 6.3]{heuer2023padic}. 
        \item Let $\Ss$ be a flat formal scheme over $\Z_p$ that locally admits a finitely generated ideal of definition and with adic generic fiber $S$. Let $\mathfrak{G} \rightarrow \Ss$ be a $p$-divisible group, that is, an fpqc abelian sheaf on the category of schemes over $\Ss$ such that $[p]\colon \mathfrak{G} \rightarrow \mathfrak{G}$ is surjective, $\mathfrak{G}[p]$ is representable by a finite locally free formal scheme over $\Ss$ and $\mathfrak{G} = \varinjlim_n \mathfrak{G}[p^n]$. Consider the sheaf $\mathfrak{G}_{\eta}$ on $\Perf_S$, obtained as the sheafification for the analytic topology of the presheaf
        \[ (R,R^+)\in \Perf_S \mapsto \mathfrak{G}(R^+).\]
        If $\mathfrak{G}$ is representable by a formal scheme, then $\mathfrak{G}_{\eta}$ is simply the adic generic fiber. Then by \cite[Lemma 4.1.6]{graham2025padicfouriertheoryfamilies}, the sheaf $\mathfrak{G}_\eta$ is representable by an analytic $p$-divisible group. For example, the group $\G_m\langle p^{\infty} \rangle \rightarrow \Spa(\Q_p)$ arises as the generic fiber of $\mu_{p^{\infty}} = \Spf(\Z_p[[ X-1]]) \rightarrow \Spf(\Z_p)$.
    \end{enumerate}
\end{example}

\begin{lemma}{(\cite[Prop. 3.2.4]{heuer2024curve})}\label{lemma: log is right-exact for an pdiv}
    Let $\Gs \rightarrow S$ be an analytic $p$-divisible group. Then the logarithm sequence (\ref{logarithm exact sequence}) is right-exact and thus defines a short exact sequence
    % % https://q.uiver.app/#q=WzAsNSxbMSwwLCJcXEdzW3Bee1xcaW5mdHl9XSJdLFsyLDAsIlxcR3MiXSxbMywwLCJcXGdhIl0sWzAsMCwiMCJdLFs0LDAsIjAuIl0sWzMsMF0sWzEsMiwiXFxsb2dfe1xcR3N9Il0sWzAsMV0sWzIsNF1d
\begin{equation}\label{eq: logarithm exact sequence for pdiv groups}\begin{tikzcd}
	0 & {\Gs[p^{\infty}]} & \Gs & \ga & {0.}
	\arrow[from=1-1, to=1-2]
	\arrow[from=1-2, to=1-3]
	\arrow["{\log_{\Gs}}", from=1-3, to=1-4]
	\arrow[from=1-4, to=1-5]
\end{tikzcd}\end{equation}
More generally, this holds if $\Gs$ is a smooth commutative $S$-group such that there exists an open $S$-subgroup $U\sub \Gs$ with $[p]\colon U \rightarrow U$ étale and surjective. 
\end{lemma}
\begin{proof}
    Let $U$ be a subgroup as above. Then $[p]$ is also surjective on $\cj{U} = U/U\langle p^{\infty} \rangle$. But $\cj{U}$ is $p$-torsionfree, by \cite[Lemma 2.11]{heuer2022geometric}, hence it is uniquely $p$-divisible. By the Snake Lemma, $[p]$ is also surjective and étale on $U\langle p^{\infty} \rangle$. Hence $\log(U\langle p^{\infty} \rangle) \sub \ga$ is an open subgroup on which $[p]$ is surjective. The only such subgroup is $\ga$, so that $\log$ is surjective, as required.
\end{proof}

\begin{definition}\label{def: T_pG is Z_p local system}
    Let $\Gs$ be an analytic $p$-divisible group over $S$.
    \begin{enumerate}
        \item The dimension of $\Gs$ is the rank of the étale vector bundle $\Lie(\Gs)$.
        \item The height of $\Gs$ is the rank of the $\F_p$-local system $\Gs[p]$.
        \item The Tate module of $\Gs$ is
    \[ T_p\Gs \coloneqq \varprojlim_n \Gs[p^n],\]
    where the inverse limit is computed in the category of sheaves on $S_v$. It is easy to see that $T_p\Gs$ is a $\Z_p$-local system on $S_v$ of rank equal to the height of $\Gs$. We define the rational Tate module of $\Gs$ to be the $\Q_p$-local system
    \[ V_p\Gs = T_p\Gs[\tfrac{1}{p}].\]
    \end{enumerate}
\end{definition}

There is a good notion of isogenies of analytic $p$-divisible groups.
\begin{definition}
    Let $\Gs,\Gs'$ be analytic $p$-divisible groups over $S$.
    \begin{enumerate}
        \item An isogeny $\varphi\colon \Gs \rightarrow \Gs'$ is a finite étale surjective morphism of $S$-group.
        \item A quasi-isogeny $\varphi\colon \Gs \rightarrow \Gs'$ is a global section of the sheaf on $S_{\an}$
        \[ \underline{\Hom}_S(\Gs,\Gs')[\tfrac{1}{p}]\]
        such that, analytic-locally on $S$, $p^n\varphi$ is an isogeny for some $n\geq 0$.
    \end{enumerate}
\end{definition}
Given an isogeny $\varphi\colon \Gs \rightarrow \Gs'$, its kernel $\Ker(\varphi)$ is a finite étale $S$-group and we have an isomorphism of $v$-sheaves
\[ \Gs/\Ker(\varphi) \cong \Gs'.\]
Conversely, we will see later in Corollary \ref{cor: quotients of finite etale subgroups} that analytic $p$-divisible groups are stable under quotients by finite étale subgroups.

\subsection{Analytic Weil pairings}

Let $S$ be an adic space over $\Q_p$ which is either a sousperfectoid space or a rigid space over some non-archimedean field $(K,K^+)$. Given an analytic $p$-divisible group $\Gs \rightarrow S$, we may consider its twisted dual Tate module
\[ T_p\Gs^{\vee}(1) = \underline{\Hom}_{\underline{\Z_p}}(T_p\Gs,\underline{\Z_p}(1)).\]
By writing $\underline{\Z_p}(1) = \varprojlim_n \mu_{p^{n}}$, we get
\[ T_p\Gs^{\vee}(1) =\varprojlim_n \,\underline{\Hom}_{S_v}(\Gs[p^n],\mu_{p^n}) = \underline{\Hom}_{S_v}(\Gs[p^{\infty}],\mu_{p^{\infty}}).\]
This yields a morphism of $v$-sheaves, the \emph{Weil pairing}
\[ \Gs[p^{\infty}] \times T_p\Gs^{\vee}(1) \rightarrow \mu_{p^{\infty}}.\]
The main result of this subsection is the following.
\begin{thm}\label{thm: existence of analytic Weil pairings}
    Let $S$ be a good adic space over $\Q_p$ (cf. Definition \ref{def: good adic spaces}) and let $\Gs \rightarrow S$ be an analytic $p$-divisible group. Then the Weil pairing extends uniquely to a pairing of diamonds, natural in $S$ and $\Gs$
    \begin{align} e_{\Gs}\colon \Gs \times T_p\Gs^{\vee}(1) \rightarrow \G_m\langle p^{\infty} \rangle.\end{align}
    We call $e_{\Gs}$ the analytic Weil pairing for the group $\Gs$.
\end{thm}

Analytic Weil pairings were already constructed by Tate for $p$-divisible groups over $\Os_K$ where $K$ is a $p$-adic field \cite[§4]{Tat67}. Outside the case of good reduction, Fontaine \cite[Prop. 1.1]{Fontaine03} constructs, for any abelian variety $A$ over a finite extension $K$ of $\Q_p$, a Galois-equivariant map
\[ A\langle p^{\infty} \rangle(\cj{K}) \rightarrow T_pA(-1) \otimes_{\Z_p} (1+\ma_{\C_p})= \Hom_{\cts, G_K}(T_pA^{\vee}(1),\G_m\langle p^{\infty} \rangle(\C_p)).\]
A similar pairing appears in the work of Deninger--Werner \cite{DeningerWerner2005} in the context of the $p$-adic Simpson correspondence: In their work, given a connected smooth projective curve $X$ over $\cj{\Q}_p$ with fixed base point $x\in X(\cj{\Q}_p)$, they establish a functor from the category of certain vector bundles on $X_{\C_p}$ to the category of continuous $\C_p$-linear representations of the fundamental group. When specified to the case of rank $1$, this yields a group homomorphism 
\[ \alpha\colon \PPic_X^0(\C_p) \rightarrow \Hom_{\cts}(\pi_1(X,x),\C_p^{\times}) = \Hom_{\cts}(TA^{\vee},\C_p^{\times}),\]
where $A = \PPic_X^0$ is the Picard variety and $TA = \prod_{\ell} T_{\ell}A$ is the Tate module. This morphism is studied further in \cite{Deninger2005}: Its domain is enlarged to $\PPic_X^{\tau}(\C_p)$, the line bundles with torsion Neron--Severi classes, and the construction of $\alpha$ is extended to proper smooth varieties $X$ over $\cj{\Q}_p$ verifying a certain good reduction assumption. It was later shown in \cite{heuer2022geometric} that the morphism $\alpha$ can be geometrized: For an arbitrary proper smooth rigid space $X$ over a complete algebraically closed field $C/\Q_p$, there is a topologically torsion Picard variety $\PPic_X^{\ttt}$. This is a rigid open subgroup of $\PPic_{X}$ with the property that, if $C=\C_p$,
\[ \PPic_X^{\ttt}(\C_p) = \PPic_X^{\tau}(\C_p),\] 
and that comes with a morphism of rigid groups
\[ \PPic_X^{\ttt} \rightarrow \underline{\Hom}_{\cts}(\pi_1(X,x),\G_m).\]

\begin{example}
    Let $\Gs = \G_m\langle p^{\infty} \rangle$ over $S=\Spa(\Q_p)$. In that case, the analytic Weil pairing is
    \[ \G_m\langle p^{\infty} \rangle \times \Z_p \rightarrow \G_m\langle p^{\infty} \rangle, \quad (1+z,a) \mapsto (1+z)^a = \sum_{n=0}^{\infty}\binom{a}{n}z^n.\]
    This is the natural $\Z_p$-module structure on $\G_m\langle p^{\infty} \rangle$.
\end{example}

\begin{remark}\label{rmk: p-divisible groups over LSD}
    One could define an analytic $p$-divisible group over an arbitrary locally spatial diamond $S$ over $\Q_p$ to be a group diamond $\Gs \rightarrow S$ such that $[p]\colon \Gs \rightarrow \Gs$ is finite étale surjective and such that $\Gs$ fits in a short exact sequence of $v$-sheaves
% https://q.uiver.app/#q=WzAsNSxbMCwwLCIwIl0sWzEsMCwiXFxHc1twXntcXGluZnR5fV0iXSxbMiwwLCJcXEdzIl0sWzMsMCwiXFxnYSJdLFs0LDAsIjAsIl0sWzAsMV0sWzEsMl0sWzMsNF0sWzIsM11d
\[\begin{tikzcd}
	0 & {\Gs[p^{\infty}]} & \Gs & \ga & {0,}
	\arrow[from=1-1, to=1-2]
	\arrow[from=1-2, to=1-3]
	\arrow[from=1-3, to=1-4]
	\arrow[from=1-4, to=1-5]
\end{tikzcd}\]
where $\ga = E\otimes_{\Os_{S_{\et}}} \Os_{S_v}$ for a vector bundle $E$ on $S_{\et}$. With this definition, Theorem \ref{thm: existence of analytic Weil pairings} and its proof extend without changes. To consider universal families over moduli stacks, it would be interesting to further extend this definition to small $v$-stacks. Due to the pathological behavior of the étale sites of small $v$-stacks in general, it would be preferable to restrict to “diamond stacks”, roughly those small $v$-stacks that admit a quasi-pro-étale surjection from a perfectoid space. We will not pursue this here.
\end{remark}

We now turn to the proof of Theorem \ref{thm: existence of analytic Weil pairings}. We start with some preparations.
    \begin{lemma}\label{Lemma: turn Hom into tensor}
Let $S$ be a good adic space over $\Q_p$.
    \begin{enumerate}
        \item Let $\Fs, \Gs$ be two sheaves of $\Z_p$-modules on $S_v$ and assume that evaluation at $1$ defines an isomorphism
        \[ \underline{\Hom}_{\Z}(\underline{\Z_p},\Gs) \xrightarrow{\cong}  \Gs. \]
        Then
    \[ \underline{\Hom}_{\Z}(\Fs,\Gs) = \underline{\Hom}_{\Z_p}(\Fs,\Gs).\]
    \item Let $\Lb$ be a $\Z_p$-local system on $S_v$ and let $\Fs$ be any $\Z_p$-module on $S_v$, then
    \[ \underline{\Hom}_{\Z_p}(\Lb,\Fs)= \Lb^{\vee} \otimes_{\Z_p} \Fs.\]
    \end{enumerate}
\end{lemma}
\begin{proof}
The first part formally follows from the fact that $\Gs$ is its own scalar co-extension. For the second isomorphism, there is a natural map from right to left which is an isomorphism $v$-locally, hence an isomorphism. 
\end{proof}

% \begin{proof}
%     It suffices to observe that each $\Gs[p^n]$ are finite étale $X$-groups fitting in short exact sequences
%     % https://q.uiver.app/#q=WzAsNSxbMCwwLCIwIl0sWzEsMCwiXFxHc1twXm5dIl0sWzIsMCwiXFxHc1twXntuK219XSJdLFszLDAsIlxcR3NbcF57bX1dIl0sWzQsMCwiMC4iXSxbMCwxXSxbMSwyXSxbMiwzLCJbcF5uXSJdLFszLDRdXQ==
% \[\begin{tikzcd}
% 	0 & {\Gs[p^n]} & {\Gs[p^{n+m}]} & {\Gs[p^{m}]} & {0.}
% 	\arrow[from=1-1, to=1-2]
% 	\arrow[from=1-2, to=1-3]
% 	\arrow["{[p^n]}", from=1-3, to=1-4]
% 	\arrow[from=1-4, to=1-5]
% \end{tikzcd}\]
% \end{proof}

Theorem \ref{thm: existence of analytic Weil pairings} now follows from the following proposition.
\begin{proposition}\label{Prop: extension from p-torsion to the whole group}
    Let $S$ be a good adic space over $\Q_p$. Let $\Gs$ be an analytic $p$-divisible $S$-group and $\Lb$ be a $\Z_p$-local system on $S_v$. Then there is an isomorphism of Hom-sheaves on $S_v$
    \[ \underline{\Hom}(\Gs[p^{\infty}],  \Lb\otimes_{\Z_p} \G_m\langle p^{\infty} \rangle) =  \underline{\Hom}(\Gs, \Lb  \otimes_{\Z_p} \G_m\langle p^{\infty} \rangle). \]
\end{proposition}
\begin{proof}
For clarity, we write $\otimes$ for $\otimes_{\Z_p}$. The claim is local on $S_v$, hence we may assume that $S$ is a strictly totally disconnected perfectoid space. Consider the exact sequence, obtained from the logarithm sequence (\ref{eq: logarithm exact sequence for pdiv groups})
% https://q.uiver.app/#q=WzAsNCxbMCwwLCJcXHVuZGVyc2V0ez0wfXtcXHVuZGVyYnJhY2V7XFxIb20oXFxnYSwgXFxMYlxcb3RpbWVzIFxcR19tXFxsYW5nbGUgcF57XFxpbmZ0eX0gXFxyYW5nbGUpfX0iXSxbMSwwLCJcXEhvbShcXEdzLCBcXExiXFxvdGltZXNcXEdfbVxcbGFuZ2xlIHBee1xcaW5mdHl9IFxccmFuZ2xlKSJdLFsyLDAsIlxcSG9tKFxcR3NbcF57XFxpbmZ0eX1dLCBcXExiXFxvdGltZXNcXEdfbVxcbGFuZ2xlIHBee1xcaW5mdHl9IFxccmFuZ2xlKSJdLFswLDEsIlxcRXh0X3tcXGV0fV4xKFxcZ2EsIFxcTGJcXG90aW1lc1xcR19tXFxsYW5nbGUgcF57XFxpbmZ0eX0gXFxyYW5nbGUpLiJdLFswLDFdLFsxLDJdLFsyLDMsIlxcZGVsdGEiXV0=
\[\begin{tikzcd}
	{\underset{=0}{\underbrace{\Hom(\ga, \Lb\otimes \G_m\langle p^{\infty} \rangle)}}} & {\Hom(\Gs, \Lb\otimes\G_m\langle p^{\infty} \rangle)} & {\Hom(\Gs[p^{\infty}], \Lb\otimes\G_m\langle p^{\infty} \rangle)} \\
	{\Ext_{\et}^1(\ga, \Lb\otimes\G_m\langle p^{\infty} \rangle).}
	\arrow[from=1-1, to=1-2]
	\arrow[from=1-2, to=1-3]
	\arrow["\delta", from=1-3, to=2-1]
\end{tikzcd}\]
The leftmost term is zero. Indeed, this may be checked $v$-locally, so that we may assume that $\ga$ and $\Lb$ are trivial. It now follows from the fact that any bounded analytic function on $\A_S^1$ is constant. To conclude, it remains to show that
\[ \Ext_{\et}^1(\ga, \Lb\otimes \G_m\langle p^{\infty} \rangle) =0.\]
As $S$ is strictly totally disconnected, $\Lb=\Z_p^{\oplus r}$ and $\ga = \G_a^{\oplus s}$ are trivial, and we may assume that $r=s=1$. We now use that, by the Breen resolution, we have an isomorphism \cite[Lemma 4.5]{Heuer2023MannWerner}
    \begin{align*}
          \Ext_{\et}^1(\G_{a,S} ,\G_{m,S}\langle p^{\infty} \rangle) = \Ker( H_{\et}^1(\A_S^1,\G_m\langle p^{\infty} \rangle) \xrightarrow{d} H_{\et}^1(\A_S^2,\G_m\langle p^{\infty} \rangle)),
    \end{align*}
    where $d = m^*-\pi_1^*-\pi_2^*$. It applies, as there are no non-constant bounded analytic functions on $\A_S^n$ and hence no nontrivial morphisms $\G_{a,S}^n \rightarrow \G_{m,S}\langle p^{\infty} \rangle$. Consider now the exact sequence
    % https://q.uiver.app/#q=WzAsNCxbMSwwLCJIXjAoXFxBX1NeMSxcXGNqe1xcR31fbSkiXSxbMiwwLCJIX3tcXGV0fV4xKFxcQV9TXjEsXFxHX21cXGxhbmdsZSBwXntcXGluZnR5fSBcXHJhbmdsZSkiXSxbMywwLCJIX3tcXGV0fV4xKFxcQV9TXjEsXFxHX20pLiJdLFswLDAsIkheMChcXEFfU14xLFxcR19tKSJdLFszLDBdLFsxLDJdLFswLDEsIlxcZGVsdGEiXV0=
\[\begin{tikzcd}
	{H^0(\A_S^1,\G_m)} & {H^0(\A_S^1,\cj{\G}_m)} & {H_{\et}^1(\A_S^1,\G_m\langle p^{\infty} \rangle)} & {H_{\et}^1(\A_S^1,\G_m).}
	\arrow[from=1-1, to=1-2]
	\arrow["\delta", from=1-2, to=1-3]
	\arrow[from=1-3, to=1-4]
\end{tikzcd}\]
By Lemma \ref{Lemma: Picard group of product in terms of one factor}(1) below, the map $\delta$ fits in the following commutative diagram
% https://q.uiver.app/#q=WzAsNCxbMCwwLCJIXjAoXFxBX1NeMSxcXGNqe1xcR31fbSkiXSxbMSwwLCJIX3tcXGV0fV4xKFxcQV9TXjEsXFxHX21cXGxhbmdsZSBwXntcXGluZnR5fSBcXHJhbmdsZSkiXSxbMCwxLCJIXjAoUyxcXGNqe1xcR31fbSkiXSxbMSwxLCJIX3tcXGV0fV4xKFMsXFxHX21cXGxhbmdsZSBwXntcXGluZnR5fSBcXHJhbmdsZSk9MCwiXSxbMCwxLCJcXGRlbHRhIl0sWzIsMCwiXFxjb25nIl0sWzMsMV0sWzIsMywiXFxkZWx0YSciLDJdXQ==
\[\begin{tikzcd}
	{H^0(\A_S^1,\cj{\G}_m)} & {H_{\et}^1(\A_S^1,\G_m\langle p^{\infty} \rangle)} \\
	{H^0(S,\cj{\G}_m)} & {H_{\et}^1(S,\G_m\langle p^{\infty} \rangle)=0,}
	\arrow["\delta", from=1-1, to=1-2]
	\arrow["\cong", from=2-1, to=1-1]
	\arrow["{\delta'}"', from=2-1, to=2-2]
	\arrow[from=2-2, to=1-2]
\end{tikzcd}\]
so that $\delta=0$. Moreover, Lemma \ref{Lemma: Picard group of product in terms of one factor}(2) also gives the vanishing of the term on the right of the above sequence. It follows that $H_{\et}^1(\A_S^1,\G_m\langle p^{\infty} \rangle)=0$, which gives the result.
\end{proof}

\begin{lemma}\label{Lemma: Picard group of product in terms of one factor}
    Let $S$ be a good affinoid adic space over $\Spa(K)$, where $K$ is a perfectoid field extension of $\Q_p$.
    \begin{enumerate}
        \item We have 
        \[ H^0(\B_S^1,\cj{\G}_m) = H^0(S,\cj{\G}_m) = H^0(\A_S^1,\cj{\G}_m).\]
        \item If $S$ is a totally disconnected perfectoid space, then
        \[\Pic_{\et}(\B_S^1) =0 = \Pic_{\et}(\A_S^1).\]
    \end{enumerate}
\end{lemma}
\begin{proof}
    \begin{enumerate}
        \item For $\B_S^1$, this is \cite[Lemma 6.6]{heuer2021line}. The result for $\A_S^1$ is derived by considering the open covering by closed balls of increasing radii.
        \item We start with $\B_S^1$. Let $\Ls$ be a line bundle on $\B_S^1$ and let $s\in S$ be any point with corresponding map $\Spa(L,L^+) \rightarrow S$. Since $L\langle T \rangle$ is a principal ideal domain, $\Pic_{\et}(\B_{(L,L^+)}^1) =0$, so that $\Ls\restr{\B_{s}^1}=0$. By \cite[Prop. 2.3]{heuer2021lineonpefd}, there exists an open neighborhood $s \in U \sub S$ such that $\Ls\restr{\B_U^1}$ is trivial. We can thus find an affinoid open covering $S = \bigcup_i U_i$ trivializing $\Ls$. We now use that $S$ is totally disconnected, so that this open cover splits. It follows that $\Ls$ is trivial, and we have proven that $\Pic_{\et}(\B_S^1)=0$. For $\A_S^1$, the \v{C}ech-to-Derived spectral sequence associated with the covering $\Us = \{\B_S^1(n)\}$ of $\A_S^1$ by closed balls yields
        \[ \Pic_{\et}(\A_S^1) = \check{H}^1(\Us,\G_m). \]
        By \cite[Lemma 6.15]{heuer2021line}, the group on the right vanishes, which concludes the proof.
    \end{enumerate}
\end{proof}

Proposition \ref{Prop: extension from p-torsion to the whole group} has the following immediate consequence.
\begin{cor}\label{Cor: extension of Weil pairing}
    Let $S$ be a good adic space over $\Q_p$ and let $\Gs \rightarrow S$ be an analytic $p$-divisible group. Then there is a natural isomorphism of $v$-sheaves
    \[ T_p\Gs^\vee(1)= \underline{\Hom}(\Gs,\G_m\langle p^{\infty} \rangle).\]
\end{cor}
\subsection{Analytic $p$-divisible groups and Hodge--Tate triples}
The goal of this subsection is to prove the following theorem.
\begin{thm}\label{thm: extending Fargues' equivalence of categories}
    Let $S$ be a good adic space over $\Q_p$ (cf. Definition \ref{def: good adic spaces}). Then there is an equivalence of categories between
    \begin{itemize}
        \item The category of analytic $p$-divisible groups $\Gs \rightarrow S$, and
        \item The category of tuples $(\Lb,E,f)$ where $\Lb$ is a $\Z_p$-local system on $S_v$, $E$ is a vector bundle on $S_{\et}$ and $f\colon E\otimes_{\Os_{S_{\et}}} \Os_{S_v} \rightarrow \Lb(-1)\otimes_{\Z_p} \Os_{S_v}$ is a morphism of $v$-vector bundles.
\end{itemize}
If $\Gs$ and $(\Lb,E,f)$ correspond to each-other, we have $\Lb = T_p\Gs$ and $E=\Lie(\Gs)$. The equivalence sends $(\Lb, E,f)$ to the pullback in group sheaves over $S_v$
% https://q.uiver.app/#q=WzAsNCxbMCwwLCJcXEdzIl0sWzEsMCwiRVxcb3RpbWVzX3tcXE9zX1N9XFxHX2EiXSxbMSwxLCJcXExiKC0xKVxcb3RpbWVzX3tcXFpfcH1cXEdfYS4iXSxbMCwxLCJcXExiKC0xKVxcb3RpbWVzX3tcXFpfcH1cXEdfbVxcbGFuZ2xlIHBee1xcaW5mdHl9IFxccmFuZ2xlIl0sWzAsMV0sWzEsMiwiZiJdLFszLDIsIlxcaWQgXFxvdGltZXMgXFxsb2ciLDJdLFswLDNdXQ==
\[\begin{tikzcd}
	\Gs & {E\otimes_{\Os_S}\G_a} \\
	{\Lb(-1)\otimes_{\Z_p}\G_m\langle p^{\infty} \rangle} & {\Lb(-1)\otimes_{\Z_p}\G_a.}
	\arrow[from=1-1, to=1-2]
	\arrow[from=1-1, to=2-1]
	\arrow["f", from=1-2, to=2-2]
	\arrow["{\id \otimes \log}"', from=2-1, to=2-2]
\end{tikzcd}\]
The equivalence is compatible with base-change along arbitrary morphisms $S' \rightarrow S$ of good adic spaces.
\end{thm}

\begin{remark}
\begin{enumerate}
    \item When $X=\Spa(K)$, for a non-archimedean field $K$ with completed algebraic closure $C$, Theorem \ref{thm: extending Fargues' equivalence of categories} yields an equivalence of categories between analytic $p$-divisible groups over $K$ and tuples $(\Lb, V,f)$, where $\Lb$ is a finite free $\Z_p$-linear continuous representation of $\Gal_K$, $V$ is a finite-dimensional $K$-vector space and $f\colon V \otimes_K C \rightarrow  \Lb(-1) \otimes_{\Z_p} C$ is a $C$-linear $\Gal_K$-equivariant morphism. This recovers a theorem of Fargues \cite[Théorème 0.1]{Farg19}. 
    \item If $S$ is a locally spatial diamond over $\Spd(\Q_p)$, the equivalence of categories of Theorem \ref{thm: extending Fargues' equivalence of categories} and its proof extend without changes to classify analytic $p$-divisible groups over $S$ (cf. Remark \ref{rmk: p-divisible groups over LSD}).
    \item We will show later in Theorem \ref{Thm: comparison theorem for analytic p-divisible groups} that the equivalence of Theorem \ref{thm: extending Fargues' equivalence of categories} is compatible with higher direct images along proper smooth morphisms of rigid spaces.
\end{enumerate}
\end{remark}
% \begin{remark}
%     Let $E/\Q_p$ be a finite field extension. For a good adic space $S$ over $\Spa(E)$, there is a notion of analytic $p$-divisible $S$-groups $\Gs$ with $\Os_E$-action, that is, $\Gs$ comes equipped with a ring morphism $\Os_E \rightarrow \End_S(\Gs)$. We will always assume that the Kottwitz condition is satisfied, i.e. the induced $\Os_E$-action on the Lie algebra coincides with scalar multiplication through $\Os_E \sub \Os(S)$. Then Theorem \ref{thm: extending Fargues' equivalence of categories} and its proof extend to an equivalence of categories between
%     \begin{itemize}
%         \item Analytic $p$-divisible $S$-groups $\Gs$ with $\Os_E$-action, and
%         \item tuples $(\Lb,V,f)$ where $\Lb$ is an $\Os_E$-local system on $S_v$, $V$ is a vector bundle on $S_{\et}$ and $f\colon V\otimes_{\Os_{S_{\et}}} \Os_{S_v} \rightarrow \Lb(-1)\otimes_{\Os_E} \Os_{S_v}$ is a morphism of $v$-vector bundles.
% \end{itemize}
% \end{remark}

We now turn to the proof of Theorem \ref{thm: extending Fargues' equivalence of categories}. We start with some definitions.
\begin{definition}\label{def: u and f}
    Let $\Gs \rightarrow S$ be an analytic $p$-divisible group.
    \begin{enumerate}
        \item We define the map
        \[ u_{\Gs}\colon \Gs \rightarrow \underline{\Hom}(T_p\Gs^{\vee}(1),\G_m\langle p^{\infty} \rangle) = T_p\Gs(-1)\otimes_{\Z_p}\G_m\langle p^{\infty} \rangle, \quad g \mapsto e(g,-)\]
        by currying the analytic Weil pairing 
        \[ e\colon \Gs \times T_p\Gs^{\vee}(1) \rightarrow \G_m\langle p^{\infty} \rangle \]
        of Theorem \ref{thm: existence of analytic Weil pairings}.
        \item We define
\[ f_{\Gs}\colon \ga \rightarrow \underline{\Hom}(T_p\Gs^{\vee}(1),\G_a) = T_p\Gs(-1)\otimes_{\Z_p}\G_a\] obtained by currying the following composition
\[ T_p\Gs^{\vee}(1) \xrightarrow{t \mapsto e(-,t)} \underline{\Hom}(\Gs,\G_m\langle p^{\infty} \rangle) \rightarrow\underline{\Hom}(\ga,\G_a), \] 
where the last map sends a group homomorphism to its derivative at the identity.
    \end{enumerate}
\end{definition}

The following proposition follows immediately from the definitions.
\begin{proposition}\label{Prop: diagram relating f and u}
    Let $\Gs \rightarrow S$ be an analytic $p$-divisible group. Then the morphism of $v$-sheaves $u_{\Gs},f_{\Gs}$ defined above fit in the following commutative diagram of sheaves on $S_v$ with exact rows
    % https://q.uiver.app/#q=WzAsMTAsWzIsMCwiXFxHcyJdLFszLDAsIlxcZ2EiXSxbMywxLCJUX3BcXEdzKC0xKVxcb3RpbWVzX3tcXFpfcH1cXEdfYSJdLFsyLDEsIlRfcFxcR3MoLTEpXFxvdGltZXNfe1xcWl9wfVxcR19tXFxsYW5nbGUgcF57XFxpbmZ0eX0gXFxyYW5nbGUiXSxbNCwwLCIwIl0sWzQsMSwiMC4iXSxbMSwwLCJcXEdzW3Bee1xcaW5mdHl9XSJdLFsxLDEsIlRfcFxcR3MoLTEpXFxvdGltZXNfe1xcWl9wfVxcbXVfe3Bee1xcaW5mdHl9fSJdLFswLDEsIjAiXSxbMCwwLCIwIl0sWzAsMSwiXFxsb2dfe1xcR3N9Il0sWzEsMiwiZl97XFxHc30iXSxbMywyLCJcXGlkIFxcb3RpbWVzIFxcbG9nIiwyXSxbMCwzLCJ1X3tcXEdzfSJdLFsxLDRdLFsyLDVdLFs2LDBdLFs3LDNdLFs4LDddLFs5LDZdLFs2LDcsIj0iLDJdXQ==
\begin{equation}\begin{tikzcd}
	0 & {\Gs[p^{\infty}]} & \Gs & \ga & 0 \\
	0 & {T_p\Gs(-1)\otimes_{\Z_p}\mu_{p^{\infty}}} & {T_p\Gs(-1)\otimes_{\Z_p}\G_m\langle p^{\infty} \rangle} & {T_p\Gs(-1)\otimes_{\Z_p}\G_a} & {0.}
	\arrow[from=1-1, to=1-2]
	\arrow[from=1-2, to=1-3]
	\arrow["{=}"', from=1-2, to=2-2]
	\arrow["{\log_{\Gs}}", from=1-3, to=1-4]
	\arrow["{u_{\Gs}}", from=1-3, to=2-3]
	\arrow[from=1-4, to=1-5]
	\arrow["{f_{\Gs}}", from=1-4, to=2-4]
	\arrow[from=2-1, to=2-2]
	\arrow[from=2-2, to=2-3]
	\arrow["{\id \otimes \log}"', from=2-3, to=2-4]
	\arrow[from=2-4, to=2-5]
\end{tikzcd}\end{equation}
\end{proposition}

\begin{proof}[Proof of Theorem \ref{thm: extending Fargues' equivalence of categories}]
    Let us call $F$ the functor in the statement of the theorem. Let $(\Lb,E,f)$ be a tuple as in the statement and let $\Gs = F(\Lb,E,f)$ denote the resulting fiber product. By construction, $\Gs$ fits in a short exact sequence of sheaves on $S_v$
    % https://q.uiver.app/#q=WzAsNSxbMiwwLCJcXEdzIl0sWzMsMCwiRVxcb3RpbWVzX3tcXE9zX1N9IFxcR19hIl0sWzQsMCwiMC4iXSxbMSwwLCJcXExiKC0xKVxcb3RpbWVzX3tcXFpfcH1cXG11X3twXntcXGluZnR5fX0iXSxbMCwwLCIwIl0sWzAsMV0sWzEsMl0sWzMsMF0sWzQsM11d
\[\begin{tikzcd}
	0 & {\Lb(-1)\otimes_{\Z_p}\mu_{p^{\infty}}} & \Gs & {E\otimes_{\Os_S} \G_a} & {0.}
	\arrow[from=1-1, to=1-2]
	\arrow[from=1-2, to=1-3]
	\arrow[from=1-3, to=1-4]
	\arrow[from=1-4, to=1-5]
\end{tikzcd}\]
Hence it follows from \cite[Lemma 3.38]{gerth2024} that $\Gs$ is representable by an analytic $p$-divisible group, so that the functor $F$ is well-defined. We construct a functor $F^{-1}$ in the reverse direction, that sends an analytic $p$-divisible group $\Gs$ to the tuple $(T_p\Gs, \Lie(\Gs),f_{\Gs})$. By construction, the map $f_{\Gs}$ is natural in $\Gs$, so that $F^{-1}$ extends to a functor. Proposition \ref{Prop: diagram relating f and u} shows that $F^{-1} \circ F=\id$. To conclude that $F^{-1} \circ F=\id$, we need to show that the short exact sequence in the definition of $\Gs = F(\Lb,E,f)$ agrees with the logarithm sequence for $\Gs$. This is \cite[Lemma 3.38]{gerth2024}, which concludes.
\end{proof}

\begin{remark}\label{remark: an pdiv groups don't descent outside of perfd}
    It follows from Theorem \ref{thm: extending Fargues' equivalence of categories} and (\ref{eq: etale and v vb equivalent for perfectoids}) that analytic $p$-divisible groups satisfy descent along $v$-covers $S' \rightarrow S$ of perfectoid spaces. We caution that this does not hold if $S$ and $S'$ are allowed to be arbitrary good adic spaces. Indeed, given a $p$-divisible group $\Gs' \rightarrow S'$ endowed with a descent data, the Lie algebra $\Lie(\Gs')$ descends to a $v$-vector bundle on $S$, rather than an étale vector bundle. Examples of non-effective descent data can be obtained by twisting a given $p$-divisible group $\Gs$ over $S$ by a $\Z_p$-local system $\Lb$ that is trivialized by $S'$.
\end{remark}

We now collect consequences of Theorem \ref{thm: extending Fargues' equivalence of categories}.

\begin{cor}\label{cor: quotients of finite etale subgroups}
    Let $\Gs \rightarrow S$ be an analytic $p$-divisible group and let $\Gamma\sub \Gs$ be an $S$-subgroup with $\Gamma\rightarrow S$ finite étale. Consider the quotient $\Gs/\Gamma$ as étale sheaves on the category $\LSD_{S,\et}$ of locally spatial diamonds over $S$. Then $\Gs/\Gamma$ is representable by an analytic $p$-divisible $S$-group.
\end{cor}
\begin{proof}
    Let $m$ be the rank of the finite étale map $\Gamma \rightarrow S$, a locally constant function on $\vert S \vert$. Then $[m]$ annihilates $\Gamma$. Indeed, this can be checked locally, so that we may assume that $S=\Spa(R,R^+)$ is affinoid. Now we use that finite étale $S$-spaces are equivalent to finite étale $R$-algebras, so that $\Gamma$ corresponds to a finite étale group scheme over $\Spec(R)$ and it suffices to show an analogous statement in the algebraic case. This is a result of Deligne, see \cite[§1]{Tate1970}.

    Now, as $\Gs$ is a $\Z_p$-module, it is uniquely $\ell$-divisible, for any prime $\ell\neq p$. It follows that $m$ is a power of $p$, so that $\Gamma \sub \Gs[p^{\infty}]$. We let $\Lb$ be the preimage of $\Gamma$ under the surjection of $v$-sheaves
    \[ V_p\Gs \rightarrow V_p\Gs/T_p\Gs \cong \Gs[p^{\infty}].\]
    We claim that $\Lb$ is a $\Z_p$-local system with $\Lb[\tfrac{1}{p}] = V_p\Gs$. This can be checked $v$-locally, so we may assume that $S$ is strictly totally disconnected. In that case, $T_p\Gs \cong \Z_p^{\oplus n}$ is trivial and $\Gamma$ correspond to a subgroup of $\Z/p^r\Z^{\oplus n}$ for some $r\geq 1$. In that case, we have inclusions $\Z_p^{\oplus n} \sub \Lb \sub \tfrac{1}{p^r} \Z_p^{\oplus n}$ and the claim follows. We let $\Gs'$ be the analytic $p$-divisible group associated to the tuple
    \[ (\Lie(\Gs),\Lb, \Lie(\Gs) \xrightarrow{f_{\Gs}}T_p\Gs(-1) \otimes \Os_{S_v} = \Lb(-1) \otimes \Os_{S_v}).\]
    Then there is a natural map $\pi\colon \Gs \rightarrow \Gs'$ and we show that it fits in a short exact sequence of étale sheaves on $\LSD_{S,\et}$
    % https://q.uiver.app/#q=WzAsNSxbMSwwLCJcXEdhbW1hIl0sWzIsMCwiXFxHcyJdLFszLDAsIlxcR3MnIl0sWzAsMCwiMCJdLFs0LDAsIjAuIl0sWzMsMF0sWzEsMiwiXFxwaSJdLFswLDFdLFsyLDRdXQ==
\[\begin{tikzcd}
	0 & \Gamma & \Gs & {\Gs'} & {0,}
	\arrow[from=1-1, to=1-2]
	\arrow[from=1-2, to=1-3]
	\arrow["\pi", from=1-3, to=1-4]
	\arrow[from=1-4, to=1-5]
\end{tikzcd}\]
which would give us that $\Gs' = \Gs/\Gamma$. As $\Gamma \rightarrow S$ is étale, it is enough to establish an exact sequence as above for the $v$-topology.
Using the logarithm exact sequence (\ref{logarithm exact sequence}) and the Snake Lemma, it remains to show that the following sequence is exact
% https://q.uiver.app/#q=WzAsNSxbMSwwLCJcXEdhbW1hIl0sWzIsMCwiXFxHc1twXntcXGluZnR5fV0iXSxbMywwLCJcXEdzJ1twXntcXGluZnR5fV0iXSxbMCwwLCIwIl0sWzQsMCwiMC4iXSxbMywwXSxbMSwyLCJcXHBpIl0sWzAsMV0sWzIsNF1d
\[\begin{tikzcd}
	0 & \Gamma & {\Gs[p^{\infty}]} & {\Gs'[p^{\infty}]} & {0.}
	\arrow[from=1-1, to=1-2]
	\arrow[from=1-2, to=1-3]
	\arrow["\pi", from=1-3, to=1-4]
	\arrow[from=1-4, to=1-5]
\end{tikzcd}\]
This is clear, by writing $\Gs[p^{\infty}] = V_p\Gs/T_p\Gs$ and $\Gs'[p^{\infty}] = V_p\Gs/\Lb$. This concludes the proof.
\end{proof}

% We also have the following description of exact sequences of analytic $p$-divisible groups, whose proof is easy and left to the reader. Compare with the corresponding statements for $p$-divisible groups over $\Os_C$ \cite[Prop. 5.2.8(ii), Remark 5.2.9]{scholze2013moduli}.
% \begin{cor}
%     Let $\Gs_1 \rightarrow \Gs_2 \rightarrow \Gs_3$ be morphisms of analytic $p$-divisible groups over $S$. Then the sequence of sheaves on $S_v$
%     % https://q.uiver.app/#q=WzAsNSxbMCwwLCIwIl0sWzEsMCwiXFxHc18xIl0sWzIsMCwiXFxHc18yIl0sWzMsMCwiXFxHc18zIl0sWzQsMCwiMCJdLFswLDFdLFsxLDJdLFszLDRdLFsyLDNdXQ==
% \[\begin{tikzcd}
% 	0 & {\Gs_1} & {\Gs_2} & {\Gs_3} & 0
% 	\arrow[from=1-1, to=1-2]
% 	\arrow[from=1-2, to=1-3]
% 	\arrow[from=1-3, to=1-4]
% 	\arrow[from=1-4, to=1-5]
% \end{tikzcd}\]
% is exact if and only if
% % https://q.uiver.app/#q=WzAsNSxbMCwwLCIwIl0sWzEsMCwiVF9wXFxHc18xIl0sWzIsMCwiVF9wXFxHc18yIl0sWzMsMCwiVF9wXFxHc18zIl0sWzQsMCwiMCJdLFswLDFdLFsxLDJdLFszLDRdLFsyLDNdXQ==
% \[\begin{tikzcd}
% 	0 & {T_p\Gs_1} & {T_p\Gs_2} & {T_p\Gs_3} & 0
% 	\arrow[from=1-1, to=1-2]
% 	\arrow[from=1-2, to=1-3]
% 	\arrow[from=1-3, to=1-4]
% 	\arrow[from=1-4, to=1-5]
% \end{tikzcd}\]
% and
% % https://q.uiver.app/#q=WzAsNSxbMCwwLCIwIl0sWzEsMCwiXFxMaWUoXFxHc18xKSJdLFsyLDAsIlxcTGllKFxcR3NfMikiXSxbMywwLCJcXExpZShcXEdzXzMpIl0sWzQsMCwiMCJdLFswLDFdLFsxLDJdLFszLDRdLFsyLDNdXQ==
% \[\begin{tikzcd}
% 	0 & {\Lie(\Gs_1)} & {\Lie(\Gs_2)} & {\Lie(\Gs_3)} & 0
% 	\arrow[from=1-1, to=1-2]
% 	\arrow[from=1-2, to=1-3]
% 	\arrow[from=1-3, to=1-4]
% 	\arrow[from=1-4, to=1-5]
% \end{tikzcd}\]
% are exact.
% \end{cor}

\begin{cor}\label{cor: stack of pdiv gps is overconvergent}
    Fix a non-archimedean field $(K,K^+)$ over $\Q_p$. Given an analytic $p$-divisible group $G$ over $(K,\Os_K)$, there exists a unique analytic $p$-divisible group $G'$ over $(K,K^+)$ containing $G$ as the dense locus $G' \times_{\Spa(K,K^+)}\Spa(K,\Os_K)$. Moreover, the construction is compatible with taking $p^{\infty}$-torsion subgroups, Lie algebra and Tate modules, and $G'(K,K^+)=G(K,\Os_K)$. This extends to an equivalence of categories.
\end{cor}
\begin{proof}
    The Tate module $T_pG$ is identified with a continuous $\Z_p$-linear representation of $\Gal_K$ and is thus immediately seen to extend to a local system on $\Spa(K,K^+)$. Take $G'$ to be associated with the tuple $(T_pG,\Lie(G),f_G)$ over $(K,K^+)$, then $G$ and $G' \times_{\Spa(K,K^+)}\Spa(K,\Os_K)$ share the same invariant under the equivalence of Theorem \ref{thm: extending Fargues' equivalence of categories} and are therefore isomorphic. Moreover, $G'$ can be written explicitly as a fiber product of overconvergent sheaves, so that also $G'$ is overconvergent. Hence $G'(K,K^+)=G'(K,\Os_K) = G(K,\Os_K)$. The rest of the statement is clear.
\end{proof}
In contrast, it is not true that $p$-divisible groups over $K^+$ and $\Os_K$ are equivalent. Indeed, if $\BT(A)$ denote the category of $p$-divisible group over an adic ring $A$, there is an expression of $\BT(\Os_K)$ as the $2$-cartesian product \cite[Prop. 3.16]{zhang2023peltypeigusastackpadic}
 \begin{align}\label{eq: milnor square} \BT(\Os_K) \xrightarrow{\cong} \BT(K^+) \times_{\BT(\cj{K}^+)} \BT(k),\end{align}
 where $k$ is the residue field of $\Os_K$ and $\cj{K^+}$ is the image of 
 $K^+$ under the compostion $K^+ \rightarrow \Os_K \rightarrow k$.
% \begin{remark}

%  % In particular, the generic fiber functor
%  % \[ \{\,p\text{-divisible formal groups over} \Spf(K^+)\,\} \rightarrow \{\,\text{analytic } p\text{-divisible group over } \Spa(K,K^+)\,\} \]
%  % is no longer essentially surjective when $K^+\neq K$.
% \end{remark}
In the following result, we use the products of points of Gleason \cite[Def. 1.2]{Gleason2025}. Let $\{(K_i,R_i,\varpi_i)\}_i$ be a collection of perfectoid fields $K_i$ equipped with an open and bounded valuation subring $R_i$ and a pseudouniformizer $\varpi_i \in R_i$ such that $\varpi = (\varpi_i)_i$ defines an open ideal in $\prod_i R_i$. Then $A^+ = \prod_i R_i$ and $A=A^+[\tfrac{1}{\varpi}]$ form a perfectoid Tate pair. The space $\Spa(A,A^+)$ is totally disconnected in the sense of \cite[Def. 7.1]{scholze2022etale}, and even strictly totally disconnected if the fields $K_i$ are algebraically closed. Furthermore, the product of points form a basis of the $v$-topology on $\Perf$ \cite[Remark 1.3]{Gleason2025}.
\begin{cor}\label{cor: pdiv groups on product of points}
    Let $S=\Spa((\prod_i R_i)[\tfrac{1}{(\varpi_i)}],\prod_i R_i)$ be a product of points. Then the natural functor
\begin{align*}
    \left\{\text{\begin{tabular}{l} {\parbox{5.0cm}{analytic $p$-divisible groups  over $S$}}\end{tabular}}\right\}
         \longrightarrow \left\{\text{\begin{tabular}{l} {\parbox{5.1cm}{collections $(\Gs_i)_{i \in I}$ of analytic $p$-divisible groups of constant height and dimension over $(K_i,R_i)$}}\end{tabular}}\right\}.
    \end{align*}
    is faithful and essentially surjective.
\end{cor}
To see that the functor is not full, take $\Gs= \Gs' = \G_a$ over $S$, and observe that
\[ \Hom(\Gs,\Gs') = A \neq \prod_{i\in I} K_i = \prod_{i\in I}\Hom(\Gs_i,\Gs_i').\]
In the next subsection, we introduce the class of dualizable analytic $p$-divisible groups and we show in Proposition \ref{prop: dualiazable pdiv gps over product of points} that the above functor restricts to an equivalence for dualizable groups.
\begin{proof}
    Let $\Gs,\Gs'$ be analytic $p$-divisible groups over $S$. By Theorem \ref{thm: extending Fargues' equivalence of categories}, these correspond to $A$-linear maps \[ f\colon A^d \rightarrow \Z_p^n\otimes A, \quad f'\colon A^{d'} \rightarrow \Z_p^{n'}\otimes A.\]
    Here, we use that $S$ is strictly totally disconnected, so that all étale vector bundles on $S$ are trivial, as well as all $\Z_p$-local systems on $S$, by \cite[Cor. 3.21]{MannWerner2022loc} and its proof. Then a homomorphism $h\colon \Gs \rightarrow \Gs'$ amounts to a pair $(a\colon A^d \rightarrow A^{d'}, b\colon \vert S \vert \rightarrow \Z_p^{n'\times n})$ such that the following diagram commutes
    % https://q.uiver.app/#q=WzAsNCxbMCwwLCJBXmQiXSxbMCwxLCJBXntkJ30iXSxbMSwxLCJcXFpfcF57bid9IFxcb3RpbWVzIEEuIl0sWzEsMCwiXFxaX3BebiBcXG90aW1lcyBBIl0sWzAsMSwiYSIsMl0sWzEsMiwiZiciLDJdLFswLDMsImYiXSxbMywyLCJiXFxvdGltZXMgXFxpZCJdXQ==
\[\begin{tikzcd}
	{A^d} & {\Z_p^n \otimes A} \\
	{A^{d'}} & {\Z_p^{n'} \otimes A.}
	\arrow["f", from=1-1, to=1-2]
	\arrow["a"', from=1-1, to=2-1]
	\arrow["{b\otimes \id}", from=1-2, to=2-2]
	\arrow["{f'}"', from=2-1, to=2-2]
\end{tikzcd}\]
Let $s_i\in S$ be the image of the unique closed point of $\Spa(K_i,R_i)$. Then the base change $\Gs_{s_i} \rightarrow \Gs_{s_i}'$ corresponds to the pair $(a_i\colon K_i \rightarrow K_i,b_i\colon \Z_p^n \rightarrow \Z_p^{n'})$, where $a_i$ is obtained from $a$ by scalar extension along the projections
\[ A = (\prod_i R_i)[\tfrac{1}{\varpi}] \rightarrow R_i[\tfrac{1}{\varpi}] = K_i\]
and $b_i =b(s_i)\in \Z_p^{n' \times n}$. As $\Z_p^{n' \times n}$ is totally disconnected, we have
\[ \Map_{\cts}(\vert S \vert, \Z_p^{n' \times n}) = \Map_{\cts}(\pi_0(S),\Z_p^{n' \times n}).\]
The space $\pi_0(S)$ is obtained as the Stone-\v{C}ech compactification of the set $I$. Hence, we further have
\[ \Map_{\cts}(\pi_0(S),\Z_p^{n'\times n}) = \Map(\coprod_{i\in I} \{s_i\},\Z_p^{n'\times n}).\]
It follows that the data of $b$ and the collections of $b_i\in \Z_p^{n' \times n}$ are equivalent. Let $(a,b)$ and $(\widetilde{a},\widetilde{b})$ be two such pairs that agree after base-change along $\Spa(K_i,R_i) \rightarrow S$, for each $i\in I$. Then we immediately see that $b=\widetilde{b}$. Moreover, let $E^+\sub A^d$, $E'^+\sub A^{d'}$ be $A^+$-lattices with $a(E^+),\widetilde{a}(E^+) \sub E'^+$. We recall that finite free $A^+$-modules are equivalent to collections of finite free modules of constant rank over $R_i$ \cite[Lemma 3.17]{zhang2023peltypeigusastackpadic}. From this, since $a,\widetilde{a}\colon E^+ \rightarrow E'^+$ agree after base change along $A^+ \rightarrow R_i$, we conclude that $a$ and $\widetilde{a}$ coincide on $E^+$ and thus everywhere, by linearity. This shows faithfulness.

% Similarly, given a collection $a_i\colon K_i^d \rightarrow K_i^d$, fix $R_i$-lattices $E_i^+\sub K_i^d$ and $a_i(E_i^+)\sub E_i'^+\sub K_i^d$. Then
% \[ a = (\prod_ia_i)[\tfrac{1}{\varpi}] \colon A^d= (\prod_i E_i^+)[\tfrac{1}{\varpi}] \rightarrow (\prod_i E_i'^+)[\tfrac{1}{\varpi}] = A^d\]
% is an $A$-linear map, and this construction provides an inverse to the previous construction.

It remains to establish essential surjectivity. Let $(\Gs_i)_i$ be a collection of analytic $p$-divisible groups of dimension $d$ and height $n$ over $(K_i,R_i)$. It amounts to a collection of $K_i$-linear maps $f_i\colon K_i^d \rightarrow \Z_p^n \otimes K_i$. Let us choose for each $i\in I$ an $R_i$-lattice $E_i^+ \sub K_i^d$. Up to scaling, we may assume that $f_i(E_i^+) \sub \Z_p^h \otimes R_i$, so that $f_i$ restricts to an $R_i$-linear map $f_i^+\colon E_i^+ \rightarrow \Z_p^n \otimes R_i$. Then we obtain an $A$-linear map
    \[ f = (\prod_{i\in I} f_i^+)[\tfrac{1}{\varpi}]\colon  E =(\prod_{i\in I}E_i^+)[\tfrac{1}{\varpi}]  \rightarrow (\Z_p^n \otimes \prod_i R_i)[\tfrac{1}{\varpi}] = \Z_p^n\otimes A.\]
    It is clear that the scalar extension of $(E,f)$ along $A\rightarrow K_i$
    recovers $(K_i^d,f_i)$, for each $i\in I$. Hence, if we let $\Gs$ denote the group corresponding to $(E,f)$ under the equivalence of Theorem \ref{thm: extending Fargues' equivalence of categories}, the functor in the statement sends $\Gs$ to the collection $(\Gs_i)_i$, which concludes the proof. 
\end{proof}

\subsection{Cartier duality}
In this subsection, we characterize which analytic $p$-divisible groups admit a Cartier dual. We then study the properties of dualizable groups and relate our Cartier dual with other constructions.
\begin{definition}\label{def: dualizable}
    Let $S/\Q_p$ be a good adic space and let $\Gs \rightarrow S$ be an analytic $p$-divisible group. We say that $\Gs$ is dualizable if the map $f_{\Gs}\colon \Lie(\Gs) \otimes \Os_{S_v}\rightarrow T_p\Gs(-1) \otimes \Os_{S_v}$ is injective and fits in a short exact sequence of $v$-vector bundles
    % https://q.uiver.app/#q=WzAsNSxbMSwwLCJcXExpZShcXEdzKVxcb3RpbWVzIFxcT3Nfe1Nfdn0iXSxbMiwwLCJUX3BcXEdzKC0xKVxcb3RpbWVzIFxcT3Nfe1Nfdn0iXSxbMywwLCJcXG9tZWdhXFxvdGltZXMgXFxPc197U192fSgtMSkiXSxbMCwwLCIwIl0sWzQsMCwiMCwiXSxbMywwXSxbMSwyLCJnIl0sWzAsMSwiZl97XFxHc30iXSxbMiw0XV0=
\begin{equation}\label{eq: HT filtration on Tate module}\begin{tikzcd}
	0 & {\Lie(\Gs)\otimes \Os_{S_v}} & {T_p\Gs(-1)\otimes \Os_{S_v}} & {\omega\otimes \Os_{S_v}(-1)} & {0,}
	\arrow[from=1-1, to=1-2]
	\arrow["{f_{\Gs}}", from=1-2, to=1-3]
	\arrow["g", from=1-3, to=1-4]
	\arrow[from=1-4, to=1-5]
\end{tikzcd}\end{equation}
for an étale vector bundle $\omega$ on $S$. We define the Cartier dual $\Gs^D$ to be the analytic $p$-divisible group attached to the tuple $(\omega^{\vee},T_p\Gs^{\vee}(1),g(1)^{\vee})$.
\end{definition}
Since the functor (\ref{eq: from etale to v vb}) is fully faithful, there is at most one (up to unique isomorphism) étale vector bundle $\omega$ satisfying the above condition, so that $\omega$ is well-defined. A fortiori, $\omega = \omega_{\Gs^D}$ is the Hodge bundle of the dual group $\Gs^D$.

One immediately sees that $(\cdot)^D$ extends to an exact functor and that there is a canonical isomorphism $\Gs = {\Gs^D}^D$. Thus $(\cdot)^D$ defines an antiequivalence on the category of dualizable analytic $p$-divisible groups with (quasi-)isogenies. Moreover, any analytic $p$-divisible group that is isogenous to a dualizable group is again dualizable.

\begin{example}\label{Examples: dualizable groups}
\begin{enumerate}
    \item The typical example of a non-dualizable $p$-divisible group is $\Gs= \G_a$ where the map $f_{\Gs}$ is zero.
    \item The group $\Gs=\G_m\langle p^{\infty} \rangle$ is dualizable: In this case, the map $f_{\Gs}$ is an isomorphism. The Cartier dual is the discrete group
    \[ \Gs^D = \Q_p/\Z_p.\]
    \item If $S$ is perfectoid, then by the equivalence (\ref{eq: etale and v vb equivalent for perfectoids}), $\Gs$ is dualizable if and only $f_{\Gs}$ is a locally direct summand. If $S$ is instead assumed to be a rigid space, the conditions on $f_{\Gs}$ are more restrictive. 

    % \item Let $G$ be a $p$-divisible group over $\Os_C$. Then by \cite[5.1.6(iii)]{scholze2013moduli}, its generic fiber $G_{\eta}$ is a dualizable $p$-divisible group and $\Gs_{\eta}^D = (G^D)_{\eta}$ (note that this is not formal). \footnote{This should generalize to families of $p$-divisible groups $\mathfrak{G} \rightarrow \Ss$, where $\Ss$ is a fairly general formal model of the good adic space $S$. We run into technicalities of infinite type formal schemes and there is no definition of the Hodge--Tate sequence at the level of formal schemes (as far as I know).}
    \end{enumerate}
\end{example}

\begin{remark}\label{remark: dualizability does not descend well}
    We note that dualizable analytic $p$-divisible groups satisfy descent along $v$-covers of perfectoid spaces. In contrast, if $S' \rightarrow S$ is a $v$-cover of arbitrary good adic spaces over $\Q_p$, there exist non-dualizable analytic $p$-divisible groups $\Gs \rightarrow S$ that become dualizable after base-change to $S'$, e.g. take $\Gs= \mu_{p^{\infty}}$ over $S=\Spa(\Q_p)$. Other examples can be obtained as follow \cite[§2]{Farg22}: let $\Lb$ be a continuous $\Z_p$-linear representation of $\Gal_K$ and let 
\[ V \subsetneq (\Lb(-1) \otimes_{\Z_p} C)^{\Gal_K}\]
be a strict sub-$K$-vector space. Then take the group $\Gs$ associated with the tuple $(\Lb, V, f)$, where $f$ is the natural map
\[ f\colon V \otimes_K C \hookrightarrow \Lb(-1)\otimes_{\Z_p} C.\]
\end{remark}

\begin{remark}
We can make the Cartier dual a bit more explicit. Let $K/\Q_p$ be a non-archimedean field with completed algebraic closure $C$ and let $\Gs$ be a dualizable $p$-divisible group over $\Spa(K)$. The image of a continuous character
\[ \chi \colon T_p\Gs \rightarrow C^{\times}\]
under the logarithm map $\Hom(T_p\Gs,\G_m)= \Hom(T_p\Gs,\G_m\langle p^{\infty} \rangle)\rightarrow \Hom(T_p\Gs, \G_a)$ can be interpreted as its derivative
\[ D\chi\colon  V_p\Gs \rightarrow C.\]
Precomposing with the map $f_{\Gs}$ determines an operator
\[ \theta\colon \Hom_{\cts}(V_p\Gs,C)\twoheadrightarrow \omega_{\Gs} \otimes_K C(-1).\]
The analytic Weil pairing $u_{\Gs^D}$ then yields an identification
\[ \Gs^D(C) = \{\,\chi\colon T_p\Gs \rightarrow C^{\times} \mid \theta(D\chi)=0  \,\}.\]
In particular, if $K$ is a $p$-adic field, or more generally if $C(-1)^{\Gal_K}=0$, taking Galois fixed points on each side of the above equality yields
\[ \Gs^D(K) = \Hom_{\Gal_K}(T_p\Gs,C^{\times}).\]
\end{remark}

% For example, let $E/\Q_p$ be a finite field extension and let $\Gs$ be the group associated to the tuple $(\Os_E, $. Then $\Gs^D$ is identified with the moduli of locally $E$-analytic characters of Schneider--Teitelbaum 
% \[ \Gs^D(C) = \Hom_{E-\an}(\Os_E,C^{\times}).\]

With our definition of the Cartier dual, we can now define polarizations.
\begin{definition}\label{def: polarization}
    Let $\Gs \rightarrow S$ be a dualizable analytic $p$-divisible group. A polarization (resp. quasi-polarization) is an isogeny (resp. a quasi-isogeny) $\lambda\colon \Gs \rightarrow \Gs^D$ such that $\lambda^D = -\lambda$ under the natural identification $(\Gs^D)^D=\Gs$. A polarization is called principal if $\lambda$ is an isomorphism.
\end{definition}
Using Theorem \ref{thm: extending Fargues' equivalence of categories}, one easily sees that a quasi-polarization is equivalent to the data of a perfect pairing of $\Q_p$-local systems
\[ V_p\Gs \times V_p\Gs \rightarrow \Q_p(1) \]
for which $\Lie(\Gs) \otimes \Os_v(1) \sub V_p\Gs\otimes_{\Q_p}\Os_S$ is its own orthogonal subspace. Moreover, the quasi-polarization is a principal polarization if and only if the pairing restricts to a perfect pairing on $T_p\Gs$.

% \begin{proposition}\label{cor: dualiz pdiv gps satisfy descent}
%     The prestack sending a perfectoid space $S\in \Perf_{\Q_p}$ to the groupoid of dualizable analytic $p$-divisible groups $\Gs \rightarrow S$ is a small $v$-stack.
% \end{proposition}
% \begin{proof}
% Let $S' \rightarrow S$ be a $v$-cover of perfectoid spaces. By Corollary \ref{cor: pdiv gps satisfy descent}, it is enough to show that if an $p$-divisible group $\Gs \rightarrow S$ becomes dualizable after base-change to $S'$, then it is already dualizable. Since $S$ and $S'$ are assumed to be perfectoid, dualizability is equivalent to $f_{\Gs}\colon \Lie(\Gs) \rightarrow T_p\Gs(-1) \otimes \Os_S$ being a locally direct summand, hence it is clear. Hence the prestack of dualizable groups is a $v$-stack and it remains to see that is is small. Let $S= \varprojlim_i S_i$ be an $\omega_1$-cofiltered inverse limit of affinoid perfectoid spaces $S_i$, and assume that $\Gs_i \rightarrow S_i$ is an analytic $p$-divisible group that becomes dualizable after base-change to $S$. Using that
% \[ S_{\et, \qcqs} = \tworlim S_{i,\et,\qcqs},\]
% we easily deduce that $f_{\Gs_i}$ is a locally direct summand up to increasing $i$, as soon as $f_{\Gs_{i,S}}$ is, which concludes.
% \end{proof}

We now collect some properties of dualizable analytic $p$-divisible groups. We can strengthen Corollary \ref{cor: pdiv groups on product of points} by restricting to dualizable $p$-divisible groups as follows.
\begin{proposition}\label{prop: dualiazable pdiv gps over product of points}
    Let $S=\Spa((\prod_i R_i)[\tfrac{1}{(\varpi_i)}],\prod_i R_i)$ be a product of points. Then the natural map
\begin{align*}
    \left\{\text{\begin{tabular}{l} {\parbox{3.4cm}{dualizable analytic $p$-divisible groups over $S$}}\end{tabular}}\right\}
         \xlongrightarrow{\cong} \left\{\text{\begin{tabular}{l} {\parbox{5.3cm}{collections $(\Gs_i)_i$ of dualizable analytic $p$-divisible groups of constant height and dimension over $(K_i,R_i)$}}\end{tabular}}\right\}.
    \end{align*}
    is an equivalence of categories.
\end{proposition}
\begin{proof}
    The only missing part is showing that the functor is full. Let $\Gs,\Gs'$ be two dualizable $p$-divisible groups of respective dimension $d,d'$ and height $n,n'$ over $S$, corresponding to $f\colon A^d\hookrightarrow \Z_p^n\otimes A$, $f'\colon A^{d'}\hookrightarrow \Z_p^{n'}\otimes A$. Let us denote by $\Gs_i, \Gs_i'$ their base change along $\Spa(K_i,R_i) \rightarrow S$. Let $h_i\colon \Gs_i \rightarrow \Gs_i'$ be a collection of morphisms, corresponding to pairs $(a_i,b_i)$, where $a_i\colon K_i^d \rightarrow K_i^{d'}$ and $b_i\colon \Z_p^n \rightarrow \Z_p^{n'}$ are linear maps fitting in the commutative diagrams
    % https://q.uiver.app/#q=WzAsNCxbMCwwLCJLX2leZCJdLFswLDEsIktfaV57ZCd9Il0sWzEsMCwiXFxaX3Beblxcb3RpbWVzIEtfaSJdLFsxLDEsIlxcWl9wXntuJ31cXG90aW1lcyBLX2kuIl0sWzAsMSwiYV9pIiwyXSxbMSwzLCJmX2knIiwyLHsic3R5bGUiOnsidGFpbCI6eyJuYW1lIjoiaG9vayIsInNpZGUiOiJ0b3AifX19XSxbMCwyLCJmX2kiLDAseyJzdHlsZSI6eyJ0YWlsIjp7Im5hbWUiOiJob29rIiwic2lkZSI6InRvcCJ9fX1dLFsyLDMsImJfaVxcb3RpbWVzIFxcaWQiXV0=
\[\begin{tikzcd}
	{K_i^d} & {\Z_p^n\otimes K_i} \\
	{K_i^{d'}} & {\Z_p^{n'}\otimes K_i.}
	\arrow["{f_i}", hook, from=1-1, to=1-2]
	\arrow["{a_i}"', from=1-1, to=2-1]
	\arrow["{b_i\otimes \id}", from=1-2, to=2-2]
	\arrow["{f_i'}"', hook, from=2-1, to=2-2]
\end{tikzcd}\]
By the proof of Corollary \ref{cor: pdiv groups on product of points}, we have
\[ \Map_{\cts}(\vert S \vert, \Z_p^{n' \times n}) = \Map(I,\Z_p^{n' \times n}),\]
so that the $b_i$ uniquely determine a continuous map $b\colon \vert S \vert \rightarrow \Z_p^{n \times n}$. Moreover, consider the $A^+$-submodules
\[ E^+ = f^{-1}(\Z_p^n \otimes A^+), \quad  E'^+=f'^{-1}(\Z_p^{n'} \otimes A^+).\]
Then as $f$ is a locally direct summands, $E^+$ is a $d$-dimensional free $A^+$-modules with $E^+[\tfrac{1}{\varpi}] = A^d$, and similarly for $E'^+$. We may write $E^+=\prod_i E_i^+$ and $E'^+=\prod_i E_i'^+$, where
\[ E_i^+ = f_i^{-1}(\Z_p^n \otimes R_i), \quad E_i^+ = f_i'^{-1}(\Z_p^{n'} \otimes R_i).\]
Then the maps $a_i$ restrict to $R_i$-linear maps $E_i^+ \rightarrow E_i^+$. We obtain a well-defined $A$-linear map
\[ a = (\prod_i a_i)[\tfrac{1}{\varpi}] \colon A^d = (\prod_i E_i^+)[\tfrac{1}{\varpi}] \rightarrow  (\prod_i E_i'^+)[\tfrac{1}{\varpi}] = A^d, \]
and we easily check that $f' \circ a = (b\otimes \id)\circ f$, since this holds after scalar extension along $A \rightarrow K_i$ for each $i$. We obtain a morphism $h\colon \Gs \rightarrow \Gs'$ that reduces to each $h_i\colon \Gs_{i} \rightarrow \Gs_i'$, as required.
\end{proof}

Let $S$ be a smooth rigid space over a $p$-adic field $K$. In this case, we obtain a characterization of dualizable $p$-divisible $S$-groups. We will use the notion of Hodge--Tate $\Z_p$-local system, which was recalled in Definition \ref{def: HT local system}.
\begin{proposition}\label{prop: dualizable groups over smooth rig over padic field}
    Let $S$ be a smooth rigid space over a $p$-adic field $K$. Then sending $\Gs$ to $T_p\Gs$ defines an equivalence of exact categories
    \begin{align}
    \left\{\text{\begin{tabular}{l} {\parbox{4.0cm}{$\Gs \rightarrow S$ dualizable analytic $p$-divisible groups}}\end{tabular}}\right\}
         \xrightarrow{\cong} \left\{\text{\begin{tabular}{l} {\parbox{5.7cm}{Hodge--Tate $\Z_p$-local systems $\Lb$ on $S$ with Hodge--Tate weights $\in \{0,1\}$}}\end{tabular}}\right\}.
    \end{align}
\end{proposition}
\begin{proof}
    By Proposition \ref{prop: characterization of HT ls} and the definition of dualizability, the Tate module $T_p\Gs$ of a dualizable group is Hodge--Tate with weights in $\{0,1\}$, so the functor is well-defined. Let us show that it is fully faithful. Let $\Gs,\Gs'$ be two dualizable analytic $p$-divisible $S$-groups. Then we have a short exact sequence
    % https://q.uiver.app/#q=WzAsNSxbMSwwLCJcXExpZShcXEdzKVxcb3RpbWVzIFxcT3NfdiJdLFsyLDAsIlRfcFxcR3MoLTEpXFxvdGltZXMgXFxPc192Il0sWzMsMCwiXFxvbWVnYV97XFxHc15EfVxcb3RpbWVzIFxcT3NfdigtMSkiXSxbMCwwLCIwIl0sWzQsMCwiMC4iXSxbMywwXSxbMSwyLCJmX3tcXEdzXkR9XntcXHZlZX0oLTEpIl0sWzAsMSwiZl97XFxHc30iXSxbMiw0XV0=
\[\begin{tikzcd}
	0 & {\Lie(\Gs)\otimes \Os_v} & {T_p\Gs(-1)\otimes \Os_v} & {\omega_{\Gs^D}\otimes \Os_v(-1)} & {0}
	\arrow[from=1-1, to=1-2]
	\arrow["{f_{\Gs}}", from=1-2, to=1-3]
	\arrow["{f_{\Gs^D}^{\vee}(-1)}", from=1-3, to=1-4]
	\arrow[from=1-4, to=1-5]
\end{tikzcd}\]
and similarly for $\Gs'$. A morphism $h\colon \Gs \rightarrow \Gs'$ amounts to a pair $(a,b)$, where $a\colon \Lie(\Gs) \rightarrow \Lie(\Gs')$ and $b\colon T_p\Gs \rightarrow T_p\Gs'$ are morphisms with $f'\circ a = (b(-1)\otimes \id) \circ f$. As $f,f'$ are injective, we see that $a$ is uniquely determined by $b$, which gives faithfulness. Conversely, given a morphism of local systems $b\colon T_p\Gs \rightarrow T_p\Gs'$, the following composition vanishes
% https://q.uiver.app/#q=WzAsNCxbMCwwLCJcXExpZShcXEdzKVxcb3RpbWVzIFxcT3NfdiJdLFsxLDAsIlRfcFxcR3MoLTEpXFxvdGltZXMgXFxPc192Il0sWzIsMCwiVF9wXFxHcycoLTEpXFxvdGltZXMgXFxPc192Il0sWzMsMCwiXFxvbWVnYV97XFxHcydeRH1cXG90aW1lcyBcXE9zX3YoLTEpLCJdLFsxLDIsImIoLTEpXFxvdGltZXNcXGlkIl0sWzAsMSwiIiwwLHsic3R5bGUiOnsidGFpbCI6eyJuYW1lIjoiaG9vayIsInNpZGUiOiJ0b3AifX19XSxbMiwzLCIiLDAseyJzdHlsZSI6eyJoZWFkIjp7Im5hbWUiOiJlcGkifX19XV0=
\[\begin{tikzcd}
	{\Lie(\Gs)\otimes \Os_v} & {T_p\Gs(-1)\otimes \Os_v} & {T_p\Gs'(-1)\otimes \Os_v} & {\omega_{\Gs'^D}\otimes \Os_v(-1).}
	\arrow[hook, from=1-1, to=1-2]
	\arrow["{b(-1)\otimes\id}", from=1-2, to=1-3]
	\arrow[two heads, from=1-3, to=1-4]
\end{tikzcd}\]
This follows from the fact that $\nu_*(\omega \otimes \Os_v(-1)) =0$ for any étale vector bundle $\omega$, by \cite[Cor. 6.19]{scholze2013padicHodge}. Hence $b(-1)\otimes \id$ restricts to a map of $v$-vector bundles
\[ a\colon \Lie(\Gs)\otimes \Os_v \rightarrow \Lie(\Gs')\otimes \Os_v\]
which arises under the fully faithful functor (\ref{eq: from etale ls to v ls}) from a uniquely determined map of étale vector bundles
\[ a\colon \Lie(\Gs)\rightarrow \Lie(\Gs').\]
This shows fully faithfulness. For essentially surjectivity, let $\Lb$ be a $\Z_p$-local system as in the statement. Consider the group $\Gs$ associated to the tuple
\[ (\Lb, \gr^{-1}D_{\HT}(\Lb), f),\]
where $\gr^{-1}D_{\HT}(\Lb)$ is the $(-1)$th graded piece of the Higgs bundle $D_{\HT}(\Lb)$ (see (\ref{eq: DHT})) and $f$ is the inclusion of the $(-1)$th Hodge--Tate filtration (see Definition \ref{def: HT filtration})
\[ f\colon \gr^{-1}D_{\HT}(\Lb)\otimes \Os_{S_v} \hookrightarrow \Lb(-1)\otimes_{\Z_p}\Os_{S_v}.\]
Then $f$ fits in a short exact sequence
% https://q.uiver.app/#q=WzAsNSxbMSwwLCJcXGdyXnstMX1EX3tcXEhUfShcXExiKVxcb3RpbWVzIFxcT3Nfe1Nfdn0iXSxbMiwwLCJcXExiKC0xKVxcb3RpbWVzIFxcT3NfdiJdLFszLDAsIlxcZ3JeezB9RF97XFxIVH0oXFxMYilcXG90aW1lcyBcXE9zX3tTX3Z9KC0xKSJdLFswLDAsIjAiXSxbNCwwLCIwLCJdLFszLDBdLFsxLDJdLFswLDEsImYiXSxbMiw0XV0=
\[\begin{tikzcd}
	0 & {\gr^{-1}D_{\HT}(\Lb)\otimes \Os_{S_v}} & {\Lb(-1)\otimes \Os_v} & {\gr^{0}D_{\HT}(\Lb)\otimes \Os_{S_v}(-1)} & {0,}
	\arrow[from=1-1, to=1-2]
	\arrow["f", from=1-2, to=1-3]
	\arrow[from=1-3, to=1-4]
	\arrow[from=1-4, to=1-5]
\end{tikzcd}\]
which shows that $\Gs$ is dualizable. It remains to show the exactness. It is clear that the functor $\Gs \mapsto T_p\Gs$ is exact. Conversely, given a short exact sequence $T_p\Gs' \rightarrow T_p\Gs \rightarrow T_p\Gs''$, we need to show that the resulting sequence of analytic $p$-divisible groups is exact. By the logarithm sequence, it remains only to show exactness on Lie algebras. By the above, it is enough to show that the functor $D_{\HT}(-)$ sends the above exact sequence to a strict exact sequence of graded vector bundles on $S_{\et}$. This follows from the isomorphism (\ref{eq: HT comparison theorem}), the projection formula and the fact that $R^1\nu_*\OBHT = \Os_{S_{\et}}$ \cite[Prop. 6.16(ii)]{scholze2013padicHodge}.

\end{proof}

One of the main examples of dualizable analytic $p$-divisible groups are groups of good reduction. We postpone the proof of the next result to Subsection \ref{subsection: generic fibers of log p-divisible groups}, where we will prove a more general statement in Proposition \ref{prop: generic fibers of log p-divisible groups}.

\begin{proposition}\label{prop: groups of good reduction are dualizable}
    Let $\Ss$ be a flat formal scheme over $\Z_p$ that locally admits a finitely generated ideal of definition, and assume that its adic generic fiber $S=\Ss_{\eta}$ is a good adic space. Let $\mathfrak{G} \rightarrow \Ss$ be a $p$-divisible group and let $\Gs =\mathfrak{G}_{\eta}$, an analytic $p$-divisible group over $S$ (see Example \ref{ex: examples of analytic p-divisible groups}(4)). Then $\Gs$ is dualizable and we have a natural isomorphism
    \[ \Gs^D = (\mathfrak{G}^D)_{\eta}.\]
\end{proposition}

It is natural to ask when dualizable analytic $p$-divisible groups have good reduction. We recall the celebrated classification of $p$-divisible groups over $\Os_C$ of Scholze--Weinstein. In our terminology, it takes the following form.

\begin{thm}{(\cite[Thm. B]{scholze2013moduli})}\label{thm: Scholze Weinstein's result}
    Let $C$ be a complete algebraically closed field extension of $\Q_p$.
    \begin{enumerate}
        \item Let $\Gs$ be a dualizable analytic $p$-divisible group over $\Spa(C,\Os_C)$. Then $\Gs$ has good reduction, i.e. it comes from a $p$-divisible group over $\Os_C$. 
        \item The adic generic fiber functor on $p$-divisible groups over $\Os_C$ is fully faithful.
    \end{enumerate}
    In particular, the latter functor yields an equivalence
    \begin{align}
    \left\{\text{\begin{tabular}{l} {\parbox{4.0cm}{$p$-divisible groups over $\Os_C$}}\end{tabular}}\right\}
         \xrightarrow{\cong} \left\{\text{\begin{tabular}{l} {\parbox{4.5cm}{dualizable analytic $p$-divisible groups over $\Spa(C,\Os_C)$}}\end{tabular}}\right\}.
    \end{align}
\end{thm}
We stress that both points in Theorem \ref{thm: Scholze Weinstein's result} need not hold over more general bases $S$. Below we give an example of the failure of fully faithfulness of $(\cdot)_{\eta}$ already over $(C,C^+)$, where $C^+ \sub C$ is an open and bounded valuation subring. Later, in Propositions \ref{prop: semistable reduction versus semistable local system} and \ref{prop: good reduction versus crystalline local system}, we will obtain many counter-examples to the first point.

\begin{example}\label{ex: generic fiber not fully faithful.}
We fix a complete algebraically closed field $C$ and a bounded open valuation subring $C^+ \sub C$ of finite Krull dimension. Let $s$ (resp. $u$) be the closed (resp. open) point of $\vert \Spa(C,C^+) \vert = \vert \Spf(C^+) \vert$, and let $k(s)$, $k(u)$ denote the residue fields of the corresponding valuations. Let $\Es \rightarrow \Spf(C^+)$ be a formal elliptic curve with $\Es_{k(s)}$ supersingular and $\Es_{k(u)}$ ordinary. For example, $s$ could be a higher rank point in the boundary of the supersingular locus inside the good reduction locus of the modular curve, viewed as an adic space. Let $\mathfrak{G} = \Es[p^{\infty}]$. Using the hypothesis and Corollary \ref{cor: stack of pdiv gps is overconvergent}, $\mathfrak{G}_{\eta} \cong \Q_p/\Z_p \oplus \G_m\langle p^{\infty} \rangle$. Hence, $\End(\mathfrak{G}_{\eta})$ identifies with the algebra of upper-triangular matrices with $\Z_p$-coefficients. On the other hand, there is a specialization map $\End(\mathfrak{G}) \hookrightarrow \End(\mathfrak{G}_{k(s)})$ whose target is the maximal order in a  quaternion division algebra over $\Q_p$. It follows that $\End(\mathfrak{G}) \neq \End(\mathfrak{G}_{\eta})$.
\end{example}
\begin{remark}
In the above example, the $p$-divisible group $\mathfrak{G}$ and its dual are not connected. The functor $(\cdot)_{\eta}$ often is fully faithful when restricted to connected $p$-divisible groups, see \cite[Thm. 6.1]{Farg19}\cite[Thm. 4.3]{Farg22}.
\end{remark}

Nevertheless, we obtain the following partial converse to Proposition \ref{prop: groups of good reduction are dualizable}.

\begin{proposition}\label{prop: dualizable groups locally have good reduction on circ}
Let $S= \Spa(A,A^+)$ be a good affinoid adic space over $\Q_p$.
\begin{enumerate}
    \item let $\Gs$ be a dualizable analytic $p$-divisible group over $S$. Then $v$-locally on $S$, $\Gs$ has good reduction over $A^{\circ}$.
    \item Let $f\colon \Gs \rightarrow \Hs$ be a morphism of dualizable analytic $p$-divisible groups over $S$. Then $v$-locally on $S$, $f$ has good reduction over $A^{\circ}$.
    \item The adic generic fiber functor from $p$-divisible groups over $A^+$ to analytic $p$-divisible groups over $S$ is faithful.
\end{enumerate}
\end{proposition}
\begin{proof}
All statements can be checked $v$-locally on $S$. Hence we may assume that $S$ is a product of points, so that $A^+= \prod_i C_i^+$, where each $C_i^+$ is an open and bounded valuation subring in a complete algebraically closed field $C_i$. By the proof of \cite[Thm. 17.5.2]{SW20} (see also \cite[Cor. 3.18]{zhang2023peltypeigusastackpadic}), $p$-divisible groups over $A^+$ are equivalent to collections of $p$-divisible groups over $C_i^+$. By Proposition \ref{prop: dualiazable pdiv gps over product of points}, we are reduced to the case of a single point $(A,A^+)=(C,C^+)$. The first and second point now are consequence of Scholze--Weinstein's result, Theorem \ref{thm: Scholze Weinstein's result}, and the third point follows by combining it with the observation that $\Spf(\Os_C)$ is dense in $\Spf(C^+)$. 
% It remains to show the first point. By Scholze--Weinstein, $\Gs \restr{\Spa(C,\Os_C)}$ has a formal model $\mathfrak{G}_0$ over $\Os_C$ and it remains to show that any such $p$-divisible group is already defined over $C^+$. Let $k$ denote the residue field of $\Os_C$ and let $\cj{C^+} \sub k$ denote the image of $C^+$. Since $\cj{C^+}$ is a valuation subring of the algebraically closed field $k$, it must contain $\cj{\F}_p$. By the Dieudonné-Mannin classification, the special fiber $G_0 = \mathfrak{G}_0 \otimes_{\Os_C} k$ is already defined over $\cj{\F}_p$ up to isogeny. Hence we may find a $p$-divisible group $H$ over $\cj{C^+}$ and a quasi-isogeny $\rho \colon H\otimes k \dashrightarrow G_0$, which we may take to be an isogeny, up to replacing it by a multiple of $p$. Let $K_0= \Ker(\rho)$, a finite subgroup scheme of $H \otimes k$. By \cite[Lemma 8.10]{zhang2023peltypeigusastackpadic}\footnote{The referenced statement has the requirement that the ambient group $A$ is an abelian variety. However, the proof only uses that $A$ is proper, so that in our case, we may take $A=H[p^N]$ for $N \gg 0$.}, the group scheme $K_0$ extends to a finite flat subgroup scheme $K \sub H$ over $\cj{C^+}$. Then $G = H/K$ is a $p$-divisible group over $\cj{C^+}$ extending $G_0$. By glueing along the span $\cj{C^+} \leftarrow C^+ \rightarrow \Os_C$, see \cite[Lemma 3.14]{zhang2023peltypeigusastackpadic}, $G$ deforms to a $p$-divisible group $\mathfrak{G}$ over $C^+$ such that $\mathfrak{G} \otimes_{C^+} \Os_C = \mathfrak{G}_0$, as required.
\end{proof}

\begin{remark}
We leave the following statements, unknown to us, as an open question: Let $S=\Spa(A,A^+)$ be a good adic space and let $\Gs \rightarrow S$ be a dualizable analytic $p$-divisible group.
\begin{enumerate}
    \item Does $\Gs$ arises from a $p$-divisible group over $A^{\circ}$, quasi-pro-étale locally on $S$?
    \item Is there an admissible formal blowup $\Ss \rightarrow \Spf(A^+)$ such that $\Gs$ arises from a $p$-divisible group over $\Ss$?
\end{enumerate}
\end{remark}

\subsection{Effective Banach--Colmez spaces}\label{subsect: effective BC spaces}
In this subsection, we generalize the classification of analytic $p$-divisible groups of Theorem \ref{thm: extending Fargues' equivalence of categories} to effective Banach--Colmez spaces.

\begin{definition}\label{def: minuscule BC space}
    Let $S$ be a locally spatial diamond over $\Q_p$. An effective Banach--Colmez space over $S$ is a sheaf $\Fs$ of abelian group on $S_v$ such that there exists a short exact sequence of $v$-sheaves
        % https://q.uiver.app/#q=WzAsNSxbMCwwLCIwIl0sWzEsMCwiXFxWIl0sWzIsMCwiXFxGcyJdLFszLDAsIkUiXSxbNCwwLCIwLiJdLFswLDFdLFsxLDJdLFsyLDNdLFszLDRdXQ==
    \[\begin{tikzcd}
	0 & \V & \Fs & E & {0,}
	\arrow[from=1-1, to=1-2]
	\arrow[from=1-2, to=1-3]
	\arrow[from=1-3, to=1-4]
	\arrow[from=1-4, to=1-5]
\end{tikzcd}\]
where $\V$ is a $\Q_p$-local system on $S_v$ and $E$ is an étale vector bundle on $S$, viewed as a sheaf on $S_v$.
    % \begin{enumerate}
    %     \item 
    % % \item An minuscule Banach--Colmez space is called effective if it admits as a presentation as above where $\V$ admits a $\Z_p$-lattice, i.e there exists a $\Z_p$-local system $\Lb\sub \V$ with $\Lb[\tfrac{1}{p}] = \V$.
    % \end{enumerate}
\end{definition}
We stress that an effective Banach--Colmez space $\Fs$ does not intrinsically determine the short exact sequence $\V \rightarrow \Fs \rightarrow E$. We will call such a sequence a presentation of $\Fs$.

 The above definition is a relative variant of a construction of Colmez \cite{Colmez_2002} and Fontaine \cite{Fontaine03}. Those were later reworked in terms of coherent sheaves on the Fargues--Fontaine curve \cite{Le_Bras_2018}. In Subsection \ref{subsect: Analytic $p$-divisible groups and vector bundles}, we will relate our definition to the relative Fargues--Fontaine curve and the families of Banach--Colmez spaces in \cite{fargues2024geometrization}.

\begin{example}\label{ex: univ cover of unipotent units}
    Let $S=\Spd(\Q_p)$ and $\Fs = (\B_{\crys}^+)^{\varphi = p}$, where $\B_{\crys}^+$ is the crystalline period $v$-sheaf, defined by the formula 
    \[ S=\Spa(R,R^+) \in \Perf_{\Q_p} \longmapsto A_{\crys}(R^+)[\tfrac{1}{p}].\]
    Here $A_{\crys}(R^+)$ is the p-adically completed divided power envelope of the surjection $\Ainf(R^+) \rightarrow R^+/p$.
    %$p$-adic completion of the divided-power envelope of $(W(R^+),\Ker(\theta))$. 
    It comes with the exact sequence of $v$-sheaves
    % https://q.uiver.app/#q=WzAsNSxbMCwwLCIwIl0sWzEsMCwiXFxRX3AoMSkiXSxbMiwwLCIoXFxCX3tcXGNyeXN9XispXntcXHZhcnBoaSA9cH0iXSxbMywwLCJcXEdfYSJdLFs0LDAsIjAsIl0sWzAsMV0sWzEsMl0sWzIsMywiXFx0aGV0YSJdLFszLDRdXQ==
\begin{equation}\label{eq: fund eq padic HT}\begin{tikzcd}
	0 & {\Q_p(1)} & {(\B_{\crys}^+)^{\varphi =p}} & {\G_a} & {0.}
	\arrow[from=1-1, to=1-2]
	\arrow[from=1-2, to=1-3]
	\arrow["\theta", from=1-3, to=1-4]
	\arrow[from=1-4, to=1-5]
\end{tikzcd}\end{equation}
By \cite[Prop. II.2.2]{fargues2024geometrization}, there is an isomorphism
\[ (\B_{\crys}^+)^{\varphi =p} = \varprojlim_{[p]} \G_m\langle p^{\infty} \rangle, \]
under which the map $\theta$ is identified with the following composition
% https://q.uiver.app/#q=WzAsMyxbMCwwLCJcXHZhcnByb2psaW1fe1twXX1cXEdfbVxcbGFuZ2xlIHBee1xcaW5mdHl9IFxccmFuZ2xlIl0sWzEsMCwiXFxHX21cXGxhbmdsZSBwXntcXGluZnR5fSBcXHJhbmdsZSJdLFsyLDAsIlxcR19hLiJdLFsxLDIsIlxcbG9nIl0sWzAsMV1d
\[\begin{tikzcd}
	{\varprojlim_{[p]}\G_m\langle p^{\infty} \rangle} & {\G_m\langle p^{\infty} \rangle} & {\G_a.}
	\arrow[from=1-1, to=1-2]
	\arrow["\log", from=1-2, to=1-3]
\end{tikzcd}\]
\end{example}

As this example suggests, interesting examples of effective Banach--Colmez spaces are given by $p$-adic universal covers of analytic $p$-divisible groups. This was already observed in \cite{Farg19}.
\begin{lemma}\label{lemma: effective BC spaces vs univ cover}
    Let $\Gs$ be an analytic $p$-divisible group over a good adic space $S/\Q_p$ (or more generally a locally spatial diamond, using Remark \ref{rmk: p-divisible groups over LSD}). Then the $p$-adic universal cover
    \[ \widetilde{\Gs} = \varprojlim_{[p]} \Gs\]
    viewed as a group diamond over $S$, fits in a short exact sequence of $v$-sheaves
    % https://q.uiver.app/#q=WzAsNSxbMCwwLCIwIl0sWzEsMCwiVl9wXFxHcyJdLFsyLDAsIlxcd2lkZXRpbGRle1xcR3N9Il0sWzMsMCwiXFxnYSJdLFs0LDAsIjAuIl0sWzAsMV0sWzEsMl0sWzIsM10sWzMsNF1d
\begin{equation}\label{eq: inv limit of log}\begin{tikzcd}
	0 & {V_p\Gs} & {\widetilde{\Gs}} & \ga & {0.}
	\arrow[from=1-1, to=1-2]
	\arrow[from=1-2, to=1-3]
	\arrow[from=1-3, to=1-4]
	\arrow[from=1-4, to=1-5]
\end{tikzcd}\end{equation}
In particular, $\Gs$ determines an effective Banach--Colmez space $\widetilde{\Gs}$ with a presentation. Conversely, if $\V \rightarrow \Fs \rightarrow E$ is a presented effective Banach--Colmez space such that $\V$ admits a $\Z_p$-lattice $\Lb\sub \V$, then there exists a unique analytic $p$-divisible group $\Gs$ over $S$ such that $\Fs = \widetilde{\Gs}$ compatibly with the presentations and with $\Lb =T_p\Gs$. 
\end{lemma}
\begin{proof}
    The displayed exact sequence arises from applying $\varprojlim_{[p]}$ to the logarithm exact sequence (\ref{eq: logarithm exact sequence for pdiv groups}) and using that the $v$-site is replete \cite[Lemma 2.6]{heuer2021lineonpefd}. Conversely, consider the following pushout in $v$-sheaves
    % https://q.uiver.app/#q=WzAsMTAsWzAsMCwiMCJdLFsxLDAsIlxcViJdLFsyLDAsIlxcRnMiXSxbMywwLCJFIl0sWzQsMCwiMCJdLFszLDEsIkUiXSxbMiwxLCJcXEdzIl0sWzEsMSwiXFxWL1xcTGIiXSxbMCwxLCIwIl0sWzQsMSwiMC4iXSxbMCwxXSxbMSwyXSxbMiwzXSxbMyw0XSxbMyw1LCI9Il0sWzIsNl0sWzYsNV0sWzEsN10sWzcsNl0sWzgsN10sWzUsOV0sWzYsMSwiIiwwLHsic3R5bGUiOnsibmFtZSI6ImNvcm5lciJ9fV1d
\[\begin{tikzcd}
	0 & \V & \Fs & E & 0 \\
	0 & {\V/\Lb} & \Gs & E & {0.}
	\arrow[from=1-1, to=1-2]
	\arrow[from=1-2, to=1-3]
	\arrow[from=1-2, to=2-2]
	\arrow[from=1-3, to=1-4]
	\arrow[from=1-3, to=2-3]
	\arrow[from=1-4, to=1-5]
	\arrow["{=}", from=1-4, to=2-4]
	\arrow[from=2-1, to=2-2]
	\arrow[from=2-2, to=2-3]
	\arrow["\lrcorner"{anchor=center, pos=0.125, rotate=180}, draw=none, from=2-3, to=1-2]
	\arrow[from=2-3, to=2-4]
	\arrow[from=2-4, to=2-5]
\end{tikzcd}\]
Then it follows from \cite[Lemma 3.38]{gerth2024} that $\Gs$ is representable by an analytic $p$-divisible group, and we immediately see that $\Fs = \widetilde{\Gs}$.
\end{proof}

\begin{lemma}
    Let $\Gs \rightarrow S$ be a dualizable analytic $p$-divisible group. Then for any map $S' \rightarrow S$ from a strictly totally disconnected perfectoid space $S'$, the pullback $\widetilde{\Gs}_{S'}$ is perfectoid. 
\end{lemma}
\begin{proof}
    We may assume that $S$ is strictly totally disconnected, so that $T_p\Gs$ is trivial and Tate twists are trivialized. We have a commutative diagram where both squares are cartesian
    % https://q.uiver.app/#q=WzAsNixbMSwwLCJcXEdzIl0sWzIsMCwiXFxnYSJdLFsyLDEsIlRfcFxcR3MoLTEpXFxvdGltZXNfe1xcWl9wfVxcR19hLiJdLFsxLDEsIlRfcFxcR3MoLTEpXFxvdGltZXNfe1xcWl9wfVxcR19tXFxsYW5nbGUgcF57XFxpbmZ0eX0gXFxyYW5nbGUiXSxbMCwwLCJcXHdpZGV0aWxkZXtcXEdzfSJdLFswLDEsIlRfcFxcR3MoLTEpXFxvdGltZXNfe1xcWl9wfVxcd2lkZXRpbGRle1xcR19tXFxsYW5nbGUgcF57XFxpbmZ0eX0gXFxyYW5nbGV9Il0sWzAsMSwiXFxsb2dfe1xcR3N9Il0sWzEsMiwiZiIsMCx7InN0eWxlIjp7InRhaWwiOnsibmFtZSI6Imhvb2siLCJzaWRlIjoidG9wIn19fV0sWzAsMywidSIsMix7InN0eWxlIjp7InRhaWwiOnsibmFtZSI6Imhvb2siLCJzaWRlIjoidG9wIn19fV0sWzMsMiwiXFxsb2ciLDJdLFs0LDBdLFs1LDMsIlxccGkiLDJdLFs0LDUsIiIsMSx7InN0eWxlIjp7InRhaWwiOnsibmFtZSI6Imhvb2siLCJzaWRlIjoidG9wIn19fV1d
\[\begin{tikzcd}
	{\widetilde{\Gs}} & \Gs & \ga \\
	{T_p\Gs(-1)\otimes_{\Z_p}\widetilde{\G_m\langle p^{\infty} \rangle}} & {T_p\Gs(-1)\otimes_{\Z_p}\G_m\langle p^{\infty} \rangle} & {T_p\Gs(-1)\otimes_{\Z_p}\G_a.}
	\arrow[from=1-1, to=1-2]
	\arrow[hook, from=1-1, to=2-1]
	\arrow["{\log_{\Gs}}", from=1-2, to=1-3]
	\arrow["u"', hook, from=1-2, to=2-2]
	\arrow["f", hook, from=1-3, to=2-3]
	\arrow["\pi"', from=2-1, to=2-2]
	\arrow["\log"', from=2-2, to=2-3]
\end{tikzcd}\]
Observe that $T_p\Gs(-1) \otimes \G_m\langle p^{\infty} \rangle \cong \G_m\langle p^{\infty} \rangle^{\oplus n}$ is a smooth $S$-space and $T_p\Gs(-1) \otimes \widetilde{\G_m\langle p^{\infty} \rangle} \cong (\widetilde{\G_m\langle p^{\infty} \rangle})^{\oplus n}$ is perfectoid. As $f$ is a locally direct summand, it follows that $f$ and hence also $u$ are Zariski-closed immersions of smooth $S$-spaces, in the sense of \cite[Def. IV.4.20]{fargues2024geometrization}. Let us fix an open affinoid subspace $W=\Spa(A,A^+)\sub T_p\Gs(-1)\otimes \G_m\langle p^{\infty} \rangle$ and write $u^{-1}(W) = V(I)$ for some ideal $I \sub A$. Let $\widetilde{W} = \pi^{-1}(W)$ be its preimage. Then if $s_i\in A^{\times}$ are the restriction of the variables $T_i\in \Os(\G_m\langle p^{\infty} \rangle^{\oplus n})$, $i=1,\ldots, n$
\[ \widetilde{W} = \Spa(A\langle s_i^{1/p^{\infty}}\rangle ,A^+\langle s_i^{1/p^{\infty}}\rangle) \]
is affinoid perfectoid and we let $J \sub A\langle s_i^{1/p^{\infty}}\rangle$ be the ideal generated by $I$. 
% The locus $V(J)\sub \vert \widetilde{W} \vert $ can be promoted to a diamond, by expressing it as an inverse limit of rational open neighborhoods
% \[V(J) = \bigcap_{f_1,\ldots, f_m\in I} U(f_1,\ldots, f_m).\]
% It is clear that the following commutative square of diamonds is cartesian 
% % https://q.uiver.app/#q=WzAsNCxbMSwwLCJWKEkpIl0sWzEsMSwiVy4iXSxbMCwwLCJWKEopIl0sWzAsMSwiXFx3aWRldGlsZGV7V30iXSxbMCwxLCJ1IiwwLHsic3R5bGUiOnsidGFpbCI6eyJuYW1lIjoiaG9vayIsInNpZGUiOiJ0b3AifX19XSxbMiwwXSxbMywxLCJcXHBpIiwyXSxbMiwzLCIiLDEseyJzdHlsZSI6eyJ0YWlsIjp7Im5hbWUiOiJob29rIiwic2lkZSI6InRvcCJ9fX1dXQ==
% \[\begin{tikzcd}
% 	{V(J)} & {V(I)} \\
% 	{\widetilde{W}} & {W.}
% 	\arrow[from=1-1, to=1-2]
% 	\arrow[hook, from=1-1, to=2-1]
% 	\arrow["u", hook, from=1-2, to=2-2]
% 	\arrow["\pi"', from=2-1, to=2-2]
% \end{tikzcd}\]
Then it follows from the existence of perfectoidization for semiperfectoid rings (\cite[Thm. 1.12(2)]{BhattScholze2022prisms}) that $\pi^{-1}(V(I)) = V(J)$ is representable by an affinoid perfectoid space, see \cite[II.0.2]{fargues2024geometrization}. This shows that $\widetilde{\Gs}$ has an open cover by affinoid perfectoid spaces, which concludes the proof.
\end{proof}

We now have an analogous result to Proposition \ref{Prop: extension from p-torsion to the whole group}.
\begin{proposition}\label{Prop: extension from tate module to the whole BC space}
    Let $S/\Q_p$ be a locally spatial diamond, let $\V \rightarrow \Fs \rightarrow E$ be an effective Banach--Colmez space over $S$ and let $\V'$ be a $\Q_p$-local system on $S$. Then we have an isomorphism
    \[ \underline{\Hom}_{\Q_p}(\V,  \V'\otimes_{\Q_p} \widetilde{\G_m\langle p^{\infty} \rangle}) =  \underline{\Hom}_{\Q_p}(\Fs, \V'\otimes_{\Q_p} \widetilde{\G_m\langle p^{\infty} \rangle}). \]
\end{proposition}
\begin{proof}
First observe that a Banach--Colmez space $\Fs$ is its own coextension of scalar
\[ \Fs = \underline{\Hom}_{\Z}(\Q_p,\Fs).\]
Indeed, this can be checked $v$-locally, so we may assume that $\V$ has a $\Z_p$-lattice, in which case $\Fs= \widetilde{\Gs}$ for an analytic $p$-divisible group $\Gs$, by Lemma \ref{lemma: effective BC spaces vs univ cover}. It now follows from the fact that, by Lemma \ref{Lemma: turn Hom into tensor}, we have $\Gs = \underline{\Hom}_{\Z}(\Z_p,\Gs)$.

From this, morphisms of Banach--Colmez spaces are automatically $\Q_p$-linear. We are thus reduced to working in the category of sheaves of abelian groups. Consider the exact sequence, where we write $\underline{\Ext}^i(-,-)$ for $\underline{\Ext}_{\Z}^i(-,-)$
% https://q.uiver.app/#q=WzAsNCxbMCwwLCIwIl0sWzEsMCwiXFx1bmRlcmxpbmV7XFxIb219KFxcRnMsIFxcVidcXG90aW1lcyBcXHdpZGV0aWxkZXtcXEdfbVxcbGFuZ2xlIHBee1xcaW5mdHl9IFxccmFuZ2xlfSkpIl0sWzIsMCwiXFx1bmRlcmxpbmV7XFxIb219KFxcViwgXFxWJ1xcb3RpbWVzIFxcd2lkZXRpbGRle1xcR19tXFxsYW5nbGUgcF57XFxpbmZ0eX0gXFxyYW5nbGV9KSJdLFszLDAsIlxcdW5kZXJsaW5le1xcRXh0fV92XjEoRSwgXFxWJ1xcb3RpbWVzIFxcd2lkZXRpbGRle1xcR19tXFxsYW5nbGUgcF57XFxpbmZ0eX0gXFxyYW5nbGV9KS4iXSxbMCwxXSxbMSwyXSxbMiwzLCJcXGRlbHRhIl1d
\[\begin{tikzcd}
	0 & {\underline{\Hom}(\Fs, \V'\otimes \widetilde{\G_m\langle p^{\infty} \rangle}))} & {\underline{\Hom}(\V, \V'\otimes \widetilde{\G_m\langle p^{\infty} \rangle})} & {\underline{\Ext}_v^1(E, \V'\otimes \widetilde{\G_m\langle p^{\infty} \rangle}).}
	\arrow[from=1-1, to=1-2]
	\arrow[from=1-2, to=1-3]
	\arrow["\delta", from=1-3, to=1-4]
\end{tikzcd}\]
The leftmost term is zero, since 
\[ \underline{\Hom}(E,\widetilde{\G_m \langle p^{\infty}\rangle}) = \varprojlim_{[p]}\underline{\Hom}(E,\G_m \langle p^{\infty}\rangle)=0.\]
It remains to show that $\delta=0$. This can be checked locally on $S_v$, hence we may assume that $S$ is strictly totally disconnected, $\V,\V'$ are trivial $\Q_p$-local systems and $E$ is a trivial vector bundle. We easily reduce to the case of $\V'=\Q_p$. It suffices now to show that, for any strictly totally disconnected space $S'$ over $S$ (that we just rename $S$), the connecting map
\[ \delta\colon \Hom(\V,\widetilde{\G_m\langle p^{\infty} \rangle}) \rightarrow \Ext_v^1(E,\widetilde{\G_m\langle p^{\infty} \rangle})  \]
is zero. Since the underlying topos of $S_v$ is replete, we have $\widetilde{\G_m\langle p^{\infty} \rangle} = R\varprojlim_{[p]} \G_m\langle p^{\infty} \rangle$ and
\[ R\Hom_v(E,\widetilde{\G_m\langle p^{\infty} \rangle}) = R\varprojlim_{[p]}R\Hom_v(E,\G_m\langle p^{\infty} \rangle), \]
which yields a spectral sequence
\[ E_2^{ij} = R^i\varprojlim_{[p]} \Ext_v^j(E,\G_m\langle p^{\infty} \rangle) \Rightarrow \Ext_v^{i+j}(E,\widetilde{\G_m\langle p^{\infty} \rangle}).\]
The terms of the system 
\[ (\ldots \xrightarrow{[p]}\Hom(E,\G_m\langle p^{\infty} \rangle)\xrightarrow{[p]}\Hom(E,\G_m\langle p^{\infty} \rangle)\xrightarrow{[p]}\ldots )\]
all vanish, so that it has no $R^1\lim$ term. It follows that
\[ \Ext_v^{1}(E,\widetilde{\G_m\langle p^{\infty} \rangle}) = \varprojlim_{[p]}  \Ext_v^1(E,\G_m\langle p^{\infty} \rangle).\]
Fix a $\Z_p$-latice $\Lb \sub \V$ and write $\Fs =\widetilde{\Gs}$ as the universal cover of an analytic $p$-divisible group $\Gs$ with Lie algebra $E$ and Tate module $\Lb$, using Lemma \ref{lemma: effective BC spaces vs univ cover}. There is a connecting homomorphism
\[ \delta'\colon \Hom(\V/\Lb,\G_m\langle p^{\infty} \rangle) \rightarrow \Ext_{v}^1(E,\G_m\langle p^{\infty} \rangle)  \]
arising from the logarithm sequence (\ref{eq: logarithm exact sequence for pdiv groups}) and one easily sees that 
\[ \delta = \varprojlim_{[p]}\delta' \in \Ext_v^{1}(E,\widetilde{\G_m\langle p^{\infty} \rangle}).\]
Now by Proposition \ref{Prop: extension from p-torsion to the whole group}, $\Hom(\V/\Lb, \G_m\langle p^{\infty} \rangle) = \Hom(\Gs, \G_m\langle p^{\infty} \rangle)$, such that $\delta'=0$. This shows that $\delta$ vanishes and concludes the proof.
\end{proof}

\begin{cor}
    Let $\V \rightarrow \Fs \rightarrow E$ be an effective Banach--Colmez space with a presentation. The perfect pairing
    \[ \V \times \V^{\vee}(1) \rightarrow \Q_p(1)\]
    admits a unique extension
    \[ e_{\Fs}\colon\Fs \times \V^{\vee}(1) \rightarrow (\B_{\crys}^+)^{\varphi =p}.\]
\end{cor}
\begin{proof}
    This follows from Proposition \ref{Prop: extension from tate module to the whole BC space} and Example \ref{ex: univ cover of unipotent units}.
\end{proof}
\begin{definition}
    Let $\V \rightarrow \Fs \rightarrow E$ be an effective Banach--Colmez space with a presentation, we define
    \[ u_{\Fs}\colon \Fs \rightarrow \underline{\Hom}(\V^{\vee}(1), (\B_{\crys}^+)^{\varphi =p})\]
    from currying the pairing $e_{\Fs}$ of the previous corollary. By construction, it maps $\V$ to $\underline{\Hom}(\V^{\vee}(1), \Q_p(1))$. We let 
    \[ f_{\Fs}\colon E \rightarrow \underline{\Hom}(\V^{\vee}(1),\G_a)\]
    denote the map induced by passing to quotients.
\end{definition}
We immediately obtain the following.
\begin{proposition}
    Let $\V \rightarrow \Fs \rightarrow E$ be a presented effective Banach--Colmez space, then the maps $f_{\Fs}$ and $u_{\Fs}$ fit in the following commutative diagram
    % https://q.uiver.app/#q=WzAsMTAsWzEsMCwiXFxWIl0sWzIsMCwiXFxGcyJdLFszLDAsIkUiXSxbMCwwLCIwIl0sWzQsMCwiMCJdLFswLDEsIjAiXSxbMSwxLCJcXFYoLTEpXFxvdGltZXNfe1xcUV9wfSBcXFEoMSkiXSxbMiwxLCJcXFYoLTEpXFxvdGltZXNfe1xcUV9wfSAoXFxCX3tcXGNyeXN9XispXntcXHZhcnBoaSA9cH0iXSxbMywxLCJcXFYoLTEpXFxvdGltZXNfe1xcUV9wfSBcXEdfYSJdLFs0LDEsIjAuIl0sWzMsMF0sWzEsMl0sWzAsMV0sWzIsNF0sWzUsNl0sWzYsN10sWzEsNywidV97XFxGc30iXSxbMCw2LCI9IiwyXSxbNyw4XSxbMiw4LCJmX3tcXEZzfSJdLFs4LDldXQ==
\begin{equation}\label{eq: u f diagram for BC space}\begin{tikzcd}
	0 & \V & \Fs & E & 0 \\
	0 & {\V(-1)\otimes_{\Q_p} \Q(1)} & {\V(-1)\otimes_{\Q_p} (\B_{\crys}^+)^{\varphi =p}} & {\V(-1)\otimes_{\Q_p} \G_a} & {0.}
	\arrow[from=1-1, to=1-2]
	\arrow[from=1-2, to=1-3]
	\arrow["{=}"', from=1-2, to=2-2]
	\arrow[from=1-3, to=1-4]
	\arrow["{u_{\Fs}}", from=1-3, to=2-3]
	\arrow[from=1-4, to=1-5]
	\arrow["{f_{\Fs}}", from=1-4, to=2-4]
	\arrow[from=2-1, to=2-2]
	\arrow[from=2-2, to=2-3]
	\arrow[from=2-3, to=2-4]
	\arrow[from=2-4, to=2-5]
\end{tikzcd}\end{equation}
\end{proposition}

We now reach the following analog of Theorem \ref{thm: extending Fargues' equivalence of categories}.

\begin{thm}\label{thm: a fargues thm for BC spaces}
Let $S$ be a locally spatial diamond over $\Spd(\Q_p)$. The following categories are equivalent:
    \begin{enumerate}
        \item The category of presented effective Banach--Colmez spaces $\V \rightarrow \Fs \rightarrow E$, with presentation-preserving morphisms, and
        \item The category of tuples $(\V,E,f)$, where $\V$ is a $\Q_p$-local system on $S_v$, $E$ is a vector bundle over $S_{\et}$ and $f\colon E \otimes_{\Os_{S_{\et}}}\Os_{S_v} \rightarrow \V(-1)\otimes_{\underline{\Q_p}}\Os_{S_v}$ is a morphism of $v$-vector bundles.
    \end{enumerate}
    If $S$ is a strictly totally disconnected perfectoid space, these are further equivalent to
    \begin{itemize}
        \item[(3)] The category of analytic $p$-divisible groups and quasi-isogenies.
    \end{itemize}
\end{thm}
\begin{remark}
    In particular, given a $\Q_p$-local system $\V$ and a vector bundle $E$ on $S$, we obtain an isomorphism
    \begin{align} \Ext_{S_v}^1(E,\V) = \Hom_{S_v}(E,\V(-1) \otimes_{\Q_p} \Os_{S_v}).
    \end{align}
    This recovers a computation of Fontaine \cite[Prop. 3.12]{Fontaine03}, which is the case $S=\Spa(K)$ for a $p$-adic field $K$, and of Colmez \cite[Prop. 1.7]{Colmez_2002}, which is the case $S=\Spa(C)$.
\end{remark}
\begin{proof}
    This follows from the same arguments as in the proof of Theorem \ref{thm: extending Fargues' equivalence of categories}. The only non-trivial point is the following: Let $E,E'$ denote étale vector bundles over $S$, then any morphism of abelian sheaves $f\colon E\otimes_{\Os_{S_{\et}}}\Os_{S_v} \rightarrow E'\otimes_{\Os_{S_{\et}}}\Os_{S_v} $ is induced by a unique morphism of vector bundle on $S_{\et}$. As Scholze's diamond functor is fully faithful on good adic spaces, $f$ comes from a morphism of smooth $S$-groups $f\colon E \otimes_{\Os_{S_{\et}}} \G_a \rightarrow E' \otimes_{\Os_{S_{\et}}} \G_a$. It is then clear from the naturality of the logarithm that $f$ agrees with its derivative and is thus $\Os_S$-linear. 
\end{proof}

\begin{remark}
\begin{enumerate}
    \item We stress that in Theorem \ref{thm: a fargues thm for BC spaces}, we have to ask for morphisms of Banach--Colmez spaces to respect presentations. This is not automatic as when dealing with $p$-divisible groups. For example, let $C$ be an algebraically closed non-archimedean field over $\Q_p$ with residue field $k$ and let $G$ be a $p$-divisible group over $k$ that admits two non-isogenous lifts $\mathfrak{G},\mathfrak{G}'$ over $\Os_{C}$. Then by the crystalline nature of the $p$-adic universal covers, there exists an isomorphism $\widetilde{\mathfrak{G}} \cong \widetilde{\mathfrak{G}}'$. The
    generic fibers $\Gs$ and $\Gs'$ are analytic $p$-divisible groups over $\Spa(C)$ with an isomorphism $\varphi\colon \widetilde{\Gs} \xrightarrow{\cong} \widetilde{\Gs}'$. As $\Gs$ and $\Gs'$ are non-isogenous, $\varphi$ cannot respect the presentations. 
    \item It follows from Theorem \ref{thm: a fargues thm for BC spaces} that presented effective Banach--Colmez spaces satisfy descent along $v$-covers of perfectoid spaces. This does not hold if we do not ask for the presentations to be preserved. For example, let $\Ss$ be a $p$-adic formal scheme with generic fiber $S$ and let $G$ be a $p$-divisible group over $\Ss \otimes_{\Z_p} \F_p$. Then the universal cover $\widetilde{G}$ deforms canonically to a sheaf on $\Ss$ and its generic fiber $\Fs \rightarrow S$ locally admits a presentation $\V \rightarrow \Fs \rightarrow E$, but it need not admit a global presentation. This corresponds to the fact that there always exist local lifts $\mathfrak{G}$ of $G$ over affine opens $\Spf(R)\sub \Ss$ but not necessarily over the whole space $\Ss$.
\end{enumerate}    
\end{remark}
    
    There is again a notion of dualizable effective Banach--Colmez spaces, by extending
    Definition \ref{def: dualizable}.
    \begin{definition}\label{def: dualizable BC spaces}
        Let $\V \rightarrow \Fs \rightarrow E$ be an presented effective Banach--Colmez space over $S$. We say that it is dualizable if the associated map $f=f_{\Fs}$ fits in a short exact sequence of $v$-vector bundles
        % % https://q.uiver.app/#q=WzAsNSxbMCwwLCIwIl0sWzEsMCwiRVxcb3RpbWVzX3tcXE9zX3tTX3tcXGV0fX19XFxPc197U192fSJdLFsyLDAsIlxcVigtMSlcXG90aW1lc197XFxRX3B9XFxPc197U192fSJdLFszLDAsIlxcb21lZ2FcXG90aW1lc197XFxPc197U197XFxldH19fVxcT3Nfe1Nfdn0oLTEpIl0sWzQsMCwiMCwiXSxbMCwxXSxbMSwyLCJmIl0sWzIsMywiZyJdLFszLDRdXQ==
\begin{equation}\begin{tikzcd}
	0 & {E\otimes_{\Os_{S_{\et}}}\Os_{S_v}} & {\V(-1)\otimes_{\Q_p}\Os_{S_v}} & {\omega\otimes_{\Os_{S_{\et}}}\Os_{S_v}(-1)} & {0,}
	\arrow[from=1-1, to=1-2]
	\arrow["f", from=1-2, to=1-3]
	\arrow["g", from=1-3, to=1-4]
	\arrow[from=1-4, to=1-5]
\end{tikzcd}\end{equation}
for a vector bundle $\omega$ on $S_{\et}$. The Cartier dual $\Fs^D$ is the effective Banach--Colmez space associated to the tuple $(\omega^{\vee}, \V^{\vee}(1), g^{\vee}(1))$.
    \end{definition}
    \begin{remark}
        \begin{enumerate}
            \item This is compatible with analytic $p$-divisible groups, i.e. if $\Gs$ is an analytic $p$-divisible group, then $\Gs$ is dualizable if and only if $\widetilde{\Gs}$ is dualizable. 
            \item If $\V \rightarrow \Fs \rightarrow E$ is an effective Banach--Colmez space over a perfectoid space $S$, we will show later in Proposition \ref{prop: G dualizable iff E(G) vb}, exploiting the link with the Fargues--Fontaine curve, that the property of being dualizable only depends on the sheaf $\Fs$ and not on the presentation. We do not know if this holds true for a general base $S$.
            \item Let $\Fs$ be a dualizable Banach--Colmez space over an arbitrary base $S$. We will later obtain an explicit formula for $\Fs^D$ in Proposition \ref{prop: explicit formula for cartier dual}, which in particular does not depend on the choice of presentation. 
        \end{enumerate}
    \end{remark}

\section{Examples of analytic $p$-divisible groups}\label{section: Examples of analytic $p$-divisible groups}
We now discuss various constructions giving rise to interesting examples of analytic $p$-divisible groups.
\subsection{Generic fibers of log $p$-divisible groups}\label{subsection: generic fibers of log p-divisible groups}
In this subsection, we construct an adic generic fiber functor from Kato's category of log $p$-divisible groups \cite{kato2023logarithmicdieudonnetheory} to analytic $p$-divisible groups. We show that it is fully faithful in the semistable case, using log prismatic Dieudonné theory \cite{inoue2025logprismaticdieudonnetheory}\cite{würthen2023logprismaticdieudonnetheory} and the étale realization functor \cite{koshikawa2023logarithmicprismaticcohomologyii}\cite{du2024logprismaticfcrystalspurity}.

Throughout this subsection, $(\Xs,\Ms_{\Xs})$ will denote a fs log formal scheme over $\Z_p$, where $\Xs$ locally admits a finitely generated ideal of definition. Here, we endow $\Z_p$ with the standard log structure $\Q_p^{\times} \cap \Z_p \hookrightarrow \Z_p$. We will be mostly interested in the case where the structure map $(\Xs,\Ms_{\Xs}) \rightarrow \Spf(\Z_p)$ is strict, that is, the log structure on $\Xs$ is the pullback of the log structure on $\Z_p$.

We let $(\Xs,\Ms_{\Xs})_{\kfl}$ denote the Kummer log flat site \cite[§2]{Kato2021}, with underlying category the fs log formal schemes adic over $(\Xs,\Ms_{\Xs})$, and with the Grothendieck topology given by the Kummer log flat (kfl in short) coverings. We let $\Vect_{\kfl}(\Xs,\Ms_{\Xs})$ denote the category of vector bundles on $(\Xs,\Ms_{\Xs})_{\kfl}$.

For any formal scheme $\Ys$ over $\Xs$, we may endow $\Ys$ with the pullback log structure. This defines a morphism of sites 
\begin{align}
    u\colon (\Xs,\Ms_{\Xs})_{\kfl} \rightarrow \Xs_{\fppf},
\end{align}
where the target is the site consisting of formal schemes adic over $\Xs$ together with the fppf topology. By \cite[Lemma 2.15]{inoue2025logprismaticdieudonnetheory}, the resulting functor
\begin{align}
    u^*\colon \Vect(\Xs) \rightarrow \Vect_{\kfl}(\Xs,\Ms_{\Xs})
\end{align}
is fully faithful. A kfl vector bundle on $(\Xs,\Ms_{\Xs})$ is called classical if it is in the essential image of the above functor.

There is a functor taking an fs formal log scheme $(\Xs,\Ms_{\Xs})$ as above to a log adic space $(\Xs, \Ms_{\Xs})_{\eta}$ over $\Spa(\Q_p)$, extending the adic generic fiber functor on formal schemes of \cite[Prop. 2.2.2]{scholze2013moduli}, see \cite[Prop. 2.2.22]{DLLZ23}. We now assume that the log structure on $\Xs$ is the pullback of the standard log structure on $\Z_p$, such that the resulting log structure on $\Xs_{\eta}$ is trivial. We then obtain a morphism of sites $\Xs_{\eta, \an} \rightarrow (\Xs,\Ms_{\Xs})_{\Zar}$. More generally, given a Zariski sheaf $\Fs$ on adic log formal schemes over $(\Xs,\Ms_{\Xs})$, we define its adic generic fiber $\Fs_{\eta}$ as the sheafification for the analytic topology of the functor on complete Tate Huber pair $(A,A^+)$ over $\Xs_{\eta}$
\[ (A,A^+) \mapsto \varinjlim_{A_0 \sub A^+} \Fs(\Spf(A_0), \Ms_{A_0}),\]
where $A_0$ ranges over all rings of definition of $A$ contained in $A^+$ and we equip $\Spf(A_0)$ with the standard log structure, arising from the pre-log structure $A^{\times} \cap A_0 \hookrightarrow A_0$. If $\Fs$ is representable by an formal log scheme $(\Ys,\Ms_{\Ys})$ whose structure map to $(\Xs,\Ms_{\Xs})$ is strict, then $\Fs_{\eta}$ is representable by $\Ys_{\eta}$.

\begin{definition}{(\cite[(4.1)]{kato2023logarithmicdieudonnetheory})}
A log $p$-divisible group over $(\Xs, \Ms_{\Xs})$ is an abelian sheaf $\mathfrak{G}$ on $(\Xs, \Ms_{\Xs})_{\kfl}$ satisfying the following conditions
    \begin{enumerate}
        \item The map $[p]\colon \mathfrak{G} \rightarrow \mathfrak{G}$ is surjective,
        \item The sheaves $\mathfrak{G}[p^n]$ and their Cartier dual
        \begin{align*}
            \mathfrak{G}[p^n]^D = \underline{\Hom}_{(\Xs,\Ms_{\Xs})_{\kfl}}(\mathfrak{G}[p^n],\G_m)
            \end{align*}
        are representable by fs log schemes whose structure map to $(\Xs,\Ms_{\Xs})$ is Kummer log flat and whose underlying map of formal schemes is finite, and
        \item We have $\mathfrak{G} =  \varinjlim_n \mathfrak{G}[p^n]$.
    \end{enumerate}
We denote by $\BT_d(\Xs,\Ms_{\Xs})$ the category of log $p$-divisible groups\footnote{In \cite{kato2023logarithmicdieudonnetheory}, Kato defines more general categories of log $p$-divisible groups. Here, the $d$ stands for "dual- representable".}.
\end{definition}

\begin{example}
    \begin{enumerate}
        \item Let $\mathfrak{G}$ be a $p$-divisible group over $\Xs$. Then endowing each $\mathfrak{G}[p^n] \rightarrow \Xs$ with the pullback log structure defines a log $p$-divisible group on $(\Xs, \Ms_{\Xs})$. This produces a fully faithful functor
        \begin{align}
            \BT(\Xs) \hookrightarrow \BT_d(\Xs, \Ms_{\Xs}).
        \end{align}
        A log $p$-divisible group is called classical if it lies in the essential image of this functor.
        \item An non-classical example is given by the Kummer extension \cite[(1.8.1)]{kato2023logarithmicdieudonnetheory}: Let $K$ be a $p$-adic field and endow $\Spf(\Os_K)$ with the standard log structure. Then for any choice of uniformizer $\pi \in \Os_K$, this is a log $p$-divisible group $\mathfrak{G}_{\pi}$, fitting in an extension
        % https://q.uiver.app/#q=WzAsNSxbMCwwLCIwIl0sWzEsMCwiXFxtdV97cF57XFxpbmZ0eX19Il0sWzIsMCwiXFxtYXRoZnJha3tHfV97XFxwaX0iXSxbMywwLCJcXFFfcC9cXFpfcCJdLFs0LDAsIjAuIl0sWzAsMV0sWzIsM10sWzEsMl0sWzMsNF1d
\[\begin{tikzcd}
	0 & {\mu_{p^{\infty}}} & {\mathfrak{G}_{\pi}} & {\Q_p/\Z_p} & {0.}
	\arrow[from=1-1, to=1-2]
	\arrow[from=1-2, to=1-3]
	\arrow[from=1-3, to=1-4]
	\arrow[from=1-4, to=1-5]
\end{tikzcd}\]
It arises as the degeneration of the $p$-divisible group of a Tate elliptic curve over $K$ with $q$-invariant equal to $\pi$.
    \end{enumerate}
\end{example}

Given a log $p$-divisible group $\mathfrak{G}$ over $(\Xs,\Ms_{\Xs})$ we can associate to it the Tate module
    \[ T_p\mathfrak{G} = \varprojlim_n \mathfrak{G}[p^n], \]
viewed as a sheaf on $(\Xs,\Ms_{\Xs})_{\kfl}$, 
and its Lie algebra $\Lie(\mathfrak{G})$, a classical vector bundle, see \cite[Prop. 7.3]{kato2023logarithmicdieudonnetheory}. We let $\omega_{\mathfrak{G}} = \Lie(G)^{\vee}$ denote the linear dual.

We may now extend the Hodge--Tate map of Fargues \cite[§II.1]{Fargues2008} to the logarithmic setting.
\begin{definition}\label{def: HT map of fargues}
Let $\mathfrak{G} \in \BT_d(\Xs,\Ms_{\Xs})$. The Hodge--Tate map of $\mathfrak{G}$
\begin{align}\label{eq: HT map of fargues}
    \alpha_{\mathfrak{G}} \colon T_p\mathfrak{G} \otimes_{\underline{\Z_p}} \Os_{\Xs} \rightarrow \omega_{\mathfrak{G}^D}
\end{align}
is obtained by scalar extension from the following map: Given an fs log scheme $T$ over $\Xs$ and any section $f \in T_p\mathfrak{G}(T) = \Hom(\mathfrak{G}_T^D,\G_{m,T})$, we set
\[ \alpha_{\mathfrak{G}}(f) = \Lie(f) \in \Hom_{\Os_T}(\Lie(\mathfrak{G}_T^D),T) = \omega_{\mathfrak{G}^D}(T).\]
This is easily seen to define a morphism of sheaves, natural in $\mathfrak{G}$.
\end{definition}

\begin{proposition}\label{prop: generic fibers of log p-divisible groups}
    Let $(\Ss, \Ms_{\Ss})$ be an fs log formal scheme over $\Z_p$ such that $\Ss$ locally admits a finitely generated ideal of definition and is flat over $\Z_p$. Assume furthermore that the log structure is trivial or the pullback of the standard log structure on $\Z_p$, and that its adic generic fiber $S=\Ss_{\eta}$ is a rigid space or a sousperfectoid space. Let $\mathfrak{G} \in \BT_d(\Ss,\Ms_{\Ss})$.
    \begin{enumerate}
        \item The sheaf $\Gs =\mathfrak{G}_{\eta}$  is representable by an analytic $p$-divisible group.
        \item Assume that $S$ is good. Then $\Gs$ is dualizable and we have a natural isomorphism
    \[ \Gs^D = (\mathfrak{G}^D)_{\eta}.\]
    Moreover, the map corresponding to $\Gs$ under Theorem \ref{thm: extending Fargues' equivalence of categories} is
    \[ f_{\Gs} = \alpha_{\mathfrak{G}^D}[\tfrac{1}{p}]^{\vee}, \]
    %\colon \Lie(\Gs)\otimes_{\Os_{S_{\et}}} \Os_{S_v} \rightarrow (T_p\Gs^D)^{\vee} \otimes_{\Z_p} \Os_{S_v} \cong T_p\Gs(-1) \otimes_{\Z_p} \Os_{S_v},
    where $\alpha_{\mathfrak{G}^D}$ is the Hodge--Tate map of the Cartier dual $\mathfrak{G}^D$, cf. Definition \ref{def: HT map of fargues}.
    \end{enumerate} 
\end{proposition}
\begin{proof}
\begin{enumerate}
    \item We may work locally and assume that $\Ss=\Spf(R)$ is affine, for a $p$-torsionfree  ring $R$ which is $I$-adically complete, for a finitely generated ideal $I$ containing a power of $p$. Let $A$ be any $I$-adically complete and topologized $R$-algebra $A$, and equip $\Spf(A)$ with the standard log structure. We claim there is a canonical isomorphism
    \[ \log\colon \Ker(\mathfrak{G}(A) \rightarrow \mathfrak{G}(A/p^2)) \cong p^2\Lie(\mathfrak{G})(A).\]
    We may assume that $A$ is discrete. In particular, $p^nA=0$ for some $n \geq 1$, so that $\Lie(\mathfrak{G})(A) = \Lie(\mathfrak{G}[p^n])(A)$. By \cite[Lemma 3.7]{inoue2025logprismaticdieudonnetheory}, up to replacing $A$ by a strict étale extension, $\mathfrak{G}[p^n]$ is an extension of a classical finite étale group $\mathfrak{G}[p^n]^{\et}$ by a classical finite flat connected group $\mathfrak{G}[p^n]^{\circ}$. It is enough to establish the above isomorphism with $\mathfrak{G}$ replaced with $\mathfrak{G}[p^n]^{\circ}$, which follows from \cite[Lemma 2.2.5]{Messing1972}.

    From this, we may argue as in the proof of \cite[Prop. 3.4.2(1)]{scholze2013moduli}. Namely, since any element of $\mathfrak{\Gs}(A)$ is sent into the kernel of reduction modulo $p^2$ by a sufficiently high power of the map $[p]$, the logarithm extends to a natural map
    \[ \log\colon \mathfrak{G}(A) \rightarrow \Lie(\mathfrak{G}(A)[\tfrac{1}{p}]\]
    and even to a left exact sequence of sheaves
    % https://q.uiver.app/#q=WzAsNCxbMCwwLCIwIl0sWzEsMCwiXFx2YXJpbmpsaW1fe24gXFxnZXEgMX0gXFxtYXRoZnJha3tHfVtwXm5dX3tcXGV0YX0iXSxbMiwwLCJcXEdzIl0sWzMsMCwiXFxMaWUoXFxtYXRoZnJha3tHfSkgXFxvdGltZXNfe1xcT3Nfe1xcU3N9fSBcXE9zX3tTfS4iXSxbMCwxXSxbMiwzLCJcXGxvZyJdLFsxLDJdXQ==
\[\begin{tikzcd}
	0 & {\varinjlim_{n \geq 1} \mathfrak{G}[p^n]_{\eta}} & \Gs & {\Lie(\mathfrak{G}) \otimes_{\Os_{\Ss}} \Os_{S}.}
	\arrow[from=1-1, to=1-2]
	\arrow[from=1-2, to=1-3]
	\arrow["\log", from=1-3, to=1-4]
\end{tikzcd}\]
By the above claim, there is a subsheaf $U \sub \Gs$ over which the logarithm restricts to an isomorphism to an open subgroup of $\Lie(\mathfrak{G}) \otimes_{\Os_{\Ss}} \Os_{S}$, so that $U$ is representable. Moreover, $\Gs$ is covered by the preimages $U_n$ of $U$ under multiplication by $[p^n]$. By \cite[Constr. 3.8]{inoue2025logprismaticdieudonnetheory}, the sheaves $\mathfrak{G}[p^n]_{\eta}$ are representable by finite étale adic groups over $S$, so that the torsors $[p^n]\colon U_n \rightarrow U$ are representable as well. It follows that $\mathfrak{G}_{\eta} = \bigcup_{n \geq 1}U_n$ is representable, of topologically $p$-torsion and with $[p]$ finite étale, as required. 
    
    \item Let $A$ be an $R$-algebra as above and assume furthermore that $A$ is $p$-torsionfree. We have a natural map
\[ T_p\mathfrak{G}^D(A) = \Hom_{\Spf(A)}(\mathfrak{G}_A,\mu_{p^{\infty},A}) \xrightarrow{(-)_{\eta}} \Hom_{Y}(\mathfrak{G}_{\eta,Y},\G_{m,Y}\langle p^{\infty} \rangle),\]
where $Y = \Spf(A)_{\eta}$. Upon passing to the generic fibers, this defines a morphism of $v$-sheaves on $\Perf_S$
\[ T_p(\mathfrak{G}^D)_{\eta} \rightarrow  \underline{\Hom}_S(\Gs,\G_{m}\langle p^{\infty} \rangle).\]
Moreover, post-composing the above morphism with the map 
\[ \underline{\Hom}_S(\Gs,\G_{m}\langle p^{\infty} \rangle) \xrightarrow{\Lie(-)} \underline{\Hom}_S(\Lie(\Gs),\Os_S) =  \omega_{\Gs}\]
yields the Hodge--Tate map 
\[ \alpha_{\mathfrak{G}^D}[\tfrac{1}{p}]\colon T_p(\mathfrak{G}^D)_{\eta} \rightarrow \omega_{\mathfrak{G}}[\tfrac{1}{p}]\]
from Definition \ref{def: HT map of fargues}. Hence, by exchanging the variables, we obtain a map
\[ u\colon \Gs \rightarrow T_p(\mathfrak{G}^D)_{\eta}^{\vee} \otimes_{\Z_p} \G_m\langle p^{\infty} \rangle = T_p\Gs(-1) \otimes_{\Z_p} \G_m\langle p^{\infty} \rangle,\]
fitting in the following diagram
% https://q.uiver.app/#q=WzAsNCxbMCwwLCJcXEdzIl0sWzAsMSwiVF9wXFxHcygtMSlcXG90aW1lc197XFxaX3B9XFxHX21cXGxhbmdsZSBwXntcXGluZnR5fSBcXHJhbmdsZSJdLFsxLDAsIlxcTGllKFxcR3MpIl0sWzEsMSwiVF9wXFxHcygtMSlcXG90aW1lc197XFxaX3B9XFxHX2EuIl0sWzAsMSwidSIsMl0sWzAsMiwiXFxsb2ciXSxbMiwzLCJcXGFscGhhX3tcXG1hdGhmcmFre0d9XkR9W1xcdGZyYWN7MX17cH1dXntcXHZlZX0iXSxbMSwzLCJcXGxvZyIsMl1d
\[\begin{tikzcd}
	\Gs & {\Lie(\Gs)} \\
	{T_p\Gs(-1)\otimes_{\Z_p}\G_m\langle p^{\infty} \rangle} & {T_p\Gs(-1)\otimes_{\Z_p}\G_a.}
	\arrow["\log", from=1-1, to=1-2]
	\arrow["u"', from=1-1, to=2-1]
	\arrow["{\alpha_{\mathfrak{G}^D}[\tfrac{1}{p}]^{\vee}}", from=1-2, to=2-2]
	\arrow["\log"', from=2-1, to=2-2]
\end{tikzcd}\]
We conclude, by Theorem \ref{thm: extending Fargues' equivalence of categories}, that $\alpha_{\mathfrak{G}^D}[\tfrac{1}{p}]^{\vee} = f_{\Gs}$ and in particular also that $u=u_{\Gs}$. To show that $\Gs$ is dualizable, with Cartier dual $\Gs^D = (\mathfrak{G}^D)_{\eta}$, it suffices to show that the sequence of $v$-vector bundles
% https://q.uiver.app/#q=WzAsNSxbMCwwLCIwIl0sWzEsMCwiXFxMaWUoXFxHcykiXSxbMiwwLCJUX3BcXEdzKC0xKVxcb3RpbWVzXFxPc197U192fSJdLFszLDAsIlxcb21lZ2Ffe1xcbWF0aGZyYWt7R31eRH1bXFx0ZnJhY3sxfXtwfV1cXG90aW1lc1xcT3Nfe1Nfdn0oLTEpIl0sWzQsMCwiMCJdLFswLDFdLFsxLDIsIlxcYWxwaGFfe1xcbWF0aGZyYWt7R31eRH1bXFxmcmFjezF9e3B9XV57XFx2ZWV9Il0sWzIsMywiXFxhbHBoYV97XFxtYXRoZnJha3tHfX1bXFxmcmFjezF9e3B9XSgtMSkiXSxbMyw0XV0=
\[\begin{tikzcd}
	0 & {\Lie(\Gs)} & {T_p\Gs(-1)\otimes\Os_{S_v}} & {\omega_{\mathfrak{G}^D}[\tfrac{1}{p}]\otimes\Os_{S_v}(-1)} & 0
	\arrow[from=1-1, to=1-2]
	\arrow["{\alpha_{\mathfrak{G}^D}[\frac{1}{p}]^{\vee}}", from=1-2, to=1-3]
	\arrow["{\alpha_{\mathfrak{G}}[\frac{1}{p}](-1)}", from=1-3, to=1-4]
	\arrow[from=1-4, to=1-5]
\end{tikzcd}\]
is exact. This can be tested on strictly totally disconnected perfectoid spaces $S'=\Spa(A,A^+)$ over $S$. There, we may further check exactness on points, by Nakayama's lemma, so we may assume that $S'=\Spa(C,C^+)$ and even that $C^+=\Os_C$. By definition of the adic generic fiber, since any open cover of $S'$ splits, the map $S' \rightarrow S$ arises from a map of log formal schemes $\Spf(\Os_C) \rightarrow \Spf(R)$. By \cite[Prop. 3.19]{inoue2025logprismaticdieudonnetheory}, this map factors through a map of fs formal log scheme $f\colon (\Ys,\Ms_{\Ys}) \rightarrow \Spf(R)$ such that $f^*\mathfrak{G}$ is classical. This allows us to reduce to the case $R=\Os_C$ and $\mathfrak{G}$ classical. In this case, the statement follows from \cite[Thm. II.1.1]{Fargues2008}.
\end{enumerate}
\end{proof}

We now focus on the case of smooth rigid spaces over $p$-adic fields that admit a semistable formal model. We will use the notion of semistable $\Z_p$-local systems of \cite[Def. 3.27]{du2024logprismaticfcrystalspurity}.
\begin{proposition}\label{prop: semistable reduction versus semistable local system}
    Let $X$ be a smooth rigid space over a $p$-adic field $K/\Q_p$, with a semistable formal model $\Xs$ over $\Os_K$. Equip $\Xs$ with the standard log structure $\Ms_{\Xs}$. Let $\Gs \rightarrow X$ be a dualizable analytic $p$-divisible group. 
    \begin{enumerate}
        \item The generic fiber functor is fully faithful on log $p$-divisible groups over $(\Xs,\Ms_{\Xs})$.
        \item If $\Gs = \mathfrak{G}_{\eta}$ for a log $p$-divisible group $\mathfrak{G}$ on $(\Xs,\Ms_{\Xs})$, then $T_p\Gs$ is a semistable local system. 
        \item Assume that $X=\Spa(K)$. Then if $T_p\Gs$ is a semistable representation of $\Gal_K$, $\Gs$ arises from a log $p$-divisible group over $\Os_K$.
    \end{enumerate}
\end{proposition}
\begin{remark}
    It would be interesting to find a characterization of analytic $p$-divisible groups $\Gs$ arising from log $p$-divisible groups directly in terms of the geometry of $\Gs$. See \cite[Prop. 2.1]{Farg19} for the case $S=\Spa(K)$.
\end{remark}
\begin{proof}
We will use logarithmic prismatic Dieudonné theory \cite{würthen2023logprismaticdieudonnetheory}\cite{inoue2025logprismaticdieudonnetheory}. Consider the log prismatic site $(\Xs, \Ms_{\Xs})_{\Prism}$, cf \cite{koshikawa2022logarithmicprismaticcohomologyi}\cite{koshikawa2023logarithmicprismaticcohomologyii}, which comes with its structure sheaf $\Os_{\Prism}$ and ideal sheaf $\Is_{\Prism}$. Following \cite[Def. 3.3]{du2024logprismaticfcrystalspurity}, a prismatic $F$-crystal $(M,\varphi_M)$ on $(\Xs,\Ms_{\Xs})$ consists of a vector bundle $M$ on the log prismatic site $(\Xs, \Ms_{\Xs})_{\Prism}$, together with an $\Os_{\Prism}$-linear isomorphism 
\[ \varphi_M\colon \varphi^*M[\tfrac{1}{\Is_{\Prism}}] \xrightarrow{\cong} M[\tfrac{1}{\Is_{\Prism}}].\]
It is called of height $\in [a,b]$ if 
\[ \Is_{\Prism}^bM \sub \varphi_M(\varphi^*M) \sub \Is_{\Prism}^aM.\]
Analytic prismatic $F$-crystals are defined analogously as pairs $(M,\varphi_M)$, where $M$ is now a sheaf sending a log prism $(A,I,\delta_{\log})$ to a vector bundle on the locus $\Spec(A)\backslash V(p,I)$ rather than on the whole $\Spec(A)$. In \cite{inoue2025logprismaticdieudonnetheory}, Inoue constructs a log prismatic Dieudonné functor $\mathfrak{G} \mapsto (M_{\mathfrak{G}}, \varphi_{\mathfrak{G}})$, which takes the form of an antiequivalence of categories \cite[Thm. 5.14]{inoue2025logprismaticdieudonnetheory}
\begin{align*}
     \BT_d(\Xs,\Ms_{\Xs})
         \xlongrightarrow{\cong} \left\{\text{\begin{tabular}{l} {\parbox{4.2cm}{ prismatic $F$-crystals on $(\Xs,\Ms_{\Xs})$ of height $\in [0,1]$}}\end{tabular}}\right\},
    \end{align*}
On the other hand, restriction along open subsets defines a fully faithful functor \cite[Cor. 3.9]{du2024logprismaticfcrystalspurity}
\begin{align*}
     \left\{\text{\begin{tabular}{l} {\parbox{3.7cm}{ prismatic $F$-crystals on $(\Xs,\Ms_{\Xs})$}}\end{tabular}}\right\}
         \lhook\joinrel\rightarrow \left\{\text{\begin{tabular}{l} {\parbox{3.7cm}{ analytic prismatic $F$-crystals on $(\Xs,\Ms_{\Xs})$}}\end{tabular}}\right\}.
    \end{align*}
Finally, there is a (contravariant) étale realization functor \cite[Thm. 1.5]{du2024logprismaticfcrystalspurity}
\begin{align*}
 T_{\et}\colon
 \left\{\text{\begin{tabular}{l} {\parbox{3.7cm}{analytic prismatic $F$-crystals on $(\Xs,\Ms_{\Xs})$}}\end{tabular}}\right\}
         \xlongrightarrow{\cong} \left\{\text{\begin{tabular}{l} {\parbox{3.2cm}{semistable $\Z_p$-local systems on $X$}}\end{tabular}}\right\}.
    \end{align*}
It is compatible with logarithmic prismatic Dieudonné theory, in the following sense \cite[Prop. 5.13]{inoue2025logprismaticdieudonnetheory}
\[ T_{\et}(M_{\mathfrak{G}},\varphi_{\mathfrak{G}}) = T_p\mathfrak{G}_{\eta}.\]
We now turn to the proof. The functor taking a log $p$-divisible group to its Tate module factors as the composition of the log prismatic Dieudonné functor, the restriction from prismatic $F$-crystals to analytic prismatic $F$-crystals and the étale realization functor $T_{\et}$. Hence the first two points follow immediately from Proposition \ref{prop: dualizable groups over smooth rig over padic field}. It remains to show the third point. By \cite[Thm. 4.1]{du2024logprismaticfcrystalspurity}, the semistable Galois representation $T_p\Gs$ gives rise to a semistable $\Z_p$-local system on $\Spa(K)$ and hence to an analytic prismatic $F$-crystal $(M,\varphi_M)$. By \cite[Cor. 3.9]{du2024logprismaticfcrystalspurity}, $(M,\varphi_M)$ extends to a prismatic $F$-crystals, which has height $\in [0,1]$ by \cite[Thm. 5.0.18]{du2023prismaticapproachvarphihat}. Therefore, it is the log Dieudonné module of a log $p$-divisible group $\mathfrak{G}$. Then $T_p\Gs = T_p\mathfrak{G}_{\eta}$, and it follows from Proposition \ref{prop: dualizable groups over smooth rig over padic field} that $\Gs$ and $\mathfrak{G}_{\eta}$ are isomorphic. Alternatively, the third point follows from \cite[Thm. B]{bertapelle2024logpdivisiblegroupssemistable}.
\end{proof}

For rigid spaces of good reduction, we have the following analogous characterization. We use the notion of crystalline $\Z_p$-local systems of \cite{GuoReinecke2024}.
\begin{proposition}\label{prop: good reduction versus crystalline local system}
    Let $X$ be a smooth rigid space over a $p$-adic field $K/\Q_p$, with a smooth formal model $\Xs$ over $\Os_K$. Let $\Gs \rightarrow X$ be a dualizable analytic $p$-divisible group. 
    \begin{enumerate}
        \item The generic fiber functor is fully faithful on $p$-divisible groups over $\Xs$.
        \item If $\Gs = \mathfrak{G}_{\eta}$ for a $p$-divisible group $\mathfrak{G}$ on $\Xs$, then $T_p\Gs$ is a crystalline local system. 
        \item Assume that $X=\Spa(K)$. Then if $T_p\Gs$ is a crystalline representation of $\Gal_K$, $\Gs$ arises from a $p$-divisible group over $\Os_K$.
    \end{enumerate}
    \end{proposition}
    \begin{remark}
    It is interesting to know also for higher dimensional $\Xs$ whether the condition that $T_p\Gs$ is crystalline implies that $\Gs$ has good reduction. For some negative and positive answers, see \cite[Example 3.36, Remark 3.38]{Du2024}.
\end{remark}
\begin{proof}
The proof is completely analogous to that of Proposition \ref{prop: semistable reduction versus semistable local system}. In place of log prismatic Dieudonné theory, we use the prismatic Dieudonné theory of Anschütz--Le bras \cite{anschutz2022prismaticdieudonnetheory}, which is an antiequivalence 
\cite[Thm. 4.74]{anschutz2022prismaticdieudonnetheory}
\begin{align*}
     \BT(\Xs)
         \xlongrightarrow{\cong} \left\{\text{\begin{tabular}{l} {\parbox{4.9cm}{admissible prismatic $F$-crystals on $\Xs$ of height $\in [0,1]$}}\end{tabular}}\right\}.
    \end{align*}
Here, admissibility is a technical condition automatically satisfied as soon as $\Xs$ is regular and admits a quasi-syntomic cover by perfectoids, by adapting the proof of \cite[Prop. 5.10]{anschutz2022prismaticdieudonnetheory}. Restriction along open subsets defines a fully faithful functor from the category of prismatic $F$-crystals into the category of analytic prismatic $F$-crystals \cite[Prop. 3.7]{GuoReinecke2024}. Finally, there is a contravariant étale realization functor \cite[Thm. A]{GuoReinecke2024} \cite[Thm. 3.29]{Du2024}
\begin{align*}
 \left\{\text{\begin{tabular}{l} {\parbox{4.4cm}{analytic prismatic $F$-crystals on $\Xs$ of height $\in [0,1]$}}\end{tabular}}\right\}
         \xlongrightarrow{\cong} \left\{\text{\begin{tabular}{l} {\parbox{5.0cm}{crystalline $\Z_p$-local systems on $X$ of Hodge--Tate weights $\in \{0,1\}$}}\end{tabular}}\right\},
    \end{align*}
compatible with the $p$-adic Tate module and Anschütz--Le Bras' functor \cite[Prop. 3.35]{Du2024}. The proof now follows in a completely analogous way as that of Proposition \ref{prop: semistable reduction versus semistable local system}, using that over $\Spf(\Os_K)$, analytic prismatic $F$-crystals uniquely extend to prismatic $F$-crystals \cite[Prop. 3.8]{GuoReinecke2024} (or \cite[Thm. 1.2]{BhattScholze2023}) for the third point.
\end{proof}

\subsection{Higher direct images}\label{subsect: Higher direct images}
In this subsection, we study the smooth proper higher direct images of analytic $p$-divisible groups, and we prove a comparison theorem. We then derive applications to the topologically $p$-torsion Picard varieties of rigid spaces.

Let $(K,K^+)$ be a non-archimedean field over $\Q_p$ and let $\pi\colon X \rightarrow S$ be a morphism of rigid spaces over $K$. It induces morphisms of sites
\begin{align}\label{eq: piEt} \pi_{\Et}\colon \LSD_{X,\et} \rightarrow \Perf_{S,\et}\, , \quad \pi_{v}\colon \LSD_{X,v} \rightarrow \Perf_{S,v}.\end{align}

\begin{definition}{(\cite[Def. 3.14]{gerth2024})}
    Let $\Gs \rightarrow X$ be a smooth commutative group. For $\tau \in \{\Et, v\}$ and $n \geq 0$, the diamantine higher direct images of $\Gs$ are defined to be the sheaves
    \begin{align} R^n\pi_{\tau,*}\Gs \colon \Perf_{S,\tau} \rightarrow \Ab.\end{align}
\end{definition}

We will use the following results from \cite{gerth2024}.
\begin{thm}{(\cite[Prop. 3.16, Prop. 3.39]{gerth2024})}
    Let $\pi\colon X \rightarrow S$ be a proper smooth morphism of seminormal rigid spaces over $K$ and let $\Gs$ be an analytic $p$-divisible $X$-group. Assume either that $S=\Spa(K,K^+)$ or that the Lie algebra $\ga$ is the pullback of a vector bundle on $S$. 
    \begin{enumerate}
        \item For $\tau \in \{\Et,v\}$, the sheaf $R^n\pi_{\tau,*}\ga$ is a $\tau$-vector bundle on $S$. We have an isomorphism of étale sheaves on $\Perf_S$
        \[ R^n\pi_{\Et,*}\ga =  (R^n\pi_{\et,*}\ga) \otimes_{\Os_{S_{\et}}} \G_a.\]
        If $S=\Spa(K,K^+)$ and is perfectoid, we further have
        \[ R^n\pi_{v,*}\ga = H_{v}^n(X,\ga)\otimes_K \G_a.\]
        \item For each $1 \leq m \leq \infty$, we have
        \[ R^n\pi_{v,*}\Gs[p^{m}] = R^n\pi_{\Et,*}\Gs[p^{m}]  = \nu^* R^n\pi_{\et,*}\Gs[p^m], \]
        where $\nu\colon \Perf_{S,v} \rightarrow S_{\et}$ is the natural morphism of topoi, and it is representable by an étale group over $S$.
        \item For $\tau \in \{\Et,v\}$, we have a short exact sequence of sheaves on $\Perf_{S,\tau}$
 % https://q.uiver.app/#q=WzAsNSxbMiwwLCJSXm5cXHBpX3tcXHRhdSwqfVxcR3MiXSxbMywwLCJSXm5cXHBpX3tcXHRhdSwqfVxcZ2EiXSxbMCwwLCIwIl0sWzQsMCwiMC4iXSxbMSwwLCJSXm5cXHBpX3tcXHRhdSwqfVxcR3NbcF57XFxpbmZ0eX1dIl0sWzAsMSwiXFxsb2dfKiJdLFsxLDNdLFs0LDBdLFsyLDRdXQ==
\begin{equation}\label{geometric cohomological log exact sequence}\begin{tikzcd}
	0 & {R^n\pi_{\tau,*}\Gs[p^{\infty}]} & {R^n\pi_{\tau,*}\Gs} & {R^n\pi_{\tau,*}\ga} & {0.}
	\arrow[from=1-1, to=1-2]
	\arrow[from=1-2, to=1-3]
	\arrow["{\log_*}", from=1-3, to=1-4]
	\arrow[from=1-4, to=1-5]
\end{tikzcd}\end{equation}
        \item Consider the morphism of sites
        \begin{align}\label{eq: piEtrig} \pi_{\Et}^{\rig}\colon \Sm_{/X,\et} \rightarrow \Sm_{/S,\et}.\end{align}
        Then the sheaf $R^n\pi_{\Et,*}^{\rig}\Gs$ is representable by a smooth $S$-group and there is a canonical isomorphism
    \begin{align}\label{eq: comparison iso between rigid and diamantine variant}
    (R^n\pi_{\Et,*}^{\rig}\Gs)^{\diamondsuit} = R^n\pi_{\Et,*}\Gs.\end{align}
    \end{enumerate}
\end{thm}

We now have the following comparison theorem for analytic $p$-divisible groups.
\begin{thm}\label{Thm: comparison theorem for analytic p-divisible groups}
    Let $\pi\colon X \rightarrow S$ be a proper smooth morphism of seminormal rigid spaces over $(K,K^+)$. Let $\Gs \rightarrow X$ be an analytic $p$-divisible group. Assume that $S=\Spa(K,K^+)$ or that $\ga$ is the pullback of an étale vector bundle on $S$.
    \begin{enumerate}
        \item Define the lisse $\Z_p$-sheaf on $\Perf_{S,v}$
        \[ \Mb^n = \underline{\Hom}(R^n\pi_{v,*}\Gs[p^{\infty}],\Q_p/\Z_p).\]
        Then there exists a commutative diagram of sheaves on $\Perf_{S,\et}$ with exact rows
        % https://q.uiver.app/#q=WzAsMTAsWzIsMCwiUl57bn1cXHBpX3tcXEV0LCp9XFxHcyJdLFszLDAsIlJee259XFxwaV97XFxFdCwqfVxcZ2EiXSxbMywxLCJcXHVuZGVybGluZXtcXEhvbX0oXFxNYl5uKDEpLFxcR19hKSJdLFsyLDEsIlxcdW5kZXJsaW5le1xcSG9tfShcXE1iXm4oMSksXFxHX21cXGxhbmdsZSBwXntcXGluZnR5fSBcXHJhbmdsZSkiXSxbNCwwLCIwIl0sWzEsMCwiUl57bn1cXHBpX3tcXEV0LCp9XFxHc1twXntcXGluZnR5fV0iXSxbMSwxLCJcXHVuZGVybGluZXtcXEhvbX0oXFxNYl5uKDEpLFxcbXVfe3Bee1xcaW5mdHl9fSkiXSxbMCwxLCIwIl0sWzAsMCwiMCJdLFs0LDEsIjAuIl0sWzAsMV0sWzEsMl0sWzAsM10sWzMsMl0sWzEsNF0sWzUsMF0sWzYsM10sWzcsNl0sWzgsNV0sWzIsOV0sWzUsNiwiXFxjb25nIiwyXV0=
\begin{equation}\label{eq: main diagram comparison theorem}\begin{tikzcd}
	0 & {R^{n}\pi_{\Et,*}\Gs[p^{\infty}]} & {R^{n}\pi_{\Et,*}\Gs} & {R^{n}\pi_{\Et,*}\ga} & 0 \\
	0 & {\underline{\Hom}(\Mb^n(1),\mu_{p^{\infty}})} & {\underline{\Hom}(\Mb^n(1),\G_m\langle p^{\infty} \rangle)} & {\underline{\Hom}(\Mb^n(1),\G_a)} & {0.}
	\arrow[from=1-1, to=1-2]
	\arrow[from=1-2, to=1-3]
	\arrow["\cong"', from=1-2, to=2-2]
	\arrow[from=1-3, to=1-4]
	\arrow[from=1-3, to=2-3]
	\arrow[from=1-4, to=1-5]
	\arrow[from=1-4, to=2-4]
	\arrow[from=2-1, to=2-2]
	\arrow[from=2-2, to=2-3]
	\arrow[from=2-3, to=2-4]
	\arrow[from=2-4, to=2-5]
\end{tikzcd}\end{equation}
\item The smooth $S$-group $R^{n}\pi_{\Et,*}\Gs$ contains a maximal open analytic $p$-divisible subgroup $\Hs \sub R^{n}\pi_{\Et,*}\Gs$. The tuple $(E,\Lb,f)$ associated with $\Hs$ under Theorem \ref{thm: extending Fargues' equivalence of categories} is
\[ E=R^n\pi_{\et,*}\ga, \quad \Lb = R^n\pi_{v,*}T_p\Gs/\torsion,\]
and
% https://q.uiver.app/#q=WzAsNCxbMCwwLCJmXFxjb2xvbiBSXm5cXHBpX3tcXEV0LCp9XFxnYSJdLFsxLDAsIlJeblxccGlfe3YsKn1cXGdhIl0sWzIsMCwiUl5uXFxwaV97diwqfShUX3BcXEdzKC0xKSBcXG90aW1lcyBcXEdfYSkiXSxbMywwLCIoUl5uXFxwaV97diwqfVRfcFxcR3MpKC0xKSBcXG90aW1lcyBcXEdfYSwiXSxbMCwxXSxbMSwyXSxbMiwzLCJcXGNvbmciXV0=
\[\begin{tikzcd}
	{f\colon R^n\pi_{\Et,*}\ga} & {R^n\pi_{v,*}\ga} & {R^n\pi_{v,*}(T_p\Gs(-1) \otimes \G_a)} & {(R^n\pi_{v,*}T_p\Gs)(-1) \otimes \G_a,}
	\arrow[from=1-1, to=1-2]
	\arrow[from=1-2, to=1-3]
	\arrow["\cong", from=1-3, to=1-4]
\end{tikzcd}\]
where the middle map is induced by $f_{\Gs}$ and the last isomorphism comes from Scholze's Primitive Comparison Theorem, cf Theorem \ref{thm: Primitive Comparison Theorem}(2).
\item Assume that $R^{n+1}\pi_{v,*}T_p\Gs$ is torsionfree. Then $\Hs = R^n\pi_{\Et,*}\Gs$, i.e. $R^n\pi_{\Et,*}\Gs$ is analytic $p$-divisible.
    \end{enumerate}
\end{thm}

For the proof, we start with the following multiplicative variant of Scholze's Primitive Comparison Theorem.
\begin{proposition}\label{prop: mult variant of prim comp}
    Let $\pi \colon X \rightarrow S$ be a proper smooth morphism of seminormal rigid spaces over $(K,K^+)$. Let $\Lb$ be a $\Z_p$-local system on $X_v$ and let $\V = \Lb[\tfrac{1}{p}]$.
    \begin{enumerate}
        \item Let us define the following sheaf on $\Perf_S$
    \[ \Mb^n = \underline{\Hom}(R^n\pi_{\Et,*}(\V/\Lb), \Q_p/\Z_p).\]
    Then $\Mb^n$ is a lisse $\Z_p$-sheaf with
    \[ \Mb^n/\tors = (R^n\pi_{v,*}\Lb)^{\vee}.\]
        \item There is a natural isomorphism of sheaves on $S_v$
    \[ R^n\pi_{v,*}(\Lb\otimes_{\Z_p} \G_m\langle p^{\infty} \rangle)= \underline{\Hom}(\Mb^n,\G_m\langle p^{\infty} \rangle), \]
    which lives over Scholze's Primitive Comparison Theorem isomorphism (\ref{eq: prim comp thm}) via the logarithm. 
    \item If $R^n\pi_{v,*}\Lb$ and $R^{n+1}\pi_{v,*}\Lb$ are torsion-free, the isomorphism of the previous point becomes
    \[ R^n\pi_{v,*}(\Lb\otimes_{\Z_p} \G_m\langle p^{\infty} \rangle)= (R^n\pi_{v,*}\Lb)\otimes_{\Z_p} \G_m\langle p^{\infty} \rangle. \]
    %  \item It always holds that
    % \[ R^n\pi_{v,*}(\V\otimes_{\Q_p} \B^{\varphi =p})= (R^n\pi_{v,*}\V)\otimes_{\Q_p}  \B^{\varphi =p}. \]
    \end{enumerate}
\end{proposition}
\begin{proof}
\begin{enumerate}
    \item We may work $v$-locally and we fix a map $T \rightarrow S$ from a strictly totally disconnected perfectoid space $T$. Then the restriction
    \[ R^n\pi_{\Et,*}(\V/\Lb)\restr{\Perf_T}\]
    is entirely determined by the group $H_{\et}^n(X_T,\V/\Lb)$. This fits in a short exact sequence
    % https://q.uiver.app/#q=WzAsNSxbMSwwLCJIX3ZebihYX1QsXFxWKS9IX3ZebihYX1QsXFxMYikiXSxbMiwwLCJIX3tcXGV0fV5uKFhfVCxcXFYvXFxMYikiXSxbMywwLCJIX3Zee24rMX0oWF9ULFxcTGIpX3tcXHRvcn0iXSxbMCwwLCIwIl0sWzQsMCwiMC4iXSxbMywwXSxbMCwxXSxbMiw0XSxbMSwyXV0=
\[\begin{tikzcd}
	0 & {H_v^n(X_T,\V)/H_v^n(X_T,\Lb)} & {H_{\et}^n(X_T,\V/\Lb)} & {H_v^{n+1}(X_T,\Lb)_{\tor}} & {0.}
	\arrow[from=1-1, to=1-2]
	\arrow[from=1-2, to=1-3]
	\arrow[from=1-3, to=1-4]
	\arrow[from=1-4, to=1-5]
\end{tikzcd}\]
It follows that 
\[ H_{\et}(X_T,\V/\Lb) \cong (\Q_p/\Z_p)^{\oplus r} \oplus A,\]
for a finite group $A$ of $p$-power torsion, so that
\[ \Mb^n = \Hom(H_{\et}(X_T,\V/\Lb), \Q_p/\Z_p) \cong \Z_p^{\oplus r} \oplus A^{\vee},\]
where $A^{\vee} \cong A$ denotes the Pontryagin dual of $A$. Hence $\Mb^n$ is lisse. Moreover
\[ \Mb^n/\tors = \Hom(H_{v}^n(X_T,\Lb)\otimes_{\Z_p}\Q_p/\Z_p, \Q_p/\Z_p) = \Hom( H_{v}^n(X_T,\Lb),\End(\Q_p/\Z_p)) = H_{v}^n(X_T,\Lb)^{\vee}.\]
\item We consider the following long exact sequence on $\Perf_{T,v}$
% https://q.uiver.app/#q=WzAsNSxbMiwwLCJSXm5cXHBpX3t2LCp9KFxcTGIgXFxvdGltZXMgXFxHX21cXGxhbmdsZSBwXntcXGluZnR5fSBcXHJhbmdsZSkiXSxbMywwLCJSXm5cXHBpX3t2LCp9KFxcTGIgXFxvdGltZXMgXFxHX2EpIl0sWzAsMCwiXFxsZG90cyJdLFsxLDAsIlJeblxccGlfe3YsKn0oXFxMYiBcXG90aW1lcyBcXG11X3twXntcXGluZnR5fX0pIl0sWzQsMCwiXFxsZG90cyJdLFswLDFdLFszLDBdLFsyLDMsIlxcZGVsdGEiXSxbMSw0LCJcXGRlbHRhIl1d
\[\begin{tikzcd}
	\ldots & {R^n\pi_{v,*}(\Lb \otimes \mu_{p^{\infty}})} & {R^n\pi_{v,*}(\Lb \otimes \G_m\langle p^{\infty} \rangle)} & {R^n\pi_{v,*}(\Lb \otimes \G_a)} & \ldots
	\arrow["\delta", from=1-1, to=1-2]
	\arrow[from=1-2, to=1-3]
	\arrow[from=1-3, to=1-4]
	\arrow["\delta", from=1-4, to=1-5]
\end{tikzcd}\]
By Scholze's Primitive Comparison Theorem, Theorem \ref{thm: Primitive Comparison Theorem}(2), we have an isomorphism
\[ \varphi\colon R^n\pi_{v,*}(\Lb \otimes \G_a) \cong (R^n\pi_{v,*}\Lb) \otimes \G_a,\]
which shows that this sheaf is representable by a vector bundle on $T$. There are no non-zero maps from a vector bundle over $T$ to an étale $T$-group, so that $\delta =0$. From the resulting short exact sequence and \cite[Lemma 3.38]{gerth2024}, it follows that $R^n\pi_{v,*}(\Lb \otimes \G_m\langle p^{\infty} \rangle)$ is representable by a smooth $T$-group. We may assume that $\Q_p^{\cyc}\sub \Os(T)$ and we fix a system of $p$th power roots of unity to trivialize the Tate twists. As
\[ R^n\pi_{v,*}(\Lb \otimes \mu_{p^{\infty}}) = H_{\et}^n(X_T, \V/\Lb) \cong (\Q_p/\Z_p)^r\oplus A\]
and there are no non-trivial extension of $\G_a$ by the finite group $A$, it follows that
\[ R^n\pi_{v,*}(\Lb \otimes \G_m\langle p^{\infty} \rangle) = \Hs \oplus A, \]
where $\Hs$ is an analytic $p$-divisible group with
\[ \Mb^n = T_p\Hs^{\vee} \oplus A^{\vee}.\]
Moreover, the associated map $f_{\Hs}$ is the isomorphism $\varphi$ given by the Primitive Comparison Theorem. It follows that also $u_{\Hs}$ defines an isomorphism 
\[ u_{\Hs}\colon \Hs \xrightarrow{\cong} T_p\Hs \otimes \G_m\langle p^{\infty} \rangle.\]
Altogether, we obtain the isomorphism
\begin{align*} \phi\colon R^n\pi_{v,*}(\Lb \otimes \G_m\langle p^{\infty} \rangle) &= \Hs \oplus A \\
&\cong \underline{\Hom}(T_p\Hs^{\vee},\G_m\langle p^{\infty} \rangle) \oplus \underline{\Hom}(A^{\vee},\G_m\langle p^{\infty} \rangle) \\
&= \underline{\Hom}(\Mb^n,\G_m\langle p^{\infty} \rangle).\end{align*}
Since an extension of a vector bundle by an étale group has no automorphism, the constructed isomorphism $\phi$ is uniquely determined by the commutativity of the following diagram of sheaves on $\Perf_T$
% https://q.uiver.app/#q=WzAsMTAsWzIsMCwiUl5uXFxwaV97diwqfShcXExiIFxcb3RpbWVzIFxcR19tXFxsYW5nbGUgcF57XFxpbmZ0eX0gXFxyYW5nbGUpIl0sWzMsMCwiUl5uXFxwaV97diwqfShcXExiIFxcb3RpbWVzIFxcR19hKSJdLFsxLDAsIlJeblxccGlfe3YsKn0oXFxMYiBcXG90aW1lcyBcXG11X3twXntcXGluZnR5fX0pIl0sWzAsMCwiMCJdLFs0LDAsIjAiXSxbMywxLCJcXHVuZGVybGluZXtcXEhvbX0oXFxNYl5uLFxcR19hKSJdLFsyLDEsIlxcdW5kZXJsaW5le1xcSG9tfShcXE1iXm4sXFxHX21cXGxhbmdsZSBwXntcXGluZnR5fSBcXHJhbmdsZSkiXSxbMSwxLCJcXHVuZGVybGluZXtcXEhvbX0oXFxNYl5uLFxcbXVfe3Bee1xcaW5mdHl9fSkiXSxbNCwxLCIwLiJdLFswLDEsIjAiXSxbMCwxXSxbMiwwXSxbMywyXSxbMSw0XSxbMSw1LCJcXGNvbmciXSxbMCw2LCJcXGNvbmciXSxbMiw3LCJcXGNvbmciXSxbOSw3XSxbNyw2XSxbNiw1XSxbNSw4XSxbMSw1LCJcXHZhcnBoaSIsMl0sWzAsNiwiXFxwaGkiLDJdXQ==
\[\begin{tikzcd}
	0 & {R^n\pi_{v,*}(\Lb \otimes \mu_{p^{\infty}})} & {R^n\pi_{v,*}(\Lb \otimes \G_m\langle p^{\infty} \rangle)} & {R^n\pi_{v,*}(\Lb \otimes \G_a)} & 0 \\
	0 & {\underline{\Hom}(\Mb^n,\mu_{p^{\infty}})} & {\underline{\Hom}(\Mb^n,\G_m\langle p^{\infty} \rangle)} & {\underline{\Hom}(\Mb^n,\G_a)} & {0.}
	\arrow[from=1-1, to=1-2]
	\arrow[from=1-2, to=1-3]
	\arrow["\cong", from=1-2, to=2-2]
	\arrow[from=1-3, to=1-4]
	\arrow["\cong", from=1-3, to=2-3]
	\arrow["\phi"', from=1-3, to=2-3]
	\arrow[from=1-4, to=1-5]
	\arrow["\cong", from=1-4, to=2-4]
	\arrow["\varphi"', from=1-4, to=2-4]
	\arrow[from=2-1, to=2-2]
	\arrow[from=2-2, to=2-3]
	\arrow[from=2-3, to=2-4]
	\arrow[from=2-4, to=2-5]
\end{tikzcd}\]
It follows that the isomorphism $\phi$ descends to an isomorphism of sheaves on all of $\Perf_S$, as required.
\item This easily follows from the first and second points.
\end{enumerate}
\end{proof}

\begin{proof}[Proof of Theorem \ref{Thm: comparison theorem for analytic p-divisible groups}]
Let us define a map
    % https://q.uiver.app/#q=WzAsNCxbMCwwLCJ1XFxjb2xvbiBSXm5cXHBpX3tcXEV0LCp9XFxHcyJdLFsxLDAsIlJeblxccGlfe3YsKn1cXEdzIl0sWzIsMCwiUl5uXFxwaV97diwqfShUX3BcXEdzKC0xKSBcXG90aW1lcyBcXEdfbVxcbGFuZ2xlIHBee1xcaW5mdHl9IFxccmFuZ2xlKSJdLFszLDAsIlxcdW5kZXJsaW5le1xcSG9tfShcXE1iXm4oMSksXFxHX21cXGxhbmdsZSBwXntcXGluZnR5fSBcXHJhbmdsZSksIl0sWzAsMV0sWzEsMl0sWzIsMywiXFxjb25nIl1d
\[\begin{tikzcd}
	{u\colon R^n\pi_{\Et,*}\Gs} & {R^n\pi_{v,*}\Gs} & {R^n\pi_{v,*}(T_p\Gs(-1) \otimes \G_m\langle p^{\infty} \rangle)} & {\underline{\Hom}(\Mb^n(1),\G_m\langle p^{\infty} \rangle),}
	\arrow[from=1-1, to=1-2]
	\arrow[from=1-2, to=1-3]
	\arrow["\cong", from=1-3, to=1-4]
\end{tikzcd}\]
where the middle map is the map induced by $u_{\Gs}$ on cohomology and the last map is the isomorphism of Proposition \ref{prop: mult variant of prim comp}. Then it restricts to an isomorphism
% https://q.uiver.app/#q=WzAsNCxbMCwwLCJSXm5cXHBpX3tcXEV0LCp9XFxHc1twXntcXGluZnR5fV0iXSxbMSwwLCJSXm5cXHBpX3t2LCp9XFxHc1twXntcXGluZnR5fV0iXSxbMiwwLCJSXm5cXHBpX3t2LCp9KFRfcFxcR3MoLTEpIFxcb3RpbWVzIFxcbXVfe3Bee1xcaW5mdHl9fSkiXSxbMywwLCJcXHVuZGVybGluZXtcXEhvbX0oXFxNYl5uKDEpLFxcbXVfe3Bee1xcaW5mdHl9fSksIl0sWzAsMSwiXFxjb25nIl0sWzEsMiwiXFxjb25nIl0sWzIsMywiXFxjb25nIl1d
\[\begin{tikzcd}
	{R^n\pi_{\Et,*}\Gs[p^{\infty}]} & {R^n\pi_{v,*}\Gs[p^{\infty}]} & {R^n\pi_{v,*}(T_p\Gs(-1) \otimes \mu_{p^{\infty}})} & {\underline{\Hom}(\Mb^n(1),\mu_{p^{\infty}}),}
	\arrow["\cong", from=1-1, to=1-2]
	\arrow["\cong", from=1-2, to=1-3]
	\arrow["\cong", from=1-3, to=1-4]
\end{tikzcd}\]
Hence $u$ fits in the following commutative diagram of sheaves on $\Perf_{S,\et}$ with exact rows
% https://q.uiver.app/#q=WzAsMTAsWzIsMCwiUl57bn1cXHBpX3tcXEV0LCp9XFxHcyJdLFszLDAsIlJee259XFxwaV97XFxFdCwqfVxcZ2EiXSxbMywxLCJcXHVuZGVybGluZXtcXEhvbX0oXFxNYl5uKDEpLFxcR19hKSJdLFsyLDEsIlxcdW5kZXJsaW5le1xcSG9tfShcXE1iXm4oMSksXFxHX21cXGxhbmdsZSBwXntcXGluZnR5fSBcXHJhbmdsZSkiXSxbMSwwLCJSXntufVxccGlfe1xcRXQsKn1cXEdzW3Bee1xcaW5mdHl9XSJdLFsxLDEsIlxcdW5kZXJsaW5le1xcSG9tfShcXE1iXm4oMSksXFxtdV97cF57XFxpbmZ0eX19KSJdLFs0LDAsIjAiXSxbNCwxLCIwLiJdLFswLDAsIjAiXSxbMCwxLCIwIl0sWzAsMV0sWzEsMiwiZiJdLFswLDMsInUiXSxbMywyXSxbMiw3XSxbMSw2XSxbNCwwXSxbNSwzXSxbOCw0XSxbOSw1XSxbNCw1LCJcXGNvbmciLDJdXQ==
\[\begin{tikzcd}
	0 & {R^{n}\pi_{\Et,*}\Gs[p^{\infty}]} & {R^{n}\pi_{\Et,*}\Gs} & {R^{n}\pi_{\Et,*}\ga} & 0 \\
	0 & {\underline{\Hom}(\Mb^n(1),\mu_{p^{\infty}})} & {\underline{\Hom}(\Mb^n(1),\G_m\langle p^{\infty} \rangle)} & {\underline{\Hom}(\Mb^n(1),\G_a)} & {0.}
	\arrow[from=1-1, to=1-2]
	\arrow[from=1-2, to=1-3]
	\arrow["\cong"', from=1-2, to=2-2]
	\arrow[from=1-3, to=1-4]
	\arrow["u", from=1-3, to=2-3]
	\arrow[from=1-4, to=1-5]
	\arrow["f", from=1-4, to=2-4]
	\arrow[from=2-1, to=2-2]
	\arrow[from=2-2, to=2-3]
	\arrow[from=2-3, to=2-4]
	\arrow[from=2-4, to=2-5]
\end{tikzcd}\]
where $f$ is the map in the statement of the theorem. It follows that the right square is cartesian. As $R^n\pi_{\Et,*}\Gs$ is a pullback square of $v$-sheaves, it thus naturally is a $v$-sheaf as well. It then follows from \cite[Lemma 3.38]{gerth2024} that it is representable by an admissible locally $p$-divisible group. There is a maximal analytic $p$-divisible subgroup $\Hs$ corresponding to the pullback along $u$ of the open subspace
% https://q.uiver.app/#q=WzAsMyxbMiwwLCJcXHVuZGVybGluZXtcXEhvbX0oXFxNYl5uKDEpLFxcR19tXFxsYW5nbGUgcF57XFxpbmZ0eX0gXFxyYW5nbGUpLiJdLFsxLDAsIlxcdW5kZXJsaW5le1xcSG9tfShcXE1iXm4oMSkvXFx0b3IsXFxHX21cXGxhbmdsZSBwXntcXGluZnR5fSBcXHJhbmdsZSkiXSxbMCwwLCJSXm5cXHBpX3t2LCp9VF9wXFxHcy9cXHRvcigtMSkgXFxvdGltZXNcXEdfbVxcbGFuZ2xlIHBee1xcaW5mdHl9IFxccmFuZ2xlIl0sWzEsMCwiIiwyLHsic3R5bGUiOnsidGFpbCI6eyJuYW1lIjoiaG9vayIsInNpZGUiOiJ0b3AifX19XSxbMiwxLCJcXGNvbmciXV0=
\[\begin{tikzcd}
	{R^n\pi_{v,*}T_p\Gs/\tor(-1) \otimes\G_m\langle p^{\infty} \rangle} & {\underline{\Hom}(\Mb^n(1)/\tor,\G_m\langle p^{\infty} \rangle)} & {\underline{\Hom}(\Mb^n(1),\G_m\langle p^{\infty} \rangle).}
	\arrow["\cong", from=1-1, to=1-2]
	\arrow[hook, from=1-2, to=1-3]
\end{tikzcd}\]
The rest of the statement is clear.
\end{proof}

As a particular case, we may study the topologically $p$-torsion Picard variety of \cite{heuer2023diamantine} in a relative setting.

\begin{cor}\label{cor: ptop picard is dualizable}
Let $\pi \colon X \rightarrow S$ be a proper smooth map of seminormal rigid spaces over $(K,K^+)$. Consider the Picard sheaf
 \[ \PPic_{X/S,\et} = R^1\pi_{\Et,*}\G_m\colon \Perf_S \rightarrow \Ab.\]
 \begin{enumerate}
     \item The topologically $p$-torsion subsheaf $\PPic_{X/S,\et}\langle p^{\infty} \rangle$ (cf (\ref{eq: ptop torsion subsheaf})) of the Picard sheaf is representable by a smooth $S$-group and is naturally isomorphic to $R^1\pi_{\Et,*}\G_m\langle p^{\infty} \rangle$.
\item Let $\Mb = \Hom(R^1\pi_{\Et,*}\mu_{p^{\infty}},\Q_p/\Z_p)$, then there is a natural short exact sequence
% https://q.uiver.app/#q=WzAsNSxbMCwwLCIwIl0sWzEsMCwiXFxQUGljX3tYL1N9XFxsYW5nbGUgcF57XFxpbmZ0eX0gXFxyYW5nbGUiXSxbMiwwLCJcXHVuZGVybGluZXtcXEhvbX0oXFxNYigxKSxcXEdfbVxcbGFuZ2xlIHBee1xcaW5mdHl9IFxccmFuZ2xlKSJdLFszLDAsIlxccGlfeyp9XFxPbWVnYV97WC9TfV4xXFxvdGltZXNfe1xcT3Nfe1Nfe1xcZXR9fX0gXFxHX2EoLTEpIl0sWzQsMCwiMC4iXSxbMCwxXSxbMiwzXSxbMyw0XSxbMSwyXV0=
\begin{equation}\label{eq: mult HT sequence}\begin{tikzcd}
	0 & {\PPic_{X/S}\langle p^{\infty} \rangle} & {\underline{\Hom}(\Mb(1),\G_m\langle p^{\infty} \rangle)} & {\pi_{*}\Omega_{X/S}^1\otimes_{\Os_{S_{\et}}} \G_a(-1)} & {0.}
	\arrow[from=1-1, to=1-2]
	\arrow[from=1-2, to=1-3]
	\arrow[from=1-3, to=1-4]
	\arrow[from=1-4, to=1-5]
\end{tikzcd}\end{equation}
It lies over the Hodge--Tate sequence (\ref{eq: contravariant HT seq}) via the logarithm.

\item The maximal analytic $p$-divisible subgroup $\Hs \sub \PPic_{X/S,\et}\langle p^{\infty} \rangle$ of Theorem \ref{Thm: comparison theorem for analytic p-divisible groups}(2) is dualizable. 
 \end{enumerate} 
\end{cor}
\begin{remark}\label{remark: link with pairing of DW}
There is a canonical map 
\[ (R^1\pi_{v,*}\Z_p)^{\vee} \hookrightarrow \Mb, \]
arising from the decomposition of the lisse sheaf $\Mb = \Mb_{\tors} \oplus \Mb/\tors$ and Proposition \ref{prop: mult variant of prim comp}(1). Together with the first map in the sequence (\ref{eq: mult HT sequence}), this yields a natural pairing
\[ \PPic_{X/S,\et}\langle p^{\infty} \rangle\times R^1\pi_{v,*}\Z_p(1)^{\vee}\rightarrow \G_m\langle p^{\infty} \rangle.\]
This generalizes (the topologically $p$-torsion part of) the analytic Weil pairings of Deninger--Werner \cite{DeningerWerner2005}\cite{Deninger2005} and Heuer \cite[§4.1]{heuer2022geometric}.
\end{remark}
\begin{proof}
Let us set
\[  \PPic_{X/S,\et}^{\ptt}= R^1\pi_{\Et,*}\G_m\langle p^{\infty}\rangle.\]
The representability of $\PPic_{X/S,\et}^{\ptt}$ follows from Theorem \ref{Thm: comparison theorem for analytic p-divisible groups}(4). Let us show that the natural map 
    \[ \PPic_{X/S,\et}^{\ptt} \rightarrow \PPic_{X/S,\et}\]
    identifies the former with the topologically $p$-torsion subsheaf of the latter. Define the quotient sheaf $\cj{\G}_m = \G_m/\G_m\langle p^{\infty} \rangle$, computed in the étale topology on $\Adic_X$. Then the above morphism fits in a long exact sequence of sheaves on $\Perf_S$
    % https://q.uiver.app/#q=WzAsNyxbMiwwLCJcXHBpX3sqfVxcY2p7XFxHfV9tIl0sWzMsMCwiXFxQUGljX3tYL1MsXFxldH1ee1xccHR0fSJdLFs0LDAsIlxcUFBpY197WC9TLFxcZXR9Il0sWzUsMCwiUl4xXFxwaV97XFxFdCwqfVxcY2p7XFxHfV9tIl0sWzEsMCwiXFxwaV97Kn1cXEdfbSJdLFs2LDAsIlxcbGRvdHMiXSxbMCwwLCJcXGxkb3RzIl0sWzAsMV0sWzEsMl0sWzIsM10sWzQsMF0sWzMsNV0sWzYsNF1d
\[\begin{tikzcd}
	\ldots & {\pi_{*}\G_m} & {\pi_{*}\cj{\G}_m} & {\PPic_{X/S,\et}^{\ptt}} & {\PPic_{X/S,\et}} & {R^1\pi_{\Et,*}\cj{\G}_m} & \ldots
	\arrow[from=1-1, to=1-2]
	\arrow[from=1-2, to=1-3]
	\arrow[from=1-3, to=1-4]
	\arrow[from=1-4, to=1-5]
	\arrow[from=1-5, to=1-6]
	\arrow[from=1-6, to=1-7]
\end{tikzcd}\]
By Lemma \ref{lemma: technical lemma on Gmbar}(1) below, the natural map
\[ \cj{\G}_{m} \rightarrow \pi_{*}\cj{\G}_m\]
is an isomorphism. It follows that in the following commutative diagram
% https://q.uiver.app/#q=WzAsNCxbMSwwLCJcXHBpX3sqfVxcY2p7XFxHfV9tIl0sWzAsMCwiXFxwaV97Kn1cXEdfbSJdLFswLDEsIlxcR197bX0iXSxbMSwxLCJcXGNqe1xcR31fe219Il0sWzEsMF0sWzIsMV0sWzIsMywiIiwyLHsic3R5bGUiOnsiaGVhZCI6eyJuYW1lIjoiZXBpIn19fV0sWzMsMCwiXFxjb25nIiwyXV0=
\[\begin{tikzcd}
	{\pi_{*}\G_m} & {\pi_{*}\cj{\G}_m} \\
	{\G_{m}} & {\cj{\G}_{m}}
	\arrow[from=1-1, to=1-2]
	\arrow[from=2-1, to=1-1]
	\arrow[two heads, from=2-1, to=2-2]
	\arrow["\cong"', from=2-2, to=1-2]
\end{tikzcd}\]
the diagonal is a surjection of étale sheaves. Hence also $\pi_*\G_m \twoheadrightarrow \pi_*\cj{\G}_m$ is surjective, so that $\PPic_{X/S,\et}^{\ptt} \hookrightarrow \PPic_{X/S,\et}$ is injective. We now apply the left exact functor $\underline{\Hom}(\Z_p,-)$ to the above sequence. By Lemma \ref{lemma: technical lemma on Gmbar}(2) below, 
\[ \underline{\Hom}(\Z_p,R^1\pi_{\Et,*}\cj{\G}_m)=0,\]
hence we end up with the following commutative diagram
% https://q.uiver.app/#q=WzAsMyxbMCwwLCJcXFBQaWNfe1gvUyxcXGV0fV57XFxwdHR9Il0sWzEsMCwiXFx1bmRlcmxpbmV7XFxIb219KFxcWl9wLFxcUFBpY197WC9TLFxcZXR9KSJdLFsxLDEsIlxcUFBpY197WC9TLFxcZXR9LiJdLFswLDEsIlxcY29uZyJdLFsxLDIsIlxcZXZfMSJdLFswLDIsIiIsMix7InN0eWxlIjp7InRhaWwiOnsibmFtZSI6Imhvb2siLCJzaWRlIjoidG9wIn19fV1d
\[\begin{tikzcd}
	{\PPic_{X/S,\et}^{\ptt}} & {\underline{\Hom}(\Z_p,\PPic_{X/S,\et})} \\
	& {\PPic_{X/S,\et}.}
	\arrow["\cong", from=1-1, to=1-2]
	\arrow[hook, from=1-1, to=2-2]
	\arrow["{\ev_1}", from=1-2, to=2-2]
\end{tikzcd}\]
From this, the vertical map is injective, showing that $\PPic_{X/S,\et}^{\ptt}$ identifies with the topologically $p$-torsion subsheaf of $\PPic_{X/S,\et}$, as required.

It remains to see that the maximal analytic $p$-divisible subgroup $\Hs \sub \PPic_{X/S,\et}^{\ptt}$ is dualizable. By Theorem \ref{Thm: comparison theorem for analytic p-divisible groups}, the map $f_{\Hs}$ is the following composition
\[ R^1\pi_{\Et,*}\G_a \rightarrow R^1\pi_{v,*}\G_a \cong (R^1\pi_{v,*}\Z_p) \otimes \G_a.\]
This is the inclusion in the relative Hodge--Tate sequence (\ref{eq: contravariant HT seq})
% https://q.uiver.app/#q=WzAsNSxbMCwwLCIwIl0sWzEsMCwiKFJeMVxccGlfe1xcZXQsKn1cXEdfYSlcXG90aW1lc197XFxPc197XFxldH19IFxcT3NfdiJdLFsyLDAsIihSXjFcXHBpX3t2LCp9XFxaX3ApIFxcb3RpbWVzIFxcT3NfdiJdLFszLDAsIlxccGlfKlxcT21lZ2Ffe1gvU31eMVxcb3RpbWVzX3tcXE9zX3tcXGV0fX0gXFxPc192KC0xKSJdLFs0LDAsIjAuIl0sWzAsMV0sWzEsMl0sWzIsM10sWzMsNF1d
\[\begin{tikzcd}
	0 & {(R^1\pi_{\et,*}\G_a)\otimes_{\Os_{\et}} \Os_v} & {(R^1\pi_{v,*}\Z_p) \otimes \Os_v} & {\pi_*\Omega_{X/S}^1\otimes_{\Os_{\et}} \Os_v(-1)} & {0.}
	\arrow[from=1-1, to=1-2]
	\arrow[from=1-2, to=1-3]
	\arrow[from=1-3, to=1-4]
	\arrow[from=1-4, to=1-5]
\end{tikzcd}\]
This shows that $\Hs$ is dualizable, as required. 
\end{proof}

\begin{lemma}\label{lemma: technical lemma on Gmbar}
Let $\pi\colon X \rightarrow S$ be a proper smooth map of seminormal rigid space, and consider the sheaf $\cj{\G}_m = \G_m/\G_m\langle p^{\infty} \rangle$ on $\Adic_{\Q_p,\et}$.
    \begin{enumerate}
        \item The natural map of sheaves on $\Perf_{S,\et}$
        \[ \cj{\G}_{m} \rightarrow \pi_*\cj{\G}_{m}\]
        is an isomorphism.
        \item We have
        \[ \underline{\Hom}(\Z_p, R^n\pi_{\Et,*}\cj{\G}_{m}) =0, \quad \fa n \geq 0.\]
    \end{enumerate}
\end{lemma}
\begin{proof}
    \begin{enumerate}
        \item Let $Y$ be an affinoid perfectoid space over $S$. We need to show that the natural map
        \[ \phi\colon \cj{\G}_m(Y) \rightarrow \cj{\G}_m(Y \times_S X)\]
        is an isomorphism. If $S=\Spa(C,C^+)$ is a geometric point, this is \cite[Lemma 4.11]{heuer2023diamantine}. Moreover, by \cite[Lemma 2.17]{heuer2021line}, $\cj{\G}_m$ is a $v$-sheaf. It thus admits a unique extension to $\LSD_{C,v}$ and the map $\phi$ is an isomorphism, for any locally spatial diamond $Y$ over $\Spa(C,C^+)$. We now reduce to the case $S=\Spa(C,C^+)$ using the approximation property, cf \cite[§2.2]{heuer2022gtorsors}\cite[§3.2]{gerth2024}. Let $s\in S$ with corresponding geometric point $\cj{s}\colon \Spa(C,C^+) \rightarrow S$. We have an inverse limit of diamonds
        \[ \Spa(C,C^+) = \varprojlim_i U_i,\]
        where the inverse system ranges over all affinoid adic spaces $U_i$ étale over $S$ together with a lift $u_i$ of $\cj{s}$. The diamond $Y_{\cj{s}}$ 
        can thus be written as the inverse limit
        \[ Y_{\cj{s}} = \varprojlim_i Y_{U_i},\]
        and similarly, we have
        \[ Y_{\cj{s}}\times_C X_{\cj{s}} = \varprojlim_i Y_{U_i}\times_{U_i} X_{U_i}.\]
        By \cite[Lemma 2.20, Prop. 4.1]{heuer2022gtorsors}, we therefore have
        \[ \cj{\G}_m(Y_{\cj{s}}) = \varinjlim_i \cj{\G}_m(Y_{U_i}), \quad \cj{\G}_m(Y_{\cj{s}}\times_C X_{\cj{s}}) = \varinjlim_i \cj{\G}_m(Y_{U_i}\times_{U_i} X_{U_i}).\]
        From this and the case of a geometric point, the map $\phi$ can easily be seen to be injective and surjective étale locally on $S$, and thus an isomorphism.
        
        \item We use that, by \cite[Cor. 3.21(1)]{gerth2024}, any morphism $\varphi \colon \underline{\Z_p} \times S \rightarrow R^n\pi_{\Et,*}\cj{\G}_m$ factors through $\underline{\Z/p^m\Z} \times S \rightarrow R^n\pi_{\Et,*}\cj{\G}_m$ for some $m$. But any such morphism is zero, since $R^n\pi_{\Et,*}\cj{\G}_m$ is uniquely $p$-divisible, by \cite[Lemma 2.16]{heuer2021line}. This concludes the proof.
    \end{enumerate}
\end{proof}

\subsection{Topologically $p$-torsion subgroups of abeloid varieties}\label{Subsect: Example: abeloid varieties}

We now discuss relative abeloid varieties $\As \rightarrow S$ for a seminormal rigid space $S$. We study the covariant Hodge--Tate sequence and prove that the topologically $p$-torsion subgroup $\As\langle p^{\infty} \rangle$ of a relative abeloid variety $\As$ is a dualizable analytic $p$-divisible group, with Cartier dual $\PPic_{\As/S,\et}\langle p^{\infty} \rangle$. 

We fix a non-archimedean field $(K,K^+)$ over $\Q_p$. Let $S$ be a seminormal rigid space over $K$ and let $\As \rightarrow S$ be a relative abeloid variety. We recall that this consists in a proper smooth commutative adic $S$-group with geometrically connected fibers. Consider the Picard sheaves
\begin{align}
    &\PPic_{\As/S,\et}\coloneqq R^1\pi_{\Et,*}\G_m\colon \Perf_{S,\et} \rightarrow \Ab,\\
    &\PPic_{\As/S,\et}^{\rig}\coloneqq R^1\pi_{\Et,*}^{\rig}\G_m\colon \Sm_{/S,\et} \rightarrow \Ab,
\end{align}
where $\pi_{\Et}$, $\pi_{\Et}^{\rig}$ are defined at (\ref{eq: piEt}) and (\ref{eq: piEtrig}) respectively. The translation-invariant Picard sheaves are defined as
\begin{align} \PPic_{\As/S,\et}^{\tau} = \Ker(\PPic_{\As/S,\et} \xrightarrow{m^*-p_1^*-p_2^*} \PPic_{\As\times_S \As/S,\et})\end{align}
and similarly for $\PPic_{\As/S,\et}^{\tau,\rig}$. 
\begin{definition}
    If there exists a relative abeloid variety  representing the functor $\PPic_{\As/S,\et}^{\tau,\rig}$, it is called the \emph{dual abeloid variety} and denoted $\breve{A}$. 
\end{definition}
By \cite[Prop. 3.16(4)]{gerth2024}, if the dual abeloid variety exists, there is a an isomorphism
\[ (\PPic_{\As/S,\et}^{\tau,\rig})^{\diamondsuit} =\PPic_{\As/S,\et}^{\tau}.\]
In fact, \emph{loc. cit.} gives a meaning to the left term even if $\PPic_{\As/S,\et}^{\tau,\rig}$ is not representable, such that the above equality holds true in general.

The dual abeloid variety was constructed in many cases by Lütkebohmert \cite[§6]{Bosch1991DegeneratingAV}. This includes:
\begin{itemize}
    \item Abeloid varieties $A$ over a non-archimedean field $(K,\Os_K)$.
    \item Analytification of abelian varieties, i.e. $(\As\rightarrow S) = (A^{\alg} \rightarrow S^{\alg})^{\an}$ for an abelian scheme $A^{\alg}$ over a $K$-variety $S^{\alg}$. 
    \item Relative abeloid varieties of good reduction, i.e. $(\As\rightarrow S) = (\mathfrak{A} \rightarrow \Ss)_{\eta}$, for a formal abelian scheme $\mathfrak{A}$ over a formal model $\Ss$ of $S$ of topologically finite type over $\Spf(\Os_K)$.
    \item Relative abeloid varieties $\As \rightarrow S$ over a rigid space over $(K,\Os_K)$ that admit a Raynaud uniformization.
\end{itemize}
Here, a relative abeloid variety $\As \rightarrow S$ is said to admit a Raynaud uniformization if we can realize it as a quotient $\As = E/M$, where $E$ is a locally split extension of a relative abeloid $B \rightarrow S$ of good reduction by a torus $T$, and $M \sub E$ is a $\Z$-lattice. This is summarized in a diagram as follow
% https://q.uiver.app/#q=WzAsNSxbMCwxLCJUIl0sWzEsMSwiRSJdLFsyLDEsIkIuIl0sWzEsMiwiQSJdLFsxLDAsIk0iXSxbMCwxXSxbMSwyXSxbMSwzXSxbNCwxXV0=
\[\begin{tikzcd}
	& M \\
	T & E & {B.} \\
	& \As
	\arrow[from=1-2, to=2-2]
	\arrow[from=2-1, to=2-2]
	\arrow[from=2-2, to=2-3]
	\arrow[from=2-2, to=3-2]
\end{tikzcd}\]

\begin{definition}
    Let $n\geq 1$, we define a Weil pairing
    \[ e_n\colon \As[p^n] \times \PPic_{\As/S,\et}[p^n] \rightarrow \mu_{p^n,S}\]
    via the usual procedure: Given $Y \in \Perf_S$ and a pair 
    \[ a\in \As[p^n](Y), \quad \Ls \in \PPic_{\As/S, \et}[p^n](Y),\] 
    we have non-canonical isomorphisms
    \begin{align*}
        \Os_Y &\cong t_a^*\Ls^{\otimes p^n}\otimes \Ls^{\otimes -p^n}\\
        &= (t_a^*\Ls\otimes \Ls^{-1})^{\otimes p^n}\\
        &= t_{[p^n]a}^*\Ls\otimes \Ls^{-1} \cong \Os_Y,
    \end{align*}
    where $t_a$ denotes translation by $a$. Then the composition defines an element $\zeta \in \G_m(Y)$ which can be shown to be independant of the choices above and to satisfy $\zeta^{p^n}=1$. We then let
    \[ e_n(x,\Ls) = \zeta \in \mu_{p^n}(Y).\]
\end{definition}

The Weil pairing satisfies the following usual properties.
\begin{lemma}\label{Lemma: finite Weil pairings}
    Let $\pi\colon \As \rightarrow S$ be a relative abeloid variety over a seminormal rigid space $S$.
    \begin{enumerate}
        \item For any $n\geq 1$ the Weil pairing
    \[ e_n\colon \As[p^n]\times \PPic_{\As/S,\et}[p^n] \rightarrow \mu_{p^n}\]
    is a perfect pairing, inducing an isomorphism of sheaves on $\Perf_S$
    \[ \PPic_{\As/S,\et}[p^n] \cong \underline{\Hom}(\As[p^n],\mu_{p^n}).\]
    \item After taking the limit, the Weil pairing induces an isomorphism
    \[ R^1\pi_{v,*}\Z_p(1) = T_p\PPic_{\As/S,\et} \cong T_p\As^{\vee}(1).\]
    More generally, for any $m\geq 0$,
    \[ R^m\pi_{v,*}\Z_p(1) = \bigwedge^m T_p\As^{\vee}(1).\]
    \end{enumerate}
\end{lemma}
\begin{proof}
We need to show that the map
\[ \PPic_{\As/S,\et}[p^n] \rightarrow \underline{\Hom}(\As[p^n],\mu_{p^n,S})\]
is an isomorphism. Note that by the Kummer sequence,
\[ \PPic_{\As/S,\et}[n] \cong R^1\pi_{\Et,*}\mu_n,\]
which is representable by a finite étale $S$-group. Since also the Hom sheaf is representable by a finite étale group, we can check pointwise on $S$ whether this map is an isomorphism. This reduces us to the case $S=\Spa(C,C^+)$, for a complete algebraically closed field. There it is well known, see for example \cite[§4.2]{heuer2022geometric}. Hence, we obtained an isomorphism
\[ R^1\pi_{\Et,*}\mu_{p^n} \cong \underline{\Hom}(\As[p^n],\mu_{p^n,S}).\]
More generally, by the Künneth formula, see e.g \cite[Thm. 5.6]{heuer2024primitive}
\[ R^m\pi_{\Et,*}\mu_{p^n} \cong \bigwedge^m\underline{\Hom}(\As[p^n],\mu_{p^n,S}).\]
It now follows from Proposition \ref{prop: cohomology of local system}(2)  that
\[ R^m\pi_{v,*}\Z_p(1) = \varprojlim_n R^m\pi_{\Et,*}\mu_p^n =  \varprojlim_n \bigwedge^m\underline{\Hom}(\As[p^n],\mu_{p^n,S}) = \bigwedge T_p\As^{\vee}(1).\]
Finally, we know from Theorem \ref{Thm: comparison theorem for analytic p-divisible groups}(2) that
\[ T_p\PPic_{\As,\et} = R^1\pi_{v,*}\Z_p(1).\]
\end{proof}

We can now state our result of the topologically $p$-torsion subgroup of abeloid varieties.
\begin{thm}\label{thm: two Hodge-Tate maps agree}
    Let $\pi\colon\As \rightarrow S$ be a relative abeloid variety over a seminormal rigid space $S$ over $(K,K^+)$. Consider the topologically $p$-torsion subsheaf $\As\langle p^{\infty} \rangle$, an analytic $p$-divisible group by Example \ref{ex: examples of analytic p-divisible groups}(3). Let 
    \[ f_{\As\langle p^{\infty} \rangle}\colon \Lie(\As)\otimes \Os_v \rightarrow T_p\As(-1)\otimes \Os_v\]
    be the map of Definition \ref{def: u and f}. Let
    \[ \HT_{\As/S}\colon R^1\pi_{v,*}\Z_p \otimes \Os_v \rightarrow \pi_*\Omega_{\As/S}^1\]
    be the Hodge--Tate map of $\As$ defined at (\ref{eq: contravariant HT seq}). Then
    \[ f_{\As\langle p^{\infty} \rangle} = (\HT_{\As/S})^{\vee}(-1) \]
    through the identification
    $T_p\As(-1) \cong (R^1\pi_{v,*}\Z_p)^{\vee}(-1)$ of Lemma \ref{Lemma: finite Weil pairings}(2).
\end{thm}
This can be seen as a generalization of \cite[Prop. 4.15]{Scholze2013} and \cite[Prop. III.3.1]{Scholze2015torsion}, which is the case of abelian varieties over $S=\Spa(C)$.
\begin{proof}
    The equality in the statement can be checked after base-change along a $v$-cover $\widetilde{S} \rightarrow S$, where $\widetilde{S}$ is a disjoint union of strictly totally disconnected perfectoid spaces, that we may take to be products of points. On $\widetilde{S}$, the source and target of $f_{\As\langle p^{\infty} \rangle}$ are étale vector bundles. By Corollary \ref{cor: pdiv groups on product of points}, the claimed equality may be checked on geometric points $\Spa(C,C^+)$. As vector bundles are overconvergent, we may thus assume that $S=\Spa(C,\Os_C)$, for a complete algebraically closed field $C/\Q_p$.

    First, we show that $\As\langle p^{\infty} \rangle$ is dualizable. If $\As$ has good reduction, i.e. $\As = \mathfrak{A}_{\eta}$ for a formal abelian variety $\mathfrak{A}$, then by \cite[Prop. 2.14(iv)]{heuer2022geometric}, $\As\langle p^{\infty} \rangle = \mathfrak{A}[p^{\infty}]_{\eta}$, so that $\As\langle p^{\infty} \rangle$ is dualizable, by Proposition \ref{prop: groups of good reduction are dualizable}. In general, consider the Raynaud uniformization of $\As$ \cite[Thm. 1]{Lutkebohmert1995}
    % https://q.uiver.app/#q=WzAsNSxbMSwwLCJNIl0sWzEsMSwiRSJdLFsxLDIsIlxcQXMiXSxbMCwxLCJUIl0sWzIsMSwiQi4iXSxbMCwxXSxbMSwyXSxbMywxXSxbMSw0XV0=
\[\begin{tikzcd}
	& M \\
	T & E & {B.} \\
	& \As
	\arrow[from=1-2, to=2-2]
	\arrow[from=2-1, to=2-2]
	\arrow[from=2-2, to=2-3]
	\arrow[from=2-2, to=3-2]
\end{tikzcd}\]
Then one easily checks that 
\[ \As\langle p^{\infty} \rangle = E\langle p^{\infty} \rangle \oplus M[\tfrac{1}{p}]/M\]
and that we have an exact sequence of analytic $p$-divisible groups
% https://q.uiver.app/#q=WzAsNSxbMiwwLCJFXFxsYW5nbGUgcF57XFxpbmZ0eX0gXFxyYW5nbGUiXSxbMSwwLCJUXFxsYW5nbGUgcF57XFxpbmZ0eX0gXFxyYW5nbGUiXSxbMywwLCJCXFxsYW5nbGUgcF57XFxpbmZ0eX0gXFxyYW5nbGUiXSxbMCwwLCIwIl0sWzQsMCwiMC4iXSxbMSwwXSxbMCwyXSxbMywxXSxbMiw0XV0=
\[\begin{tikzcd}
	0 & {T\langle p^{\infty} \rangle} & {E\langle p^{\infty} \rangle} & {B\langle p^{\infty} \rangle} & {0.}
	\arrow[from=1-1, to=1-2]
	\arrow[from=1-2, to=1-3]
	\arrow[from=1-3, to=1-4]
	\arrow[from=1-4, to=1-5]
\end{tikzcd}\]
The groups $T\langle p^{\infty} \rangle \cong \G_m\langle p^{\infty} \rangle^{\oplus n}$ and $B\langle p^{\infty} \rangle$ are dualizable. By considering the following commutative diagram
% https://q.uiver.app/#q=WzAsMTAsWzIsMCwiXFxMaWUoRSkiXSxbMSwwLCJcXExpZShUKSJdLFszLDAsIlxcTGllKEIpIl0sWzAsMCwiMCJdLFs0LDAsIjAiXSxbMSwxLCJUX3BUKC0xKVxcb3RpbWVzXFxHX2EiXSxbMCwxLCIwIl0sWzIsMSwiVF9wRSgtMSlcXG90aW1lc1xcR19hIl0sWzMsMSwiVF9wQigtMSlcXG90aW1lc1xcR19hIl0sWzQsMSwiMCwiXSxbMSwwXSxbMCwyXSxbMywxXSxbMiw0XSxbMSw1LCJmX3tUXFxsYW5nbGUgcF57XFxpbmZ0eX0gXFxyYW5nbGV9IiwyXSxbNiw1XSxbMCw3LCJmX3tFXFxsYW5nbGUgcF57XFxpbmZ0eX0gXFxyYW5nbGV9IiwyXSxbMiw4LCJmX3tCXFxsYW5nbGUgcF57XFxpbmZ0eX0gXFxyYW5nbGV9IiwyLHsic3R5bGUiOnsidGFpbCI6eyJuYW1lIjoiaG9vayIsInNpZGUiOiJ0b3AifX19XSxbNSw3XSxbNyw4XSxbOCw5XSxbMSw1LCJcXGNvbmciXV0=
\[\begin{tikzcd}
	0 & {\Lie(T)} & {\Lie(E)} & {\Lie(B)} & 0 \\
	0 & {T_pT(-1)\otimes\G_a} & {T_pE(-1)\otimes\G_a} & {T_pB(-1)\otimes\G_a} & {0,}
	\arrow[from=1-1, to=1-2]
	\arrow[from=1-2, to=1-3]
	\arrow["{f_{T\langle p^{\infty} \rangle}}"', from=1-2, to=2-2]
	\arrow["\cong", from=1-2, to=2-2]
	\arrow[from=1-3, to=1-4]
	\arrow["{f_{E\langle p^{\infty} \rangle}}"', from=1-3, to=2-3]
	\arrow[from=1-4, to=1-5]
	\arrow["{f_{B\langle p^{\infty} \rangle}}"', hook, from=1-4, to=2-4]
	\arrow[from=2-1, to=2-2]
	\arrow[from=2-2, to=2-3]
	\arrow[from=2-3, to=2-4]
	\arrow[from=2-4, to=2-5]
\end{tikzcd}\]
we conclude that also $f_{E\langle p^{\infty} \rangle}$ is injective. Hence $E\langle p^{\infty} \rangle$ and thus also $\As\langle p^{\infty} \rangle$ are dualizable.

By Scholze--Weinstein's result, Theorem \ref{thm: Scholze Weinstein's result}, the group $\As\langle p^{\infty} \rangle \rightarrow \Spa(C)$ has good reduction $G \rightarrow \Spf(\Os_C)$, and we know from Proposition \ref{prop: generic fibers of log p-divisible groups} that
\[ f_{\As\langle p^{\infty} \rangle} = \alpha_{G^D}[\tfrac{1}{p}]^{\vee},\]
where $\alpha_{G^D}$ is the Hodge Tate map (\ref{eq: HT map of fargues}). Hence it remains to see that 
\[ \alpha_{G^D}[\tfrac{1}{p}]=\HT_{\As/S}(-1).\]
This is proven in \cite[Prop. 4.15]{Scholze2013} in case where $\As$ is an abelian variety of good reduction, and the same proof applies verbatim here. This concludes the proof.
\end{proof}

\begin{cor}
\label{cor: HT sequences for abeloids}
    Let $A \rightarrow S$ be a relative abeloid variety over a seminormal rigid space $S$ over $(K,K^+)$. 
    \begin{enumerate}
        \item There is a natural exact sequence of $v$-vector bundles on $\Perf_{S}$, the (covariant) Hodge--Tate sequence
    % https://q.uiver.app/#q=WzAsNSxbMCwwLCIwIl0sWzEsMCwiXFxMaWUoXFxBcylcXG90aW1lc197XFxPc197XFxldH19XFxPc192Il0sWzIsMCwiVF9wXFxBcygtMSkgXFxvdGltZXNfe1xcWl9wfSBcXE9zX3YiXSxbMywwLCIoUl4xXFxwaV97XFxldCwqfVxcT3Nfe1xcQXN9KV57XFx2ZWV9IFxcb3RpbWVzX3tcXE9zX3tcXGV0fX0gXFxPc192KC0xKSJdLFs0LDAsIjAuIl0sWzAsMV0sWzIsM10sWzMsNF0sWzEsMl1d
\begin{equation}\label{eq: covariant HT sequence}\begin{tikzcd}
	0 & {\Lie(\As)\otimes_{\Os_{\et}}\Os_v} & {T_p\As(-1) \otimes_{\Z_p} \Os_v} & {(R^1\pi_{\et,*}\Os_{\As})^{\vee} \otimes_{\Os_{\et}} \Os_v(-1)} & {0.}
	\arrow[from=1-1, to=1-2]
	\arrow[from=1-2, to=1-3]
	\arrow[from=1-3, to=1-4]
	\arrow[from=1-4, to=1-5]
\end{tikzcd}\end{equation}
% \item Under the Weil pairing
% \[ T_p\As \times T_p\breve{\As} \rightarrow \Z_p(1),\]
% the orthogonal complement of $\Lie(A)$ is $\Lie(\breve{\As})$.  
\item There is a short exact sequence of $v$-sheaves on $\Perf_S$, the multiplicative Hodge--Tate sequence
% https://q.uiver.app/#q=WzAsNSxbMCwwLCIwIl0sWzEsMCwiXFxBc1xcbGFuZ2xlIHBee1xcaW5mdHl9IFxccmFuZ2xlIl0sWzIsMCwiVF9wXFxBcygtMSkgXFxvdGltZXNfe1xcWl9wfSBcXEdfbVxcbGFuZ2xlIHBee1xcaW5mdHl9IFxccmFuZ2xlIl0sWzMsMCwiKFJeMVxccGlfe1xcZXQsKn1cXE9zX0EpXntcXHZlZX1cXG90aW1lc197XFxPc197XFxldH19XFxPc192KC0xKSJdLFs0LDAsIjAuIl0sWzAsMV0sWzEsMl0sWzIsM10sWzMsNF1d
\begin{equation}\label{eq: mult HT sequence}\begin{tikzcd}
	0 & {\As\langle p^{\infty} \rangle} & {T_p\As(-1) \otimes_{\Z_p} \G_m\langle p^{\infty} \rangle} & {(R^1\pi_{\et,*}\Os_A)^{\vee}\otimes_{\Os_{\et}}\Os_v(-1)} & {0.}
	\arrow[from=1-1, to=1-2]
	\arrow[from=1-2, to=1-3]
	\arrow[from=1-3, to=1-4]
	\arrow[from=1-4, to=1-5]
\end{tikzcd}\end{equation}
The logarithm defines a natural transformation from this sequence to the Hodge--Tate sequence above.
\item The analytic $p$-divisible group $\As\langle p^{\infty} \rangle$ is dualizable, and its Cartier dual is 
\[ \As\langle p^{\infty} \rangle^D= \PPic_{\As/S,\et}\langle p^{\infty} \rangle,\]
the topologically $p$-torsion Picard variety of Corollary \ref{cor: ptop picard is dualizable}.
    \end{enumerate}
\end{cor}
\begin{proof}
The Hodge--Tate sequence is defined to be the short exact sequence obtained by applying the functor $(\cdot)^{\vee}(-1)$ to the sequence (\ref{eq: contravariant HT seq})
% https://q.uiver.app/#q=WzAsNSxbMCwwLCIwIl0sWzEsMCwiKFJeMVxccGlfe1xcZXQsKn1cXE9zX3tcXEFzfSkgXFxvdGltZXMgXFxPc192Il0sWzIsMCwiKFJeMVxccGlfe3YsKn1cXFpfcClcXG90aW1lcyBcXE9zX3YiXSxbMywwLCJcXHBpX3sqfVxcT21lZ2Ffe1xcQXMvU31eMVxcb3RpbWVzIFxcT3NfdigtMSkiXSxbNCwwLCIwLCJdLFswLDFdLFsyLDMsIlxcSFRfe1xcQXMvU30iXSxbMyw0XSxbMSwyXV0=
\[\begin{tikzcd}
	0 & {(R^1\pi_{\et,*}\Os_{\As}) \otimes \Os_v} & {(R^1\pi_{v,*}\Z_p)\otimes \Os_v} & {\pi_{*}\Omega_{\As/S}^1\otimes \Os_v(-1)} & {0,}
	\arrow[from=1-1, to=1-2]
	\arrow[from=1-2, to=1-3]
	\arrow["{\HT_{\As/S}}", from=1-3, to=1-4]
	\arrow[from=1-4, to=1-5]
\end{tikzcd}\]
together with the identification of Lemma \ref{Lemma: finite Weil pairings}(2) 
\[ (R^1\pi_{v,*}\Z_p)^{\vee} = T_p\As(-1).\]
The first map of the multiplicative Hodge--Tate sequence is defined to be the map
\[u_{\As\langle p^{\infty} \rangle} \colon \As\langle p^{\infty} \rangle \rightarrow T_p\As(-1)\otimes \G_m\langle p^{\infty} \rangle\]
of Definition \ref{def: u and f}. By Proposition \ref{Prop: diagram relating f and u} and Theorem \ref{thm: two Hodge-Tate maps agree}, we have a commutative diagram of $v$-sheaves with exact rows
% https://q.uiver.app/#q=WzAsMTAsWzIsMCwiXFxBc1xcbGFuZ2xlIHBee1xcaW5mdHl9IFxccmFuZ2xlIl0sWzIsMSwiVF9wXFxBcygtMSlcXG90aW1lcyBcXEdfbVxcbGFuZ2xlIHBee1xcaW5mdHl9IFxccmFuZ2xlIl0sWzMsMCwiXFxMaWUoXFxBcykiXSxbMywxLCJUX3BcXEFzKC0xKVxcb3RpbWVzIFxcR19hIl0sWzEsMCwiXFxBc1twXntcXGluZnR5fV0iXSxbNCwwLCIwIl0sWzAsMCwiMCJdLFswLDEsIjAiXSxbMSwxLCJUX3BcXEFzKC0xKVxcb3RpbWVzIFxcbXVfe3Bee1xcaW5mdHl9fSJdLFs0LDEsIjAuIl0sWzAsMSwidV97XFxBc1xcbGFuZ2xlIHBee1xcaW5mdHl9IFxccmFuZ2xlfSIsMl0sWzAsMiwiXFxsb2ciXSxbMiwzLCIoXFxIVF97XFxBcy9TfSlee1xcdmVlfSgtMSkiXSxbMSwzXSxbNCwwXSxbMiw1XSxbNiw0XSxbNyw4XSxbNCw4LCI9IiwyXSxbOCwxXSxbMyw5XV0=
\[\begin{tikzcd}
	0 & {\As[p^{\infty}]} & {\As\langle p^{\infty} \rangle} & {\Lie(\As)} & 0 \\
	0 & {T_p\As(-1)\otimes \mu_{p^{\infty}}} & {T_p\As(-1)\otimes \G_m\langle p^{\infty} \rangle} & {T_p\As(-1)\otimes \G_a} & {0.}
	\arrow[from=1-1, to=1-2]
	\arrow[from=1-2, to=1-3]
	\arrow["{=}"', from=1-2, to=2-2]
	\arrow["\log", from=1-3, to=1-4]
	\arrow["{u_{\As\langle p^{\infty} \rangle}}"', from=1-3, to=2-3]
	\arrow[from=1-4, to=1-5]
	\arrow["{(\HT_{\As/S})^{\vee}(-1)}", from=1-4, to=2-4]
	\arrow[from=2-1, to=2-2]
	\arrow[from=2-2, to=2-3]
	\arrow[from=2-3, to=2-4]
	\arrow[from=2-4, to=2-5]
\end{tikzcd}\]
It follows by the Snake Lemma that $u_{\As\langle p^{\infty} \rangle}$ is injective with cokernel $(R^n\pi_{\et,*}\Os_A)^{\vee}\otimes\Os_v(-1)$. Hence $\As\langle p^{\infty} \rangle$ is dualizable, as required. Finally, by definition of Cartier dual, Definition \ref{def: dualizable}, we see that $\As\langle p^{\infty} \rangle$ and $\PPic_{\As/S,\et}\As\langle p^{\infty} \rangle^D$ match to the same tuple
\[ (\Lie(\As),T_p\As, f_{\As\langle p^{\infty} \rangle} = (\HT_{\As/S})^{\vee}(-1))\]
under the equivalence of Theorem \ref{thm: extending Fargues' equivalence of categories}. Hence they are canonically isomorphic, which concludes the proof.
\end{proof}

\section{Dieudonné theory}\label{section Dieudonne theory}
In the course of this section, given a perfectoid space $S/\Q_p$, we construct an equivalence $\Gs \mapsto \Ms(\Gs)$ between dualizable analytic $p$-divisible group over $S$ and certain local shtukas over $S$ in the sense of \cite{SW20}. We then use this to define a $v$-vector bundle $D(\Gs)$ functorially associated with $\Gs$, the de Rham module, and we relate our functors to previous constructions.

\subsection{Reminders about the Fargues--Fontaine curve}\label{subsection: Recollections on the Fargues--Fontaine curve}
We briefly survey some basic constructions surrounding the Fargues--Fontaine curve and its vector bundles. Readers who are familiar with this theory are invited to skip this subsection.

Let $S$ be a perfectoid space of characteristic $p$. We let $\Yss_{S}$ denote the adic space $S \mathbin{\dot\times} \Spa(\Z_p)$, where we use the notation of \cite[§11.2]{SW20}. This is an analytic adic space over $\Z_p$, constructed via glueing from the affinoid perfectoid case $S=\Spa(A,A^+)$, in which case it is explicitely given by
\[ \Yss_{S} = \Spa(W(A^+))\backslash  \{[\varpi] =0\},\]
for any pseudouniformizer $\varpi \in A^+$. There is a continuous function, depending on $\varpi$ \cite[§12.2]{SW20}
\[ \kappa\colon \vert \Yss_{S} \vert\rightarrow [0,\infty[.\]
For any interval $I \sub [0,\infty[$, we let $\Yss_{S,I}$ denote the interior of $\kappa^{-1}(I)$. The space $\Yss_{S,]0,\infty[}$ is independant of $\varpi$, agrees with $S \mathbin{\dot\times} \Spa(\Q_p)$ and has the following description in the affinoid case
\[  \Yss_{S,]0,\infty[} = \Spa(W(A^{+}))\backslash \{p[\varpi] =0\}.\]
By \cite[Prop. II.1.1]{fargues2024geometrization}, the spaces $\Yss_{S}$ and $\Yss_{S,]0,\infty[}$ are sousperfectoid adic spaces, and by \cite[Prop. 9.9]{kim2025uniquenessfunctorialityigusastacks}, $\Yss_{S,]0,\infty[}$ is a good adic space\footnote{The cited result shows that $\Yss_{S,]0,\infty[}$ admits a basis of affinoid opens $\Spa(A,A^+)$ such that $A$ is $v$-complete. Since any étale map is locally finite étale and since $v$-completeness is stable under finite étale maps by \cite[Lemma 9.8(a)]{Kedlaya2020Sheafiness}, it follows that $\Yss_{S,]0,\infty[}$ is a good adic space.}. They are functorial in $S$ and therefore carry a natural Frobenius $\varphi$ which restricts to the usual Frobenius on $S \cong \{p =0\} \sub \Yss_{S}$. It acts properly discontinuously on $\Yss_{S,]0,\infty[}$ and we define the relative Fargues--Fontaine curve to be the adic space over $\Q_p$
\[ X_S = \Yss_{S,]0,\infty[}/\varphi^{\Z}.\]
It comes with a natural line bundle $\Os_{X_S}(1)$ and for a vector bundle $\Es$ on $X_S$ and $n\in \Z$, we let $\Es(n) = \Es \otimes_{\Os_{X_S}} \Os_{X_S}(1)^{\otimes n}$. 

For any untilt $S^{\sharp}$ of $S$ of characteristic $0$, there is an associated map $i_{S^{\sharp}}\colon S^{\sharp} \rightarrow \Yss_{S,]0,\infty[}$. In the affinoid case, it is induced by Fontaine's map
\[ \theta\colon W(A^+) \rightarrow A^{+,\sharp}.\]
We will denote by $\xi$ a distinguished element that generates the kernel of $\theta$. By \cite[Prop. II.1.18]{fargues2024geometrization}, the map $S^{\sharp} \rightarrow \Yss_{S,]0,\infty[}$ is a closed Cartier divisor in the sense of \cite[§5.3]{SW20}. Moreover, the composition with the projection $i\colon S\rightarrow X_{S}$ is still a closed Cartier divisor.

Let $G$ be any algebraic group over $\Q_p$. By \cite[Prop. 19.5.3]{SW20}, the prestack 
\begin{align*} 
&S\in \Perf_{\F_p} \mapsto \{ \text{étale }G \text{-torsors on }X_S\}
\end{align*}
is a small $v$-stack. Moreover, we have the following classification of $G$-bundles for the absolute relative curve. Let $\breve{\Q}_p = W(\cj{\F}_p)[\tfrac{1}{p}]$ and consider the Kottwitz set \cite[§1.7]{Kottwitz85}
\begin{align}
    B(G) = G(\breve{\Q}_p)/\simeq
\end{align} 
where $b,b' \in G(\breve{\Q}_p)$ are equivalent if there exists $c\in G(\breve{\Q}_p)$ with $
\varphi(c)bc^{-1} =b'$. It parametrizes isocrystals with $G$-structures. Moreover, by \cite[Thm. 1.1]{Viehmann2024}, there is a canonical bijection
\begin{align}\label{eq: from Kotwitz set to bundles}
B(G) \xrightarrow{\cong } \vert \Bun_G\restr{\Perf_{\cj{\F}_p}} \vert, \quad b \mapsto \Es^b.
\end{align}
Given an element $b\in G(\breve{\Q}_p)$, we let $\Es^b$ denote the corresponding $G$-torsor on the Fargues--Fontaine curve.

Taking a perfectoid space over $S$ to its associated relative Fargues--Fontaine curve defines a morphism of sites
\begin{align}\label{eq: the map tau} 
\tau \colon X_{S,v} \rightarrow S_v.
\end{align}
Given a $\Q_p$-local system $\V$ on $S_v$, we obtain a $v$-vector bundle $\tau^{*}\V \otimes_{\tau^*\Q_p} \Os_{X_S}$ on $X_S$, which we will abbreviate by $\V \otimes_{\Q_p} \Os_{X_S}$. Since it is trivialized by $X_{S'} \rightarrow X_S$, where $S' \rightarrow S$ is a $v$-cover trivializing $\V$, by the above descent result, $\V \otimes_{\Q_p} \Os_{X_S}$ is an étale vector bundle on $X_S$. Analogous constructions exist over $\Yss_{S,]0,\infty[}$ and $\Yss_{S}$. By \cite[§4.5]{kedlaya2019relative2}\cite[Thm. 22.3.1]{SW20}, the functor $\V \mapsto \V \otimes_{\Q_p} \Os_{X_S}$ defines an equivalence
\begin{align}\label{eq: equiv between ls and ss slope zero}
         \{\,\Q_p\text{-local systems on }S_v\,\}  \xrightarrow{\cong} \{\,\text{vector bundles of slope }0 \text{ on }X_{S,\et}\,\},
     \end{align}
with inverse given by $\Es \mapsto \tau_*\Es$.

Let $S$ be a perfectoid space over $\F_p$ with untilt $S^{\sharp}$ over $\Q_p$. Let $\Es,\Es_0$  be vector bundles on $X_S$. Recall \cite[§II.3]{fargues2024geometrization} that a modification between $\Es$ and $\Es_0$ at $S^{\sharp}$ is an isomorphism 
\[ \alpha\colon \Es \restr{X_S \backslash S^{\sharp}} \xrightarrow{\cong}\Es_0\restr{X_S \backslash S^{\sharp}}  \]
which is meromorphic along $S^{\sharp}$, i.e. locally on $S$, $\alpha$ extends to a morphism $\Es \rightarrow \Es_0(nS^{\sharp})$ along the inclusion $\Es_0 \sub \Es_0(nS^{\sharp})$, for some $n\gg 0$. We denote a modification by $\alpha\colon \Es \dashrightarrow \Es_0$. We will be particularly interested in modifications of a particularly simple type.
\begin{definition}\label{def: minuscule modification}
    A modification $\alpha\colon \Es \dashrightarrow \Es_0$ at $S^{\sharp}$ is said to be minuscule if $\alpha$ extends to a short exact sequence 
    % https://q.uiver.app/#q=WzAsNSxbMCwwLCIwIl0sWzEsMCwiXFxFcyJdLFsyLDAsIlxcRXNfMCJdLFszLDAsImlfe1Nee1xcc2hhcnB9LCp9RSJdLFs0LDAsIjAsIl0sWzAsMV0sWzEsMiwiXFxhbHBoYSJdLFsyLDNdLFszLDRdXQ==
\[\begin{tikzcd}
	0 & \Es & {\Es_0} & {i_{S^{\sharp},*}E} & {0,}
	\arrow[from=1-1, to=1-2]
	\arrow["\alpha", from=1-2, to=1-3]
	\arrow[from=1-3, to=1-4]
	\arrow[from=1-4, to=1-5]
\end{tikzcd}\]
for some vector bundle $E$ on $S_{\et}^{\sharp}$.
\end{definition}

\begin{example}
\begin{enumerate}
    \item The main example of a minuscule modification arises from the line bundle $\Os_{X_S}(1)$. For any untilt $S^{\sharp}$ of $S$ over $\Q_p$, there is a natural short exact sequence \cite[Prop. II.2.3]{fargues2024geometrization}
% https://q.uiver.app/#q=WzAsNSxbMCwwLCIwIl0sWzEsMCwiXFxRX3AoMSlcXG90aW1lcyBcXE9zX3tYX3tTfX0iXSxbMiwwLCIgXFxPc197WF97U319KDEpIl0sWzMsMCwiaV97U157XFxzaGFycH0sKn1cXE9zX3tTXntcXHNoYXJwfX0iXSxbNCwwLCIwLiJdLFswLDFdLFsxLDJdLFsyLDNdLFszLDRdXQ==
\begin{equation}\label{eq: sequence for OFF(1)}\begin{tikzcd}
	0 & {\Q_p(1)\otimes \Os_{X_{S}}} & { \Os_{X_{S}}(1)} & {i_{S^{\sharp},*}\Os_{S^{\sharp}}} & {0.}
	\arrow[from=1-1, to=1-2]
	\arrow[from=1-2, to=1-3]
	\arrow[from=1-3, to=1-4]
	\arrow[from=1-4, to=1-5]
\end{tikzcd}\end{equation}
    \item Any modification of rank $n$ vector bundles bounded (in the sense of \cite[§23.3]{SW20}) by a minuscule cocharacter $\mu$ of $\GL_n$ is minuscule. 
\end{enumerate}
\end{example}

We now reinterpret modifications of vector bundles in terms of lattices in period rings. Consider the Rham period sheaf $\B_{\dR}^+$ on $\Perf_{\Q_p,v}$. By \cite[Cor. 17.1.9]{SW20}, given an affinoid perfectoid space $S$ over $\Q_p$, finite locally free $\B_{\dR}^+$-modules on $S_v$ are equivalent to finite projective $\B_{\dR}^+(S)$-modules. In particular, since an analogous statement holds on the pro-étale site \cite[Thm. 6.5]{scholze2013padicHodge}, we conclude that, for a locally noetherian adic space $X$ over $\Q_p$, the $\B_{\dR}^+$-locally free sheaves on $X_v^{\diamondsuit}$ and on $X_{\proet}$ are equivalent. We will sometimes go back and forth between both categories without further comments.

Given an affinoid perfectoid space $S^{\sharp}$ with tilt $S$, the ring $\B_{\dR}^+(S^{\sharp})$ corepresent the formal neighborhood of the closed Cartier divisor $i\colon S^{\sharp} \rightarrow X_S$. Let $\Es_0$ be a vector bundle on $X_S$ with completed stalk $\Xi_0 = \Es_0 \otimes_{\Os_{X_S}} \B_{\dR}^+(S^{\sharp})$. Then using Beauville--Laszlo glueing, the functor taking a vector bundle to its completed stalk yields an equivalence \cite[Prop. 19.1.2]{SW20}
    \begin{align}\label{eq: from modif to lattices} \{\, \text{modifications } \alpha\colon \Es \dashrightarrow \Es_0 \text{ at }S^{\sharp}\,\}\xrightarrow{\cong}\{ \,\B_{\dR}^+\text{-lattices } \Xi \sub \Xi_0[\xi^{-1}] \text{ over }S^{\sharp}  \, \}.\end{align}
Of relevance to us are $\B_{\dR}^+$-lattices whose relative positions are particularly simple.
\begin{definition}\label{def: minuscule lattice}
    Let $S^{\sharp}$ be a perfectoid space over $\Q_p$ and let $\Xi_0$ be a finite locally free $\B_{\dR}^+$-module over $S^{\sharp}$. A $\B_{\dR}^+$-lattice $\Xi \sub \Xi_0[\xi^{-1}]$ is called minuscule if 
    \[ \Xi \sub \Xi_0 \sub \xi^{-1}\Xi.\]
\end{definition}
We claim that the minuscule modifications correspond to minuscule lattices under the equivalence (\ref{eq: from modif to lattices}). The only non-trivial part is showing that given a minuscule lattice $\Xi \sub \Xi_0$, the sheaf $E = \Xi_0/\Xi$ is a vector bundle on $S^{\sharp}$. By the Snake Lemma, it is equivalent to showing that $W= \Xi/\xi\Xi_0$ is a vector bundle, which is done in the following lemma.

\begin{lemma}\label{lemma: minuscule lattices vs flags}
    Let $\Xi_0$ be a finite locally free $\B_{\dR}^+$-module on a perfectoid space $S^{\sharp}$ over $\Q_p$ and let $D_0 = \Xi_0 \otimes_{\B_{\dR}^+,\theta}\Os_{S_v}$, a vector bundle. Then the assignement $ \Xi \mapsto \frac{\Xi}{\xi\Xi_0}$ defines a bijection
\begin{align*}
    \left\{\text{\begin{tabular}{l} {\parbox{4.2cm}{minuscule $\B_{\dR}^+$-lattices $\Xi \sub \Xi_0[\xi^{-1}]$ }}\end{tabular}}\right\}
         \xlongrightarrow{\cong} \left\{\text{\begin{tabular}{l} {\parbox{4.5cm}{locally direct summands $W\hookrightarrow D_0$}}\end{tabular}}\right\}.
    \end{align*} 
\end{lemma}
\begin{proof}
We may assume that $S=\Spa(R,R^+)$ is affinoid perfectoid. We may identity $D_0$ with its $S$-rational points $D_0(S) = \Xi_0(S)/\xi\Xi_0(S)$, a finite projective $R$-module, and similarly for $\Xi$ and $\Xi_0$. Let us first show that the inverse functor is well-defined. Given a split injection of $R$-modules $W\sub D_0$, we let $\Xi = q^{-1}(W) \sub \Xi_0$, where $q\colon \Xi_0 \rightarrow D_0$ is the natural projection. Then we have a short exact sequence
% https://q.uiver.app/#q=WzAsNSxbMCwwLCIwIl0sWzEsMCwiXFx4aVxcWGlfMCJdLFsyLDAsIlxcWGkiXSxbMywwLCJXIl0sWzQsMCwiMC4iXSxbMCwxXSxbMiwzXSxbMyw0XSxbMSwyXV0=
\[\begin{tikzcd}
	0 & {\xi\Xi_0} & \Xi & W & {0.}
	\arrow[from=1-1, to=1-2]
	\arrow[from=1-2, to=1-3]
	\arrow[from=1-3, to=1-4]
	\arrow[from=1-4, to=1-5]
\end{tikzcd}\]
Here, $W$ is a finite projective $R$-module and thus has projective dimension $\leq 1$ over $\BdR+(R)$, since $\xi$ is a non-zero divisor. We conclude that $\Xi$ is finite projective over $\BdR+(R)$. Conversely, given a minuscule $\BdR+(R)$ lattice $\Xi \sub \Xi_0$, we need to show that $W = \frac{\Xi(S)}{\xi\Xi_0(S)}$ is a finite projective module. We argue as in \cite[Remark 4.25]{anschutz2022prismaticdieudonnetheory}: the module $W$ is perfect as an $\BdR+(R)$ module, as $\xi$ is a non-zero divisor. In particular, it is pseudocoherent as a $\BdR+(R)$ module and thus also as an $R$-module. Let $\Spa(K,K^+) \rightarrow \Spa(R,R^+)$ be any point, then the derived tensor product
\[ W \otimes_{R}^{L} K = W \otimes_{\BdR+(R)}^{L} \BdR+(K)\]
is a perfect complex of $\BdR+(K)$-modules and is thus bounded. It now follows from \cite[\href{https://stacks.math.columbia.edu/tag/068W}{Tag 068W}]{stacks-project} that $W$ admits a finite resolution of finite projective $R$-modules. Since by construction it has projective dimension $1$ over $\BdR+(R)$, it follows that $W$ is a finite projective $R$-module, as required.
\end{proof}

\begin{example}
     The sequence (\ref{eq: sequence for OFF(1)}) gives rise to the following exact sequence over $S^{\sharp}$ after taking completed stalks
% https://q.uiver.app/#q=WzAsNSxbMCwwLCIwIl0sWzEsMCwiXFxRX3AoMSlcXG90aW1lc1xcQl97XFxkUn1eKyJdLFsyLDAsIlxceGleey0xfShcXFFfcCgxKVxcb3RpbWVzXFxCX3tcXGRSfV4rKSJdLFszLDAsIlxcT3Nfe1Nee1xcc2hhcnB9fSJdLFs0LDAsIjAuIl0sWzAsMV0sWzEsMl0sWzIsM10sWzMsNF1d
\begin{equation}\begin{tikzcd}
	0 & {\Q_p(1)\otimes\B_{\dR}^+} & {\xi^{-1}(\Q_p(1)\otimes\B_{\dR}^+)} & {\Os_{S^{\sharp}}} & {0.}
	\arrow[from=1-1, to=1-2]
	\arrow[from=1-2, to=1-3]
	\arrow[from=1-3, to=1-4]
	\arrow[from=1-4, to=1-5]
\end{tikzcd}\end{equation}
\end{example}

Finally, we will also need local shtukas \cite[Def. 23.1.1]{SW20}. Given a characteristic $p$ perfectoid space $S$ with an untilt $S^{\sharp}$ over $\Q_p$, a local shtuka with a leg at $S^{\sharp}$ is a pair $(\Ms, \varphi_{\Ms})$ consisting of a vector bundle $\Ms$ over $\Yss_{S}$ together with an isomorphism
\[ \varphi_{\Ms}\colon (\varphi^*\Ms)\restr{\Yss_{S} \backslash S^{\sharp}} \xrightarrow{\cong} \Ms\restr{\Yss_{S} \backslash S^{\sharp}} \]
that is meromorphic along $S^{\sharp}$. We say that a sequence of local shtukas is exact if the sequence formed by the underlying vector bundles is exact. We will again impose a minuscule condition.

\begin{definition}\label{def: minuscule shtuka}
    Let $S$ be a characteristic $p$ perfectoid space with untilt $S^{\sharp}$ over $\Q_p$. A local shtuka $(\Ms,\varphi_{\Ms})$ with a leg at $S^{\sharp}$ is said to be minuscule if
    \[ \Ms \sub \varphi_{\Ms}(\varphi^*\Ms) \sub \tfrac{1}{\xi}\Ms.\]
\end{definition}

\begin{proposition}{(\cite[Prop. 23.3.1]{SW20})}\label{prop: from modification to shtukas}
Let $S$ be a characteristic $p$ perfectoid space with untilt $S^{\sharp}$ over $\Q_p$. Then there is an exact equivalence of categories between
\begin{enumerate}
    \item local shtukas $(\Ms,\varphi_{\Ms})$ over $\Yss_S$ with a leg at $\varphi^{-1}(S^{\sharp})$, and
    \item tuples $(\Lb, \Es, \alpha)$ consisting in a $\Z_p$-local system $\Lb$ on $S$ (or equivalently, on $S^{\sharp}$), a vector bundle $\Es$ on $X_S$ and a modification $\alpha\colon \Lb \otimes_{\Z_p} \Os_{X_S} \dashrightarrow \Es$ at $S^{\sharp}$.
\end{enumerate}
    % \begin{align*}
    %  \left\{\text{\begin{tabular}{l} {\parbox{3.7cm}{$(\Lb, \Es, \alpha)$ modifications $\alpha\colon \Es \dashrightarrow \Es_0$ on $X_S$ at $S^{\sharp}$}}\end{tabular}}\right\}
    %      \xrightarrow{\cong} \left\{\text{\begin{tabular}{l} {\parbox{5.2cm}{local shtukas $(\Ms,\varphi_{\Ms})$ over $\Yss_S$ with a leg at $\varphi^{-1}(S^{\sharp})$}}\end{tabular}}\right\}.
    % \end{align*}
    Moreover, it restricts to an equivalence between minuscule objects on each side. In that case, if $\Lb \otimes_{\Z_p} \Os_{X_S} \xrightarrow{\alpha} \Es \rightarrow i_*E$ is a minuscule modification corresponding to the shtuka $(\Ms,\varphi_{\Ms})$, we have a natural identification of vector bundles on $S^{\sharp}$
    \[ E = \varphi_{\Ms}(\varphi^*\Ms)/\Ms. \]
\end{proposition}
\begin{proof}
    This follows from the proof of \cite[Prop. 23.3.1]{SW20} (see also \cite[Prop. 12.4.6]{SW20} for the case of a geometric point) and we reproduce it here for the reader's convenience. We may assume that $S$ is affinoid and we fix a pseudouniformizer $\varpi \in \Os^+(S)$ such that $S^{\sharp} \sub \Yss_{S,\{ 1\}}$. Let $\alpha \colon \Lb \otimes_{\Z_p}\Os_{X_S}\dashrightarrow \Es$ be a modification and consider the cover $\pi\colon \Yss_{S,]0,\infty[} \rightarrow X_S$. Let $\Ms = \pi^*\Es$, which come with a semilinear isomorphism $\varphi_{\Ms}\colon \varphi^*\Ms \rightarrow \Ms$, and let $(\Ms',\varphi_{\Ms'
    })$ be obtained similarly from $\Lb \otimes_{\Z_p}\Os_{X_S}$. Then $\Ms' \cong \Lb \otimes \Os_{\Yss_{S,]0,\infty[}}$ trivially spreads out to a $\varphi$-equivariant vector bundle on $\Yss_S$ and $\alpha$ gives rise to a Frobenius-equivariant isomorphism
    \[\iota\colon \Ms'\restr{\Yss_{S,]0,\infty[}\backslash \bigcup_{n \in \Z}\varphi^n(S)} \xrightarrow{\cong } \Ms\restr{\Yss_{S,]0,\infty[}\backslash \bigcup_{n \in \Z}\varphi^n(S)}\]
    that is meromorphic at each $\varphi^n(S)$. From this, we may view $\iota$ as a glueing data for the covering of $\Yss_S$ by the open subsets $\Yss_{S,[0,1[}$ and $\Yss_{S,]p^{-1},\infty[}$. This lets us spread $\Ms$ to a vector bundle on $\Yss_S$ and extend its Frobenius to a map $\varphi_{\Ms}\colon \varphi^*\Ms \restr{\Yss_{S}\backslash \varphi^{-1}(S)} \rightarrow \Ms \restr{\Yss_{S}\backslash \varphi^{-1}(S)}$. Hence $(\Ms,\varphi_{\Ms})$ promotes to a shtuka with one leg at $\varphi^{-1}(S)$. If $\alpha$ is assumed to be minusucule, the map $\iota$ extends to an injection over $\Yss_{S,]0,\infty[}$ and one easily deduces that $(\Ms,\varphi_{\Ms})$ is minuscule. The formula for $\varphi_{\Ms}(\varphi^*\Ms)/\Ms$ is left to the reader.
    
    Conversely, given a shtuka $(\Ms,\varphi_{\Ms})$ with one leg at $\varphi^{-1}(S)$, the restriction $\Ms\restr{\Yss_{S,]0,1[}}$ is $\varphi^{-1}$-equivariant and thus defines a vector bundle $\Es$ on $X_S$, using that the latter is covered by $\Yss_{S,]0,1[}$. On the other hand, $\Ms$ restricts to a $\varphi^{-1}$-equivariant bundle $\Ms'$ on $\Yss_{S,[0,1[}$. By \cite[Prop. 22.6.1]{SW20}, we have $\Ms' = \Lb \otimes \Os_{\Yss_{S,[0,1[}}$, for a $\Z_p$-local system $\Lb$ on $S$. By construction, $(\Ms,\varphi_\Ms)$ and $(\Ms',\varphi_{\Ms'})$ are identified away from $\bigcup_{n \in \Z} \varphi^n(S)$. This yields a modification $\alpha \colon \Lb \otimes_{\Z_p}\Os_{X_S}\dashrightarrow \Es$. Again, it can be seen to be minuscule as soon as $(\Ms,\varphi_{\Ms})$ is. 
    
    It remains to verify the claim about exactness. One direction is clear. For the converse, starting with an exact sequence of local shtukas $\Ms_1 \rightarrow \Ms_2 \rightarrow \Ms_2$, by restricting it to open subsets of $\Yss_{S,]0,\infty[}$ over which the projection $\pi$ is an isomorphism, we immediately see that the resulting sequence $\Es_1 \rightarrow \Es_2 \rightarrow \Es_2$ of vector bundles on $X_S$ is exact. It remains to see that the associated sequence $\Lb_1 \rightarrow \Lb_2 \rightarrow \Lb_3$ of $\Z_p$-local systems on $S$ is exact. This can be checked on $S$-points, where $S$ is a product of geometric points $\Spa(C_i,C_i^+)$. Then by Proposition \ref{prop: dualiazable pdiv gps over product of points} applied to $\Lb_k[\tfrac{1}{p}]/\Lb_k$ for $1 \leq k \leq 3$, we find that $\Lb_k(S) = \prod_{i \in I}\Lb_k(C_i,C^+)$, so that we may further reduce to the case $S=\Spa(C,C^+)$. In that case, the sequence is exact for dimension reasons, using Artin--Schreier--Witt theory, see \cite[Thm. 12.3.4]{SW20}.
\end{proof}

\begin{example}\label{ex: examples of shtukas}
    \begin{enumerate}
        \item The Breuil--Kisin twist: Let $S$ be a characteristic $p$ perfectoid space with untilt $S^{\sharp}$, and assume for simplicity that the latter lives over $\Q_p^{\cyc}$. Let $\epsilon = (1,\zeta_p, \zeta_{p^2},\ldots)$ be a system of compatible $p$th roots of unity and let $\mu = [\epsilon]-1$ and $\xi = \tfrac{\mu}{\varphi^{-1}(\mu)}$. Then $\varphi^{-n}(\xi) \vert \mu$ for all $n \geq 0$, and in fact the vanishing locus of $\mu$ in $\Yss_{S}$ agrees with the reunion of the positive Frobenius-translates of $S^{\sharp}$:
    \[ \{ \mu =0\} = \bigcup_{n \geq 0} \varphi^n(S^{\sharp}).\]
    The minuscule local shtuka corresponding to the sequence (\ref{eq: sequence for OFF(1)}) under Proposition \ref{prop: from modification to shtukas} then is
    \[ \Os_{\Yss_S}\{1 \} \coloneqq \tfrac{1}{\mu}(\Z_p(1) \otimes_{\Z_p} \Os_{\Yss_{S}}) \sub \Yss_S,\]
    with its natural Frobenius.
    \item Let $S^{\sharp}=\Spa(R,R^+)$ be affinoid perfectoid and let $(M,\varphi_M)$ be a Breuil--Kisin--Fargues module over $R^+$ \cite[Def. 4.22]{Bhatt2018}. Recall that this consists in a finite projective $A_{\inf}(R^+)$-module together with a $\varphi$-semilinear automorphism 
\[ \varphi_M\colon (\varphi^*M)[\tfrac{1}{\varphi(\xi)}] \rightarrow M[\tfrac{1}{\varphi(\xi)}].\]
It is called minuscule if
\[ M\sub \varphi_M(\varphi^*M) \sub \tfrac{1}{\varphi(\xi)}M.\]
Then one can associate to it a local shtuka $(\widetilde{M},\varphi_{\widetilde{M}})$ over $\Yss_S$ with one leg at $\varphi^{-1}(S^{\sharp})$, by pulling back along the map $\Yss_S \sub \Spa(A_{\inf}(A^+))$. Clearly $(\widetilde{M},\varphi_{\widetilde{M}})$ is minuscule if $(M,\varphi_M)$ is. The shtuka $\Os_{\Yss_S}\{1\}$ arises this way from the Breuil--Kisin--Fargues twist $\Ainf\{1 \}$ \cite[Ex. 4.24]{Bhatt2018}.
    \end{enumerate}
\end{example}

We record the following lemma about duality in this situation.
\begin{lemma}\label{lemma: dual of vector bundles and shtukas}
    Let $(\V \otimes_{\Q_p}\Os_{X_S} \xrightarrow{\alpha} \Es \rightarrow i_*E)$ be a minuscule modification at $S^{\sharp}$. 
    \begin{enumerate}
        \item Upon applying the twisted dual functor $(\cdot)^{\vee} \otimes_{\Os_{X_S}} \Os_{X_S}(1)$, we obtain a minuscule modification
    % https://q.uiver.app/#q=WzAsNSxbMCwwLCIwIl0sWzEsMCwiXFxWXntcXHZlZX0oMSlcXG90aW1lcyBcXE9zX3tYX1N9Il0sWzIsMCwiXFxFc157XFx2ZWV9XFxvdGltZXNfe1xcT3Nfe1hfU319XFxPc197WF9TfSgxKSJdLFszLDAsImlfKlxcb21lZ2Fee1xcdmVlfSJdLFs0LDAsIjAsIl0sWzAsMV0sWzEsMiwiXFxhbHBoYSciXSxbMiwzXSxbMyw0XV0=
\[\begin{tikzcd}
	0 & {\V^{\vee}(1)\otimes \Os_{X_S}} & {\Es^{\vee}\otimes_{\Os_{X_S}}\Os_{X_S}(1)} & {i_*\omega^{\vee}} & {0,}
	\arrow[from=1-1, to=1-2]
	\arrow["{\alpha'}", from=1-2, to=1-3]
	\arrow[from=1-3, to=1-4]
	\arrow[from=1-4, to=1-5]
\end{tikzcd}\]
where $\alpha'$ arises from the inclusion $\Es \sub \xi^{-1}(\V \otimes \Os_{X_S}) = \V(-1) \otimes \Os_{X_s}(1)$ upon taking twisted dual, and $\omega$ is a vector bundle on $S^{\sharp}$ such that $i_*\omega(-1)$ is the cokernel of this inclusion.
\item Assume that $\V$ admits a $\Z_p$-lattice $\Lb$ and let $(\Ms,\varphi_{\Ms})$ be the minuscule local shtuka corresponding to $(\Lb,\alpha)$ via Proposition \ref{prop: from modification to shtukas}. Then the shtuka associated with $(\Lb^{\vee}(1),\alpha')$ is $(\Ms,\varphi_{\Ms})^{\vee} \otimes_{\Os_{\Yss_S}}\Os_{\Yss_S}\{1\}$.
    \end{enumerate} 
\end{lemma}
\begin{proof}
Using Beauville--Laszlo glueing (\ref{eq: from modif to lattices}) to each of the inclusion $\V \otimes \B_{\dR}^+ \sub \Xi \sub \xi^{-1}(\V \otimes \B_{\dR}^+)$, where $\Xi$ is the completed stalk of $\Es$ at $S^{\sharp}$, the sheaf $\Es$ fits in a commutative diagram
% https://q.uiver.app/#q=WzAsMTAsWzAsMCwiMCJdLFsxLDAsIlxcVlxcb3RpbWVzIFxcT3Nfe1hfU30iXSxbMiwwLCJcXEVzIl0sWzMsMCwiaV8qRSJdLFszLDEsImlfKlxcVigtMSlcXG90aW1lcyBcXE9zX1MiXSxbNCwwLCIwIl0sWzQsMSwiMCwiXSxbMiwxLCJcXFYoLTEpXFxvdGltZXNcXE9zX3tYX1N9KDEpIl0sWzAsMSwiMCJdLFsxLDEsIlxcVigtMSlcXG90aW1lc1xcUV9wKDEpIFxcb3RpbWVzIFxcT3Nfe1hfU30iXSxbMCwxXSxbMSwyLCJcXGFscGhhIl0sWzIsM10sWzMsNCwiZiJdLFszLDVdLFs0LDZdLFs3LDRdLFsyLDcsInUiLDJdLFsxLDksIj0iLDJdLFs4LDldLFs5LDddXQ==
\[\begin{tikzcd}
	0 & {\V\otimes \Os_{X_S}} & \Es & {i_*E} & 0 \\
	0 & {\V(-1)\otimes\Q_p(1) \otimes \Os_{X_S}} & {\V(-1)\otimes\Os_{X_S}(1)} & {i_*\V(-1)\otimes \Os_S} & {0,}
	\arrow[from=1-1, to=1-2]
	\arrow["\alpha", from=1-2, to=1-3]
	\arrow["{=}"', from=1-2, to=2-2]
	\arrow[from=1-3, to=1-4]
	\arrow["u"', from=1-3, to=2-3]
	\arrow[from=1-4, to=1-5]
	\arrow["f", from=1-4, to=2-4]
	\arrow[from=2-1, to=2-2]
	\arrow[from=2-2, to=2-3]
	\arrow[from=2-3, to=2-4]
	\arrow[from=2-4, to=2-5]
\end{tikzcd}\]
 where $u$ is a minuscule modification. By the Snake Lemma, $u$ has cokernel $i_*\omega(-1)$, where $\omega$ is the cokernel of $f(1)$, a vector bundle on $S^{\sharp}$. We now apply the functor $(\cdot)^{\vee} \otimes \Os_{X_{S}}(1)$ to the resulting exact sequence. This yields an exact sequence
% https://q.uiver.app/#q=WzAsNixbMiwwLCJcXHVuZGVyc2V0ez1cXFZee1xcdmVlfSgxKSBcXG90aW1lcyBcXE9zX3tYX3tTfX19e1xcdW5kZXJicmFjZXtcXHVuZGVybGluZXtcXEhvbX0oXFxWKC0xKVxcb3RpbWVzXFxPc197WF97U319KDEpLFxcT3Nfe1hfe1N9fSgxKSl9fSJdLFszLDAsIlxcdW5kZXJsaW5le1xcSG9tfShcXEVzLFxcT3Nfe1hfe1N9fSgxKSkiXSxbMSwwLCJcXHVuZGVyc2V0ez0wfXtcXHVuZGVyYnJhY2V7XFx1bmRlcmxpbmV7XFxIb219KGlfKlxcb21lZ2EoLTEpLFxcT3Nfe1hfe1N9fSgxKSl9fSJdLFswLDAsIjAiXSxbMSwxLCJcXHVuZGVybGluZXtcXEV4dH1eMShpXypcXG9tZWdhKC0xKSxcXE9zX3tYX3tTfX0oMSkpIl0sWzIsMSwiXFx1bmRlcnNldHs9MH17XFx1bmRlcmJyYWNle1xcdW5kZXJsaW5le1xcRXh0fV4xKFxcVigtMSlcXG90aW1lc1xcT3Nfe1hfe1N9fSgxKSxcXE9zX3tYX3tTfX0oMSkpfX0uIl0sWzAsMV0sWzIsMF0sWzMsMl0sWzEsNF0sWzQsNV1d
\[\begin{tikzcd}
	0 & {\underset{=0}{\underbrace{\underline{\Hom}(i_*\omega(-1),\Os_{X_{S}}(1))}}} & {\underset{=\V^{\vee}(1) \otimes \Os_{X_{S}}}{\underbrace{\underline{\Hom}(\V(-1)\otimes\Os_{X_{S}}(1),\Os_{X_{S}}(1))}}} & {\underline{\Hom}(\Es,\Os_{X_{S}}(1))} \\
	& {\underline{\Ext}^1(i_*\omega(-1),\Os_{X_{S}}(1))} & {\underset{=0}{\underbrace{\underline{\Ext}^1(\V(-1)\otimes\Os_{X_{S}}(1),\Os_{X_{S}}(1))}}.}
	\arrow[from=1-1, to=1-2]
	\arrow[from=1-2, to=1-3]
	\arrow[from=1-3, to=1-4]
	\arrow[from=1-4, to=2-2]
	\arrow[from=2-2, to=2-3]
\end{tikzcd}\]
We claim that
\[ \underline{\Ext}^1(i_*\omega(-1),\Os_{X_{S}}(1)) = i_*\omega^{\vee}.\]
First, we have
\[ \underline{\Ext}^1(i_*\omega(-1),\Os_{X_{S}}(1)) = \underline{\Ext}^1(i_*\omega(-1)\otimes_{\Os_{X_S}}\Os_{X_{S}}(-1),\Os_{X_S}) = \underline{\Ext}^1(i_*\omega(-1),\Os_{X_S}),\]
combining the fact that $ i^*\Os_{X_S}(-1)=\Os_S$ and the projection formula. By \cite[Cor. 3.10]{anschutz2021fourier}, we may check the above isomorphism after applying the functor $\tau_*$, in which case the claimed identity reduces to
\[ \underline{\Ext}_{S_v^{\sharp}}^1(\omega(-1),\Q_p) = \omega^{\vee}.\]
This follows from Theorem \ref{thm: a fargues thm for BC spaces}, by interpreting the above Ext-group as a Yoneda Ext-group. This concludes the proof of the first part. The second part of the statement is easy and left to the reader.
\end{proof}

\subsection{Effective Banach--Colmez spaces and vector bundles}
In this subsection, we connect the effective Banach--Colmez spaces introduced in Subsection \ref{subsect: effective BC spaces} with vector bundles on the Fargues--Fontaine curve. We will exclusively work with perfectoid spaces over $\Q_p$. Therefore, we shift the notations a bit.
\begin{notn}\label{notn: FF curve has variable in char 0}
Given a perfectoid space $S$ over $\Q_p$, we denote by $X_{S}$ the Fargues--Fontaine curve relative to the tilt $S^{\flat}$. Likewise, we let $\Yss_{S}$ denote the corresponding object associated with $S^{\flat}$.
\end{notn}

The main result of this section is the following.

\begin{thm}\label{thm: vb on FF associated to an pdiv groups}
    Let $S$ be a perfectoid space over $\Q_p$ and let $\Gs \rightarrow S$ be an analytic $p$-divisible group.
    \begin{enumerate}
        \item There exists a functorial sheaf of $\Os_{X_S}$-modules $\Es(\Gs)$ on the Fargues--Fontaine curve $X_S$, coming with a modification
        % https://q.uiver.app/#q=WzAsNSxbMCwwLCIwIl0sWzEsMCwiVF9wXFxHcyBcXG90aW1lc197XFxaX3B9IFxcT3Nfe1hfU30iXSxbMiwwLCJcXEVzKFxcR3MpIl0sWzMsMCwiaV8qXFxMaWUoXFxHcykiXSxbNCwwLCIwLiJdLFswLDFdLFsxLDJdLFsyLDNdLFszLDRdXQ==
\begin{equation}\label{eq: modification defining E(G)}\begin{tikzcd}
	0 & {T_p\Gs \otimes_{\Z_p} \Os_{X_S}} & {\Es(\Gs)} & {i_*\Lie(\Gs)} & {0.}
	\arrow[from=1-1, to=1-2]
	\arrow[from=1-2, to=1-3]
	\arrow[from=1-3, to=1-4]
	\arrow[from=1-4, to=1-5]
\end{tikzcd}\end{equation}
It recovers the exact sequence (\ref{eq: inv limit of log}) upon applying the functor $\tau_*$ (\ref{eq: the map tau}). Concretely, $\Es(\Gs)$ is defined to be the pullback
% https://q.uiver.app/#q=WzAsNCxbMCwwLCJcXEVzKFxcR3MpIl0sWzEsMCwiaV8qXFxMaWUoXFxHcykiXSxbMCwxLCJUX3BcXEdzKC0xKSBcXG90aW1lc197XFxaX3B9XFxPc197WF9TfSgxKSJdLFsxLDEsImlfKihUX3BcXEdzKC0xKSBcXG90aW1lc197XFxaX3B9XFxPc19TKS4iXSxbMCwxXSxbMiwzXSxbMSwzLCJpXypmX3tcXEdzfSJdLFswLDJdXQ==
\begin{equation}\label{eq: defining diagram for E(G)}\begin{tikzcd}
	{\Es(\Gs)} & {i_*\Lie(\Gs)} \\
	{T_p\Gs(-1) \otimes_{\Z_p}\Os_{X_S}(1)} & {i_*(T_p\Gs(-1) \otimes_{\Z_p}\Os_S).}
	\arrow[from=1-1, to=1-2]
	\arrow[from=1-1, to=2-1]
	\arrow["{i_*f_{\Gs}}", from=1-2, to=2-2]
	\arrow[from=2-1, to=2-2]
\end{tikzcd}\end{equation}
The formation of $\Es(\Gs)$ and of its modification is natural in $\Gs$ and in $S$.
\item The sheaf $\Es(\Gs)$ is functorial in $\widetilde{\Gs}$ and the functor $\Es(-)$ factors through a fully faithful exact functor $\Fs \mapsto \Es(\Fs)$ on effective Banach--Colmez spaces. 
\item The group $\Gs$ is dualizable if and only if $\Es(\Gs)$ is a vector bundle. In this case, we have
\[ \Es(\Gs^D) = \Es(\Gs)^{\vee}\otimes_{\Os_{X_S}}\Os_{X_S}(1).\]
\item If $S=\Spa(C)$ for a complete algebraically closed field $C/\Q_p$, the functor sending a $p$-divisible group $G$ over $\Os_C$ to $\Es(G_{\eta})$ recovers Scholze--Weinstein's functor \cite[Prop. 5.1.6]{scholze2013moduli}.
    \end{enumerate}
\end{thm}

To allow for some degenerate cases, we introduce a weaker notion of modifications.

\begin{definition}\label{def: effective sheaf}
    Let $S$ be a perfectoid space over $\Q_p$. An effective sheaf on the Fargues--Fontaine curve $X_S$ is a sheaf $\Es$ of $\Os_{X_S}$-modules fitting in a short exact sequence
    % https://q.uiver.app/#q=WzAsNSxbMCwwLCIwIl0sWzEsMCwiXFxWXFxvdGltZXNfe1xcUV9wfVxcT3Nfe1hfe1N9fSJdLFsyLDAsIlxcRXMiXSxbMywwLCJpXypFIl0sWzQsMCwiMCwiXSxbMCwxXSxbMSwyXSxbMiwzXSxbMyw0XV0=
\[\begin{tikzcd}
	0 & {\V\otimes_{\Q_p}\Os_{X_{S}}} & \Es & {i_*E} & {0,}
	\arrow[from=1-1, to=1-2]
	\arrow[from=1-2, to=1-3]
	\arrow[from=1-3, to=1-4]
	\arrow[from=1-4, to=1-5]
\end{tikzcd}\]
where $\V$ is a $\Q_p$-local system on $S$ and $E$ is a vector bundle on $S$. 
\end{definition}

% As both $\V \otimes_{\Q_p}\Os_{X_S}$ and $i_*E$ have slopes $\leq 0$, it follows that also $\Es$ has slopes $\leq 0$. 
Note that we always have the trivial example $\Es = i_*E$. In particular, $\Es$ need not be a vector bundle in general.

\begin{definition}
    Let $\Es$ be an effective sheaf on $X_S$. Its associated Banach--Colmez space is defined to be the following $v$-sheaf
\[ \BC(\Es)\colon \Perf_{S} \rightarrow \Ab , \, \quad T \mapsto H^0(X_{T},\Es\otimes_{\Os_{X_S}}\Os_{X_{T}}).\]
\end{definition}
\begin{remark}
    Given an effective sheaf $\Es$ on $X_S$, we will sometimes use the notation 
    \[ \BC(\Es) = \tau_*\Es,\]
    where $\tau$ is the map (\ref{eq: the map tau}). To make this rigorous, one would need to consider the associated $v$-sheaf $\Es_v = \Es \otimes_{\Os_{X_S}} \Os_{X_{S,v}}$ and apply the functor $\tau_*$ to this sheaf instead.
    %However, we do not know whether $\nu_*\Es_v = \Es$ for the natural map of sites $\nu\colon X_{S,v} \rightarrow X_{S,\an}$. This can be reduced to the case where $\Es$ is a vector bundle, which would hold if $X_S$ was a good adic space, see Definition \ref{def: good adic spaces}. 
    Alternatively, one can view any effective sheaf $\Es$ as a perfect complex on $X_S$, and these satisfy $v$-descent with respect to $S$ by \cite[Prop. 2.4]{anschutz2021fourier}. This is enough to define a left exact functor $\tau_*$ on effective sheaves, with values in $v$-sheaves on $\Perf_S$. We will use this notation without comments in what follows.
\end{remark}

\begin{example}
    Consider the line bundle $\Os_{X_{S}}(1)$ with its modification (\ref{eq: sequence for OFF(1)})
    % https://q.uiver.app/#q=WzAsNSxbMCwwLCIwIl0sWzEsMCwiXFxRX3AoMSlcXG90aW1lcyBcXE9zX3tYX3tTfX0iXSxbMiwwLCIgXFxPc197WF97U319KDEpIl0sWzMsMCwiaV8qXFxPc19TIl0sWzQsMCwiMC4iXSxbMCwxXSxbMSwyXSxbMiwzXSxbMyw0XV0=
\[\begin{tikzcd}
	0 & {\Q_p(1)\otimes \Os_{X_{S}}} & { \Os_{X_{S}}(1)} & {i_*\Os_S} & {0.}
	\arrow[from=1-1, to=1-2]
	\arrow[from=1-2, to=1-3]
	\arrow[from=1-3, to=1-4]
	\arrow[from=1-4, to=1-5]
\end{tikzcd}\]
Then by the proof of \cite[II.2.2]{fargues2024geometrization}, passing to Banach--Colmez spaces yields back the sheaf of Example \ref{ex: univ cover of unipotent units}
\[ \BC(\Os_{X_{S}}(1)) = (\B_{\crys}^+)^{\varphi =p},\]
together with its associated short exact sequence
% https://q.uiver.app/#q=WzAsNSxbMiwwLCIoXFxCX3tcXGNyeXN9XispXntcXHZhcnBoaT1wfSJdLFszLDAsIiBcXE9zX1MiXSxbNCwwLCIwLiJdLFswLDAsIjAiXSxbMSwwLCJcXFFfcCgxKSJdLFswLDEsIlxcdGhldGEiXSxbMSwyXSxbMyw0XSxbNCwwXV0=
\[\begin{tikzcd}
	0 & {\Q_p(1)} & {(\B_{\crys}^+)^{\varphi=p}} & { \Os_S} & {0.}
	\arrow[from=1-1, to=1-2]
	\arrow[from=1-2, to=1-3]
	\arrow["\theta", from=1-3, to=1-4]
	\arrow[from=1-4, to=1-5]
\end{tikzcd}\]
\end{example}

\begin{lemma}\label{lemma: tau is fully faithful in most cases}
Let $S$ be a perfectoid space over $\Q_p$.
    \begin{enumerate}
        \item Let $\Es$ be an effective sheaf on $X_{S}$, then the presheaf $\BC(\Es)$ is an effective Banach--Colmez space (cf. Definition \ref{def: minuscule BC space}). Moreover,
        \[ R^j\tau_*\Es = 0, \quad \fa j> 0.\]
        \item Let $\Es'$ be another effective sheaf on $X_{S}$. Then we have an isomorphism
    \[ \underline{\Hom}_S(\BC(\Es),\BC(\Es')) = \tau_*\underline{\Hom}_{\Os_{X_{S}}}(\Es,\Es').\]
    \end{enumerate}
\end{lemma}
\begin{proof}
    \begin{enumerate}
        \item %  Let us first show that $\BC(\Es)$ is a $v$-sheaf. Let $S' \rightarrow S$ be a $v$-cover and let $S'' = S' \times_S S'$. We need to show that 
% \[ \Es(X_S) = \Ker(H^0(\Es \otimes_{\Os_{X_S}}\Os_{X_{S'}}) \xrightarrow{p_1^*-p_2^*} H^0(\Es \otimes_{\Os_{X_S}}\Os_{X_{S''}}). \]
Let us fix an exact sequence
        % https://q.uiver.app/#q=WzAsNSxbMCwwLCIwIl0sWzEsMCwiXFxWXFxvdGltZXNfe1xcUV9wfVxcT3Nfe1hfe1N9fSJdLFsyLDAsIlxcRXMiXSxbMywwLCJpXypFIl0sWzQsMCwiMC4iXSxbMCwxXSxbMSwyXSxbMiwzXSxbMyw0XV0=
\[\begin{tikzcd}
	0 & {\V\otimes_{\Q_p}\Os_{X_{S}}} & \Es & {i_*E} & {0.}
	\arrow[from=1-1, to=1-2]
	\arrow[from=1-2, to=1-3]
	\arrow[from=1-3, to=1-4]
	\arrow[from=1-4, to=1-5]
\end{tikzcd}\]
We obtain a left exact sequence of $v$-sheaves
    % https://q.uiver.app/#q=WzAsNCxbMCwwLCIwIl0sWzEsMCwiXFxCQyhcXFZcXG90aW1lc197XFxRX3B9XFxPc197WF97U319KSJdLFsyLDAsIlxcQkMoXFxFcykiXSxbMywwLCJcXEJDKGlfKkUpLiJdLFswLDFdLFsxLDJdLFsyLDNdXQ==
\[\begin{tikzcd}
	0 & {\BC(\V\otimes_{\Q_p}\Os_{X_{S}})} & {\BC(\Es)} & {\BC(i_*E).}
	\arrow[from=1-1, to=1-2]
	\arrow[from=1-2, to=1-3]
	\arrow[from=1-3, to=1-4]
\end{tikzcd}\]
By the equivalence (\ref{eq: equiv between ls and ss slope zero}), there is an isomorphism $\BC(\V\otimes_{\Q_p} \Os_{X_{S}})= \V$. It is moreover clear that $H^0(X_{S},i_*E) = H^0(S,E)$. As the formation of the closed Cartier divisor $S \hookrightarrow X_S$ commutes with pullbacks \cite[Remark II.1.5]{fargues2024geometrization}, we find that $\BC(i_*E) = E$. Finally, by \cite[Prop. II.2.5(ii)]{fargues2024geometrization}, the sheaf $R^1\tau_*(\V\otimes \Os_{X_{S}})$ vanishes, so that the above sequence is also right-exact. Hence $\BC(\Es)$ is an effective Banach--Colmez space. Finally, we have $H^j(X_S,i_*E) = H^j(S,E)$, so that also $R^j\tau_*(i_*E)=0$, for all $j>0$. Since vector bundles on $X_S$ only have cohomology in range $[0,1]$, see the beggining of \cite[§II.2]{fargues2024geometrization}, we conclude that $R^j\tau_*\Es=0$, for all $j>0$, as required.
\item We use the following result of Anschütz--Le Bras \cite[Cor. 3.10]{anschutz2021fourier}, which says that
    \[ R\tau_*\colon \Perf(X_{S}) \rightarrow D(S_v,\Q_p)\]
    is fully faithful. In other words, we have a quasi-isomorphism
    \[R\tau_* R\underline{\Hom}_{\Os_{X_{S}}}(\Es,\Es') \simeq R\underline{\Hom}_{S}(R\tau_*\Es,R\tau_*\Es').\]
    The result follows by looking at the degree zero part of the above quasi-isomorphism, using that, by the first point,
    \[ R\tau_*\Es = \BC(\Es)\]
    and similarly for $\Es'$.
    \end{enumerate}
\end{proof}

We now reach the following.
\begin{proposition}\label{prop: minuscule sheaves vs minuscule BC spaces}
    Let $S$ be a perfectoid space over $\Q_p$. The functor $\BC(-)$ defines an exact equivalence of categories
    \begin{align*}
         \{\,\text{effective sheaves on } X_S\,\}  \xrightarrow{\cong} \{\,\text{effective Banach--Colmez spaces } \Fs\rightarrow S\,\}.
     \end{align*}
\end{proposition}
\begin{proof}
    The functor is fully faithful and exact, by Lemma \ref{lemma: tau is fully faithful in most cases}, hence it only remains to show essential surjectivity. Given an effective Banach--Colmez space $\V \rightarrow \Fs \rightarrow E$, we have a commutative diagram with exact rows, by Theorem \ref{thm: a fargues thm for BC spaces}
    % https://q.uiver.app/#q=WzAsMTAsWzIsMCwiXFxGcyJdLFszLDAsIkUiXSxbMiwxLCJcXFYoLTEpXFxvdGltZXMgKFxcQl97XFxjcnlzfV4rKV57XFx2YXJwaGk9cH0iXSxbMywxLCJcXFYoLTEpXFxvdGltZXMgXFxPc19TIl0sWzEsMCwiXFxWIl0sWzAsMCwiMCJdLFs0LDAsIjAiXSxbMSwxLCJcXFYoLTEpXFxvdGltZXMgXFxRKDEpIl0sWzQsMSwiMC4iXSxbMCwxLCIwIl0sWzAsMiwidSJdLFsxLDMsImYiXSxbMiwzXSxbMCwxXSxbNCwwXSxbNSw0XSxbMSw2XSxbMyw4XSxbNywyXSxbNCw3XSxbOSw3XV0=
\[\begin{tikzcd}
	0 & \V & \Fs & E & 0 \\
	0 & {\V(-1)\otimes \Q(1)} & {\V(-1)\otimes (\B_{\crys}^+)^{\varphi=p}} & {\V(-1)\otimes \Os_S} & {0.}
	\arrow[from=1-1, to=1-2]
	\arrow[from=1-2, to=1-3]
	\arrow[from=1-2, to=2-2]
	\arrow[from=1-3, to=1-4]
	\arrow["u", from=1-3, to=2-3]
	\arrow[from=1-4, to=1-5]
	\arrow["f", from=1-4, to=2-4]
	\arrow[from=2-1, to=2-2]
	\arrow[from=2-2, to=2-3]
	\arrow[from=2-3, to=2-4]
	\arrow[from=2-4, to=2-5]
\end{tikzcd}\]
We now define a sheaf $\Es$ on $X_{S}$ as the following pullback
% https://q.uiver.app/#q=WzAsNCxbMCwwLCJcXEVzIl0sWzEsMCwiaV8qRSJdLFswLDEsIlxcVigtMSlcXG90aW1lcyBcXE9zX3tYX3tTfX0oMSkiXSxbMSwxLCJcXFYoLTEpXFxvdGltZXMgaV8qXFxPc19TLCJdLFswLDJdLFsxLDMsImlfKmYiXSxbMiwzXSxbMCwxXV0=
\[\begin{tikzcd}
	\Es & {i_*E} \\
	{\V(-1)\otimes \Os_{X_{S}}(1)} & {\V(-1)\otimes i_*\Os_S,}
	\arrow[from=1-1, to=1-2]
	\arrow[from=1-1, to=2-1]
	\arrow["{i_*f}", from=1-2, to=2-2]
	\arrow[from=2-1, to=2-2]
\end{tikzcd}\]
where we use the projection formula to write
\[ i_*(\V(-1)\otimes_{\Q_p} \Os_S) = \V(-1)\otimes_{\Q_p}i_*\Os_S. \]
It then follows from the exactness of $\BC(-)$ on effective sheaves that we have an isomorphism $\BC(\Es)=\Fs$.
\end{proof}

\begin{proposition}\label{prop: G dualizable iff E(G) vb}
    Let $S$ be a perfectoid space over $\Q_p$ and let $\Es$ be an effective sheaf on $X_S$. Then the associated Banach--Colmez space $\Fs=\BC(\Es)$ is dualizable in the sense of Definition \ref{def: dualizable BC spaces} if and only if $\Es$ is a vector bundle. In that case, the Cartier dual is
     \[ \Fs^D = \BC(\Es^{\vee}\otimes_{\Os_{X_S}}\Os_{X_S}(1)).\]
\end{proposition}
\begin{remark}
    We compare our Cartier duality with \cite{anschutz2021fourier}, where Anschütz--Le Bras define a category of “very nice stacks” $\Fs$ of $\Q_p$-vector spaces and define a dual 
    \[ \Fs^{\star} = \underline{\RHom}_{S_v}(\Fs,[S/\underline{\Q}_p]). \] 
    The relation with the Fargues--Fontaine curve is as follows: If $\Fs \rightarrow S$ is a dualizable effective Banach--Colmez space, with corresponding vector bundle $\Es$ on the curve $X_S$, then by \cite[Cor. 3.10]{anschutz2021fourier},
    \[ \Fs^{\star} = R\tau_*(\Es^{\vee})[1].\]
\end{remark}
\begin{proof}
Let us set $\Fs=\BC(\Es)$. By the proof of Proposition \ref{prop: minuscule sheaves vs minuscule BC spaces}, the sheaf $\Es$ fits in a commutative diagram
% https://q.uiver.app/#q=WzAsMTAsWzAsMCwiMCJdLFsxLDAsIlxcVlxcb3RpbWVzIFxcT3Nfe1hfU30iXSxbMiwwLCJcXEVzIl0sWzMsMCwiaV8qRSJdLFszLDEsImlfKlxcVigtMSlcXG90aW1lcyBcXE9zX1MiXSxbNCwwLCIwIl0sWzQsMSwiMCwiXSxbMiwxLCJcXFYoLTEpXFxvdGltZXNcXE9zX3tYX1N9KDEpIl0sWzAsMSwiMCJdLFsxLDEsIlxcVigtMSlcXG90aW1lc1xcUV9wKDEpIFxcb3RpbWVzIFxcT3Nfe1hfU30iXSxbMCwxXSxbMSwyXSxbMiwzXSxbMyw0LCJpXypmIl0sWzMsNV0sWzQsNl0sWzcsNF0sWzIsNywidSIsMl0sWzEsOSwiPSIsMl0sWzgsOV0sWzksN11d
\[\begin{tikzcd}
	0 & {\V\otimes \Os_{X_S}} & \Es & {i_*E} & 0 \\
	0 & {\V(-1)\otimes\Q_p(1) \otimes \Os_{X_S}} & {\V(-1)\otimes\Os_{X_S}(1)} & {i_*\V(-1)\otimes \Os_S} & {0,}
	\arrow[from=1-1, to=1-2]
	\arrow[from=1-2, to=1-3]
	\arrow["{=}"', from=1-2, to=2-2]
	\arrow[from=1-3, to=1-4]
	\arrow["u"', from=1-3, to=2-3]
	\arrow[from=1-4, to=1-5]
	\arrow["{i_*f}", from=1-4, to=2-4]
	\arrow[from=2-1, to=2-2]
	\arrow[from=2-2, to=2-3]
	\arrow[from=2-3, to=2-4]
	\arrow[from=2-4, to=2-5]
\end{tikzcd}\]
where $\V$ is a $\Q_p$-local system, $E$ is a vector bundle on $S$ and $f=f_{\Fs}$ is the map of Theorem \ref{thm: a fargues thm for BC spaces}. By Beauville--Laszlo glueing (\ref{eq: from modif to lattices}) and Lemma \ref{lemma: minuscule lattices vs flags}, we obtain that $\Es$ is a vector bundle if and only if the map of $S$-vector bundles $f\colon E \rightarrow \V(-1) \otimes \Os_S$ is a locally direct summand. In that case, by the Snake lemma, the map $u$ is injective with cokernel $i_*\omega(-1)$. By Lemma \ref{lemma: dual of vector bundles and shtukas} and its proof, the vector bundle $\Es^{\vee}\otimes \Os(1)$ fits in a minuscule modification with kernel $\V^{\vee}(1)\otimes \Os_X$ and cokernel $i_*\omega^{\vee}$, which embedds in a commutative diagram
% https://q.uiver.app/#q=WzAsMTAsWzEsMCwiXFxWXntcXHZlZX0oMSlcXG90aW1lcyBcXE9zX3tYX1N9Il0sWzIsMCwiXFxFc157XFx2ZWV9XFxvdGltZXMgXFxPc197WF9TfSgxKSJdLFszLDAsImlfKlxcb21lZ2Fee1xcdmVlfSJdLFswLDAsIjAiXSxbNCwwLCIwIl0sWzMsMSwiXFxWXntcXHZlZX1cXG90aW1lcyBpXypcXE9zX1MiXSxbMiwxLCJcXFZee1xcdmVlfVxcb3RpbWVzIFxcT3Nfe1hfU30oMSkiXSxbMSwxLCJcXFZee1xcdmVlfVxcb3RpbWVzIFxcUV9wKDEpXFxvdGltZXMgXFxPc197WF9TfSJdLFswLDEsIjAiXSxbNCwxLCIwLiJdLFszLDBdLFsxLDJdLFswLDEsIlxcYWxwaGEnIl0sWzIsNF0sWzIsNSwiZ157XFx2ZWV9KC0xKSJdLFsxLDZdLFs2LDVdLFs3LDZdLFswLDcsIj0iXSxbOCw3XSxbNSw5XV0=
\[\begin{tikzcd}
	0 & {\V^{\vee}(1)\otimes \Os_{X_S}} & {\Es^{\vee}\otimes \Os_{X_S}(1)} & {i_*\omega^{\vee}} & 0 \\
	0 & {\V^{\vee}\otimes \Q_p(1)\otimes \Os_{X_S}} & {\V^{\vee}\otimes \Os_{X_S}(1)} & {\V^{\vee}\otimes i_*\Os_S} & {0.}
	\arrow[from=1-1, to=1-2]
	\arrow["{\alpha'}", from=1-2, to=1-3]
	\arrow["{=}", from=1-2, to=2-2]
	\arrow[from=1-3, to=1-4]
	\arrow[from=1-3, to=2-3]
	\arrow[from=1-4, to=1-5]
	\arrow["{g^{\vee}(-1)}", from=1-4, to=2-4]
	\arrow[from=2-1, to=2-2]
	\arrow[from=2-2, to=2-3]
	\arrow[from=2-3, to=2-4]
	\arrow[from=2-4, to=2-5]
\end{tikzcd}\]
Applying the functor $\BC(-)$ and comparing with the defining diagram for the Cartier dual $\Fs^D$, we find that $\BC(\Es^{\vee}\otimes \Os_{X_S}(1))=\Fs^D$, as required.
\end{proof}

We now prove Theorem \ref{thm: vb on FF associated to an pdiv groups}.
\begin{proof}[Proof of Theorem \ref{thm: vb on FF associated to an pdiv groups}]
Given an analytic $p$-divisible group $\Gs$, its universal cover $\widetilde{\Gs}$ is an effective Banach--Colmez space, coming with the exact sequence $V_p\Gs \rightarrow \widetilde{\Gs} \rightarrow \Lie(\Gs)$. By Proposition \ref{prop: minuscule sheaves vs minuscule BC spaces}, there exists a canonically defined effective sheaf $\Es(\Gs)$, together with an exact sequence $ V_p\Gs \otimes \Os_{X_S} \rightarrow \Es(\Gs) \rightarrow i_*\Lie(\Gs)$ that recovers the above sequence upon applying the functor $\BC(-)$. Explicitely, the proof of Proposition \ref{prop: minuscule sheaves vs minuscule BC spaces} shows that $\Es(\Gs)$ is the fiber product
% https://q.uiver.app/#q=WzAsNCxbMCwwLCJcXEVzKFxcR3MpIl0sWzEsMCwiaV8qXFxMaWUoXFxHcykiXSxbMSwxLCJUX3BcXEdzKC0xKVxcb3RpbWVzX3tcXFpfcH0gaV8qXFxPc19TLCJdLFswLDEsIlRfcFxcR3MoLTEpXFxvdGltZXNfe1xcWl9wfSBcXE9zX3tYX1N9KDEpIl0sWzAsMV0sWzEsMiwiaV8qZl97XFxHc30iXSxbMywyXSxbMCwzXV0=
\[\begin{tikzcd}
	{\Es(\Gs)} & {i_*\Lie(\Gs)} \\
	{T_p\Gs(-1)\otimes_{\Z_p} \Os_{X_S}(1)} & {T_p\Gs(-1)\otimes_{\Z_p} i_*\Os_S,}
	\arrow[from=1-1, to=1-2]
	\arrow[from=1-1, to=2-1]
	\arrow["{i_*f_{\Gs}}", from=1-2, to=2-2]
	\arrow[from=2-1, to=2-2]
\end{tikzcd}\]
where $f_{\Gs}$ is the map associated with $\Gs$ from Definition \ref{def: u and f}. This proves the first two points. The third point follows from Proposition \ref{prop: G dualizable iff E(G) vb} below. Finally, if $G$ is a $p$-divisible group over $\Os_C$, the vector bundle $\Es(G)$ of Scholze--Weinstein agrees with our construction of $
\Es(G_{\eta})$, by the explicit formula  above Proposition $6.2.2$ in \cite{scholze2013moduli}.
\end{proof}

We now show that a polarization on an analytic $p$-divisible group $\Gs$ induces a pairing on the vector bundle $\Es(\Gs)$.

\begin{lemma}\label{lemma: pairings on vb induced by polarization}
    Let $S$ be a perfectoid space and let $\Gs \rightarrow S$ be a dualizable analytic $p$-divisible group. Let $\lambda \colon \Gs \rightarrow \Gs^D$ be a quasi-polarization (cf. Definition \ref{def: polarization}). Then $\lambda$ induces a perfect pairing of vector bundles
    \[ e\colon \Es(\Gs) \times \Es(\Gs) \rightarrow \Os_{X_S}(1).\]
    Moreover, the map
    \[ \alpha\colon V_p\Gs \otimes \Os_{X_S} \rightarrow \Es(\Gs)\]
    from the modification (\ref{eq: modification defining E(G)}) is compatible with the pairings on $V_p\Gs$ and $\Es(\Gs)$, as in the following diagram
    % https://q.uiver.app/#q=WzAsNCxbMCwwLCIoVl9wXFxHcyBcXG90aW1lcyBcXE9zX3tYX1N9KSBcXHRpbWVzIChWX3BcXEdzIFxcb3RpbWVzIFxcT3Nfe1hfU30pIl0sWzEsMCwiXFxRX3AoMSlcXG90aW1lcyBcXE9zX3tYX1N9Il0sWzAsMSwiXFxFcyhcXEdzKVxcdGltZXMgXFxFcyhcXEdzKSJdLFsxLDEsIlxcT3Nfe1hfU30oMSkuIl0sWzAsMV0sWzAsMiwiXFxhbHBoYSBcXHRpbWVzIFxcYWxwaGEiLDIseyJzdHlsZSI6eyJ0YWlsIjp7Im5hbWUiOiJob29rIiwic2lkZSI6InRvcCJ9fX1dLFsyLDNdLFsxLDMsIiIsMCx7InN0eWxlIjp7InRhaWwiOnsibmFtZSI6Imhvb2siLCJzaWRlIjoidG9wIn19fV1d
\[\begin{tikzcd}
	{(V_p\Gs \otimes \Os_{X_S}) \times (V_p\Gs \otimes \Os_{X_S})} & {\Q_p(1)\otimes \Os_{X_S}} \\
	{\Es(\Gs)\times \Es(\Gs)} & {\Os_{X_S}(1).}
	\arrow[from=1-1, to=1-2]
	\arrow["{\alpha \times \alpha}"', hook, from=1-1, to=2-1]
	\arrow[hook, from=1-2, to=2-2]
	\arrow[from=2-1, to=2-2]
\end{tikzcd}\]
\end{lemma}
\begin{proof}
    The quasi-polarization $\lambda$ is immediately seen to induce a presentation-preserving isomorphism $\widetilde{\Gs} \cong \widetilde{\Gs}^D$ and an isomorphism of modifications
    % https://q.uiver.app/#q=WzAsMTAsWzEsMCwiVl9wXFxHc1xcb3RpbWVzIFxcT3Nfe1hfU30iXSxbMCwwLCIwIl0sWzIsMCwiXFxFcyhcXEdzKSJdLFszLDAsImlfKlxcTGllKFxcR3MpIl0sWzQsMCwiMCJdLFsxLDEsIlZfcFxcR3NeRFxcb3RpbWVzIFxcT3Nfe1hfU30iXSxbMCwxLCIwIl0sWzIsMSwiXFxFcyhcXEdzXkQpIl0sWzMsMSwiaV8qXFxMaWUoXFxHc15EKSJdLFs0LDEsIjAuIl0sWzEsMF0sWzAsMiwiXFxhbHBoYSJdLFsyLDNdLFszLDRdLFswLDUsIlxcY29uZyJdLFs2LDVdLFs1LDcsIlxcYWxwaGFeRCIsMl0sWzIsNywiXFxjb25nIl0sWzMsOCwiXFxjb25nIl0sWzcsOF0sWzgsOV1d
\[\begin{tikzcd}
	0 & {V_p\Gs\otimes \Os_{X_S}} & {\Es(\Gs)} & {i_*\Lie(\Gs)} & 0 \\
	0 & {V_p\Gs^D\otimes \Os_{X_S}} & {\Es(\Gs^D)} & {i_*\Lie(\Gs^D)} & {0.}
	\arrow[from=1-1, to=1-2]
	\arrow["\alpha", from=1-2, to=1-3]
	\arrow["\cong", from=1-2, to=2-2]
	\arrow[from=1-3, to=1-4]
	\arrow["\cong", from=1-3, to=2-3]
	\arrow[from=1-4, to=1-5]
	\arrow["\cong", from=1-4, to=2-4]
	\arrow[from=2-1, to=2-2]
	\arrow["{\alpha^D}"', from=2-2, to=2-3]
	\arrow[from=2-3, to=2-4]
	\arrow[from=2-4, to=2-5]
\end{tikzcd}\]
By Proposition \ref{prop: G dualizable iff E(G) vb} and Lemma \ref{lemma: dual of vector bundles and shtukas}, the map $\alpha^D$ is identified with
% https://q.uiver.app/#q=WzAsMixbMCwwLCJcXHVuZGVybGluZXtcXEhvbX0oVl9wXFxHcygtMSlcXG90aW1lcyBcXE9zX3tYX1N9KDEpLFxcT3Nfe1hfU30oMSkpIl0sWzEsMCwiXFx1bmRlcmxpbmV7XFxIb219KFxcRXMoXFxHcyksXFxPc197WF9TfSgxKSkuIl0sWzAsMSwidV4qIl1d
\[\begin{tikzcd}
	{\underline{\Hom}(V_p\Gs(-1)\otimes \Os_{X_S}(1),\Os_{X_S}(1))} & {\underline{\Hom}(\Es(\Gs),\Os_{X_S}(1)).}
	\arrow["{u^*}", from=1-1, to=1-2]
\end{tikzcd}\]
Here, $u$ is the left vertical map in (\ref{eq: defining diagram for E(G)}). By exchanging variables, the middle vertical map induces a pairing 
\[ e\colon \Es(\Gs) \times \Es(\Gs) \rightarrow \Os_{X_S}(1)\]
satisfying for any local sections $x\in V_p\Gs$ and $y\in \Es(\Gs)$, 
\[ e(\alpha(x),y) = e'(x,u(y)),\]
where 
\[ e'\colon (V_p\Gs \otimes \Os_{X_S}) \times (V_p\Gs(-1) \otimes \Os_{X_S}(1)) \rightarrow \Os_{X_S}(1)\]
is the unique $\Os_{X_S}$-bilinear pairing extending the pairing on $V_p\Gs$. Upon noticing that the composition
% https://q.uiver.app/#q=WzAsMyxbMCwwLCJWX3BcXEdzIFxcb3RpbWVzIFxcT3Nfe1hfU30iXSxbMSwwLCJcXEVzKFxcR3MpIl0sWzIsMCwiVl9wXFxHcygtMSkgXFxvdGltZXMgXFxPc197WF9TfSgxKSJdLFsxLDIsInUiXSxbMCwxLCJcXGFscGhhIl1d
\[\begin{tikzcd}
	{V_p\Gs \otimes \Os_{X_S}} & {\Es(\Gs)} & {V_p\Gs(-1) \otimes \Os_{X_S}(1)}
	\arrow["\alpha", from=1-1, to=1-2]
	\arrow["u", from=1-2, to=1-3]
\end{tikzcd}\]
agrees with the natural inclusion induced by $\Q_p(1)\otimes \Os_{X_S} \hookrightarrow \Os_{X_S}(1)$, we obtain the desired relation.
\end{proof}

Using $v$-descent, the constructions of this subsection may be extended to non-perfectoid bases.
\begin{cor}\label{cor: BdR+ lattice assoc to pdiv groups in general}
    Let $\Gs$ be an analytic $p$-divisible group over a good adic space $S/\Q_p$. Then there exists a functorial $\B_{\dR}^+$-module $\Xi(\Gs)$ on $S_v$ together with a short exact sequence on $S_v$
    % https://q.uiver.app/#q=WzAsNSxbMSwwLCJUX3BcXEdzIFxcb3RpbWVzX3tcXFpfcH1cXEJfe1xcZFJ9XisiXSxbMiwwLCJcXFhpKFxcR3MpIl0sWzMsMCwiXFxMaWUoXFxHcykiXSxbMCwwLCIwIl0sWzQsMCwiMC4iXSxbMywwXSxbMSwyXSxbMCwxXSxbMiw0XV0=
\[\begin{tikzcd}
	0 & {T_p\Gs \otimes_{\Z_p}\B_{\dR}^+} & {\Xi(\Gs)} & {\Lie(\Gs)} & {0.}
	\arrow[from=1-1, to=1-2]
	\arrow[from=1-2, to=1-3]
	\arrow[from=1-3, to=1-4]
	\arrow[from=1-4, to=1-5]
\end{tikzcd}\]
In particular, it induces a natural isomorphism
\begin{align}\label{eq: BdeR lattice of G iso to tate module}  \Xi(\Gs)[\xi^{-1}] \cong T_p\Gs \otimes_{\Z_p} \B_{\dR}.
    \end{align}
The sheaf $\Xi(\Gs)$ is functorial in $\widetilde{\Gs}$. If $\Gs$ is dualizable, $\Xi(\Gs)$ is a finite locally free $\B_{\dR}^+$-module.
\end{cor}
\begin{proof}
    If $S$ is perfectoid, we let $\Xi(\Es) = \Es(\Gs) \otimes_{\Os_{X_{S}}} \B_{\dR}^+$. This construction clearly commutes with pullback along morphisms of perfectoid spaces $S' \rightarrow S$, using that the formation of $\Es(\Gs)$ commutes with pullback. In general, we fix a $v$-cover $S' \rightarrow S$ by a perfectoid space so that $S' \times_S S'$ is again perfectoid. Then $\Xi(\Gs_{S'})$ together with its short exact sequence is naturally endowed with a descent data to $S$, and thus descends to a $\B_{\dR}^+$-module $\Xi(\Gs)$ on $S$ that has the required properties.
\end{proof}

We conclude this subsection with the following explicit formula for the Cartier dual of an effective Banach--Colmez space.
\begin{proposition}\label{prop: explicit formula for cartier dual}
    Let $S$ be a locally spatial diamond over $\Q_p$ and let $\Fs \rightarrow S$ be a dualizable effective Banach--Colmez space. Then there is a natural isomorphism
        \[ \Fs^D = \underline{\Hom}_{S_v}(\Fs,(\B_{\crys}^+)^{\varphi =p}).\]
\end{proposition}
\begin{proof}
When $S$ is perfectoid, if $\Es$ is the corresponding vector bundle on $X_S$, we have, by Proposition \ref{prop: G dualizable iff E(G) vb}
\[ \Fs^D = \BC(\Es^D) = \BC(\underline{\Hom}_{\Os_{X_S}}(\Es,\Os_{X_S}(1)).\]
By Lemma \ref{lemma: tau is fully faithful in most cases}(2), this is equal to
\[ \underline{\Hom}_{S_v}(\BC(\Es),\BC(\Os_{X_S}(1))) = \underline{\Hom}_{S_v}(\Fs,(\B_{\crys}^+)^{\varphi=p}).\]
Since these isomorphisms are natural in $S$, they descend to an isomorphism over an arbitrary base $S$, as required.
\end{proof}

\subsection{Analytic $p$-divisible groups and local shtukas}\label{subsect: Analytic $p$-divisible groups and vector bundles}
We are now ready to define our Dieudonné functor. We start with Theorem \ref{thm: equivalence for dualizable groups on perfd} in the case where the base is perfectoid, and we conclude with Theorem \ref{thm: analytic Dieudonné theory in general} which deals with arbitrary bases.

\begin{thm}\label{thm: equivalence for dualizable groups on perfd}
    Let $S$ be a perfectoid space over $\Q_p$, then there exists canonical exact equivalences of categories between the following:
    \begin{enumerate}
        \item dualizable analytic $p$-divisible groups $\Gs \rightarrow S$,
        \item pairs $(\Lb,E)$, where $\Lb$ is a $\Z_p$-local system on $S$ and $E \hookrightarrow\Lb(-1)\otimes_{\Z_p} \Os_v$ is a locally direct summand,
        \item pairs $(\Lb,\Xi)$ consisting of a $\Z_p$-local system $\Lb$ and a minuscule $\B_{\dR}^+$-lattice $\Lb \otimes_{\Z_p} \B_{\dR}^+ \sub \Xi \sub \xi^{-1}(\Lb \otimes_{\Z_p}\B_{\dR}^+)$,
        \item tuples $(\Lb,\Es,\alpha)$ consisting of a $\Z_p$-local system $\Lb$, a vector bundle $\Es$ on $X_S$ and a minuscule modification $\alpha\colon \Lb\otimes_{\Z_p} \Os_{X_S} \dashrightarrow \Es$, and
        \item minuscule shtukas $(\Ms,\varphi_{\Ms})$ over $\mathcal{Y}_S$ with one leg at $\varphi^{-1}(S)$.
    \end{enumerate}
\end{thm}
\begin{remark}
    The reader is invited to compare the above theorem with Fargues' theorem \cite[Thm. 14.1.1]{SW20}. The main difference is that, in contrast with the case $S=\Spa(C,\Os_C)$, the above objects are no longer equivalent to minuscule Breuil--Kisin--Fargues modules over $\Ainf$. 
\end{remark}
\begin{proof}
    The first two categories are equivalent by Theorem \ref{thm: extending Fargues' equivalence of categories}. Here, we use that $S$ is perfectoid, so that étale and $v$-vector bundles agree. Hence, an analytic $p$-divisible group $\Gs$ is dualizable if and only if the map $f_{\Gs}$ is a locally direct summand. The equivalence between $(2)$ and $(3)$ follow from Lemma \ref{lemma: minuscule lattices vs flags}. It follows from Beauvilles--Laszlo glueing (\ref{eq: from modif to lattices}) that the third and fourth categories are also equivalent. Finally, it follows from Proposition \ref{prop: from modification to shtukas} that the fourth and fifth categories are equivalent. 
\end{proof}

We now prove an analogue of this result over arbitrary good adic spaces. For this, consider the category fibered in groupoid
\[ \Sht\colon S \in \Perf_{\Q_p} \mapsto \{\text{local shtukas } (\Ms,\varphi_{\Ms}) \text{ over } X_S \text{ with a leg at }\varphi^{-1}(S) \}.\]
We let $\Sht_{\min}$ denote the substack consisting of minuscule objects in the sense of Definition \ref{def: minuscule shtuka}. Then by \cite[Prop. 19.5.3]{SW20}, the prestacks $\Sht$ and $\Sht_{\min}$ are $v$-stacks, which can easily seen to be small by arguing as in the proof of \cite[III.1.3]{fargues2024geometrization}. Therefore, it makes sense to evaluate them on arbitrary locally spatial diamonds. 

Let $S$ be a good adic space over $\Q_p$. By Proposition \ref{prop: from modification to shtukas}, given an object $(\Ms,\varphi_{\Ms}) \in \Sht_{\min}(S)$, the rule
\[ T \in \Perf_S \mapsto \varphi_{\Ms_T}(\varphi^*\Ms_T)/\Ms_T\]
defines a $v$-vector bundle $E(\Ms)$ over $S$. We will abbreviate it as $E(\Ms) =\varphi_{\Ms}(\varphi^*\Ms)/\Ms$.
\begin{definition}\label{def: analytic Dieudonné crystal}
    Let $S$ be a good adic space over $\Q_p$. An analytic Dieudonné crystal over $S$ is a minuscule shtuka $(\Ms,\varphi_{\Ms}) \in \Sht_{\min}(S)$ such that the $v$-vector bundles $E(\Ms)$ and $E(\Ms^{\vee} \otimes_{\Os_{\Yss_S}} \Os_{\Yss_S}\{1\})$
        arise from étale vector bundles over $S$ (which are necessarily unique, by (\ref{eq: diamond functor ff on good adic spaces})).
\end{definition}

We obtain the following result.
\begin{thm}\label{thm: analytic Dieudonné theory in general}
    Let $S$ be a good adic space over $\Q_p$, then there is a canonical equivalence of categories 
    \begin{align*}
 \left\{\text{\begin{tabular}{l} {\parbox{4.4cm}{dualizable analytic $p$-divisible groups $\Gs$ over $S$}}\end{tabular}}\right\}
         \xlongrightarrow{\cong} \left\{\text{\begin{tabular}{l} {\parbox{4.1cm}{analytic Dieudonné crystal $(\Ms,\varphi_{\Ms})$ over $S$}}\end{tabular}}\right\}.
    \end{align*}
The functor and its inverse are exact and are compatible with Cartier duality in the following sense:
\[ (\Ms_{\Gs^D},\varphi_{\Ms_{\Gs^D}}) = (\Ms_{\Gs},\varphi_{\Ms_{\Gs}})^{\vee}\otimes_{\Os_{\Yss_S}} \Os_{\Yss_S}\{1\}.\]
\end{thm}
\begin{remark}
    If $S$ is any locally spatial diamond over $\Q_p$, the above theorem and its proof extend without change to classify dualizable analytic $p$-divisible groups over $S$ in the sense of Remark \ref{rmk: p-divisible groups over LSD}.
\end{remark}
\begin{proof}
By Theorem \ref{thm: extending Fargues' equivalence of categories} and the definition of dualizability, it follows that dualizable analytic $p$-divisible groups over $S$ are equivalent to a functorial rule $\Gs$ associating to a perfectoid space $T \rightarrow S$ a dualizable analytic $p$-divisible group $\Gs_T$ over $T$, subject to the condition that the $v$-vector bundles $T \mapsto \Lie(\Gs_T)$ and $T \mapsto \omega_{\Gs_T^D}$ arise from étale vector bundles over $S$. The equivalence then follows directly from Theorem \ref{thm: equivalence for dualizable groups on perfd}. The formulas for the dual follow from Lemma \ref{lemma: dual of vector bundles and shtukas} and (the proof of) 
Proposition \ref{prop: G dualizable iff E(G) vb}.
\end{proof}

\subsection{The de Rham module}
We now define the de Rham module $D(\Gs)$ of an analytic $p$-divisible group $\Gs$, as well as its Hodge filtration.

\begin{definition}\label{def: Dieudonné module}
    Let $S$ be a good adic space over $\Q_p$ and let $\Gs \rightarrow S$ be an analytic $p$-divisible group. We define the de Rham module of $\Gs$ to be the $\Os_{S_v}$-module
    \[ D(\Gs) = \Xi(\Gs)\otimes_{\B_{\dR}^+,\theta} \Os_{S_v},\]
    where $\Xi(\Gs)$ is the $\B_{\dR}^+$-module associated to $\Gs$ (cf. Corollary \ref{cor: BdR+ lattice assoc to pdiv groups in general}).
\end{definition}
It follows by the above corollary that the sheaf $D(\Gs)$ is functorial in $\widetilde{\Gs}$ and that it is a $v$-vector bundle if $\Gs$ is dualizable. We stress that $D(\Gs)$ need not be an étale vector bundle. Later in Proposition \ref{prop: VpG de rham vs MG étale vb}, we show that for a dualizable $S$-group $\Gs$, where $S$ is smooth rigid space $S$ over a $p$-adic field, $D(\Gs)$ is an étale vector bundle if and only if $V_p\Gs$ is a de Rham local system.

Let $S$ be a perfectoid space over $\Q_p$. For any sheaf of $\Os_{X_{S}}$-modules $\Es$ on $X_{S}$, there is a canonical map $\Es \rightarrow i_*i^*\Es$. Here, $i^*$ stands for pullback of $\Os$-modules. If $\Es$ is an effective sheaf (cf. Definition \ref{def: effective sheaf}), this map induces a map between their associated Banach--Colmez spaces, natural in $\Es$
\begin{align}\label{eq: qlog in general} 
\qlog\colon \BC(\Es) \rightarrow \BC(i_*i^*\Es) = i^*\Es.\end{align}
This map is called the \emph{quasi-logarithm}.
% By descent, we obtain a map
% \begin{align}
%     \qlog\colon \Fs \rightarrow \BC(i_*i^*\Es) = i^*\Es.
% \end{align}
% we can glue this construction to arbitrary minuscule Banach---Colmez space $\Fs$ over a locally spatial diamond $S$

\begin{definition}
    Let $\Gs$ be an analytic $p$-divisible group over a good adic space $S/\Q_p$. The quasi-logarithm of $\Gs$ 
\begin{align} 
\qlog_{\Gs}\colon \widetilde{\Gs} \rightarrow D(\Gs),
\end{align}
is the map of $v$-sheaves obtained from (\ref{eq: qlog in general}) applied to the sheaf $\Es=\Es(\Gs)$ and using $v$-descent.
\end{definition}

We now check that the Dieudonné module has the expected properties.

\begin{lemma}\label{lemma: Hodge filtration on Dieudonné module}
Let $\Gs \rightarrow S$ be a dualizable analytic $p$-divisible group. The Dieudonné module carries a natural Hodge filtration
    % https://q.uiver.app/#q=WzAsNSxbMSwwLCJcXG9tZWdhX3tcXEdzXkR9Il0sWzIsMCwiRChcXEdzKSJdLFszLDAsIlxcTGllKFxcR3MpIl0sWzAsMCwiMCJdLFs0LDAsIjAuIl0sWzMsMF0sWzEsMl0sWzAsMV0sWzIsNF1d
\begin{equation}\label{eq: Hodge filtration}\begin{tikzcd}
	0 & {\omega_{\Gs^D}} & {D(\Gs)} & {\Lie(\Gs)} & {0.}
	\arrow[from=1-1, to=1-2]
	\arrow[from=1-2, to=1-3]
	\arrow[from=1-3, to=1-4]
	\arrow[from=1-4, to=1-5]
\end{tikzcd}\end{equation}
It fits in a commutative diagram of $v$-sheaves with exact rows
% https://q.uiver.app/#q=WzAsMTAsWzEsMCwiVl9wXFxHcyJdLFsyLDAsIlxcd2lkZXRpbGRle1xcR3N9Il0sWzMsMCwiXFxMaWUoXFxHcykiXSxbMCwwLCIwIl0sWzQsMCwiMCJdLFsxLDEsIlxcb21lZ2Ffe1xcR3NeRH0iXSxbMiwxLCJEKFxcR3MpIl0sWzMsMSwiXFxMaWUoXFxHcykiXSxbNCwxLCIwLiJdLFswLDEsIjAiXSxbMywwXSxbMSwyXSxbMCwxXSxbMiw0XSxbMSw2LCJcXHFsb2ciXSxbMCw1XSxbNSw2XSxbNiw3XSxbMiw3LCI9Il0sWzcsOF0sWzksNV1d
\begin{equation}\label{eq: hodge filtration and quasilogarithm}\begin{tikzcd}
	0 & {V_p\Gs} & {\widetilde{\Gs}} & {\Lie(\Gs)} & 0 \\
	0 & {\omega_{\Gs^D}} & {D(\Gs)} & {\Lie(\Gs)} & {0.}
	\arrow[from=1-1, to=1-2]
	\arrow[from=1-2, to=1-3]
	\arrow[from=1-2, to=2-2]
	\arrow[from=1-3, to=1-4]
	\arrow["\qlog", from=1-3, to=2-3]
	\arrow[from=1-4, to=1-5]
	\arrow["{=}", from=1-4, to=2-4]
	\arrow[from=2-1, to=2-2]
	\arrow[from=2-2, to=2-3]
	\arrow[from=2-3, to=2-4]
	\arrow[from=2-4, to=2-5]
\end{tikzcd}\end{equation}
Here, the left vertical map is the composition
\[ V_p\Gs \rightarrow V_p\Gs \otimes_{\Q_p} \Os_S\xrightarrow{f_{\Gs^D}^{\vee}} \omega_{\Gs^D},\]
where $f_{\Gs^D}$ is the map associated to $\Gs^D$ under Theorem \ref{thm: extending Fargues' equivalence of categories}.
\end{lemma}
\begin{proof}
By naturality of the quasi-logarithm, we have a commutative diagram with exact top row and right-exact bottom row
% https://q.uiver.app/#q=WzAsOSxbMSwwLCJWX3BcXEdzIl0sWzIsMCwiXFx3aWRldGlsZGV7XFxHc30iXSxbMywwLCJcXExpZShcXEdzKSJdLFswLDAsIjAiXSxbNCwwLCIwIl0sWzEsMSwiVl9wXFxHc1xcb3RpbWVzX3tcXFFfcH1cXEdfYSJdLFsyLDEsIkQoXFxHcykiXSxbMywxLCJcXExpZShcXEdzKSJdLFs0LDEsIjAuIl0sWzMsMF0sWzEsMl0sWzAsMV0sWzIsNF0sWzEsNiwiXFxxbG9nIl0sWzAsNSwiXFxxbG9nIl0sWzUsNl0sWzYsN10sWzIsNywiXFxxbG9nPVxcaWQiXSxbNyw4XV0=
\[\begin{tikzcd}
	0 & {V_p\Gs} & {\widetilde{\Gs}} & {\Lie(\Gs)} & 0 \\
	& {V_p\Gs\otimes_{\Q_p}\G_a} & {D(\Gs)} & {\Lie(\Gs)} & {0.}
	\arrow[from=1-1, to=1-2]
	\arrow[from=1-2, to=1-3]
	\arrow["\qlog", from=1-2, to=2-2]
	\arrow[from=1-3, to=1-4]
	\arrow["\qlog", from=1-3, to=2-3]
	\arrow[from=1-4, to=1-5]
	\arrow["{\qlog=\id}", from=1-4, to=2-4]
	\arrow[from=2-2, to=2-3]
	\arrow[from=2-3, to=2-4]
	\arrow[from=2-4, to=2-5]
\end{tikzcd}\]
We claim that the bottom left horizontal map factors as
% https://q.uiver.app/#q=WzAsMyxbMCwwLCJWX3BcXEdzXFxvdGltZXNfe1xcUV9wfVxcR19hIl0sWzEsMCwiRChcXEdzKSwiXSxbMCwxLCJcXG9tZWdhX3tcXEdzXkR9Il0sWzAsMV0sWzAsMiwiZl97XFxHc15EfV57XFx2ZWV9IiwyXSxbMiwxLCIiLDIseyJzdHlsZSI6eyJib2R5Ijp7Im5hbWUiOiJkYXNoZWQifX19XV0=
\[\begin{tikzcd}
	{V_p\Gs\otimes_{\Q_p}\G_a} & {D(\Gs),} \\
	{\omega_{\Gs^D}}
	\arrow[from=1-1, to=1-2]
	\arrow["{f_{\Gs^D}^{\vee}}"', from=1-1, to=2-1]
	\arrow[dashed, from=2-1, to=1-2]
\end{tikzcd}\]
where $f_{\Gs^D}$ is the map naturally associated with $\Gs^D$ (see Definition \ref{def: dualizable}). As $\rank D(\Gs) = \rank(\Es(\Gs)) = \rank(T_p\Gs)$, it will follow that $\omega_{\Gs^D} \rightarrow \Ker(D(\Gs) \rightarrow \Lie(\Gs))$ is a surjection between vector bundles of the same rank and thus an isomorphism. 

The claim is local, so we may assume that $S$ is perfectoid. We use Theorem \ref{thm: vb on FF associated to an pdiv groups}(3) to write the occurring objects in terms of the dual group $\Gs^D$. This yields
\begin{align*}
    D(\Gs) &= i^*\underline{\Hom}_{\Os_{X_{S}}}(\Es(\Gs^D),\Os_{X_{S}}(1))\\
    &= \underline{\Hom}_{\Os_S}(i^*\Es(\Gs^D),\underset{=\Os_S}{\underbrace{i^*\Os_{X_{S}}(1)}}) \\
    &= i^*\underline{\Hom}_{\Os_{X_{S}}}(\Es(\Gs^D),i_*\Os_S),
\end{align*}
and similarly
\[ \omega_{\Gs^D} = \underline{\Hom}_{\Os_S}(\Lie(\Gs^D),\Os_S) = i^*\underline{\Hom}_{\Os_{X_{S}}}(i_*\Lie(\Gs^D),i_*\Os_S)\]
Applying $\underline{\Hom}(-, i_*\Os_S)$ to the defining diagram (\ref{eq: defining diagram for E(G)}) for $\Es(\Gs^D)$ and abbreviating $V_p\Gs^D$ by $V^D$, we obtain the following commutative diagram with left exact rows
% https://q.uiver.app/#q=WzAsNCxbMCwwLCJcXHVuZGVybGluZXtcXEhvbX0oaV8qXFxMaWUoXFxHc15EKSwgaV8qXFxPc19TKSJdLFsxLDAsIlxcdW5kZXJsaW5le1xcSG9tfShcXEVzKFxcR3NeRCksIGlfKlxcT3NfUykiXSxbMCwxLCJcXHVuZGVybGluZXtcXEhvbX0oVl5EKC0xKVxcb3RpbWVzIGlfKlxcR19hLCBpXypcXE9zX1MpIl0sWzEsMSwiXFx1bmRlcmxpbmV7XFxIb219KFZeRCgtMSlcXG90aW1lcyBcXE9zX3tYX3tTfX0oMSksIGlfKlxcT3NfUykuIl0sWzIsMywiXFxjb25nIiwyXSxbMiwwLCJmXntcXHZlZX0iXSxbMywxLCJ1XntcXHZlZX0iLDJdLFswLDFdXQ==
\[\begin{tikzcd}
	{\underline{\Hom}(i_*\Lie(\Gs^D), i_*\Os_S)} & {\underline{\Hom}(\Es(\Gs^D), i_*\Os_S)} \\
	{\underline{\Hom}(V^D(-1)\otimes i_*\G_a, i_*\Os_S)} & {\underline{\Hom}(V^D(-1)\otimes \Os_{X_{S}}(1), i_*\Os_S).}
	\arrow[from=1-1, to=1-2]
	\arrow["{f^{\vee}}", from=2-1, to=1-1]
	\arrow["\cong"', from=2-1, to=2-2]
	\arrow["{u^{\vee}}"', from=2-2, to=1-2]
\end{tikzcd}\]
The bottom horizontal map is easily seen to be an isomorphism, using that $i^*\Os_{X_S}(1)=\Os_S$. After applying $i^*$, we obtain the commutative diagram
% https://q.uiver.app/#q=WzAsNCxbMCwwLCJcXG9tZWdhX3tcXEdzXkR9Il0sWzEsMCwiRChcXEdzKSJdLFswLDEsIlZfcFxcR3NcXG90aW1lcyBcXE9zX1MiXSxbMSwxLCJWX3BcXEdzXFxvdGltZXMgXFxPc19TLCJdLFswLDFdLFsyLDAsImZfe1xcR3NeRH1ee1xcdmVlfSJdLFsyLDMsIlxcY29uZyIsMl0sWzMsMV1d
\[\begin{tikzcd}
	{\omega_{\Gs^D}} & {D(\Gs)} \\
	{V_p\Gs\otimes \Os_S} & {V_p\Gs\otimes \Os_S,}
	\arrow[from=1-1, to=1-2]
	\arrow["{f_{\Gs^D}^{\vee}}", from=2-1, to=1-1]
	\arrow["\cong"', from=2-1, to=2-2]
	\arrow[from=2-2, to=1-2]
\end{tikzcd}\]
which yields the required factorization.
\end{proof}

Next, we give an explicit formula for the de Rham module.
\begin{proposition}\label{prop: simpler description of Dieudonné module}
Let $\Gs$ be a dualizable analytic $p$-divisible group over a good adic space $S$. We have a natural isomorphism of $v$-sheaves
    \[ D(\Gs) = \underline{\Hom}_S(\widetilde{\Gs^D},\G_a).\]
    Under this isomorphism, the Hodge filtration is identified with
\begin{align}
    \omega_{\Gs^D}= \underline{\Hom}_S(\Lie(\Gs^D),\G_a) \hookrightarrow  \underline{\Hom}_S(\widetilde{\Gs^D},\G_a)
\end{align}
arising from the surjection $\widetilde{\Gs^D} \twoheadrightarrow\Lie(\Gs^D)$.
\end{proposition}
\begin{proof}
    We work $v$-locally and assume that $S$ is perfectoid. Let us write
    \[D(\Gs) = i^*\underline{\Hom}_{\Os_{X_{S}}}(\Es(\Gs^D),i_*\G_a) = \tau_*\underline{\Hom}_{\Os_{X_{S}}}(\Es(\Gs^D),i_*\G_a), \]
    where the first equality was shown in the proof of Lemma \ref{lemma: Hodge filtration on Dieudonné module} and the second comes from the fact that the $\underline{\Hom}$-sheaf under consideration is supported on $i(S)$. Hence the result follows from Lemma \ref{lemma: tau is fully faithful in most cases}(2).
\end{proof}
\begin{remark}
    When $S=\Spa(C)$, Proposition \ref{prop: simpler description of Dieudonné module} together with \cite[Thm. 3.3]{Farg22} imply that
    \[ D(\Gs) = \QHom(\Gs^D,\G_a),\]
    where the right-hand side is the group of quasi-logarithms of $\Gs^D$, see \cite[§8.2]{Farg19}.
\end{remark}
\subsection{Comparison with integral Dieudonné theory}
We compare our analytic Dieudonné theory with various integral constructions. 

\begin{proposition}\label{prop: comparison with integral variant}
Let $S/\Q_p$ be a good adic space and let $\Gs \rightarrow S$ be a dualizable analytic $p$-divisible group. Assume that $\Gs$ has good reduction, i.e. there exists a formal model $\Ss$ for $S$ flat over $\Z_p$ that locally admits a finitely generated ideal of definition, and a $p$-divisible group $\mathfrak{G}$ over $\Ss$ with adic generic fiber $\Gs$. Then there is a natural isomorphism
\[ D(\Gs) = D(\mathfrak{G})\otimes_{\Os_{\Ss}} \Os_S,\]
compatible with the Hodge filtrations. Here, $D(\mathfrak{G})$ is the Dieudonné crystal of $\mathfrak{G}$, i.e. the Lie algebra of the universal vector extension $\omega_{\mathfrak{G}^D} \rightarrow E(\mathfrak{G}) \rightarrow \mathfrak{G}$ \cite{Messing1972}.
\end{proposition}
\begin{proof}
There exists a canonical map $s\colon \widetilde{\mathfrak{G}} \rightarrow E(\mathfrak{G})$ \cite[§3.2]{scholze2013moduli}, realizing the universal vector extension as the pushforward of the $p$-adic universal cover $T_p\mathfrak{G} \rightarrow \widetilde{\mathfrak{G}} \rightarrow \mathfrak{G}$, so that the following diagram commutes
% https://q.uiver.app/#q=WzAsMTAsWzAsMSwiMCJdLFsxLDEsIlxcb21lZ2Ffe1xcbWF0aGZyYWt7R31eRH0iXSxbMiwxLCJFKFxcbWF0aGZyYWt7R30pIl0sWzMsMSwiXFxtYXRoZnJha3tHfSJdLFs0LDEsIjAuIl0sWzMsMCwiXFxtYXRoZnJha3tHfSJdLFsyLDAsIlxcd2lkZXRpbGRle1xcbWF0aGZyYWt7R319Il0sWzQsMCwiMCJdLFsxLDAsIlRfcFxcbWF0aGZyYWt7R30iXSxbMCwwLCIwIl0sWzAsMV0sWzEsMl0sWzIsM10sWzMsNF0sWzUsMywiPSJdLFs2LDIsInMiXSxbNiw1XSxbNSw3XSxbOSw4XSxbOCwxLCJcXGFscGhhX3tcXG1hdGhmcmFre0d9fSIsMl0sWzgsNl1d
\[\begin{tikzcd}
	0 & {T_p\mathfrak{G}} & {\widetilde{\mathfrak{G}}} & {\mathfrak{G}} & 0 \\
	0 & {\omega_{\mathfrak{G}^D}} & {E(\mathfrak{G})} & {\mathfrak{G}} & {0.}
	\arrow[from=1-1, to=1-2]
	\arrow[from=1-2, to=1-3]
	\arrow["{\alpha_{\mathfrak{G}}}"', from=1-2, to=2-2]
	\arrow[from=1-3, to=1-4]
	\arrow["s", from=1-3, to=2-3]
	\arrow[from=1-4, to=1-5]
	\arrow["{=}", from=1-4, to=2-4]
	\arrow[from=2-1, to=2-2]
	\arrow[from=2-2, to=2-3]
	\arrow[from=2-3, to=2-4]
	\arrow[from=2-4, to=2-5]
\end{tikzcd}\]
Here, $\alpha_{\mathfrak{G}}$ is the Hodge--Tate map of $\mathfrak{G}$, see Definition \ref{def: HT map of fargues}. We now apply the generic fiber functor $(-)_{\eta}$ of Scholze--Weinstein \cite[§2.2]{scholze2013moduli}. Let $\qlog_{\mathfrak{G}}$ denote the composition of $s_{\eta}\colon \widetilde{\mathfrak{G}}_{\eta} \rightarrow E(\mathfrak{G})_{\eta}$ with the logarithm map $\log\colon E(\mathfrak{G})_{\eta} \rightarrow D(\mathfrak{G})\otimes_{\Os_{\Ss}} \Os_S$. Then we obtain the following commutative diagram \cite[Lemma 3.2.5]{scholze2013moduli}, which realizes the bottom sequence as the pushout of the top sequence along the map $\alpha_{\mathfrak{G}}[\tfrac{1}{p}]$
% https://q.uiver.app/#q=WzAsMTAsWzAsMSwiMCJdLFsxLDEsIlxcb21lZ2Ffe1xcbWF0aGZyYWt7R31eRH1cXG90aW1lc197XFxPc197XFxTc319XFxPc19TIl0sWzIsMSwiRChcXG1hdGhmcmFre0d9KVxcb3RpbWVzX3tcXE9zX3tcXFNzfX1cXE9zX1MiXSxbMywxLCJcXExpZShcXG1hdGhmcmFre0d9KVxcb3RpbWVzX3tcXE9zX3tcXFNzfX1cXE9zX1MiXSxbNCwxLCIwLiJdLFszLDAsIlxcTGllKFxcbWF0aGZyYWt7R30pXFxvdGltZXNfe1xcT3Nfe1xcU3N9fVxcT3NfUyJdLFsyLDAsIlxcd2lkZXRpbGRle1xcbWF0aGZyYWt7R319X3tcXGV0YX0iXSxbNCwwLCIwIl0sWzEsMCwiVl9wXFxtYXRoZnJha3tHfV97XFxldGF9Il0sWzAsMCwiMCJdLFswLDFdLFsxLDJdLFsyLDNdLFszLDRdLFs1LDMsIj0iXSxbNiwyLCJcXHFsb2dfe1xcbWF0aGZyYWt7R319Il0sWzYsNV0sWzUsN10sWzgsNl0sWzksOF0sWzgsMSwiXFxhbHBoYV97XFxtYXRoZnJha3tHfX1bXFx0ZnJhY3sxfXtwfV0iLDJdXQ==
\[\begin{tikzcd}
	0 & {V_p\mathfrak{G}_{\eta}} & {\widetilde{\mathfrak{G}}_{\eta}} & {\Lie(\mathfrak{G})\otimes_{\Os_{\Ss}}\Os_S} & 0 \\
	0 & {\omega_{\mathfrak{G}^D}\otimes_{\Os_{\Ss}}\Os_S} & {D(\mathfrak{G})\otimes_{\Os_{\Ss}}\Os_S} & {\Lie(\mathfrak{G})\otimes_{\Os_{\Ss}}\Os_S} & {0.}
	\arrow[from=1-1, to=1-2]
	\arrow[from=1-2, to=1-3]
	\arrow["{\alpha_{\mathfrak{G}}[\tfrac{1}{p}]}"', from=1-2, to=2-2]
	\arrow[from=1-3, to=1-4]
	\arrow["{\qlog_{\mathfrak{G}}}", from=1-3, to=2-3]
	\arrow[from=1-4, to=1-5]
	\arrow["{=}", from=1-4, to=2-4]
	\arrow[from=2-1, to=2-2]
	\arrow[from=2-2, to=2-3]
	\arrow[from=2-3, to=2-4]
	\arrow[from=2-4, to=2-5]
\end{tikzcd}\]
By Proposition \ref{prop: generic fibers of log p-divisible groups}, the map $\alpha_{\mathfrak{G}}[\tfrac{1}{p}]$ coincides with $f_{\Gs^D}^{\vee}$. Comparing the above diagram to the pushout diagram (\ref{eq: hodge filtration and quasilogarithm}), we conclude that $D(\Gs)=D(\mathfrak{G}) \otimes_{\Os_{\Ss}}\Os_S$, as required.
\end{proof}

\begin{remark}\label{rmk: Universal vector extension for pdivisible groups}
    In the above proof, we are naturally led to consider the following analytic variant of the universal vector extension: For $\Gs$ a dualizable $p$-divisible group over $S$, define $E(\Gs)$ as the following pushout in abelian sheaves on $S_v$
    % https://q.uiver.app/#q=WzAsMTAsWzEsMCwiVF9wXFxHcyJdLFsyLDAsIlxcd2lkZXRpbGRle1xcR3N9Il0sWzMsMCwiXFxHcyJdLFswLDAsIjAiXSxbNCwwLCIwIl0sWzEsMSwiXFxvbWVnYV97XFxHc15EfSJdLFsyLDEsIkUoXFxHcykiXSxbMywxLCJcXEdzIl0sWzQsMSwiMC4iXSxbMCwxLCIwIl0sWzMsMF0sWzEsMl0sWzAsMV0sWzIsNF0sWzAsNSwiZl97XFxHc15EfV57XFx2ZWV9IiwyXSxbMSw2XSxbMiw3LCJcXGlkIl0sWzYsN10sWzcsOF0sWzUsNl0sWzksNV1d
\begin{equation}\label{eq: universal extension}\begin{tikzcd}
	0 & {T_p\Gs} & {\widetilde{\Gs}} & \Gs & 0 \\
	0 & {\omega_{\Gs^D}} & {E(\Gs)} & \Gs & {0.}
	\arrow[from=1-1, to=1-2]
	\arrow[from=1-2, to=1-3]
	\arrow["{f_{\Gs^D}^{\vee}}"', from=1-2, to=2-2]
	\arrow[from=1-3, to=1-4]
	\arrow[from=1-3, to=2-3]
	\arrow[from=1-4, to=1-5]
	\arrow["=", from=1-4, to=2-4]
	\arrow[from=2-1, to=2-2]
	\arrow[from=2-2, to=2-3]
	\arrow[from=2-3, to=2-4]
	\arrow[from=2-4, to=2-5]
\end{tikzcd}\end{equation}
When the de Rham module $D(\Gs)$ is an étale vector bundle, the sheaf $E(\Gs)$ can be shown to be representable by a smooth $S$-group, whose Lie algebra recovers $D(\Gs)$. If $\Gs=\mathfrak{G}_{\eta}$ has good reduction, the universal vector extensions of $\mathfrak{G}$ and $\Gs$ are compatible, in the sense that the following diagram commutes
% https://q.uiver.app/#q=WzAsMTAsWzAsMSwiMCJdLFsxLDEsIlxcb21lZ2Ffe1xcR3NeRH0iXSxbMiwxLCJFKFxcR3MpIl0sWzMsMSwiXFxHcyJdLFs0LDEsIjAuIl0sWzMsMCwiXFxtYXRoZnJha3tHfV97XFxldGF9Il0sWzIsMCwiRShcXG1hdGhmcmFre0d9KV97XFxldGF9Il0sWzQsMCwiMCJdLFsxLDAsIihcXG9tZWdhX3tcXG1hdGhmcmFre0d9XkR9KV97XFxldGF9Il0sWzAsMCwiMCJdLFswLDFdLFsxLDJdLFsyLDNdLFszLDRdLFs1LDMsIj0iXSxbNiwyLCIiLDAseyJzdHlsZSI6eyJ0YWlsIjp7Im5hbWUiOiJob29rIiwic2lkZSI6InRvcCJ9fX1dLFs2LDVdLFs1LDddLFs4LDZdLFs5LDhdLFs4LDEsIiIsMix7InN0eWxlIjp7InRhaWwiOnsibmFtZSI6Imhvb2siLCJzaWRlIjoidG9wIn19fV1d
\[\begin{tikzcd}
	0 & {(\omega_{\mathfrak{G}^D})_{\eta}} & {E(\mathfrak{G})_{\eta}} & {\mathfrak{G}_{\eta}} & 0 \\
	0 & {\omega_{\Gs^D}} & {E(\Gs)} & \Gs & {0.}
	\arrow[from=1-1, to=1-2]
	\arrow[from=1-2, to=1-3]
	\arrow[hook, from=1-2, to=2-2]
	\arrow[from=1-3, to=1-4]
	\arrow[hook, from=1-3, to=2-3]
	\arrow[from=1-4, to=1-5]
	\arrow["{=}", from=1-4, to=2-4]
	\arrow[from=2-1, to=2-2]
	\arrow[from=2-2, to=2-3]
	\arrow[from=2-3, to=2-4]
	\arrow[from=2-4, to=2-5]
\end{tikzcd}\]
If $S=\Spa(C)$ and $\Gs= A\langle p^{\infty} \rangle$, for an abeloid variety $A/C$, then $E(\Gs)$ is the topologically $p$-torsion subgroup of the universal vector extension $E_{\et}(A)$ of \cite[§4.1]{Heuer2023MannWerner}. Note however that the extension (\ref{eq: universal extension}) is not universal among extensions of $\Gs$ by vector bundles in the usual sense, i.e. given an étale vector bundle $V$ on $S$, the natural map
\[ \Hom(\omega_{\Gs^D},V) \rightarrow \Ext_{\et}^1(V,\Gs), \quad \varphi \mapsto \varphi_*(E(\Gs))\]
need not be an isomorphism. For example, taking $\Gs =\Q_p/\Z_p$ over $S=\Spa(C)$, the universal vector extension $E(\Gs)$ is the extension
% https://q.uiver.app/#q=WzAsNSxbMSwwLCJcXEdfYSJdLFsyLDAsIlxcdGZyYWN7XFxHX2EgXFxvcGx1cyBcXFFfcH17XFxaX3B9Il0sWzMsMCwiXFxRX3AvXFxaX3AiXSxbNCwwLCIwLCJdLFswLDAsIjAiXSxbMSwyXSxbMiwzXSxbMCwxXSxbNCwwXV0=
\[\begin{tikzcd}
	0 & {\G_a} & {\frac{\G_a \oplus \Q_p}{\Z_p}} & {\Q_p/\Z_p} & {0,}
	\arrow[from=1-1, to=1-2]
	\arrow[from=1-2, to=1-3]
	\arrow[from=1-3, to=1-4]
	\arrow[from=1-4, to=1-5]
\end{tikzcd}\]
where the quotient is via the diagonal embedding of $\Z_p$ and the left map is the inclusion in the left factor. The latter has a left inverse given by $(a,b) \mapsto a-b$, so that the above sequence is split. One proves similarly that any extension of $\Q_p/\Z_p$ by a vector bundle $V$ splits, while 
\[ \Hom(\omega_{\Gs^D},V) = V\]
is non-zero.
\end{remark}

Next, we compare our construction with the Dieudonné theory over integral perfectoid rings of Scholze--Weinstein \cite[Appendix to Lecture 17]{SW20} and Lau \cite{Lau_2018}. Recall that, given an integral perfectoid ring $R$, this is an equivalence between the categories of $p$-divisible groups $\mathfrak{G}$ over $R$ and minuscule Breuil--Kisin--Fargues modules over $R$. Assume from now on that $R$ is $p$-torsionfree and let $S=\Spa(R[\tfrac{1}{p}],\cj{R})$ be its adic generic fiber, an affinoid perfectoid space over $\Q_p$. Recall the functor of Example \ref{ex: examples of shtukas}(2)
\begin{align}\label{eq: vb from BKF module} (M,\varphi_M) \mapsto  (\widetilde{M},\varphi_{\widetilde{M}}), \end{align}
attaching to a (minuscule) Breuil--Kisin--Fargues module $(M,\varphi_M)$ over $R$ a (minuscule) shtuka $(\widetilde{M},\varphi_{\widetilde{M}})$ over $\mathcal{Y}_S$ with one leg at $\varphi^{-1}(S)$. We can now state our comparison result.
\begin{proposition}\label{prop: prismatic dieudonné versus E(G)}
    Let $R$ be a $p$-torsionfree integral perfectoid ring and let $S$ be the adic generic fiber of $\Spf(R)$. Let $\mathfrak{G}$ be a $p$-divisible group over $\Spf(R)$ with generic fiber $\Gs \rightarrow S$. Let $(M,\varphi_M)$ denote the Breuil--Kisin--Fargues module associated to $\mathfrak{G}$ over $\Ainf(R)$ and let $(\Ms(\Gs),\varphi_{\Ms(\Gs)})$ be the local shtuka over $\Yss_S$ associated with $\Gs$ under Theorem \ref{thm: equivalence for dualizable groups on perfd}. Then we have a canonical isomorphism
    \[ (\Ms(\Gs),\varphi_{\Ms(\Gs)}) = (\widetilde{M}, \varphi_{\widetilde{M}}).\]
\end{proposition}
\begin{proof}
    Let $R^{\flat}$ be the tilt of $R$, a perfect ring, and let $\Spec(T) \rightarrow \Spec(R^{\flat})$ be a $v$-cover in the sense of \cite[Def. 2.1]{Bhatt2017}. This yields a $p$-torsionfree integral perfectoid ring $T^{\sharp} = W(R^{\flat})/(\xi_R)$ and a map $\Spf(T^{\sharp}) \rightarrow \Spf(R)$ which becomes a $v$-cover after passing to the adic generic fiber. By Remark \ref{remark: dualizability does not descend well}, dualizable analytic $p$-divisible groups over $S$ satisfy descent along $v$-covers of perfectoid spaces. Moreover, the minuscule shtuka $(\widetilde{M},\varphi_{\widetilde{M}})$ uniquely determines, by Theorem \ref{thm: equivalence for dualizable groups on perfd}, an analytic $p$-divisible group $\Gs'$ over $S$, functorially associated with $\mathfrak{G}$, with $\Ms(\Gs) = \widetilde{M}$. If we can show that, for a basis $\Bs$ of the $v$-topology on $S$, we have a collection of isomorphisms $\Gs \restr{Y_{\eta}} \cong \Gs'\restr{Y_{\eta}}$, for all $Y = \Spf(T^{\sharp})_{\eta}$ with $T\in \Bs$, compatible with pullbacks, this will imply that $\Gs$ and $\Gs'$ are isomorphic and conclude.     
    
    From this, we may assume that $R = \prod_{i\in I} R_i$ is a product of $p$-torsionfree perfectoid valuation rings with algebraically closed fraction fields, in which case $S=\Spa(R[p^{-1}],R)$ is a product of points. By Proposition \ref{prop: dualiazable pdiv gps over product of points}, a natural isomorphism $\Gs \cong \Gs'$ on $S$ is equivalent to a collection of isomorphism $\Gs_i \cong \Gs_i'$ on $\Spa(R_i[p^{-1}],R_i)$. Therefore, we may further assume that $R$ is a valuation ring with algebraically closed fraction field $C$. By Corollary \ref{cor: stack of pdiv gps is overconvergent}, any isomorphism $\Gs \cong \Gs'$ on $S=\Spa(C,R)$ is equivalent to an isomorphism over the open locus $\Spa(C,\Os_C)$. Hence, we may further assume that $R=\Os_C$. Consider the sub-vector space $E \sub T(-1)\otimes C$ corresponding to $\Gs'$ via Theorem \ref{thm: equivalence for dualizable groups on perfd}. By \cite[Thm. 12.1.1, Thm. 14.1.1]{SW20}, we have $T=T_p\Gs$, $E=\Lie(\Gs)$ and the inclusion is given by $\alpha_{\mathfrak{G}^D}^{\vee}$, the dual of the Hodge-Tate map of the Cartier dual. By Proposition \ref{prop: generic fibers of log p-divisible groups}, this map coincides with $f_{\Gs}$, so that $\Gs = \mathfrak{G}_{\eta}$ corresponds to the same tuple as $\Gs'$ under the equivalence of Theorem \ref{thm: extending Fargues' equivalence of categories}. We thus obtain a canonical isomorphism $\Gs' \cong \Gs$, uniquely determined by this identification. This concludes the proof.
\end{proof}
\begin{remark}
    Let $\mathfrak{G}$ be a $p$-divisible group over $R$ with generic fiber $\Gs$ over $S = \Spa(A,A^+)$ and let 
    \begin{align}\label{eq: map of minscule lattices} \alpha\colon T_p\Gs \otimes_{\Z_p} \B_{\dR}^+ \hookrightarrow \Xi(\Gs)\end{align}
    denote the corresponding minuscule $\B_{\dR}^+$-lattice under Theorem \ref{thm: equivalence for dualizable groups on perfd}. We can obtain $\alpha$ directly from the group $\mathfrak{G}$ as follows: Consider the natural map
    \[ v\colon T_p\mathfrak{G}(R) \otimes_{\Z_p} \Q_p/\Z_p \rightarrow \mathfrak{G}.\]
    By functoriality of the prismatic Dieudonné module, this induces a $\varphi$-equivariant morphism 
    \begin{align}\label{eq: explicit prismatic dieudonné map} \alpha'\colon T_p\mathfrak{G}(R) \otimes_{\Z_p} \Ainf(R) = M( T_p\mathfrak{G} \otimes_{\Z_p} \Q_p/\Z_p) \rightarrow M(\mathfrak{G}). \end{align}
    We claim that it induces the map (\ref{eq: map of minscule lattices}) after base change to $\B_{\dR}^+$. This follows from the following commutative diagram
    % https://q.uiver.app/#q=WzAsNCxbMCwwLCJUX3BcXEdzXFxvdGltZXNfe1xcWl9wfVxcQl97XFxkUn1eKyJdLFsxLDAsIlxcWGkoVF9wXFxHcyBcXG90aW1lcyBcXFFfcC9cXFpfcCkiXSxbMSwxLCJcXFhpKFxcR3MpLCJdLFswLDEsIlRfcFxcR3NcXG90aW1lc197XFxaX3B9XFxCX3tcXGRSfV4rIl0sWzAsMSwiPSJdLFsxLDIsIlxcYWxwaGEnXFxvdGltZXMgXFxpZCJdLFswLDMsIj0iXSxbMywyLCJcXGFscGhhIl1d
\[\begin{tikzcd}
	{T_p\Gs\otimes_{\Z_p}\B_{\dR}^+} & {\Xi(T_p\Gs \otimes \Q_p/\Z_p)} \\
	{T_p\Gs\otimes_{\Z_p}\B_{\dR}^+} & {\Xi(\Gs),}
	\arrow["{=}", from=1-1, to=1-2]
	\arrow["{=}", from=1-1, to=2-1]
	\arrow["{\alpha'\otimes \id}", from=1-2, to=2-2]
	\arrow["\alpha", from=2-1, to=2-2]
\end{tikzcd}\]
induced by the map $v$ using functoriality of the sheaf $\Xi(\Gs)$, and Proposition \ref{prop: prismatic dieudonné versus E(G)}.
\end{remark}

We may more generally compare our construction with the prismatic Dieudonné theory of Anschütz--Le Bras \cite{anschutz2022prismaticdieudonnetheory}. Let $R$ be a quasi-syntomic ring \cite{BhattScholze2022prisms}. For example, $R$ could be an integral perfectoid ring, or the $p$-completion of a smooth algebra over $\Os_K$, where $K$ is a $p$-adic field or a perfectoid field. The authors define \cite[Def. 4.5]{anschutz2022prismaticdieudonnetheory} a category $\DM^{\adm}(R)$ of admissible prismatic Dieudonné modules as certain $\varphi$-modules on the small quasi-syntomic site $(R)_{\qsyn}$, and a contravariant equivalence \cite[Thm. 4.74]{anschutz2022prismaticdieudonnetheory}
\begin{align} \BT(R) \cong \DM^{\adm}(R), \quad \mathfrak{G} \mapsto (\Ms_{\Prism}(\mathfrak{G}), \varphi_{\Ms_{\Prism}(\mathfrak{G})})^{\vee}.\end{align}
Assume that the adic generic fiber $S=\Spf(R)_{\eta}$ is a good adic space. We construct a natural contravariant functor 
    \begin{align}\label{eq: from admissible DM to analytic DC} 
    \widetilde{(\cdot)^{\vee}} \colon \DM^{\adm}(R) \rightarrow \{ \text{analytic Dieudonné crystals over }S
    \} \end{align}
as follows. Let $T=\Spa(A,A^+)$ be an affinoid perfectoid space over $S$ such that the structure map is the generic fiber of a map $\Spf(A^+) \rightarrow \Spf(R)$. Given a $p$-divisible group $\mathfrak{G}$ over $R$ with associated prismatic Dieudonné module $(\Ms_{\Prism},\varphi_{\Ms_{\Prism}})^{\vee}$ over $R$, extending scalars to $A^+$, passing to linear duals and taking global sections produces a minuscule Breuil--Kisin--Fargues module $(M,\varphi_M)$ which corresponds to $\mathfrak{G} \otimes_R A^+$ under Scholze--Weinstein's functor \cite[Prop. 4.47-48]{anschutz2022prismaticdieudonnetheory}. Moreover, using the admissibility condition in \cite[Remark. 1.14]{anschutz2022prismaticdieudonnetheory}, the $A^+$-module $\varphi_{M}(\varphi^*M)/M$ arises via scalar extension from a finite projective $R$-module, namely $\Lie(\mathfrak{G})$. We may pass to the associated minuscule shtuka $(\widetilde{M},\varphi_{\widetilde{M}})$ over $\Yss_T$. Upon noticing that perfectoid spaces $T \rightarrow S$ satisfying the above form a basis of the $v$-site of $S$, we obtain a well-defined object 
\[ (\widetilde{\Ms}_{\Prism},\varphi_{\widetilde{\Ms}_{\Prism}}) \in \Sht_{\min}(S),\]
satisfying the additional property that $\varphi_{\widetilde{\Ms}_{\Prism}}(\varphi^*\widetilde{\Ms}_{\Prism})/\widetilde{\Ms}_{\Prism}$ is a vector bundle arising from a finite projective $R$-module. By using the compatibility of the prismatic Dieudoné functor with Cartier duality \cite[Prop. 4.73]{anschutz2022prismaticdieudonnetheory}, we conclude that the same can be said of $\widetilde{\Ms}_{\Prism}^{\vee} \otimes \Os_{\Yss_S}\{1\}$. It follows that $(\widetilde{\Ms}_{\Prism},\varphi_{\widetilde{\Ms}_{\Prism}})$ is an analytic Dieudonné crystal over $S$ in the sense of Definition \ref{def: analytic Dieudonné crystal}. We then reach the following.
\begin{proposition}\label{prop: analytic DT vs prismatic DT}
    Let $R$ be a $p$-torsionfree quasi-syntomic ring, and assume that the adic generic fiber $S=\Spf(R)_{\eta}$ is a good adic space. Then  the functor (\ref{eq: from admissible DM to analytic DC}) is compatible with prismatic Dieudonné theory \cite[Thm. 4.74]{anschutz2022prismaticdieudonnetheory} and the equivalence of Theorem \ref{thm: analytic Dieudonné theory in general}, in the sense that the following diagram commutes
    % https://q.uiver.app/#q=WzAsNCxbMCwwLCJcXEJUKFIpIl0sWzEsMCwiXFxETV57XFxhZG19KFIpIl0sWzAsMSwiXFx7XFx0ZXh0e2FuYWx5dGljIH1wXFx0ZXh0ey1kaXZpc2libGUgZ3JvdXBzIG92ZXIgfVNcXH0iXSxbMSwxLCJcXHsgXFx0ZXh0e2FuYWx5dGljIERpZXVkb25uw6kgY3J5c3RhbHMgb3ZlciB9U1xcfS4iXSxbMCwxLCJcXGNvbmciXSxbMCwyLCIoXFxjZG90KV97XFxldGF9IiwyXSxbMSwzLCJcXHdpZGV0aWxkZXsoXFxjZG90KX1ee1xcdmVlfSJdLFsyLDMsIlxcY29uZyIsMl1d
\[\begin{tikzcd}
	{\BT(R)} & {\DM^{\adm}(R)} \\
	{\{\text{analytic }p\text{-divisible groups over }S\}} & {\{ \text{analytic Dieudonné crystals over }S\}.}
	\arrow["\cong", from=1-1, to=1-2]
	\arrow["{(\cdot)_{\eta}}"', from=1-1, to=2-1]
	\arrow["{\widetilde{(\cdot)}^{\vee}}", from=1-2, to=2-2]
	\arrow["\cong"', from=2-1, to=2-2]
\end{tikzcd}\]
\end{proposition}
\begin{proof}
    This follows from the above discussion and Proposition \ref{prop: prismatic dieudonné versus E(G)}.
\end{proof}
\begin{remark}
    A similar compatibility result holds for the log prismatic Dieudonné theory of \cite{würthen2023logprismaticdieudonnetheory}\cite{inoue2025logprismaticdieudonnetheory} on $p$-adic semistable formal schemes and the adic generic fiber functor on log $p$-divisible groups of Proposition \ref{prop: generic fibers of log p-divisible groups}. We do not pursue this here.
\end{remark}

Next, we show that if $\Gs$ is an analytic $p$-divisible group over a perfectoid space $S = \Spa(R,R^+)$ of good reduction $\mathfrak{G} \rightarrow \Spf(R^+)$, then the vector bundle $\Es(\Gs)$ on $X_S$ only depends on the special fiber $\mathfrak{G} \otimes_{R^+} R^+/p$ up to isogeny. For this, we recall a construction of Zhang (\cite{zhang2023peltypeigusastackpadic}, below Proposition $8.5$) which attaches to a $p$-divisible group $G$ over $R^+/\varpi$ a vector bundle $\Es(G)$ on $X_S$. Here $\varpi\in R^+$ is any perfectoid pseudo-uniformizer. This is done as follows: Consider the rational crystalline Dieudonné module $M[\tfrac{1}{p}] \coloneqq M_{\crys}(G)[\tfrac{1}{p}]$. This is a finite free $\Bcr+(R^+/\varpi)$-module, where $\Bcr+(R^+/\varpi) = \Acrys(R^+/\varpi)[\tfrac{1}{p}]$. Therefore, the group
\[ \bigoplus_{d\geq 0} (M[\tfrac{1}{p}])^{\varphi =p^{d+1}}\]
is a graded module over the ring
\[ P = \bigoplus_{d\geq 0} \Bcr+(R^+/\varpi)^{\varphi =p^{d}}.\]
This thus defines a vector bundle on the \emph{algebraic Fargues--Fontaine curve}
\[ X_S^{\alg} = \Proj(P).\]
We now use the morphism of locally ringed space \cite[§7.1]{zhang2023peltypeigusastackpadic}
\[ X_S \rightarrow X_S^{\alg}\]
which allows to pull back this vector bundle to a vector bundle $\Es(G)$ on $X_S$.

\begin{cor}\label{cor: E(G) only depends on special fiber}
    Let $S=\Spa(R,R^+)$ be an affinoid perfectoid space over $\Q_p$. Let $G$ be a $p$-divisible group over $R^+/\varpi$ that admits a deformation $\mathfrak{G}$ over $R^+$, and let $\Gs$ be its adic generic fiber over $S$. Then we have a canonical isomorphism of vector bundle on $X_S$, natural in $S$
    \[ \Es(\Gs) \cong \Es(G).\]
\end{cor}
\begin{proof}
    Let $(M,\varphi_M)$ be the prismatic Dieudonné module of $\mathfrak{G}$ and let $\Es(M,\varphi_M)$ denote the resulting vector bundle on the Fargues--Fontaine curve. By Proposition \ref{prop: prismatic dieudonné versus E(G)}, it is enough to give a natural isomorphism
    \[ \Es(M,\varphi_M) \cong \Es(G).\]
    We adapt the proof of \cite[Lemma 8.6]{zhang2023peltypeigusastackpadic}. Consider the Frobenius-equivariant map (\ref{eq: explicit prismatic dieudonné map})
    \[ T_p\mathfrak{G}(R^+) \otimes_{\Z_p} \Ainf(R^+) \rightarrow M.\]
    After base-change to $\Acrys(R^+)$, we obtain a map
    \[ T_p\mathfrak{G}(R^+) \otimes_{\Z_p} \Acrys(R^+) \rightarrow M \otimes_{\Ainf(R^+)} \Acrys(R^+) = M_{\crys}(G),\]
    where we use \cite[Thm. 17.5.2]{SW20} for the last equality. From this, we obtain that both $\Es(M,\varphi_M)$ and $\Es(G)$ are the modification of the vector bundle $T_p\Gs \otimes_{\Z_p} \Os_{X_S}$ at the $\BdR+(R)$-lattice $M \otimes_{\Ainf(R^+)} \BdR+(R)$ along the isomorphism 
    \[ T_p\Gs \otimes_{\Q_p} \BdeR(R) = M \otimes_{\Ainf(R^+)} \BdeR(R)\]
    resulting from Proposition \ref{prop: prismatic dieudonné versus E(G)}. This concludes the proof.
\end{proof}

We derive the following.
\begin{cor}\label{cor: E(G) = E(D,phi) over padic field}
    Let $k$ be a perfect field of characteristic $p$. Let $G$ be a $p$-divisible group over $k$ with Dieudonné isocrystal $(D,\varphi_D)$ over $K=W(k)[\tfrac{1}{p}]$. Let $\mathfrak{G} \rightarrow \Spf(W(k))$ be a $p$-divisible group deforming $G$ and let $\Gs \rightarrow \Spa(K)$ be its generic fiber. Then we have a canonical isomorphism
    \[ \Es(\Gs) = \Es(D,\varphi_D) \in \Bun_{\GL_n}(\Spd(K)).\]
\end{cor}
\begin{proof}
Let $S = \Spa(R,R^+)$ be any affinoid perfectoid space over $K$. We need to produce a canonical isomorphism of vector bundles on $X_S$, compatible with base-change in $S$
\[ \Es(\Gs_S) = \Es(D,\varphi_D).\]
In that case, the rational crystalline Dieudonné module of $G\otimes_k R^+/p$ is given by 
\[ M_{\crys}(G\otimes_k R^+/p)[\tfrac{1}{p}] = (D,\varphi) \otimes_K \Bcr+(R^+).\]
Therefore, it follows from the description of $\Es(D,\varphi_D)$ in terms of the algebraic curve, see \cite[Thm. 13.5.4]{SW20}, and Corollary \ref{cor: E(G) only depends on special fiber}.
\end{proof}

\subsection{Comparison with de Rham cohomology}\label{subsection: Comparison with de Rham cohomology}
We now compare our Dieudonné functor with $p$-adic cohomology theories. We will need Liu--Zhu's de Rham and Hodge--Tate functors $D_{\dR}$, $D_{\HT}$, whose definitions are recalled at (\ref{eq: DdR}) and (\ref{eq: DHT}) respectively.

\begin{proposition}\label{prop: VpG de rham vs MG étale vb}
    Let $S$ be a smooth rigid space over a $p$-adic field $K$ and let $\Gs \rightarrow S$ be a dualizable analytic $p$-divisible group. 
    \begin{enumerate}
        \item The de Rham module $D(\Gs)$ is an étale vector bundle if and only if $V_p\Gs$ is a de Rham $\Q_p$-local system. In that case, we have a canonical isomorphism
    \[ D_{\dR}(V_p\Gs) = D(\Gs),\]
    compatible with the filtrations\footnote{Beware that the filtration on $D_{\dR}(V_p\Gs)$ has range in $[-1,0]$, i.e. we have
    \[ \omega_{\Gs^D} = \Fil^0 \sub \Fil^{-1} = D_{\dR}(V_p\Gs),\]
    while we want the Hodge filtration on $D(\Gs)$ to have range in $[0,1]$.}.
    \item In general, we have a canonical isomorphism
    \[ D_{\HT}(V_p\Gs) = \omega_{\Gs^D} \oplus \Lie(\Gs),\]
    compatible with the gradings.
    \end{enumerate} 
\end{proposition}
\begin{proof}
    \begin{enumerate}
        \item Assume first that $D(\Gs)$ is an étale vector bundle. Let $\nu\colon X_{\proet} \rightarrow X_{\et}$ and consider
        \begin{align*} 
        D_{\dR}(V_p\Gs) = \nu_*(V_p\Gs \otimes_{\Q_p}\OBdeR) =\nu_*(\Xi(\Gs) \otimes_{\B_{\dR}^+}\OBdeR),
        \end{align*}
        where $\Xi(\Gs)$ is the lattice from Corollary \ref{cor: BdR+ lattice assoc to pdiv groups in general}. We equip $\Xi(\Gs) _{\B_{\dR}^+}\OBdeR$ with the filtration $\Xi(\Gs) \otimes_{\B_{\dR}^+} \Fil^i\OBdeR$, with graded pieces $\Xi(\Gs)/\xi\Xi(\Gs) \otimes_{\Os_{X_{\proet}}} \OC(i)$. By the projection formula, we have for all $i\neq 0$ and $j>0$
        \[ R^j\nu_*(\Xi(\Gs)/\xi\Xi(\Gs) \otimes_{\Os_{X_{\proet}}} \OC(i)) = D(\Gs) \otimes_{\Os_{X_{\et}}} R^j\nu_*\OC(i) = 0,\]
        where the last equality follows from \cite[Prop. 6.16]{scholze2011perfectoid}. We conclude that
        \[ \nu_*(\Xi(\Gs) \otimes_{\B_{\dR}^+}\OBdeR) = \nu_*(\Xi(\Gs)/\xi\Xi(\Gs) \otimes_{\Os_{X_{\proet}}}\OC) = D(\Gs).\]
        This shows that $\rank D_{\dR}(V_p\Gs) = \rank_{\Q_p} V_p\Gs$, so that $V_p\Gs$ is de Rham.

        Conversely, assume that $V_p\Gs$ is a de Rham local system, so that also $V_p\Gs(-1)$ is de Rham. Let $\Xi \sub V_p\Gs(-1) \otimes_{\Q_p} \B_{\dR}$ be the $\B_{\dR}^+$-lattice of Lemma \ref{lemma: B+ lattice of Scholze}. Then by Lemma \ref{lemma: BdR+-lattice induces HT filtration in arithmetic case}, there are inclusions
        \[ \xi(V_p\Gs(-1) \otimes_{\Z_p}\B_{\dR}^+) \sub \Xi \sub V_p\Gs(-1) \otimes_{\Z_p}\B_{\dR}^+\]
        with
        \[(\id\otimes\theta)(\Xi) = \Lie(\Gs) \sub V_p\Gs(-1)\otimes_{\Q_p} \Os_{S_v},\]
        where 
        \[ (\id \otimes\theta)\colon V_p\Gs(-1) \otimes_{\Z_p}\B_{\dR}^+ \rightarrow V_p\Gs(-1) \otimes_{\Z_p}\Os_v \]
        is the projection. But we also know from Corollary \ref{cor: BdR+ lattice assoc to pdiv groups in general} that $\Xi(\Gs)$ satisfies
        \[ V_p\Gs \otimes_{\Z_p}\B_{\dR}^+ \sub \Xi(\Gs) \sub \xi^{-1}(V_p\Gs \otimes_{\Z_p}\B_{\dR}^+)\]
        with 
        \[(\id\otimes\theta)(\Xi(\Gs)) = \Lie(\Gs) \sub V_p\Gs\otimes_{\Q_p} \Os_{S_v}(-1).\]
        Since there clearly is at most one lattice $\Xi(\Gs)$ satisfying these properties, it follows that 
        \[ \Xi(\Gs) = \Xi \otimes_{\B_{\dR}^+} \xi^{-1}(\Q_p(1)\otimes_{\Q_p} \B_{\dR}^+).\]
        In particular,
        \begin{align*} 
        D(\Gs) = \Xi(\Gs)\otimes_{\B_{\dR}^+,\theta} \Os_{S_v} = \Xi\otimes_{\B_{\dR}^+,\theta} \Os_{S_v} = D_{\dR}(V_p{\Gs}(-1)). 
        \end{align*} 
        This shows that $D(\Gs)$ is an étale vector bundle, as required. Finally, the last sentence follows from the fact that $D_{\dR}(V_p{\Gs}(-1)) = D_{\dR}(V_p\Gs)$, with a shift in the filtration index. 
        \item This follows from Proposition \ref{prop: characterization of HT ls} and the sequence (\ref{eq: HT filtration on Tate module}).
    \end{enumerate}
\end{proof}

\begin{cor}\label{cor: Dieudonné module recovers first de Rham cohomology}
    Let $X\rightarrow S$ be a proper smooth morphism of smooth rigid spaces over a $p$-adic field $K$. Consider the topologically $p$-torsion Picard variety $\PPic_{X/S,\et}\langle p^{\infty} \rangle$ of Corollary \ref{cor: ptop picard is dualizable} and let $\Hs \sub \PPic_{X/S,\et}\langle p^{\infty} \rangle$ be its maximal analytic $p$-divisible subgroup. Then we have a canonical isomorphism of vector bundles on $S_{\et}$
    \[ D(\Hs) = R^1\pi_{\dR,*}\Os_X,\]
    compatible with the Hodge filtrations on each side.
\end{cor}
We expect this result to extend to arbitrary non-archimedean fields $K$ over $\Q_p$. We show in Proposition \ref{prop: dRFF cohomology from ptop picard variety} that it holds for $S=\Spa(C)$, when $C$ is algebraically closed.
\begin{proof}
    This follows from the de Rham Comparison Theorem, \cite[Thm. 8.8]{scholze2013padicHodge}\cite[Thm. 3.9(iv)]{Liu2017} and Proposition \ref{prop: VpG de rham vs MG étale vb}. 
\end{proof}

We now fix a complete algebraically closed field $C$ over $\Q_p$. Given a proper smooth rigid space $Z$ over $C$, in Proposition \ref{prop: dRFF cohomology from ptop picard variety} below, we give a geometric interpretation of the vector bundle $\Es$ on the Fargues--Fontaine curve $X_C$ associated with the topologically $p$-torsion Picard variety of $Z$. For this, we briefly recall some refined $p$-adic cohomology theories.

In \cite[Thm. 13.1]{Bhatt2018}, Bhatt--Morrow--Scholze construct a cohomology theory
\[ H_{\crys}^n(Z/\BdR+)\]
with coefficients in finite free $\BdR+=\BdR+(C)$-modules. It comes with natural comparison isomorphisms with de Rham cohomology
    \begin{align}\label{eq: BdR+ cohomology deforms de Rham cohomology} 
    H_{\crys}^n(Z/\BdR+) \otimes_{\BdR+,\theta} C = H_{\dR}^n(Z/C),
    \end{align}
and with pro-étale cohomology
    \begin{align}\label{eq: generalized dR comparison theorem}  
    H_{\crys}^n(Z/\BdR+)[\xi^{-1}] = H_{\et}^n(Z,\Q_p) \otimes_{\Q_p} \BdeR. \end{align}

 If $Z$ arises as the pullback of a smooth proper rigid space $Z_0$ over a $p$-adic field $K \sub C$, the authors show moreover in \emph{loc. cit.} that
\[ H_{\crys}^n(Z/\BdR+) = H_{\dR}^n(Z_0/K)\otimes_K \BdR+.\]
This cohomology theory was later revisited in the work of Guo \cite{Guo2021cryscoh} and Bosco \cite{bosco2023padicproetalecohomologydrinfeld}\cite{bosco2023rationalpadichodgetheory} and given site-theoretic definitions applying to arbitrary rigid spaces over $C$.

Observe that (\ref{eq: generalized dR comparison theorem}) realises $H_{\et}^n(Z,\Q_p)\otimes \BdR+$ as a $\BdR+$-lattice inside of $H_{\crys}^n(Z/\BdR+)[\xi^{-1}]$. We show below that we can recognize the induced filtration under Proposition \ref{lemma: minuscule lattices vs flags} when $n=1$.

\begin{proposition}\label{prop: recognize the fils induced by BB map}
    Let $Z$ be a proper smooth rigid space over $C$.
    \begin{enumerate}
        \item There is an inclusion of $\BdR+$-lattice $H_{\crys}^1(Z/\BdR+) \sub H_{\et}^1(Z,\Q_p)\otimes_{\Q_p} \BdR+ \sub \tfrac{1}{\xi}H_{\crys}^1(Z/\BdR+)$, and the lattice $\xi (H_{\et}^1(Z,\Q_p)\otimes_{\Q_p} \BdR+)$ induces the Hodge filtration
        \[ \Fil_{\Hdg}^{1}\]
         on $H_{\dR}^1(Z/C) \cong  H_{\crys}^1(Z/\BdR+) \otimes_{\BdR+} C$ under the map of Proposition \ref{lemma: minuscule lattices vs flags}.
         \item Dually, the minuscule lattice $H_{\crys}^1(Z/\BdR+)$ induces the Hodge--Tate filtration
         \[ \Fil_{\HT}^{1}.\]
         on $H_{\et}^1(Z,\Q_p)\otimes_{\Q_p} C$. Here, $\Fil_{\HT}^{\bullet}$ is induced by the Hodge--Tate spectral sequence of Theorem \ref{thm: relative linear HT sequence}.
    \end{enumerate}
\end{proposition}
\begin{proof}
\begin{enumerate}
    \item Let $n \geq 0$ be arbitrary. We equip the $\BdR+$-module $H_{\crys}^n(Z/\BdR+)$ with its Hodge filtration in the sense of \cite[Def. 5.7]{bosco2023rationalpadichodgetheory}. Then by \cite[Thm. 7.4]{bosco2023rationalpadichodgetheory}, we have a natural filtered isomorphism
    \begin{align*}
    H_{\et}^n(Z,\Q_p)\otimes_{\Q_p}\BdR+ = \Fil^0(H_{\crys}^n(Z/\BdR+) \otimes_{\BdR+}\BdeR) = \sum_{i=0}^{n} \Fil^iH_{\crys}^n(Z/\BdR+)\xi^{-i}.\end{align*}
    Now by combining \cite[Prop. 5.8]{bosco2023rationalpadichodgetheory} and \cite[Thm. 1.2.7(i), Thm. 1.2.1(i)]{Guo2021cryscoh}, the Hodge filtration on $H_{\crys}^n(Z/\BdR+)$ reduces modulo $\xi$ to the Hodge filtration $\Fil_{\Hdg}^i$ on de Rham cohomology. When $n=1$, this immediately implies the desired equality.
    \item We will use the following spreading out argument of Conrad--Gabber, see \cite[Cor. 13.16]{Bhatt2018}: There exists a $p$-adic field $K$ and a proper smooth morphism $\pi\colon Y \rightarrow S$ of smooth rigid spaces over $K$ such that $X=Y \times_{S} \Spa(C)$, for some point $s\in S(C)$. Up to shrinking $S$, we may assume that $S=\Spa(R)$ is affinoid. By formal smoothness of $R$, the map $R \rightarrow C$ admits a continuous lift $\varphi\colon R \rightarrow \BdR+$. By the proof of \cite[Thm. 13.19]{Bhatt2018}, we then have a canonical isomorphism
    \[ H_{\dR}^n(Y/S)\otimes_{R,\varphi} \BdR+ \cong H_{\crys}^n(Z/\BdR+),\]
     lying over the base-change isomorphism for de Rham cohomology
      \[ H_{\dR}^n(Y/S)\otimes_{R} C\cong H_{\dR}^n(Z/C).\]
    Let $\Xi \sub R^n\pi_{v,* }\Q_p\otimes_{\Q_p} \B_{\dR}$ be the $\B_{\dR}^+$-lattice of Lemma \ref{lemma: B+ lattice of Scholze}, which comes with a canonical isomorphism
    \[ \Xi/\xi\Xi = H_{\dR}^n(Y/S).\]
    We fix a choice of isomorphism of $\BdR+$-modules
    \[ \Xi(C) \cong H_{\dR}^n(Y/S) \otimes_{R,\varphi} \BdR+\]
    lifting the identity on $H_{\dR}^n(Y/S) \otimes_{R}C$ and we consider the composition
    \[ \tau\colon \Xi(C)\cong H_{\dR}^n(Y/S) \otimes_{R,\varphi} \BdR+ \cong H_{\crys}^n(Z/\BdR+).\]
    By Lemma \ref{lemma: B+ lattice of Scholze}(2), the image of the $\BdR+$-lattice 
    \[ (R^n\pi_{v,*}\Q_p\otimes_{\Q_p} \B_{\dR}^+)(C) \sub \Xi(C)[\xi^{-1}] \overset{\tau[\xi^{-1}]}{\cong} H_{\crys}^n(Z/\BdR+)[\xi^{-1}]\]
    induces the Hodge filtration on $H_{\dR}^n(Z/C)$. When $n=1$, by Lemma \ref{lemma: minuscule lattices vs flags} and the first point, this implies that 
    \[ (R^1\pi_{v,*}\Q_p\otimes_{\Q_p} \B_{\dR}^+)(C) = H_v^1(Z,\Q_p)\otimes_{\Q_p} \BdR+ \]
    as $\BdR+$-lattices inside of $H_{\crys}^n(Z/\BdR+)[\xi^{-1}]$. In other words, the isomorphism $\tau$ is compatible with base-change for $\Q_p$-cohomology (up to some non-canonical isomorphism of $\BdR+$-modules).
    
From this, the desired interaction between $\BdR+$-lattices may be checked on $S$. In other words, it is enough to show that the $\B_{\dR}^+$-lattice $\Xi \sub R^1\pi_{v,*}\Q_p \otimes \B_{\dR}^+$ induces the filtration
    \[ \Fil_{\HT}^{1} \sub R^1\pi_{v,*}\Q_p \otimes \Os_{S_v}\]
    under the map of Proposition \ref{lemma: minuscule lattices vs flags}. This follows from Lemma \ref{lemma: BdR+-lattice induces HT filtration in arithmetic case}, using Lemma \ref{lemma: comparison between the two HT filtrations} to identify the Hodge--Tate filtration on $R^n\pi_{v,*}\Q_p \otimes \Os_v$. This concludes the proof.
\end{enumerate}
\end{proof}
\begin{remark}
    We saw in the proof of Proposition \ref{prop: recognize the fils induced by BB map} that, given a proper smooth morphism $Y\rightarrow S$ of smooth rigid spaces over a $p$-adic field, a point $s\in S(C)$ with corresponding fiber $Z=\pi^{-1}(s)$, then the first $\BdR+$-cohomology group $H_{\crys}^1(Z/\BdR+)$ of $Z$ is already defined over $S$. More, precisely, we proved that it arises from the $\B_{\dR}^+$-lattice
    \[  \Xi^1 \sub (R^1\pi_{v,*}\Q_p) \otimes_{\Q_p} \B_{\dR}\]
    of Lemma \ref{lemma: B+ lattice of Scholze}, up to some non-canonical isomorphism. We believe that this statement can be upgraded to a canonical isomorphism 
    \[ \Xi^n(C) = H_{\crys}^n(X/\BdR+),\]
    for arbitrary $n\geq 0$. Here, $\Xi^n$ is the analog of the $\B_{\dR}^+$-lattice $\Xi^1$ in the $n$th relative pro-étale cohomology of $Y/S$. We will not need such a generalization. 
\end{remark}

Let $Z$ be a proper smooth rigid space over $C$. By Beauville--Laszlo glueing, there is a unique vector bundle $H_{B}^n(Z)$ on the Fargues--Fontaine curve $X_C=X_C^{\FF}$ together with a natural isomorphism
\begin{align}
    H_{\FF}^n(Z) \otimes_{\Os_{X_C}}\BdR+ = H_{\crys}^n(Z/\BdR+)
\end{align}
and a modification of vector bundles
\begin{align}\label{eq: dRFF curve modifies etale coh}
\alpha\colon H_{\FF}^n(Z) \dashrightarrow H_{\et}^n(Z,\Q_p) \otimes_{\Q_p} \Os_{X_C}\end{align}
that recovers the isomorphism (\ref{eq: generalized dR comparison theorem}) under the equivalence (\ref{eq: from modif to lattices}). Bosco shows \cite[Def. 6.16]{bosco2023rationalpadichodgetheory} that the vector bundles $H_{\FF}^n(Z)$ extends to a cohomology theory $R\Gamma_{\FF}(Z)$, the \emph{Fargues--Fontaine cohomology}, applying to rigid spaces $Z$ over $C$. It takes value in solid quasi-coherent sheaves on the Fargues--Fontaine curve and deforms the de Rham cohomology $R\Gamma_{\dR}(Z/C)$. A related construction was also given by Le Bras--Vezzani in \cite{Le_Bras_vezzani_2023}.\\

We now record the consequences of Proposition \ref{prop: recognize the fils induced by BB map} for the case $n=1$ in the following proposition.
\begin{proposition}
    Let $Z$ be a proper smooth rigid space over $C$, for a complete algebraically closed $C/\Q_p$. Then the first de Rham--Fargues--Fontaine cohomology group $H_{\FF}^1(Z)$ fits in a minuscule modification of vector bundles on $X_C$
    % https://q.uiver.app/#q=WzAsNSxbMSwwLCJIX3tcXEZGfV4xKFopIl0sWzAsMCwiMCJdLFsyLDAsIkhfe1xcZXR9XjEoWixcXFFfcClcXG90aW1lcyBcXE9zX3tYX0N9Il0sWzMsMCwiaV8qSF4wKFosXFxPbWVnYV97Wi9DfV4xKSgtMSkiXSxbNCwwLCIwLCJdLFsxLDBdLFswLDJdLFsyLDNdLFszLDRdXQ==
\[\begin{tikzcd}
	0 & {H_{\FF}^1(Z)} & {H_{\et}^1(Z,\Q_p)\otimes \Os_{X_C}} & {i_*H^0(Z,\Omega_{Z/C}^1)(-1)} & {0,}
	\arrow[from=1-1, to=1-2]
	\arrow[from=1-2, to=1-3]
	\arrow[from=1-3, to=1-4]
	\arrow[from=1-4, to=1-5]
\end{tikzcd}\]
which extends to a commutative diagram with exact rows
% https://q.uiver.app/#q=WzAsNixbMCwwLCJIX3tcXGV0fV4xKFosXFxRX3AoMSkpXFxvdGltZXMgXFxPc197WF9DfSgtMSkiXSxbMSwwLCJIX3tcXGV0fV4xKFosXFxRX3ApXFxvdGltZXMgXFxPc197WF9DfSJdLFsyLDAsImlfKkhfe1xcZXR9XjEoWixcXFFfcClcXG90aW1lcyBDIl0sWzAsMSwiSF97XFxGRn1eMShaKSJdLFsxLDEsIkhfe1xcZXR9XjEoWixcXFFfcClcXG90aW1lcyBcXE9zX3tYX0N9Il0sWzIsMSwiaV8qSF4wKFosXFxPbWVnYV97Wi9DfV4xKSgtMSkuIl0sWzAsMV0sWzEsMl0sWzAsMywiIiwwLHsic3R5bGUiOnsidGFpbCI6eyJuYW1lIjoiaG9vayIsInNpZGUiOiJ0b3AifX19XSxbMyw0XSxbMiw1LCIiLDIseyJzdHlsZSI6eyJoZWFkIjp7Im5hbWUiOiJlcGkifX19XSxbNCw1XSxbMSw0LCI9IiwyXV0=
\begin{equation}\label{eq: diagram for dRFF coh 1}\begin{tikzcd}
	{H_{\et}^1(Z,\Q_p(1))\otimes \Os_{X_C}(-1)} & {H_{\et}^1(Z,\Q_p)\otimes \Os_{X_C}} & {i_*H_{\et}^1(Z,\Q_p)\otimes C} \\
	{H_{\FF}^1(Z)} & {H_{\et}^1(Z,\Q_p)\otimes \Os_{X_C}} & {i_*H^0(Z,\Omega_{Z/C}^1)(-1).}
	\arrow[from=1-1, to=1-2]
	\arrow[hook, from=1-1, to=2-1]
	\arrow[from=1-2, to=1-3]
	\arrow["{=}"', from=1-2, to=2-2]
	\arrow[two heads, from=1-3, to=2-3]
	\arrow[from=2-1, to=2-2]
	\arrow[from=2-2, to=2-3]
\end{tikzcd}\end{equation}
 Moreover, this diagram uniquely characterizes $H_{\FF}^1(Z)$ in the vector bundle $H_{\et}^1(Z,\Q_p)\otimes \Os_{X_C}$.
\end{proposition}
\begin{proof}
     By Proposition \ref{prop: recognize the fils induced by BB map}(2), we see that $H_{\crys}^1(Z/\BdR+) \sub H_{\et}^1(Z,\Q_p)\otimes \BdR+ \sub \xi^{-1}H_{\crys}^1(Z/\BdR+)$ and that
     \[ \frac{H_{\crys}^1(Z/\BdR+) \cap (H_{\et}^1(Z,\Q_p)\otimes \BdR+) }{H_{\crys}^1(Z/\BdR+) \cap \xi(H_{\et}^1(Z,\Q_p)\otimes \BdR+) } = \Fil_{\HT}^1 = H_{\et}^1(Z,\Os_Z).\]
     % \[ \frac{\xi( H_{\et}^1(Z,\Q_p)\otimes \BdR+) \cap H_{\crys}^1(Z/\BdR+)}{\xi (H_{\et}^1(Z,\Q_p)\otimes \BdR+) \cap \xi H_{\crys}^1(Z/\BdR+)} = \Fil_{\Hdg}^1 = H^0(Z,\Omega_{Z/C}).\]
     The commutative diagram in the statement is easily derived from this and the equivalence (\ref{eq: from modif to lattices}). The unicity part follows from Proposition \ref{lemma: minuscule lattices vs flags}.
\end{proof}

We are now finally ready to state our result.
\begin{proposition}\label{prop: dRFF cohomology from ptop picard variety}
    Let $Z$ be a proper smooth rigid space over $\Spa(C)$, for a complete and algebraically closed field $C/\Q_p$. Let $\PPic_{Z/C,\et}\langle p^{\infty} \rangle$ be the topologically $p$-torsion Picard variety and let $\Hs \sub \PPic_{Z/C}\langle p^{\infty} \rangle$ be its maximal analytic $p$-divisible subgroup, see Corollary \ref{cor: ptop picard is dualizable}. 
    \begin{enumerate}
        \item We have a canonical isomorphism
        \[ D(\Hs) = H_{\dR}^1(Z/C),\]
         compatible with the Hodge filtrations on each side.
        \item We have a canonical isomorphism of $\BdR+=\B_{\dR}^+(C)$-modules
        \[ \Xi(\Hs) = H_{\crys}^1(Z/\BdR+) \otimes_{\BdR+} \xi^{-1}(\Q_p(1)\otimes_{\Q_p}\BdR+).\]
        compatible with the isomorphisms (\ref{eq: BdeR lattice of G iso to tate module}) and (\ref{eq: generalized dR comparison theorem}).
        \item We have a canonical isomorphism of vector bundles on the Fargues--Fontaine curve $X_C$
        \[ \Es(\Hs) = H_{\FF}^1(Z)\otimes_{\Os_{X_C}} \Os_{X_C}(1),\]
        compatible with the modifications (\ref{eq: modification defining E(G)}) and (\ref{eq: dRFF curve modifies etale coh}).
    \end{enumerate}
\end{proposition}
\begin{proof}
    We start with the third point. Consider the following diagram, obtained from (\ref{eq: defining diagram for E(G)}) by twisting by $\Os_{X_C}(-1)$
    % https://q.uiver.app/#q=WzAsMTAsWzAsMCwiMCJdLFsxLDAsIlZfcFxcSHMgXFxvdGltZXNfe1xcUV9wfSBcXE9zX3tYX0N9KC0xKSJdLFsyLDAsIlxcRXMoXFxIcylcXG90aW1lc197XFxPc197WF9DfX1cXE9zX3tYX0N9KC0xKSJdLFszLDAsImlfKlxcTGllKFxcSHMpIl0sWzQsMCwiMCJdLFsxLDEsIlZfcFxcSHNcXG90aW1lc197XFxRX3B9IFxcT3Nfe1hfQ30oLTEpIl0sWzAsMSwiMCJdLFsyLDEsIlZfcFxcSHMoLTEpIFxcb3RpbWVzX3tcXFpfcH1cXE9zX3tYX0N9Il0sWzMsMSwiaV8qKFZfcFxcR3MoLTEpIFxcb3RpbWVzX3tcXFFfcH1DKSJdLFs0LDEsIjAuIl0sWzAsMV0sWzEsMl0sWzIsM10sWzMsNF0sWzEsNSwiPSIsMl0sWzYsNV0sWzUsN10sWzcsOF0sWzMsOCwiaV8qZl97XFxIc30iLDAseyJzdHlsZSI6eyJ0YWlsIjp7Im5hbWUiOiJob29rIiwic2lkZSI6InRvcCJ9fX1dLFsyLDcsIiIsMCx7InN0eWxlIjp7InRhaWwiOnsibmFtZSI6Imhvb2siLCJzaWRlIjoidG9wIn19fV0sWzgsOV1d
\[\begin{tikzcd}
	0 & {V_p\Hs \otimes_{\Q_p} \Os_{X_C}(-1)} & {\Es(\Hs)\otimes_{\Os_{X_C}}\Os_{X_C}(-1)} & {i_*\Lie(\Hs)} & 0 \\
	0 & {V_p\Hs\otimes_{\Q_p} \Os_{X_C}(-1)} & {V_p\Hs(-1) \otimes_{\Z_p}\Os_{X_C}} & {i_*(V_p\Gs(-1) \otimes_{\Q_p}C)} & {0.}
	\arrow[from=1-1, to=1-2]
	\arrow[from=1-2, to=1-3]
	\arrow["{=}"', from=1-2, to=2-2]
	\arrow[from=1-3, to=1-4]
	\arrow[hook, from=1-3, to=2-3]
	\arrow[from=1-4, to=1-5]
	\arrow["{i_*f_{\Hs}}", hook, from=1-4, to=2-4]
	\arrow[from=2-1, to=2-2]
	\arrow[from=2-2, to=2-3]
	\arrow[from=2-3, to=2-4]
	\arrow[from=2-4, to=2-5]
\end{tikzcd}\]
After identifying
\[ \Lie(\Hs) = H_{\et}^1(Z,\Os_Z), \quad V_p\Hs = H_{\et}^1(Z,\Q_p(1)),\]
we see that $\Es(\Hs)\otimes_{\Os_{X_S}}\Os_{X_S}(-1)$ is the preimage of $H_{\et}^1(Z,\Os_Z) \sub H_{\et}^1(Z,\Q_p)\otimes C$ under the projection $H_{\et}^1(Z,\Q_p) \otimes_{\Q_p}\Os_{X_C} \twoheadrightarrow i_*(H_{\et}^1(Z,\Q_p)\otimes C)$. By comparing with (\ref{eq: diagram for dRFF coh 1}), we conclude that 
\[ \Es(\Hs)\otimes_{\Os_{X_S}}\Os_{X_S}(-1) = H_{\FF}^1(Z),\]
compatibly with the modifications. This proves the third point. The second (resp. first point) follow by passing to the completed stalks (resp. by pulling back along $i\colon \Spa(C) \rightarrow X_C$).
\end{proof}

\section{Analytic $p$-divisible groups and local Shimura varieties}\label{section: EL and PEL local Shimura varieties}
We now turn to our first application to moduli spaces. In this section, we give a moduli description of the local Shimura varieties of EL and PEL type in terms of analytic $p$-divisible groups, starting with the group $\GL_n$. 
\subsection{The $\GL_n$ case}\label{subsection: The GLn case}
 
We fix a $p$-adic field $E$ and a dualizable analytic $p$-divisible group $\Gs_0$ over $\Spa(E)$ such that $V_p\Gs_0$ is a de Rham Galois representation of $\Gal_E$. 
\begin{definition}\label{def: moduli of analytic pdiv gp with frame}
We define a presheaf $\Ms_{\Gs_0}$ on the category of good adic spaces as follows
\[  \Ms_{\Gs_0}\colon S\in \Adic_E \mapsto \{\,(\Gs,\beta)\,\}/\cong,\]
where $\Gs$ is an analytic $p$-divisible group over $S$ and $\beta\colon \widetilde{\Gs} \xrightarrow{\cong} \widetilde{\Gs}_0\times_{\Spa(E)}$ is an isomorphism of group diamonds over $S$.
\end{definition}

We will also consider variants that allow for level structures. We start with the definition of level structure on analytic $p$-divisible groups.
\begin{definition}
    Let $K\sub \GL_n(\Q_p)$ be a compact open subgroup. Let $\Gs$ be an analytic $p$-divisible group of height $n$ over a good adic space $S$. 
    \begin{enumerate}
        \item A $K$-level structure on $V_p\Gs$ is a $K$-orbit $\cj{\gamma}$ of isomorphisms $\gamma \colon \Q_p^{\oplus n}\rightarrow V_p\Gs$, i.e. a global section
    \[ \cj{\gamma} \in \Gamma(S,\underline{\Isom}_{S_v}(\Q_p^{\oplus n},V_p\Gs)/\underline{K}). \]
    \item Assume that $K\sub \GL_n(\Z_p)$. A $K$-level structure on $T_p\Gs$ is defined analogously as a $K$-orbit of isomorphisms $\Z_p^{\oplus n} \rightarrow T_p\Gs$.
    \end{enumerate}
\end{definition}
For example, consider the principal congruence subgroups
\[ K(N) = \Ker(\GL_n(\Z_p) \rightarrow \GL_n(\Z/p^N\Z)), \quad N\geq 0.\]
Then a $K(N)$-level structure on $T_p\Gs$ can be represented by a classical level-$N$ structure, that is, an isomorphism of finite étale adic groups
\[ (\Z/p^N\Z)^{\oplus n} \xrightarrow{\cong} \Gs[p^N].\]

\begin{definition}
    Let $K \sub \GL_n(\Q_p)$ be a compact open subgroup. We let $\Ms_{\Gs_0,K}$ denote the presheaf
    \[  \Ms_{\Gs_0,K}\colon S\in \Adic_E \mapsto \{\,(\Gs,\beta,\cj{\gamma}) \mid (\Gs,\beta)\in \Ms_{\Gs_0}(S), \, \cj{\gamma} \text{ a } K\text{-level structure on }V_p\Gs\,\}/\cong. \]
     Here, an isomorphism $(\Gs,\beta,\cj{\gamma}) \rightarrow (\Gs',\beta',\cj{\gamma}')$ is defined to be an quasi-isogeny $\varphi\colon \Gs \rightarrow \Gs'$ such that $\beta' \circ \widetilde{\varphi} = \beta$ and $\cj{\gamma}' \circ V_p\varphi = \cj{\gamma}$.
\end{definition}

If $K$ is contained in
$\GL_n(\Z_p)$, we have a more practical description of $\Ms_{\Gs_0,K}$: It sends a good adic space $S$ to the set of isomorphism classes $(\Gs,\beta,\cj{\gamma})$ where $(\Gs,\beta)\in \Ms_{\Gs_0}(S)$ and $\cj{\gamma}$ is a $K$-level structure on the integral Tate module $T_p\Gs$. An isomorphism $(\Gs,\beta,\cj{\gamma}) \rightarrow (\Gs',\beta',\cj{\gamma}')$ is now defined to be an isomorphism $\varphi\colon \Gs \rightarrow \Gs'$ such that $\beta' \circ \widetilde{\varphi} = \beta$ and $\cj{\gamma}' \circ T_p\varphi = \cj{\gamma}$.

 \begin{proposition}\label{prop: properties of local moduli}
 Let $\Gs_0$ and $K\sub \GL_n(\Q_p)$ be as above.
     \begin{enumerate}
         \item The presheaf $\Ms_{\Gs_0,K}$ is a small $v$-sheaf on $\Adic_{E}$.
         \item For any good adic space $S$ over $E$ and any $(\Gs,\beta) \in \Ms_{\Gs_0}(S)$, $\Gs$ is dualizable and we have $\height(\Gs) = \height(\Gs_0)$ and $\dim(\Gs) = \dim(\Gs_0)$. In particular, 
         \[ \Ms_{\Gs_0} = \Ms_{\Gs_0,\GL_n(\Z_p)}. \]
         If $S$ is a smooth rigid space over $E$, $T_p\Gs$ is a de Rham local system on $S$.
         \item Let $D_0=D(\Gs_0)$ and $d=\dim \Gs_0$. There is an étale Grothendieck--Messing period map
         \[ \pi_{\GM}\colon \Ms_{\Gs_0,K} \rightarrow \Fl_{D_0,d,E},\]
         where the target is the flag variety of $d$-dimensional quotients of $D_0$, viewed as a diamond over $E$. It is defined by sending a pair $(\Gs,\beta,\cj{\gamma})$ defined over $S \in \Adic_E$ to the surjection
         \[ D_0 \otimes_E \Os_S = D(\Gs_0\times_{\Spa(E)} S) \overset{\beta}{\cong} D(\Gs) \twoheadrightarrow \Lie(\Gs). \]
    \item The sheaf $\Ms_{\Gs_0,K}$ is representable by smooth rigid space over $E$.
     \end{enumerate}
 \end{proposition}
 \begin{remark}
     In view of Remarks \ref{remark: an pdiv groups don't descent outside of perfd} and \ref{remark: dualizability does not descend well}, it may come as a surprise that $\Ms_{\Gs_0,K}$ satisfies the sheaf property along $v$-covers of arbitrary good adic spaces, not just perfectoid spaces. As a consequence, there is a universal family $\Gs_{\univ} \rightarrow \Ms_{\Gs_0,K}$, together with a de Rham local system $T_p\Gs_{\univ}$ on $\Ms_{\Gs_0,K}$.
 \end{remark}
%  \begin{remark}
%     If $K\sub \GL_n(\Q_p)$ is an arbitrary compact open subgroup, we may also define a moduli $\Ms_{\Gs_0,K}$, which sends $S$ to the set of tuples $(\Fs,\beta,\cj{\gamma})$, where $\Fs$ is an effective Banach--Colmez space, $\beta\colon \Fs \cong \Gs_0$ is a (not necessarily filtered) isomorphism and $\cj{\gamma}$ is a $K$-level structure on $V_p\Fs$, that is, a $K$-orbit of isomorphism $V_p\Fs \xrightarrow{\cong} \Q_p^{\oplus n}$. An analog of the Proposition holds for this moduli space. The proofs are easily seen to adapt -- alternatively, one can use that, for any compact open subgroup $K$, a conjugate of $K$ is contained in $\GL_n(\Z_p)$, which allows to reduce to this case.
% \end{remark}
 \begin{proof}
     \begin{enumerate}
\item Up to replacing $K$ by a conjugate subgroup, we may assume that $K\sub \GL_n(\Z_p)$. Let $S' \rightarrow S$ be a $v$-cover of good adic spaces and let $S'' = S'\times_S S'$ with projections $p_1,p_2\colon S'' \rightarrow S'$. Let $(\Gs',\beta',\cj{\gamma}') \in \Ms_{\Gs_0,K}(S')$ be given together with an isomorphism $\sigma\colon p_1^*(\Gs',\beta',\cj{\gamma}')\cong p_2^*(\Gs',\beta',\cj{\gamma}')$, automatically satisfying the cocycle conditions, as the category of such tuples is easily seen to be rigid. By descent, we obtain a tuple $(\Gs,\beta,\cj{\gamma})$, where $\Gs$ is a $v$-sheaf of abelian groups on $S$, $\beta$ is an isomorphism $\widetilde{\Gs} \cong \widetilde{\Gs}_0 \times_{\Spa(E)} S$ and $\cj{\gamma}$ is a $K$-level structure on $V_p\Gs$. We need to show that $\Gs$ is representable by an analytic $p$-divisible group.

We may assume without loss of generality that $S'$ is perfectoid. As an intermediary step, we show that $\Gs'$ is dualizable. It is equivalent to show that the effective Banach--Colmez space $
V_p\Gs'\rightarrow \widetilde{\Gs}' \rightarrow \Lie(\Gs')$ is dualizable. But by Proposition \ref{prop: G dualizable iff E(G) vb}, this only depends on the isomorphism class of the abelian sheaf $\widetilde{\Gs}'$, hence, it follows from the fact that $\Gs_0$ is dualizable. 

The tuple $(T_p\Gs', \Lie(\Gs'),f_{\Gs'})$ associated to $\Gs'$ clearly descends to a tuple $(\Lb,L,f)$ on $S$, consisting of a $\Z_p$-local system $\Lb$, a $v$-vector bundle $L$ and a map of $v$-vector bundle $f\colon L \rightarrow \Lb(-1)\otimes_{\Z_p} \Os_{S_v}$. The map $f$ is a $v$-locally direct summand since it holds true after base-change to $S'$. Consequently, if we let
\[ \omega = \Coker(f(1)\colon L(1) \rightarrow \Lb \otimes_{\Z_p} \Os_{S_v}),\]
then $\omega$ is a $v$-vector bundle. Consider now the sheaf $D$ obtained by descending the bundle $D(\Gs')$ along the $v$-cover $S' \rightarrow S$. It comes with a sequence of $v$-vector bundles
% https://q.uiver.app/#q=WzAsNSxbMSwwLCJcXG9tZWdhIl0sWzIsMCwiRCJdLFszLDAsIkwiXSxbMCwwLCIwIl0sWzQsMCwiMCwiXSxbMywwXSxbMCwxXSxbMiw0XSxbMSwyXV0=
\[\begin{tikzcd}
	0 & \omega & D & L & {0,}
	\arrow[from=1-1, to=1-2]
	\arrow[from=1-2, to=1-3]
	\arrow[from=1-3, to=1-4]
	\arrow[from=1-4, to=1-5]
\end{tikzcd}\]
obtained from descending the sequence (\ref{eq: Hodge filtration}). But since the de Rham module $D(\Gs')$ is functorial in the $p$-adic universal cover $\widetilde{\Gs}'$, the sheaf $D\cong D(\Gs_0)\otimes_E \Os_S$ is an étale vector bundle. By Lemma \ref{lemma: quotient of étale vb is étale vb} below, we conclude that $L$ and $\omega$ are étale vector bundles as well. By Theorem \ref{thm: extending Fargues' equivalence of categories}, the tuple $(\Lb,L,f)$ determines a dualizable analytic $p$-divisible group, which is easily seen to represent $\Gs$, as required.

\item That $\Gs$ is dualizable was already proven in the previous step. The height and dimension of $\Gs$ are locally constant on $S$, so that they can be read at a single point. Hence, we may assume that $S=\Spa(C,\Os_C)$. In that case, by Scholze--Weinstein's result, Theorem \ref{thm: Scholze Weinstein's result}, $\Gs$ and $\Gs_0$ have good reduction $\mathfrak{G}$ and $\mathfrak{G}_0$ over $\Os_C$. By \cite[Lemme 3.2, Théorème 3.3]{Farg22}, the isomorphism $\beta\colon\widetilde{\Gs} \cong \widetilde{\Gs}_0$ over $\Spa(C)$ uniquely arises from a quasi-isogeny $\mathfrak{G} \otimes_{\Os_C} \Os_C/p \rightarrow \mathfrak{G}_0 \otimes_{\Os_C} \Os_C/p$. Since the height and dimensions can be read on these groups, we win. Finally, as we saw in the proof of the first point, $D(\Gs) \cong D(\Gs_0)\otimes_E \Os_S$ is an étale vector bundle. Hence, if $S$ is a smooth rigid space over $E$, by Proposition \ref{prop: VpG de rham vs MG étale vb}, $T_p\Gs$ is a de Rham local system on $S$.
\item To see that the map $\pi_{\GM}$ is étale, we follow the strategy of the proof of \cite[Prop. 23.3.3]{SW20}. Fix a map $h\colon S \rightarrow \Fl_{D_0,d,E}$ from a strictly totally disconnected perfectoid space $S$, corresponding to a modification $\alpha\colon \Es \dashrightarrow \Es(\Gs_0)$ on the Fargues--Fontaine curve $X_S$, using Lemma \ref{lemma: minuscule lattices vs flags} and the equivalence (\ref{eq: from modif to lattices}). By the equivalence (\ref{eq: equiv between ls and ss slope zero}) and Theorem \ref{thm: equivalence for dualizable groups on perfd}, the fiber of $\pi_{\GM}$ over $h$ then identifies with the projection
\[ \GL_n(\Q_p)/K \times S^a \rightarrow S,\]
where $S^a\sub S$ is the locus of points $s\in S$ where $\Es\restr{X_s}$ is trivial. This is an open subspace, by \cite[Thm. 22.6.2]{SW20}, which shows that $\pi_{\GM}$ is étale.
\item This follows from the previous point and \cite[Lemma 15.6]{scholze2022etale}.
\end{enumerate}
 \end{proof}

 The following lemma was used in the above proof.
\begin{lemma}\label{lemma: quotient of étale vb is étale vb}
    Let $S$ be a good adic space over $\Q_p$ and assume we are given a short exact sequence of $v$-vector bundles on $S$
    % https://q.uiver.app/#q=WzAsNSxbMSwwLCJcXG9tZWdhIl0sWzIsMCwiRCJdLFszLDAsIkwiXSxbMCwwLCIwIl0sWzQsMCwiMC4iXSxbMywwXSxbMCwxXSxbMiw0XSxbMSwyXV0=
\[\begin{tikzcd}
	0 & \omega & D & L & {0.}
	\arrow[from=1-1, to=1-2]
	\arrow[from=1-2, to=1-3]
	\arrow[from=1-3, to=1-4]
	\arrow[from=1-4, to=1-5]
\end{tikzcd}\]
If $M$ is an étale vector bundle, then so are $\omega$ and $L$.
\end{lemma}
\begin{proof}
We adapt the proof of \cite[Prop. 3.2]{heuer2021line}. It suffices to show that $L$ is an étale vector bundle, by passing to duals. We have a commutative diagram of morphism of $v$-sheaves
% https://q.uiver.app/#q=WzAsNCxbMCwxLCJNIl0sWzEsMSwiTC4iXSxbMCwwLCJcXG51XypNXFxvdGltZXNfe1xcT3Nfe1xcZXR9fVxcT3NfdiJdLFsxLDAsIlxcbnVfKkxcXG90aW1lc197XFxPc197XFxldH19XFxPc192Il0sWzAsMSwiIiwwLHsic3R5bGUiOnsiaGVhZCI6eyJuYW1lIjoiZXBpIn19fV0sWzIsMCwiXFxjb25nIiwyXSxbMywxXSxbMiwzXV0=
\[\begin{tikzcd}
	{\nu_*M\otimes_{\Os_{\et}}\Os_v} & {\nu_*L\otimes_{\Os_{\et}}\Os_v} \\
	M & {L.}
	\arrow[from=1-1, to=1-2]
	\arrow["\cong"', from=1-1, to=2-1]
	\arrow[from=1-2, to=2-2]
	\arrow[two heads, from=2-1, to=2-2]
\end{tikzcd}\]
It follows that the right vertical map is surjective. To conclude, it remains to show that $\nu_*L$ is an étale vector bundle of rank $d$, where $d=\rank L$. This can be checked locally, so we may assume that $S=\Spa(A,A^+)$ is affinoid and that $M=\Os_S^{\oplus n}$ is free. It is enough to produce a standard-étale cover $\{ U_i \rightarrow S\}$ such that $L(U_i)$ is a finite free $\Os(U_i)$-module and $L(U_i)\otimes_{\Os(U_i)}\Os(U_{ij}) \cong L(U_{ij})$, for any $i,j$.

Consider a pro-finite-étale cover 
\[ \widetilde{S} = \varprojlim_i S_i \rightarrow S\] 
by an affinoid perfectoid space $\widetilde{S} = \Spa(\widetilde{A},\widetilde{A}^+)$, see e.g \cite[Lemma 15.3]{scholze2022etale}. Let $\pi$ denote the Galois group. As $\widetilde{S}$ is perfectoid, $L\restr{\widetilde{S}}$ is trivialized by standard-étale cover $\widetilde{S'} \rightarrow \widetilde{S}$. By \cite[Lemma 11.23]{scholze2022etale}, we have $\widetilde{S}_{\et, \qcqs} \cong \tworlim_i S_{i,\et,\qcqs}$, so that this cover arises via pullback from a standard-étale cover $S_i' \rightarrow S_i$. Up to replacing $S$ by $S_i'$, we may assume that $L\restr{\widetilde{S}}$ is trivial. Therefore, there exists a continuous group cocycle $c\colon \pi \rightarrow \GL_d(\widetilde{A})$ such that
\[ L(S) = \{\,t \in \widetilde{A}^{\oplus d} \mid \sigma^*(t) = c(\sigma)t, \, \fa \sigma \in \pi\,\}.\]
Let $t_1,\ldots, t_n\in L(S)$ denote the image of the generators $e_j\in M(S) \cong A^{\oplus n}$. Let $x\in \widetilde{S}$ with corresponding geometric point $\Spa(C,C^+) \rightarrow \widetilde{S}$. Then there exists distinct indices $1\leq j_1,\ldots, j_d  \leq n$ such that the vectors
\[ t_{j_1}(x),\ldots, t_{j_d}(x) \in C^{\oplus d} \]
form a basis of $L(C,C^+)$. Therefore, the map
\[ \det(t_{j_l})_{l=1}^{d}\colon \widetilde{S} \rightarrow \G_a\]
is non-vanishing at $x$. We may write
\[ \Spa(C,C^+) =\varprojlim_k \widetilde{U}_k\]
where the limit ranges over standard-étale neighborhood of $x$ in $\widetilde{S}$. By \cite[Lemma 3.13]{gerth2024}, we conclude that already the restriction
\[ \det(t_{j_l})\restr{\widetilde{U}_k}\colon \widetilde{U}_k \rightarrow \widetilde{S} \rightarrow \A^1
\]
has image in $\G_m \sub \A^1$, for some $k$. Using again \cite[Lemma 11.23]{scholze2022etale}, $\widetilde{U}_k$ comes via pullback from a standard-étale map $U_{k,i} \rightarrow S_i$ for some $i$. Then $t_{j_1},\ldots, t_{j_d}$ form a basis of the $\Os(U_{k,i})$-module $L(U_{k,i})$. Repeating this procedure at each point $x\in \widetilde{S}$ yields a standard-étale cover of $S$ with the desired properties, which concludes the proof.
    \end{proof}
% The tower of smooth rigid spaces $(\Ms_{\Gs_0,K})_{K \sub \GL_n(\Q_p)}$ of Proposition \ref{prop: properties of local moduli} carries a natural $\GL_n(\Q_p)$-action. % This means the following. For any inclusion $K\sub K'$ of compact open subgroups, there is a finite étale projection
% % \[ \Ms_{\Gs_0,K} \rightarrow \Ms_{\Gs_0,K'}, \quad (\Gs,\beta,\cj{\gamma}) \mapsto (\Gs,\beta,\cj{\gamma} K').\]
% % For any $g\in \GL_n(\Q_p)$, there is an isomorphism
% % \[  \Ms_{\Gs_0,K} \rightarrow \Ms_{\Gs_0,g^{-1}Kg}, \quad (\Gs,\beta,\cj{\gamma}) \mapsto (\Gs,\beta,\cj{\gamma} g).\]
% % Moreover, these isomorphisms satisfy the obvious compatibilities. 
% In particular, we may consider the moduli space at infinite level
% \[ \Ms_{\Gs,\infty} = \varprojlim_{K \sub \GL_n(\Q_p)} \Ms_{\Gs,K},\]
% viewed as a diamond. It has the following moduli interpretation: For any $S\in\Adic_E$
% \[ \Ms_{\Gs_0,\infty}(S) = \{(\Gs,\beta,\gamma) \mid (\Gs,\beta)\in \Ms_{\Gs_0}(S), \, \gamma\colon T_p\Gs \xrightarrow{\cong }\Z_p^{\oplus n} \, \}/\cong.\]

We now set up to relate the tower $(\Ms_{\Gs_0,K})_{K \sub \GL_n(\Q_p)}$ of moduli spaces from Definition \ref{def: moduli of analytic pdiv gp with frame} with the local Shimura varieties of Scholze--Weinstein \cite[Def. 24.1.3]{SW20}. Let $0 \leq d \leq n$ and let $\mu_d = (z, \ldots, z, 1,\ldots, 1)$ be the standard cocharacter of $\GL_n$ with $z$ appearing $d$ times. Let $b\in B(\GL_n,\mu_d^{-1})$ be an element of the Kottwitz set that is $\mu$-admissible (cf. \cite[Def. 24.1.1]{SW20}). To the data $(G,\mu_d,b)$, Scholze--Weinstein associate the local Shimura varieties $(\Sht_{\GL_n,b,\mu_d,K})_{K \sub \GL_n(\Q_p)}$, which are smooth rigid spaces over $\breve{\Q}_p$ with the following diamonds
\[ \Sht_{\GL_n,b,\mu_d,K}^{\diamondsuit}\colon S \in \Perf_{\breve{\Q}_p} \mapsto \{(\Lb,\alpha \colon \Lb \otimes_{\Z_p} \Os_{X_S} \dashrightarrow \Es^b) \}/\cong.\]
Here, $\Lb$ is a $\Z_p$-local system on $S$ and $\alpha $ is a modification on the Fargues--Fontaine curve $X_S$ whose relative position is given by $\mu_d$. Here, $\Es^b$ is the vector bundle on $X_S$ corresponding to the element $b$ under the equivalence (\ref{eq: from Kotwitz set to bundles}).

By Dieudonné theory, there exists a $p$-divisible group $G_0$ over $\cj{\F}_p$ of dimension $d$ and height $n$, unique up to isogeny, so that $b$ is the element of the Kottwitz set corresponding to the Dieudonné isocrystal of $G_0$. We fix any lift $\mathfrak{G}_0$ of $G_0$ to a $p$-divisible group over $\breve{\Z}_p$ and we let $\Gs_0$ denote its generic fiber, an analytic $p$-divisible group over $\breve{\Q}_p$. Such a lift always exists since the deformation problem is formally smooth. In this case, by Corollary \ref{cor: E(G) = E(D,phi) over padic field}, we have a canonical isomorphism 
\[ \Es(\Gs_0) = \Es^b \in \Bun_{\GL_n}(\Spd(\breve{\Q}_p)) \]
and we will identify these bundles in what follow. 

We arrive at the following theorem.

\begin{thm}\label{main thm: moduli of an pdiv vs local SV}
There exists a canonical isomorphism of smooth rigid spaces, for any compact open subgroup $K\sub \GL_n(\Q_p)$
\[ \Ms_{\Gs_0,K}\cong \Sht_{\GL_n,b,\mu_d,K},\]
compatible with the $G(\Q_p)$-actions and Grothendieck--Messing period maps. For $K=\GL_n(\Z_p)$ and $S \in \Perf_{\breve{\Q_p}}$, this isomorphism sends a pair $(\Gs,\beta\colon \widetilde{\Gs} \cong \widetilde{\Gs}_0) \in \Ms_{\Gs_0}(S)$ to
\[ T_p\Gs \otimes_{\Z_p} \Os_{X_S} \dashrightarrow \Es(\Gs) \overset{\beta}{\cong} \Es(\Gs_0),\]
where the left map is the modification given by Theorem \ref{thm: vb on FF associated to an pdiv groups} and the right map is the unique isomorphism induced by $\beta$ under the equivalence of Proposition \ref{prop: minuscule sheaves vs minuscule BC spaces}.
\end{thm}
\begin{proof}
As both towers $(\Ms_{\Gs_0,K})_{K}$ and $(\Sht_{\GL_n,b,\mu_d,K})_{K}$ admit an action of $\GL_n(\Q_p)$, up to replacing $K$ by a conjugate subgroup, we may assume that $K\sub \GL_n(\Z_p)$. In this case, $\Sht_{\GL_n,b,\mu_d,K}$ can be seen to parametrize objects $(\Lb,\alpha,\cj{\gamma})$, where $(\Lb,\alpha) \in \Sht_{\GL_n,b,\mu_d,\GL_n(\Z_p)}$ and $\cj{\gamma}$ is a $K$-orbit of isomorphism $\Z_p^{\oplus n} \rightarrow \Lb$. A similar description holds on the other side, hence it is enough to treat the case $K=\GL_n(\Z_p)$. Let $S\in \Perf_{\breve{\Q}_p}$ and $(\Lb,\alpha) \in \Sht_{\GL_n,b,\mu_d,\GL_n(\Z_p)}(S)$. Let
\[ \iota \colon \Lb \otimes_{\Z_p} \B_{\dR}^+ \hookrightarrow \Xi^b[\tfrac{1}{\xi}]\]
denote the map obtained from $\alpha$ by passing to the completed stalk at $S \hookrightarrow X_S$. Then the condition that the modification $\alpha$ has relative position $\mu_d$ translates into the fact that, pointwise on $S$, $\iota$ becomes identified with the inclusion
\[ \bigoplus_{i=1}^{d} \BdR+(C)\xi^{r_i}e_i \sub \bigoplus_{i=1}^{n}\BdeR(C)e_i,\]
where $r_i$ is equal to $1$ if $i \leq d$ and is $0$ otherwise. It follows that $\iota$ is minuscule in the sense of Definition \ref{def: minuscule lattice} and that the functor in the statement is well-defined. To see that it is an equivalence, we use that, by Theorem \ref{thm: equivalence for dualizable groups on perfd}, the pair $(\Lb, \alpha)$ uniquely determines a dualizable analytic $p$-divisible group $\Gs \rightarrow S$ with Tate module $\Lb$ such that the modification $\alpha$ coincides with the sequence (\ref{eq: modification defining E(G)})
% https://q.uiver.app/#q=WzAsNSxbMCwwLCIwIl0sWzEsMCwiVF9wXFxHc1xcb3RpbWVzX3tcXFpfcH1cXE9zX3tYX1N9Il0sWzIsMCwiXFxFcyhcXEdzKSJdLFszLDAsImlfKlxcTGllKFxcR3MpIl0sWzQsMCwiMC4iXSxbMCwxXSxbMSwyXSxbMiwzXSxbMyw0XV0=
\[\begin{tikzcd}
	0 & {T_p\Gs\otimes_{\Z_p}\Os_{X_S}} & {\Es(\Gs)} & {i_*\Lie(\Gs)} & {0.}
	\arrow[from=1-1, to=1-2]
	\arrow[from=1-2, to=1-3]
	\arrow[from=1-3, to=1-4]
	\arrow[from=1-4, to=1-5]
\end{tikzcd}\]
Moreover, by Proposition \ref{prop: minuscule sheaves vs minuscule BC spaces}, the resulting isomorphism $\Es(\Gs) \cong \Es^b =\Es(\Gs_0)$ comes from a uniquely determined isomorphism $\beta\colon \widetilde{\Gs} \cong \widetilde{\Gs}_0$. Hence we obtain a well-defined object $(\Gs,\beta) \in \Ms_{\Gs_0}(S)$. It follows from Theorem \ref{thm: equivalence for dualizable groups on perfd} that isomorphisms of objects on each side correspond to each other, from which we derive that the map under consideration defines an isomorphism of diamonds. Since both sheaves are representable by smooth rigid spaces, it follows from the fully faithful embedding (\ref{eq: diamond functor ff on good adic spaces}) that this isomorphism arises uniquely from an isomorphism of adic spaces, which concludes.
\end{proof}
\begin{remark}\label{remark: unwrapping iso with rapoport-zink space}
    Let us consider the Rapoport--Zink space $\RZ_{G_0}$ associated with $G_0$. This is a smooth formal scheme over $\Spf(\breve{\Z}_p)$ representing the functor on $p$-complete $\breve{\Z}_p$-algebras
    \[ A  \longmapsto \{\, (G,\rho)\,\}/\cong, \]
    where $G$ is a $p$-divisible group over $A$ and $\rho\colon G \otimes_A A/p \dashrightarrow G_0 \otimes_{\cj{\F}_p}A/p$ is a quasi-isogeny. By \cite[Thm. 24.2.5]{SW20}, the adic generic fiber $(\RZ_{G_0})_{\eta}$ is canonically isomorphic to the local Shimura variety $\Sht_{\GL_n,b,\mu_d,\GL_n(\Z_p)}$. Combined with Theorem \ref{main thm: moduli of an pdiv vs local SV}, this yields an isomorphism of smooth rigid spaces
    \begin{align}\label{eq: moduli of an pdiv vs RZ space} 
    (\RZ_{G_0})_{\eta}\cong \Ms_{\Gs_0}.
    \end{align}
    This isomorphism can be read on functor of points as follow, for a good adic space $S=\Spa(R,R^+)$ over $\breve{\Q}_p$
    \[ (\mathfrak{G},\rho) \in \RZ_{G_0}(R^+) \longmapsto ( \mathfrak{G}_{\eta},\widetilde{\rho}_{\eta}) \in \Ms_{\Gs_0}(S).\]
    Here,
    \[ \widetilde{\rho}\colon \widetilde{\mathfrak{G}} \rightarrow \widetilde{\mathfrak{G}}_{0}\widehat{\otimes}_{\breve{\Z}_p} R^+\]
    is the unique isomorphism deforming $\varprojlim_{[p]}\rho$, using the crystal property of the universal cover $\widetilde{\mathfrak{G}}$. 
    \end{remark}
    \begin{remark}\label{remark: good reduction locally}
         The isomorphism (\ref{eq: moduli of an pdiv vs RZ space}) implies that, for any good adic space $S$ over $\breve{\Q}_p$ and any $(\Gs,\beta) \in \Ms_{\Gs_0}(S)$, the group $\Gs$ has good reduction analytic-locally on $\vert S \vert$. Note that if $S$ is a rigid space with smooth formal model $\Ss$ of $S$ over $\breve{\Z}_p$, the related question of the good reduction of $\Gs$ Zariski-locally on $\vert \Ss \vert$ was studied by Fargues in \cite[§5.3]{Farg22}.
    \end{remark}

\subsection{The EL and PEL case}\label{The EL and PEL case}
We now extend Theorem \ref{main thm: moduli of an pdiv vs local SV} to the EL and PEL case. We start with one of the following situations.
\begin{itemize}
    \item (EL case) Let $B$ be a semisimple $\Q_p$-algebra, and let $V$ be a finite left $B$-module.
    \item (PEL case) In addition to the above, we assume that $p\neq 2$, $B$ is equipped with an anti-involution $(\cdot)^{*} \colon B \rightarrow B$ and $V$ is endowed with a non-degenerate alternating pairing satisfying
    \[ (dv,v') = (v,d^*v'),\]
    for all $d\in B$, $v,v'\in V$.
\end{itemize}
% We also fix some integral structures. Let $\Os_B\sub B$ be a maximal order and let $\Lambda \sub V$ be a $\Os_B$-lattice. In the PEL case, we further assume that $\Os_B$ is stable under the involution and that $\Lambda$ is self-dual for the pairing on $V$.

In the EL case, we let $G=\GL_B(V)$, viewed as an algebraic group over $\Q_p$. In the PEL case, we let $G$ be the algebraic group of symplectic similitudes of $V$, which has the following $R$-rational points, for $R$ a $\Q_p$-algebra
\[ G(R) = \{\,g \in \GL_B(V\otimes_{\Q_p}R) \mid (gv,gv') = c(g)(v,v'), \, c(g)\in R^{\times}\}.\]
This defines the similitude character $c\colon G \rightarrow \G_m$.

We fix a conjugacy class $[\mu]$ of minuscule cocharacters $\mu\colon \G_m \rightarrow G_{\cj{\Q}_p}$ with only weights $0,1$. We let $E$ denote the field of definition of $[\mu ]$. We thus have a weight decomposition
\[ V\otimes_{\Q_p} E = V_0 \oplus V_1.\]
In the PEL case, we further assume that $c \circ \mu =\id_{\G_m}$. Equivalently, the $B$-linear subspaces $V_0,V_1$ are isotropic for the pairing on $V\otimes_{\Q_p} E$.

We now fix a $p$-divisible group $G_0$ over $\cj{\F}_p$ with an action $\iota\colon B \rightarrow \End^0(G_0) = \End(G_0)\otimes_{\Z_p}\Q_p$ such that there exists an isomorphism of $B\otimes_{\Q_p}\breve{\Q}_p$-modules
\[ D(G_0)\otimes_{\breve{\Z}_p} \breve{\Q}_p \cong V \otimes_{\Q_p} \breve{\Q}_p,\]
where $D(G_0)$ is the Dieudonné crystal of $G_0$. In the PEL case, we fix a quasi-polarization $\lambda_0\colon G_0 \rightarrow G_0^D$, compatible with the involution on $B$, and we ask for the above isomorphism to respect the pairings. We may identify the Frobenius on the left hand side with $b(\id\otimes \varphi)$ for some $b \in G(\breve{\Q}_p)$, where $\varphi$ is the Frobenius on $\breve{\Q}_p$. Then this yields a well-defined element $b$ in the Kottwitz set $B(G)$.

Let $\breve{E} = E \cdot \breve{\Q}_p$. We further assume that the group $G_0$ together with its $B$-action (and its quasi-polarization) admit a deformation $\mathfrak{G}_0$ over $\Spf(\Os_{\breve{E}})$ such that the Hodge filtration
% https://q.uiver.app/#q=WzAsNSxbMSwwLCJcXG9tZWdhX3tcXG1hdGhmcmFre0d9XzBeRH1bXFx0ZnJhY3sxfXtwfV0iXSxbMiwwLCJEKFxcbWF0aGZyYWt7R31fMClbXFx0ZnJhY3sxfXtwfV0iXSxbMywwLCJcXExpZShcXG1hdGhmcmFre0d9XzApW1xcdGZyYWN7MX17cH1dIl0sWzAsMCwiMCJdLFs0LDAsIjAiXSxbMywwXSxbMSwyXSxbMCwxXSxbMiw0XV0=
\[\begin{tikzcd}
	0 & {\omega_{\mathfrak{G}_0^D}[\tfrac{1}{p}]} & {D(\mathfrak{G}_0)[\tfrac{1}{p}]} & {\Lie(\mathfrak{G}_0)[\tfrac{1}{p}]} & 0
	\arrow[from=1-1, to=1-2]
	\arrow[from=1-2, to=1-3]
	\arrow[from=1-3, to=1-4]
	\arrow[from=1-4, to=1-5]
\end{tikzcd}\]
coincides with the weight decomposition\footnote{Here, we follow the convention of Scholze--Weinstein. Note that the weights $0$ and $1$ are interchanged in comparison with Rapoport--Zink's book, cf \cite[Above Def. 21.6.8]{SW20}.}
% https://q.uiver.app/#q=WzAsNSxbMSwwLCJWXzBcXG90aW1lc19FIFxcYnJldmV7RX0iXSxbMiwwLCJWIFxcb3RpbWVzX3tcXFFfcH0gXFxicmV2ZXtFfSJdLFszLDAsIlZfMVxcb3RpbWVzX0UgXFxicmV2ZXtFfSJdLFswLDAsIjAiXSxbNCwwLCIwLiJdLFszLDBdLFsxLDJdLFswLDFdLFsyLDRdXQ==
\[\begin{tikzcd}
	0 & {V_0\otimes_E \breve{E}} & {V \otimes_{\Q_p} \breve{E}} & {V_1\otimes_E \breve{E}} & {0.}
	\arrow[from=1-1, to=1-2]
	\arrow[from=1-2, to=1-3]
	\arrow[from=1-3, to=1-4]
	\arrow[from=1-4, to=1-5]
\end{tikzcd}\]
This ensures that $b$ is $\mu$-admissible, i.e. $b\in B(G,\mu^{-1})$, cf \cite[(3.19)]{rapoportzink96}\footnote{It is not clear whether $G_0$ always admits such a lift. However, by the rigidity of $p$-divisible groups up to quasi-isogeny, the universal cover $\widetilde{G}_0$ admits a unique deformation to a formal group scheme $\mathfrak{F}_0$ with $B$-action over $\Spf(\Os_{\breve{E}})$. With appropriate care, one can give a variant of Definition \ref{def: moduli of an pdiv gps in EL and PEL case} below purely in terms of the abelian sheaf $\mathcal{F}_0 = \mathfrak{F}_{0,\eta}$, after observing that this construction only depends on the universal cover $\widetilde{\Gs}_0$.}. 

We now let $\Gs_0 \rightarrow \Spa(\breve{E})$ denote the generic fiber of $\mathfrak{G}_0$, a dualizable analytic $p$-divisible group with a $B$-action $\iota_0 \colon B \rightarrow \End^0(\Gs_0)$ and, in the PEL case, a quasi-polarization $\lambda_0\colon \Gs_0 \rightarrow \Gs_0^D$ (cf. Definition \ref{def: polarization}).
% We denote by $\Ds = (B,V,\mu,b,\Gs_0,\iota_0 (,\lambda_0))$ denote all the data
% \begin{definition}
%     Let $S$ be a good adic space over $\breve{E}$. An analytic $p$-divisible group with $G$-structure is a tuple $(\Gs,\iota, (\lambda))$, where
%     \begin{enumerate}
%         \item $\Gs$ is a dualizable analytic $p$-divisible group,
%         \item $\iota\colon B \rightarrow \End^0(B)$ is a $\Q_p$-linear homomorphism
%          \item (In the PEL case) $\lambda\colon \Gs \rightarrow \Gs^D$ is a 
%     \end{enumerate}
% \end{definition}

We are now ready to define the moduli problem. 
\begin{definition}\label{def: moduli of an pdiv gps in EL and PEL case}
    Let $K \sub G(\Q_p)$ be a compact open subgroup. We define a moduli problem $\Ms_{\Gs_0,\iota_0,(\lambda_0),K}$, taking a good adic space $S\in \Adic_{\breve{E}}$ to the set of isomorphism classes of tuplets $(\Gs,\iota, (\lambda),\beta, \cj{\gamma})$, where
    \begin{itemize}
        \item $\Gs$ is a dualizable analytic $p$-divisible group over $S$,
        \item $\iota\colon B \rightarrow \End_S^0(\Gs)$ is a $\Q_p$-algebra homomorphism, 
         \item (in the PEL case) $\lambda \colon \Gs \rightarrow \Gs^D$ is a quasi-polarization, such that for $d\in B$, we have $\lambda^{-1} \circ \iota(d)^D \circ \lambda = \iota(d^*)$,
         \item $\beta\colon \widetilde{\Gs} \rightarrow \widetilde{\Gs}_0 \times_{\Spa(\breve{E})} S$ is a $B$-linear isomorphism. In the PEL case, we ask that $\beta^D \circ \widetilde{\lambda}_0 \circ \beta = c\widetilde{\lambda}$, for $c\in \underline{\Q_p^{\times}}(S)$,
         \item $\cj{\gamma}$ is a $K$-orbit of isomorphisms of $B$-local systems $V \xrightarrow{ \cong} V_p\Gs$ which is a symplectic similitude in the PEL case\footnote{To make this precise, we need to identify the targets $\Q_p\cong \Q_p(1)$ of the pairings. Such an isomorphism exists $v$-locally and is unique up to a factor in $\Q_p^{\times}$}.
    \end{itemize} 
   We further assume that the map $\iota$ satisfies the Kottwitz condition:
    \begin{align} \det(d; \Lie(\Gs)) = \det(d; V_1 \otimes_{E} \Os_S),\end{align}
    as polynomial functions in $d\in B$. An isomorphism $(\Gs,\iota,(\lambda),\beta, \cj{\gamma}) \rightarrow (\Gs',\iota', (\lambda'),\beta', \cj{\gamma}')$ is defined to be a $B$-linear quasi-isogeny $\varphi\colon \Gs \rightarrow \Gs'$ such that $\beta'\circ \widetilde{\varphi} = \beta$ and $V_p\varphi \circ \cj{\gamma} = \cj{\gamma}'$. In the PEL case, we also ask that $\varphi^D \circ \lambda' \circ \varphi = c\lambda$, for $c\in \underline{\Q_p^{\times}}(S)$.
\end{definition}

We can simplify the moduli description at the cost of introducing integral data. Let $\Os_B\sub B$ be a maximal order and let $\Lambda \sub V$ be an $\Os_B$-lattice. In the PEL case, we assume furthermore that $\Os_B$ is stable under the involution on $B$ and that $\Lambda$ is self-dual with respect to the pairing on $V$. Let us assume that $K\sub \Stab_{G(\Q_p)}(\Lambda)$. In that case, $\Ms_{\Gs_0,\iota_0,(\lambda_0),K}(S)$ parametrizes the tuples $(\Gs,\iota,(\lambda),\beta, \cj{\gamma})$, where 
\begin{itemize}
    \item $\Gs$ is an analytic $p$-divisible group over $S$,
    \item $\iota\colon \Os_B \rightarrow \End_S(\Gs)$ is a $\Z_p$-linear map, satisfying the Kottwitz condition,
    \item (PEL case) $\lambda \colon \Gs \rightarrow \Gs^D$ is a principal polarization, compatible with the involutions,
    \item $\beta\colon \widetilde{\Gs} \xrightarrow{\cong} \widetilde{\Gs}_0$ is a $B$-linear isomorphism. In the PEL case, we ask that $\beta^D \circ \widetilde{\lambda}_0 \circ \beta = c\widetilde{\lambda}$, for $c\in \underline{\Q_p^{\times}}(S)$, and
    \item $\cj{\gamma}$ is a $K$-orbit of isomorphism of $\Os_B$-local systems $\Lambda \xrightarrow{ \cong} T_p\Gs$, which is a symplectic similitude in the PEL case.
    \end{itemize}  
    An isomorphism $(\Gs,\iota,(\lambda),\beta, \cj{\gamma}) \rightarrow (\Gs',\iota',(\lambda'),\beta', \cj{\gamma}')$ is defined to be an isomorphism of $p$-divisible groups $\varphi\colon \Gs \rightarrow \Gs'$ that is compatible with the extra data. In the PEL case, the map $\varphi$ pulls back $\lambda'$ to $\lambda$ up to a factor in $\underline{\Z_p^{\times}}(S)$.

    It follows from Proposition \ref{prop: properties of local moduli}(1) that the presheaves $\Ms_{\Gs_0,\iota_0,(\lambda_0),K}$ are small $v$-sheaves on $\Adic_{\breve{E}}$. We also have a Grothendieck--Messing period map
    \begin{align} \pi_{\GM} \colon \Ms_{\Gs_0,\iota_0,(\lambda_0),K} \rightarrow \Fl_{\GM}^{\diamondsuit}, \end{align}
    where the right-hand side is the flag variety over $\Spa(\breve{E})$ whose $S = \Spa(R,R^+)$-rational points parametrizes $B$-equivariant quotients $V\otimes_{\Q_p} R \twoheadrightarrow L$ such that, locally on $S$, $L$ is isomorphic to $V_1 \otimes_{E} R$ as a $B\otimes_{\Q_p} R$-module. In the PEL case, we further ask that $L$ is isotropic. This admits an interpretation as a $\B_{\dR}^+$-affine Grassmannian. Namely, we have a natural isomorphism
    \[ \Fl_{\GM}^{\diamondsuit} \cong \Gr_{\Xi_0,\mu,\Spd(\breve{E})}, \]
    where the right-hand side is the $v$-sheaf parametrizing $B\otimes_{\Q_p} \B_{\dR}^+$-lattices $\Xi \sub \Xi_0[\xi^{-1}]$ of relative position $\mu$. Here, $\Xi_0 = \Xi_0(\Gs_0)$ is the $\BdR+$-lattice associated with $\Gs_0$ from Corollary \ref{cor: BdR+ lattice assoc to pdiv groups in general}. Explicitely, for $S \in \Perf_{\breve{E}}$, the $S$-rational points of this sheaf consist in a $B\otimes_{\Q_p} \B_{\dR}^+(S)$-lattice $\xi \Xi_0 \sub \Xi \sub \Xi_0$ such that the quotient $L = \Xi_0/\Xi$ is locally on $S$ isomorphic to $V_1 \otimes_E R$. Arguing as in the proof of Proposition \ref{prop: properties of local moduli}(3), we find that the fibers of $\pi_{\GM}$ are locally given by
    \[ G(\Q_p)/K \times S^a\rightarrow S,\]
    so that $\pi_{\GM}$ is étale. Consequently, we obtain a tower $(\Ms_{\Gs_0,\iota_0,(\lambda_0), K})_{ K \sub G(\Q_p)}$ of smooth rigid spaces over $\breve{E}$. 
    % The infinite level moduli
    % \[ \Ms_{\Gs_0,\iota_0,(\lambda_0), \infty} = \varprojlim_{K \sub G(\Q_p)} \Ms_{\Gs_0,\iota_0,(\lambda_0), K}\]
    % sends a good adic space $S$ to the set of quasi-isogeny classes of tuples $(\Gs,\iota, (\lambda),\beta,\gamma)$, where this time, $\gamma$ is an isomorphism $V \xrightarrow{\cong} V_p\Gs$ compatible with the $B$-actions. In the PEL case, $\gamma$ is required to identify the pairings up to a scalar in $\underline{\Q_p^{\times}}$ after an identification $\Q_p \cong \Q_p(1)$ has been fixed. In particular, any point on the infinite level space $\Ms_{\Gs_0,\iota_0,(\lambda_0), \infty}$ lives over $\Q_p^{\cyc}$.

    We can now state our theorem. We have a local Shimura data $(G,[\mu],b)$, to which Scholze--Weinstein associate a local Shimura variety $(\Sht_{G,[\mu],b,K})_{K \sub G(\Q_p)}$ \cite[Def. 24.1.3]{SW20}. We then have the following result.
    
    \begin{thm}\label{main thm: EL and PEL moduli}
    There is a canonical isomorphism of smooth rigid spaces
    \[ \Ms_{\Gs_0,\iota_0,(\lambda_0),K} \cong \Sht_{G,[\mu],b,K},\]
    compatible with $G(\Q_p)$-actions and with the Grothendieck--Messing period maps.
\end{thm}
\begin{proof}
    It is enough to produce an isomorphism between the spaces at infinite level 
    \[ \Ms_{\Gs_0,\iota_0,(\lambda_0),\infty} \cong \Sht_{G,\mu,b,\infty}\]
that identifies the $G(\Q_p)$-actions. The embedding $i\colon G \hookrightarrow \GL_{\Q_p}(V)\cong \GL_n$ induces a cocharacter $i\mu\colon \G_m \rightarrow \GL_n$ and an element $i(b)\in B(\GL_n)$, that we simply name $\mu$ and $b$, and an embedding 
\[ \Sht_{G,\mu,b,\infty} \hookrightarrow \Sht_{\GL_n,\mu,b,\infty}. \]
Similarly, on the left-hand side, we have an embedding
\[ \Ms_{\Gs_0,\iota_0,(\lambda_0),\infty} \hookrightarrow \Ms_{\Gs_0,\infty},\]
obtained by forgetting the extra structure. It is enough to show that the two infinite level moduli spaces are equal as subsheaves of 
\[ \Ms_{\Gs_0,\infty} \cong \Sht_{\GL_n,\mu,b,\infty},\]
where the above is the isomorphism from Theorem \ref{main thm: moduli of an pdiv vs local SV}. Let 
\[ (\Gs,\iota,(\lambda),\beta,\gamma) \in \Ms_{\Gs_0,\iota_0,(\lambda_0),\infty}(S).\]
Here, $\gamma$ is an honest isomorphism $V \xrightarrow{\cong} V_p\Gs$ compatible with the $B$-actions. In the PEL case, $\gamma$ is required to identify the pairings up to a scalar in $\underline{\Q_p^{\times}}$ after an identification $\Q_p \cong \Q_p(1)$ has been fixed. The corresponding point of $\Sht_{\GL_n,\mu,b,\infty}(S)$ is given by the modification
\[ \alpha\colon V\otimes_{\Q_p}\Os_{X_S}\overset{\gamma }{\cong}V_p\Gs \otimes_{\Q_p} \Os_{X_S} \dashrightarrow \Es(\Gs) \overset{\beta}{\cong} \Es^b.\]
The $B$-action on $\Gs$ and the compatibility conditions on $\beta$ and $\gamma$ allow us to upgrade $\alpha$ to a modification of $B\otimes_{\Q_p} \Os_{X_S}$-modules. In the PEL case, let
\[e\colon \Es^b \times \Es^b\rightarrow \Os_{X_S}(1)\]
be the perfect pairing from Lemma \ref{lemma: pairings on vb induced by polarization} obtained by identifying $\Es^b = \Es(\Gs_0)$. Then $\alpha$ upgrades to a symplectic similitude of $B\otimes_{\Q_p}\Os_{X_S}$-bundles. This shows that $\alpha$ defines an element of $\Sht_{G,\mu,b,\infty}(S)$ and yields one inclusion of subsheaves. The reverse inclusion is showed similarly, using Theorem \ref{thm: equivalence for dualizable groups on perfd} to reconstruct the $B$-actions and quasi-polarizations on $\Gs$ from the extra structure on $V_p\Gs$ and $\Es(\Gs)$. This concludes the proof.
\end{proof}
\section{The moduli stack of analytic $p$-divisible groups}\label{section: moduli stack of an pdiv groups}
In this last section, we define small $v$-stacks of dualizable analytic $p$-divisible groups. We then show that they naturally are the target of the Hodge--Tate period maps on global PEL Shimura varieties.

\subsection{Definitions}
\begin{definition}
\begin{enumerate}
    \item We let $\Ns$ denote the prestack on perfectoid spaces over $\Q_p$ defined by
    \[ S \in \Perf_{\Q_p} \mapsto \{ \, \text{dualizable analytic }p\text{-divisible groups }\Gs \text{ over }S \, \}.\]
    \item For integers $0 \leq d \leq n$, we let $\Ns_{n,d}$ be the subfunctor of $\Ns$ consisting of those groups $\Gs$ of dimension $d$ and height $n$.
    \item For an element $b \in B(\GL_n, \mu_d^{-1})$, we let $\Ns_{n,d}^b \sub \Ns_{n,d}$ be the subfunctor given by those groups $\Gs$ such that, for any geometric point $\cj{s}\colon \Spa(C,C^+) \rightarrow S$, there exists an isomorphism $\Es(\Gs_{\cj{s}}) \cong \Es^b$ of vector bundle on the Fargues--Fontaine curve $X_{(C,C^+)}$.
\end{enumerate}
\end{definition}

\begin{proposition}\label{Prop: properties of moduli of an pdiv gps}
    \begin{enumerate}
        \item The prestacks $\Ns$, $\Ns_{d,n}$ and $\Ns_{d,n}^b$ are small $v$-stack. We have an open and closed decomposition
        \begin{align} 
        \Ns = \coprod_{0 \leq d \leq n} \Ns_{n,d},
        \end{align}
        and locally closed decompositions
        \begin{align}\label{eq: Newton decomposition}
        \Ns_{n,d} = \coprod_{b \in B(\GL_n,\mu_d^{-1})} \Ns_{n,d}^b.
        \end{align}
        \item We have a canonical isomorphism
        \begin{align}\label{eq: n,d moduli as quotient of grassmanian}
        \Ns_{n,d} = \Big[ \Gr(n,d)^{\diamondsuit}/\GL_n(\Z_p) \Big], 
        \end{align}
        where $\Gr(n,d)$ denotes the Grassmanian of $d$-dimensional subspaces of the $n$-dimensional affine space, considered as an adic space over $\Q_p$. In particular, the stack $\Ns_{n,d}$ (resp. $\Ns$) is proper (resp. partially proper) over $\Spd(\Q_p)$.
        \item The canonical map of $v$-stacks over $\Q_p$
        \begin{align*}
            \Ns_{n,d}& \rightarrow \Bun_{\GL_n},\\
            \Gs & \mapsto \Es(\Gs)
        \end{align*}
        comes from the Beauville--Laszlo morphism \cite[Prop. III.3.1]{fargues2024geometrization}
    \[ \BL\colon \Gr(n,d) \rightarrow \Bun_{\GL_n}\]
    under the isomorphism (\ref{eq: n,d moduli as quotient of grassmanian}). In particular, the locally closed decomposition (\ref{eq: Newton decomposition}) is induced from the Newton decomposition on $\Gr(n,d)$ (cf. \cite[Thm. 1.11]{ScholzeCariani2017}).
    \end{enumerate}
\end{proposition}
\begin{proof}
    By Remark \ref{remark: dualizability does not descend well}, $\Ns$ and $\Ns_{n,d}$ are $v$-stacks on $\Perf_{\Q_p}$, and it remains to see that they are small. Given any $\omega_1$-cofiltered inverse limit $S= \varprojlim_i S_i$ of affinoid perfectoid spaces $S_i$, we observe that we have
\[\{ \,\text{analytic }p\text{-divisible groups } \Gs \rightarrow S \, \} = \tworlim_i \{ \,\text{analytic }p\text{-divisible groups } \Gs \rightarrow S_i \, \}.\]
Indeed, by Theorem \ref{thm: extending Fargues' equivalence of categories}, it is enough to establish the corresponding statements for the stacks of étale vector bundles and $\Z_p$-local systems respectively. The case of étale vector bundles follows from the fact that by \cite[Prop. 11.23]{scholze2022etale}
\[ S_{\et, \qcqs} = \tworlim S_{i,\et,\qcqs}.\]
The case of local systems follows similarly from this equivalence of sites, by writing a $\Z_p$-local system $\Lb = \varprojlim_m \Lb/p^m$ as an inverse limit of étale $\Z/p^m\Z$-local systems and using that the inverse system $(S_i)_i$ is $\omega_1$-cofiltered. Moreover, if an analytic $p$-divisible group $\Gs_i \rightarrow S_i$ becomes dualizable after base-change to $S$, then it becomes so over $S_j$ for $j \gg i$. Indeed, over perfectoid spaces, dualizability is equivalent to $f_{\Gs}$ being a locally direct summand, whence it is clear. By arguing as in \cite[Prop. III.1.3]{fargues2024geometrization}, we conclude that $\Ns$ and $\Ns_{n,d}$ are small $v$-stacks.

By the semicontinuity of the Newton polygon \cite[Thm. 22.2.1]{SW20}, the $\Ns_{n,d}^b$ are locally closed in $\Ns_{n,d}$. The second point readily follows Theorem \ref{thm: extending Fargues' equivalence of categories} and the definition of dualizability. Finally, given a dualizable analytic $p$-divisible groups $\Gs \rightarrow S$, it follows from the explicit formula (\ref{eq: defining diagram for E(G)}) that the map $T_p\Gs \otimes \Os_{X_S} \dashrightarrow \Es(\Gs)$ is the unique modification of vector bundles on the relative Fargues--Fontaine curve spreading out the inclusion $T_p\Gs \otimes \B_{\dR}^+ \rightarrow \Xi$, where $\Xi$ is the $\B_{\dR}^+$-local system corresponding to the Hodge--Tate filtration on $T_p\Gs(-1) \otimes \Os_v$ under the Bialynicki--Birula equivalence (\ref{eq: from modif to lattices}). Since by construction \cite[Def. 3.5.6]{ScholzeCariani2017} the Newton stratification on $\Gr(n,d)$ is defined to be the pullback of the stratification $\Bun_{\GL_n} = \coprod_{ b \in B(\GL_n)} \{\Es^b\}$, the third point follows.
\end{proof}

\begin{example}
    Let us make Proposition \ref{Prop: properties of moduli of an pdiv gps} more explicit by considering the case $d=1$ and $b\in B(\GL_n,\mu_1^{-1})$ the basic element, corresponding to the isocrystal of a connected $p$-divisible groups over $\cj{\F}_p$ of height $n$. Let $\Gs \in \Ns_{n,1}(S)$, that is $\Gs$ is a dualizable analytic $p$-divisible group of dimension $1$ and height $n$. Then one can show that $\Gs$ is in $\Ns_{n,1}^b(S)$ if and only if it is connected, in the sense that, after pullback along any point $\Spa(C,C^+) \rightarrow S$, the group $\Gs$ is isomorphic to an open unit ball. On the other hand, the open Newton strata corresponding to $b$ on the Grassmanian is Drinfeld's upper half space
    \[ \Omega_{\Q_p}^{n} = \Pro_{\Q_p}^{n-1}\backslash \bigcup_{H} H, \]
    obtained as the complement in $\Pro_{\Q_p}^{n-1}$ of the union of all $\Q_p$-rational hyperplanes. Then Proposition \ref{Prop: properties of moduli of an pdiv gps} yields an isomorphism of small $v$-stacks
    \[ \Ns_{n,1}^b = \Big[ \Omega_{\Q_p}^{n}/\GL_n(\Z_p)\Big].\] 
    % In that case, there is exactly two elements $b,b'\in B(\GL_n,\mu_1^{-1})$, corresponding to isocrystal to $p$-divisible groups $G, G'$ over $\cj{\F_p}$ of height $n$, where $G$ is connected and $G' = \Q_p/\Z_p \oplus H$ for a connected $p$-divisible group $H$ of height $n-1$. For example, if $n=2$, this is the dichotomy of supersingular elliptic curve and 
\end{example}

We may relate our stacks with the generic fiber of the stack of $p$-divisible groups. For integers $0 \leq d \leq n$, we let $\BT$ (resp. $\BT_{n,d}$) denote the algebraic stack of $p$-divisible groups (resp. of dimension $d$ and height $n$). We let $\widehat{\BT}$, resp. $\widehat{\BT}_{n,d}$ denote their $p$-adic completion, obtained by restricting their domains to the category $\Nil_p$ of discrete rings $R$ such that $p$ is nilpotent in $R$. We view them as formal stacks over $\Spf(\Z_p)$.

\begin{definition}[\cite{Fargues2024}]
    The overconvergent, diamantine generic fiber $\widehat{\BT}_{\eta}^{\diamondsuit, \dagger}$ is the $v$-stack on $\Perf_{\Q_p}$ obtained as the stackification of
    \[ (R,R^+) \mapsto \widehat{\BT}(R^{\circ}).\]
    We define $\widehat{\BT}_{n,d,\eta}^{\diamondsuit, \dagger}$ in an analogous manner.
\end{definition}

The following is then an immediate consequence of Proposition \ref{prop: dualizable groups locally have good reduction on circ}.
\begin{proposition}\label{prop: anBT vs generic fiber of BT}
    The adic generic fiber functor defines isomorphisms of $v$-stacks
    \begin{align*} 
    \widehat{\BT}_{\eta}^{\diamondsuit, \dagger} \cong \Ns, \quad \widehat{\BT}_{n,d,\eta}^{\diamondsuit, \dagger} \cong \Ns_{n,d}.
    \end{align*}
\end{proposition}
\begin{remark}
    Combined with Proposition \ref{Prop: properties of moduli of an pdiv gps}(2), this yields an isomorphism of $v$-stacks
    \begin{align} 
    \widehat{\BT}_{n,d,\eta}^{\diamondsuit, \dagger} \cong \Big[ \Gr(n,d)^{\diamondsuit}/\GL_n(\Z_p) \Big].
    \end{align}
    This was already proven by Fargues \cite[Prop. 3.1, Rem. 3.2]{Fargues2024}.
\end{remark} 

We now consider variants with additional structures. We fix a local EL data $(B,V,[\mu])$ or PEL data $(B,(\cdot)^{*},V,(\cdot,\cdot),[\mu])$ as in Subsection \ref{The EL and PEL case}. In particular, we have the corresponding algebraic group $G$ over $\Q_p$. We continue to denote by $E$ the field of definition of the conjugacy class $[\mu]$, a finite extension of $\Q_p$. Let $K\sub G(\Q_p)$ be a compact open subgroup.
\begin{definition}
\begin{enumerate}
    \item We define a pre-stack $\Ns_{G,\mu,K}$ on $\Perf_{E}$, sending a perfectoid space $S$ to the groupoid of analytic $p$-divisible groups with extra structure $(\Gs,\iota,(\lambda),\cj{\gamma})$ over $S$. These are the same as in Definition \ref{def: moduli of an pdiv gps in EL and PEL case} except that we do not impose the isomorphism $\beta$. Concretely:
    \begin{itemize}
        \item $\Gs$ is a dualizable analytic $p$-divisible group over $S$,
        \item $\iota\colon B \rightarrow \End_S^0(\Gs)$ is a $\Q_p$-algebra homomorphism, satisfying the Kottwitz condition,         
         \item (in the PEL case) $\lambda \colon \Gs \rightarrow \Gs^D$ is a quasi-polarization, such that for $d\in B$, we have $\lambda^{-1} \circ \iota(d)^D \circ \lambda = \iota(d^*)$, and
         \item $\cj{\gamma}$ is a $K$-orbit of isomorphisms of $B$-local systems $V \xrightarrow{ \cong} V_p\Gs$ which are required to be symplectic similitudes in the PEL case.
    \end{itemize} 
    The isomorphisms $(\Gs,\iota,(\lambda),\cj{\gamma}) \rightarrow (\Gs',\iota',(\lambda'),\cj{\gamma}')$ in this groupoid are $B$-linear quasi-isogenies $\varphi \colon \Gs \rightarrow \Gs'$ such that $V_p\varphi \circ \cj{\gamma} = \cj{\gamma}'$. In the PEL case, we ask that $\varphi^D \circ \lambda' \circ \varphi = c\lambda$, for some $c\in \underline{\Q_p^{\times}}(S)$. 
    \item For $b\in B(G,\mu^{-1})$, we define $\Ns_{G,\mu,K}^b\sub \Ns_{G,\mu,K}$ to be the subfunctor determined by those tuples $(\Gs,\iota, (\lambda),\cj{\gamma})$ such that, pointwise on $S$, there exists an isomorphism of vector bundles $\Es(\Gs) \cong \Es^b$ on the Fargues--Fontaine curve, respecting the extra structure.
\end{enumerate}
\end{definition}
Again, the functors $\Ns_{G,\mu,K}$ and $\Ns_{G,\mu,K}^b$ are small $v$-stacks and there is a locally closed stratification
\begin{align} \Ns_{G,\mu,K} = \coprod_{b \in B(G,\mu^{-1})} \Ns_{G,\mu,K}^b.\end{align}
By Theorem \ref{main thm: EL and PEL moduli}, there is a surjective map of $v$-stacks $\Ms_{G,\mu,b,K} \rightarrow \Ns_{G,\mu,K}^b$ from the corresponding local Shimura variety, for each $b$. Here we use that, by \cite[Lemma 12.5]{scholze2022etale}, the surjectivity can be checked pointwise. This surjection yields an isomorphism of small $v$-stacks
\begin{align} 
\Ns_{G,\mu,K}^b = \Big[\Ms_{G,\mu,b,K}/\widetilde{G}_b\Big],
\end{align}
where $\widetilde{G}_b$ is the $v$-sheaf of groups
\[ S \in \Perf_{\Q_p} \mapsto \Aut_{X_S}(\Es^b). \]
If $b$ is basic, the corresponding stratum $\Ns_{G,\mu,K}^b$ is open in $\Ns_{G,\mu,K}$ and $\widetilde{G}_b= \underline{G_b(\Q_p)}$ is a locally profinite sheaf, where $G_b$ is the corresponding inner form of $G$.

On the other hand, consider the flag variety $\Fl_{\HT}$ over $\Spa(E)$, whose $S$-points parametrize $B\otimes_{\Q_p}\Os_S$-equivariant (isotropic) quotients 
\[  V\otimes_{\Q_p}\Os_S \twoheadrightarrow \omega \]
such that, locally on $S$, $\omega$ is isomorphic to $V_0\otimes_E \Os_S$. By Theorem \ref{thm: extending Fargues' equivalence of categories}, such a quotient uniquely determines a dualizable analytic $p$-divisible group $\Gs$ with extra structure, together with an identification $V_p\Gs \cong V$, such that the above quotient map identifies with the surjection in the sequence
% https://q.uiver.app/#q=WzAsNSxbMCwwLCIwIl0sWzEsMCwiXFxMaWUoXFxHcykoMSkiXSxbMiwwLCJWX3BcXEdzXFxvdGltZXNfe1xcUV9wfSBcXE9zX3tTfSJdLFszLDAsIlxcb21lZ2Ffe1xcR3NeRH0iXSxbNCwwLCIwLiJdLFswLDFdLFsxLDJdLFsyLDNdLFszLDRdXQ==
\[\begin{tikzcd}
	0 & {\Lie(\Gs)(1)} & {V_p\Gs\otimes_{\Q_p} \Os_{S}} & {\omega_{\Gs^D}} & {0.}
	\arrow[from=1-1, to=1-2]
	\arrow[from=1-2, to=1-3]
	\arrow[from=1-3, to=1-4]
	\arrow[from=1-4, to=1-5]
\end{tikzcd}\]
Therefore, we obtain a map $\Fl_{\HT} \rightarrow \Ns_{G,\mu,K},$ which induces an isomorphism
\begin{align}\label{eq: stack of an pdiv as quotient of flag var} \Ns_{G,\mu,K} = \Big[ \Fl_{\HT}/K\Big].\end{align}
Under this isomorphism, the stratification on $\Ns_{G,\mu,K}$ is again induced by the Newton stratification on $\Fl_{\HT}$, cf. \cite[Thm. 1.11]{ScholzeCariani2017}.

\subsection{The Hodge--Tate period map}
We now present our reinterpretation of the Hodge--Tate period map on global Shimura varieties of PEL type. For this, we shift our notations. Let $(B,*,V,(\cdot,\cdot),h)$ be a global PEL datum \cite[§5]{kottwitz1992}. This consists in 
\begin{itemize}
    \item a semisimple $\Q$-algebra $B$,
    \item a positive anti-involution $*$ on $B$,
    \item a finite left $B$-module $V$,
    \item a non-degenerate alternating pairing $(\cdot,\cdot)$ on $V$ such that $(dv,v') = (v,d^*v')$ for $d\in B$, $v,v'\in V$. 
    \end{itemize}
    We let $G$ be the algebraic group over $\Q$ of symplectic similitudes of $V$
    \begin{itemize}
    \item a group homomorphism $h\colon \Sph \rightarrow G_{\R}$, where $\Sph$ is Deligne's torus, such that $h(z)^* = h(\cj{z})$, the symmetric real-valued form $(v,h(i)v')$ on $V_{\R}$ is positive-definite and the induced Hodge structure on $V$ is of type $(1,0), \, (0,1)$.
\end{itemize}
There is a corresponding Shimura datum $(G,X)$, a number field $E$, the reflex field, and a Shimura variety $(S_K)_{K \sub G(\A_f)}$, where $K$ ranges over compact open subgroups of the adelic points $G(\A_f)$ of $G$. It has the following moduli interpretation: For a scheme $T$ over $E$, $S_K(T)$ parametrizes the tuples $(\As,\iota,\lambda,\gamma)$, where
\begin{itemize}
    \item $\As \rightarrow T$ is an abelian scheme of dimension $g = \tfrac{1}{2}\dim_{\Q} V$,
    \item $\iota\colon B \rightarrow \End^0(\As) = \End(\As) \otimes \Q$ is a ring homomorphism, satisfying the Kottwitz condition \cite[§5]{kottwitz1992},
    \item $\lambda\colon \As \rightarrow \breve{\As}$ is a quasi-polarization, compatible with the involutions, and
    \item $\gamma$ is a $K$-orbit of symplectic similitude $V(\A_f) \xrightarrow{\cong} V\As$, where $V\As$ is the adelic Tate module of $\As$.
\end{itemize}
The moduli problem $S_K$ is representable by a Deligne--Mumford stack over $\Spec(E)$. If $K$ is small enough, it is even representable by a smooth quasi-projective $E$-scheme.

We now transport the situation back to $p$-adic geometry. We fix a prime ideal $\ip \sub \Os_E$ lying over $p$ and we let $E_{\ip}$ denote the completion of $E$ at $\ip$. We let 
\[ \Ss_K = (S_K \otimes_E E_{\ip})^{\an}\]
denote the analytification of the Shimura variety, a smooth rigid space over $E_{\ip}$. Its functor of points is the sheafification for the analytic topology of the presheaf
\[ \Spa(R,R^+) \in \Adic_{E_{\ip}} \mapsto S_K(R).\]
We will restrict our attention to level subgroups of the form $K = K_p K^p$, for subgroups $K_p \sub G(\Q_p)$ and $K^p\sub G(\A_f^p)$ with $K^p$ fixed. In that case, we can consider the Shimura variety at infinite level
\[ \Ss_{K^p} = \varprojlim_{K_p \sub G(\Q_p)} \Ss_{K_pK^p},\]
where this inverse limit is understood as a diamond. It is shown in \cite{Scholze2015torsion} that $\Ss_{K^p}$ is representable by a perfectoid space.

The global PEL datum and the choice of $\ip$ induce a local PEL datum $(B_p,(\cdot)^{*},V_p,(\cdot,\cdot),\mu)$. In particular, we obtain the algebraic group $G_p = G\otimes_{\Q} \Q_p$ of symplectic similitudes of $V_p$. We may consider again the associated flag variety $\Fl_{\HT}$ over $E_{\ip}$, whose $S$-rational points parametrize certain quotients of $V_p \otimes_{\Q_p} \Os_S$. In this situation, we have the Hodge--Tate period map
\begin{align} \pi_{\HT}\colon \Ss_{K^p} \rightarrow \Fl_{\HT},\end{align}
whose definition was first given in \cite{Scholze2015torsion} for the Siegel case and extended in  \cite{ScholzeCariani2017} to Shimura varieties of Hodge type. This is a morphism of diamonds which is defined as follows: Let $\As \rightarrow S_{K_pK^p}$ be the universal abelian scheme and denote by the same symbol its analytification over $\Ss_{K_pK^p}$ and its pullback to the infinite level variety $\Ss_{K^p}$. Then the map $\pi_{\HT}$ is the point in $\Fl_{\HT}(\Ss_{K^p})$ corresponding to the covariant Hodge--Tate sequence of $\As$
% https://q.uiver.app/#q=WzAsNSxbMCwwLCIwIl0sWzEsMCwiXFxMaWUoXFxBcykoMSkiXSxbMiwwLCJWX3BcXEFzXFxvdGltZXNfe1xcUV9wfSBcXE9zX3tcXFNzX3tLXnB9fSJdLFszLDAsIlxcb21lZ2Ffe1xcYnJldmV7QX19Il0sWzQsMCwiMC4iXSxbMCwxXSxbMSwyXSxbMiwzXSxbMyw0XV0=
\[\begin{tikzcd}
	0 & {\Lie(\As)(1)} & {V_p\As\otimes_{\Q_p} \Os_{\Ss_{K^p}}} & {\omega_{\breve{A}}} & {0}
	\arrow[from=1-1, to=1-2]
	\arrow[from=1-2, to=1-3]
	\arrow[from=1-3, to=1-4]
	\arrow[from=1-4, to=1-5]
\end{tikzcd}\]
under the identification $V_p\As \overset{\gamma}{\cong} V_p$ that arises on the infinite level Shimura variety $\Ss_{K^p}$.

The map $\pi_{\HT}$ is $G(\Q_p)$-equivariant, where $G(\Q_p)$ acts on the flag variety through its natural action on $V_p \otimes_{\Q_p} E_{\ip}$. By descent, we obtain a morphism of small $v$-stacks
\begin{align}\label{eq: descended HT period map} \pi_{\HT,K_p}\colon \Ss_{K_pK^p} \rightarrow \Big[ \Fl_{\HT}/K_p \Big].  \end{align}

Our result is then the following.
 \begin{thm}\label{thm: HT map reinterpreted}
     Under the isomorphism (\ref{eq: stack of an pdiv as quotient of flag var}), the Hodge--Tate period map $\pi_{\HT,K_p}$ sends an abelian variety $\As$ with extra structure to the topologically $p$-torsion subgroup of its analytification $\As^{\an}\langle p^{\infty} \rangle \sub \As^{\an}$.
 \end{thm}
 \begin{proof}
     Let $\As \rightarrow \Ss_{K_pK^p}$ denote the analytification of the universal abelian variety. By Theorem \ref{thm: two Hodge-Tate maps agree}, the topologically $p$-torsion subgroup $\As\langle p^{\infty} \rangle$ corresponds under the equivalence of Theorem \ref{thm: extending Fargues' equivalence of categories} to the tuple
     \[ (T_p\As,\Lie(\As), \Lie(\As) \hookrightarrow V_p\As(-1) \otimes_{\Q_p} \Os_{S_v}), \]
     where the above map is the inclusion in the covariant Hodge--Tate sequence of $\As$. Moreover, let $\lambda\colon \As \rightarrow \check{\As}$ be the universal quasi-polarization. Then the resulting symplectic pairing on $V_p\As$ preserves the Hodge--Tate sequence, by functoriality of the latter. Hence, the resulting map $\lambda\langle p^{\infty} \rangle \colon \As\langle p^{\infty} \rangle \rightarrow \check{\As}\langle p^{\infty} \rangle \cong \As\langle p^{\infty} \rangle^D$, where we use Corollary \ref{cor: HT sequences for abeloids}(3) for the last equality, is a quasi-polarization. This concludes the proof. 
 \end{proof}

  \medskip
 
\printbibliography[
%heading=bibintoc, %makes an entry in bibliography
title={References}
]
\end{document}